\documentclass[12pt,a4paper,reqno]{amsbook}%
\usepackage{amsbsy,amsopn,amscd,amstext,amsgen,amsthm,amsmath,amsfonts,amssymb,amsxtra,appendix,bookmark,dsfont,latexsym,color,epsfig,xspace,mathrsfs,dsfont,srcltx}

\makeindex

\makeatletter
\usepackage{hyperref}
\usepackage{delarray,a4,color}
\usepackage{euscript}
\usepackage[latin1]{inputenc}

\numberwithin{equation}{chapter}
\numberwithin{section}{chapter}

\makeatletter
\def\@tocline#1#2#3#4#5#6#7{\relax
  \ifnum #1>\c@tocdepth 
  \else
    \par \addpenalty\@secpenalty\addvspace{#2}%
    \begingroup \hyphenpenalty\@M
    \@ifempty{#4}{%
      \@tempdima\csname r@tocindent\number#1\endcsname\relax
    }{%
      \@tempdima#4\relax
    }%
    \parindent\z@ \leftskip#3\relax \advance\leftskip\@tempdima\relax
    \rightskip\@pnumwidth plus4em \parfillskip-\@pnumwidth
    #5\leavevmode\hskip-\@tempdima
      \ifcase #1
       \or\or \hskip 1em \or \hskip 2em \else \hskip 3em \fi%
      #6\nobreak\relax
      \dotfill
      \hbox to\@pnumwidth{\@tocpagenum{#7}}
    \par
    \nobreak
    \endgroup
  \fi}
\makeatother

\newtheorem{theorem}{Theorem}[chapter]

\newtheorem{corollary}[theorem]{Corollary}

\newenvironment{argument}{\emph{Argument.}}{\hfill \qed \medskip}
\newtheorem{lemma}[theorem]{Lemma}

\newtheorem{assumption}[theorem]{Assumption}
\newtheorem{definition}[theorem]{Definition}

\theoremstyle{remark}

\numberwithin{equation}{chapter}

\usepackage{color}

\newcommand{\norm}[1]{\left\lVert #1 \right\rVert}

\newcommand{\1}{{\ensuremath {\mathds 1} }}
\newcommand{\Z}{\mathbb{Z}}

\newcommand{\R}{\mathbb{R}}

\newcommand{\N}{\mathbb{N}}

\newcommand{\C}{\mathbb{C}}

\newcommand{\E}{\mathcal{E}}
\newcommand{\cE}{\mathcal{E}}

\newcommand{\cG}{\mathcal{G}}
\newcommand{\cK}{\mathcal{K}}
\newcommand{\cW}{\mathcal{W}}
\newcommand{\cF}{\mathcal{F}}

\newcommand{\cT}{\mathcal{T}}
\newcommand{\cA}{\mathcal{A}}
\newcommand{\cB}{\mathcal{B}}

\newcommand{\bH}{\mathbb{H}}
\newcommand{\bG}{\mathbb{G}}
\newcommand{\bHt}{\widetilde{\mathbb{H}}}

\newcommand{\cP}{\mathcal{P}}
\newcommand{\cS}{\mathcal{S}}
\newcommand{\cI}{\mathcal{I}}

\newcommand{\cM}{\mathcal{M}}

\newcommand{\gS}{\mathfrak{S}}
\newcommand{\gH}{\mathfrak{H}}
\newcommand{\cU}{\mathcal{U}}

\newcommand{\dd}{\partial}

\newcommand{\sspan}{\mathrm{span}}

\newcommand{\wto}{\rightharpoonup}
\newcommand{\eps}{\varepsilon}

\def\XXint#1#2#3{{\setbox0=\hbox{$#1{#2#3}{\int}$}
     \vcenter{\hbox{$#2#3$}}\kern-.5\wd0}}

\newcommand{\tr}{\mathrm{Tr}}

\newcommand{\bA}{\mathbf{A}}
\newcommand{\bbA}{\mathbb{A}}
\newcommand{\bx}{\mathbf{x}}
\newcommand{\bt}{\mathbf{t}}
\newcommand{\bk}{\mathbf{k}}
\newcommand{\bp}{\mathbf{p}}
\newcommand{\bq}{\mathbf{q}}
\newcommand{\bX}{\mathbf{X}}
\newcommand{\by}{\mathbf{y}}
\newcommand{\bY}{\mathbf{Y}}
\newcommand{\bz}{\mathbf{z}}
\newcommand{\bZ}{\mathbf{Z}}
\newcommand{\bP}{\mathbf{P}}
\newcommand{\gF}{\mathfrak{F}}
\newcommand{\cN}{\mathcal{N}}
\newcommand{\mubf}{\boldsymbol{\mu}}

\newcommand{\ada}{a ^{\dagger}}

\newcommand{\im}{\mathrm{i}}

\newcommand{\sym}{\mathrm{sym}}
\newcommand{\asym}{\mathrm{asym}}

\newcommand{\cEh}{\mathcal{E}^{\rm H}}
\newcommand{\cEnls}{\mathcal{E}^{\rm NLS}}
\newcommand{\cEGP}{\mathcal{E}^{\rm GP}}
\newcommand{\cEMF}{\mathcal{E}^{\rm MF}}
\newcommand{\cEmix}{\mathcal{E}^{\rm mix}}

\newcommand{\Eh}{E^{\rm H}}
\newcommand{\Enls}{E^{\rm NLS}}
\newcommand{\EGP}{E^{\rm GP}}
\newcommand{\EMF}{E^{\rm MF}}
\newcommand{\Emix}{E^{\rm mix}}

\newcommand{\uh}{u^{\rm H}}
\newcommand{\unls}{u^{\rm NLS}}
\newcommand{\uGP}{u^{\rm GP}}
\newcommand{\uMF}{u^{\rm MF}}

\newcommand{\Mh}{\cM^{\rm H}}
\newcommand{\Mnls}{\cM^{\rm NLS}}
\newcommand{\MGP}{\cM^{\rm GP}}
\newcommand{\MMF}{\cM^{\rm MF}}

\newcommand{\htilde}{\widetilde{h}}
\newcommand{\Htilde}{\widetilde{H}}
\newcommand{\Psit}{\widetilde{\Psi}}

\newcommand{\eB}{E^{\rm Bog}}
\newcommand{\HB}{\mathbb{H}^{\rm Bog}}

\title[Scaling limits of bosonic ground states]{Scaling limits of bosonic ground states, from many-body to nonlinear Schr\"odinger}

\author[Nicolas Rougerie]{Nicolas Rougerie\\ 
\small Unit\'e de Math\'ematiques Pures et Appliqu\'ees, Ecole Normale Sup\'erieure de Lyon \& CNRS.}
\address{\small Ecole Normale Sup\'erieure de Lyon \& CNRS, UMPA, Lyon, France.}
\email{nicolas.rougerie@ens-lyon.fr}

\date{November 2020}

\begin{document}

%


\maketitle

\begin{abstract}
How and why could an interacting system of many particles be described as if all particles were independent and identically distributed ? This question is at least as old as statistical mechanics itself. Its quantum version has been rejuvenated by the birth of cold atoms physics. In particular the experimental creation of Bose-Einstein condensates leads to the following variant: why and how can a large assembly of very cold interacting bosons (quantum particles deprived of the Pauli exclusion principle) all populate the same quantum state ? 

In this text I review the various mathematical techniques allowing to prove that the lowest energy state of a bosonic system forms, in a reasonable macroscopic limit of large particle number, a Bose-Einstein condensate. This means that indeed in the relevant limit all particles approximately behave as if independent and identically distributed, according to a law determined by minimizing a non-linear Schr\"odinger energy functional. This is a particular instance of the justification of the mean-field approximation in statistical mechanics, starting from the basic many-body Schr\"odinger Hamiltonian. 
\end{abstract}

\setcounter{tocdepth}{2}

\tableofcontents

\noindent \textbf{Acknowledgments.} A collective thank you to all the colleagues who contributed to my understanding of the field through their writings and talks, and especially through private conversations. Special thanks to Laure Saint-Raymond, whose suggestion that I should write such a review provided the necessary motivation. Financial support was provided by the European Research Council (ERC) under the European Union's Horizon 2020 Research and Innovation Programme (Grant agreement CORFRONMAT No 758620).

The present version of the text has benefited from critical readings and/or remarks by Niels Benedikter, Christian Brennecke, Serena Cenatiempo, Christian Hainzl, Elliott Lieb, Phan Th\`anh Nam,  Sergio Simonella, Laure Saint-Raymond, Jakob Yngvason,  and three anonymous referees.

\chapter{Aims and scope}

This introductory chapter is perhaps long. Readers already acquainted with quantum statistical mechanics will probably want to skip to Section~\ref{sec:sca lim} after glancing at Section~\ref{sec:intro intro}, and very briefly at Sections~\ref{sec:basic stat} and~\ref{sec:second quant} to get familiar with the notation\footnote{I have tried to make it so that the glance need only be very brief. Perhaps you will instead prefer to jump to Section~\ref{sec:sca lim} immediately and go back as needed for notational issues.} I use. The first three sections are intended as a very basic introduction to newcomers in the field. 

\section{Introduction}\label{sec:intro intro}

We start with the basic mathematical description of $N$ non-relativistic $d-$dimensional quantum particles in a scalar (electric-like) potential $V:\R^d\mapsto \R$ and a gauge (magnetic-like) vector potential $\bA:\R^d \mapsto \R ^d$, interacting via an even pair-interaction potential $w:\R^d \mapsto \R$. It is done via the action of the \index{Many-body Schr\"odinger operator}{$N$-body Schr\"odinger operator}
\begin{equation}\label{eq:intro Schro op}
H_N := \sum_{j=1} ^N \left(-\im \nabla_{\bx_j} + \bA (\bx_j) \right) ^2 + V (\bx_j) + \sum_{1\leq i < j \leq N} w (\bx_i - \bx_j)
\end{equation}
on the space of wave-functions $L^2 (\R^{dN},\C)$. The coordinates of the $N$-particles in the Euclidean space are the vectors $\bx_1,\ldots,\bx_N$ and the units are such that $\hbar = 2m = 1$ (reduced Planck constant, twice the mass, set equal to $1$). This means that if $\bA$ represents a real magnetic field coupling to the electric charge of the particles, it is proportional to the square root of the fine structure constant $\alpha = e^2 / (\hbar c)$ ($e$= electric charge, $c=$ speed of light). We rarely consider actual Coulomb interactions between our particles, but in these units their strength would be proportional to $\alpha$. Spin is ignored as irrelevant for most of the topics discussed below. I refer to~\cite[Section~2.17]{LieSei-09} for a discussion of units. In most of the text our applications will be to cold alkali gases~\cite{PetSmi-01,PitStr-03,BloDalZwe-08,DalGerJuzOhb-11,Fetter-09,Cooper-08,Viefers-08}, where the particles are neutral atoms, the magnetic field is artificial, the external potential is a magneto-optic trap and interactions are via van der Waals forces  (actually only $s$-wave scattering most of the time, because of dilution).  

Under standard suitable assumptions on $\bA,V,w$, the above operator is well-defined and self-adjoint on some domain related to that of the $N$-body basic kinetic energy operator (the Laplacian on $\R^{dN}$). For physical reasons recalled below, one is in fact interested in the action of $H_N$ 
\begin{itemize}
 \item either on $L^2_{\asym} (\R^{dN},\C)$, the subspace of functions totally \emph{antisymmetric} under the exchange of the coordinates $\bx_1,\ldots,\bx_N$,
 \item or on $L^2_{\sym} (\R^{dN},\C)$, the subspace of functions totally \emph{symmetric} under the exchange of the coordinates $\bx_1,\ldots,\bx_N$.
\end{itemize}
The former option is relevant for fermions, i.e. quantum particles that obey the Pauli exclusion principle (colloquially, ``no more than one particle in a single quantum state''). The latter option is appropriate for bosons, quantum particles subject to no such exclusion rule. 

This review is solely concerned with the second case, that of bosons, and its main message might be summarized as 

\medskip

\begin{center}
\textbf{Particles that \emph{may} populate only a single quantum state} 

\textbf{\emph{do} populate a single quantum state.}
\end{center}

\medskip

This is a particular instance of the surprising efficiency of the mean-field approximation, wherein one assumes all particles to be independent and identically distributed according to a common statistical law (which, roughly speaking is the classical mechanics equivalent for ``populate a single quantum state''). 

\bigskip 

Let me be a bit more precise regarding the meaning of the above bold statement. Indeed, (as any statement so colloquially formulated) it must be taken with a few grains of salt. First let us recall how exactly the action of $H_N$ specifies the physics of a system of spinless non-relativistic bosons. The state of the system is described by a wave-function $\Psi_N \in L^2 _{\rm sym} (\R ^{dN}, \C)$ (or, perhaps, a statistical ensemble of such wave-functions). For reasons of interpretation recalled below, $\Psi_N$ must be $L^2$-normalized:
\begin{equation}\label{eq:intro norm}
\int_{\R^{dN}} |\Psi_N| ^2 = 1. 
\end{equation}
The dynamics is prescribed by the many-body Schr\"odinger equation 
\begin{equation}\label{eq:intro Schro dyn}
 \im \partial_t \Psi_N (t) = H_N \Psi_N (t) 
\end{equation}
that we supplement with a Cauchy datum $\Psi_N (0)$. Consequently, the equilibrium states are the eigenfunctions (or, again, statistical ensembles of such) of $H_N$. Said differently, equilibria are the critical points of the energy functional 
\begin{equation}\label{eq:intro energy}
\cE_N [\Psi_N] = \left\langle \Psi_N | H_N | \Psi_N \right\rangle_{L^2}  
\end{equation}
under the mass constraint~\eqref{eq:intro norm}.

All these candidate wave-functions are in addition symmetric, i.e.\footnote{An antisymmetric, \index{Quantum statistics, bosons and fermions}{fermionic}, function, would have $(-1) ^{\mathrm{sgn} (\sigma)}$ multiplying the right-hand side, with $\mathrm{sgn} (\sigma)$ the signature of a permutation.} 
\begin{equation}\label{eq:intro sym}
\Psi_N (\bx_1,\ldots,\bx_N) = \Psi_N (\bx_{\sigma(1)},\ldots,\bx_{\sigma(N)}) 
\end{equation}
for all permutation $\sigma$. Now, what is the simplest symmetric wave-function ? I think we all agree it is of the form 
\begin{equation}\label{eq:intro BEC}
\Psi_N (\bx_1,\ldots,\bx_N) = u ^{\otimes N} (\bx_1,\ldots,\bx_N) := u(\bx_1) \ldots u (\bx_N)
\end{equation}
with $u\in L^2 (\R^d,\C)$ a function of a single variable. The above represents a pure \index{Bose-Einstein condensate (BEC)}{Bose-Einstein condensate}, with all particles in the quantum state $u$. 

The surprising fact is that, in great generality, a large bosonic system ($N\to \infty$) will have a very strong tendency to prefer simple states of the form~\eqref{eq:intro BEC}. This is the phenomenon that is our main concern here. We shall limit ourselves to the case of the most stable equilibria, called \emph{ground states}, the minimizers of the energy functional~\eqref{eq:intro energy} under the unit mass constraint~\eqref{eq:intro norm}. Reviews of the dynamical pendant of this theory (the manifold of states of the form~\eqref{eq:intro BEC} is approximately invariant under the Schr\"odinger flow~\eqref{eq:intro Schro dyn}) may be found e.g. in~\cite{BenPorSch-15,Golse-13,Schlein-08,Spohn-12}.   

The first grain of salt regarding our main statement is that it can be valid only in special scaling limits of the many-body problem. In fact, what is required is a $N$-dependent scaling of the interaction potential $w$ (we take as reference length scale that of the trapping potential $V$) in the limit $N\to \infty$. There are several ways of achieving this, with a wide range of physical relevance and mathematical difficulty (unfortunately but unsurprisingly, the two aspects are rather positively correlated). This review aims at a systematic exposition of the known  means to give mathematical rigor to the above vague bold statement, i.e. prove its validity (or rather, that of its mathematically precise version) in the scaling limits just mentioned. These come in different types, and after introducing more background material in the rest of this introductory chapter, three families of limits will be considered in the next chapters.

I have tried to review the material in a pedagogical rather than chronological order. As the reader will see, several methods exist to deal with the problem at hand, and I have tried to be rather exhaustive in that regard. In particular, I have not limited myself to the one or two only known methods that are able to obtain the full results we aim at in the most general/difficult case. This is to avoid having to ``kill a fly with an atomic bomb'', as one colleague would put it. There is thus a gradual build-up in mathematical sophistication in the following, and at each stage I try to be exhaustive as to what exactly each method can achieve, and in which circumstances. Throughout the text I give mathematically precise statements of most lemmas and theorems that serve as our main tools, but, for want of space and time, I refer to the literature for most proofs.

\section{Basic quantum statistical mechanics}\label{sec:basic stat}

Our focus shall be on \emph{many-body quantum mechanics}. A lucid exposition of this topic usually starts from \emph{one-body classical mechanics}. 
  
\medskip

\noindent \textbf{One classical particle.} To know the state of this simplest of systems we need to specify a position  $\bx \in \R^d$ and a momentum $\bp \in \R^d$. The energy of the particle is then given by a Hamilton function $H(\bx,\bp)$ from the phase-space $\R^d \times \R ^d$ to $\R$. In non-relativistic classical mechanics (with mass $m = 1/2$) this is given as 
\begin{equation}\label{eq:1 class hamil}
 H (\bx, \bp) = |\bp|  ^2 + V (\bx) 
\end{equation}
where the first term is the kinetic energy and $V:\R^d \mapsto \R$ an external potential. 

The dynamics is given by Newton's equations (here in Hamiltonian form)
$$ \begin{cases}
    \partial_t \bx = \nabla_{\bp} H(\bx,\bp) = 2 \bp \\
    \partial_t \bp = - \nabla_{\bx} H(\bx,\bp) = - \nabla V (\bx).
   \end{cases}
$$
The zero-temperature equilibria are the minima of the Hamilton function $H(\bx,\bp)$. They are not particularly exciting, being specified by $\bp = 0$ and $\bx$ being a minimum point of $V$. This changes when one moves to quantum mechanics.

A bit of vocabulary is in order to introduce the sequel. We shall call \index{States, observables}{\emph{observables}} the functions (say continuous and bounded) on the phase-space $\R^d \times \R ^d$. They correspond to all quantities that could in principle be measured. The \index{States, observables}{\emph{state}} of a system is then a recipe to input an observable and output a number, ensuring that the process is linear. The latter assumption is natural if we want the measurement of $A+B$ to be\footnote{This does not go without subtleties however: Eugen Wigner liked to ask colleagues the question ``How do you measure $\bx + \bp$ ?''} the measurement of $A$ plus the measurement of $B$.  

In other words, the state of a system is a probability measure $\mu$ on $\R^d \times \R^d$. Given an observable $f$ we interpret the number 
\begin{equation}\label{eq:exp value}
 \int f(\bx,\bp) d\mu (\bx,\bp) 
\end{equation}
as the expected value of the quantity modeled by $f$, if the state of the system is described (probabilistically) by $\mu$. Note that the latter could be a Dirac measure at some point $\bx,\bp$, in which case the result of measurements is deterministic. 

The most noteworthy point for the sequel is that the algebra of classical observables is a commutative one, that of bounded functions from $\R^d \times \R ^d$ to $\R$.

\medskip

\noindent \textbf{One quantum particle.} Let us a make a long story short: quantization is the art of introducing a constraint as to the possible values of the position and the momentum of a particle. In hindsight (this is the content of the Gelfand-Naimark-Segal theorem), one can realize that any modification of the previous classical set-up that preserves commutativity of observables leads to a description equivalent to classical mechanics. Therefore we need to introduce some non-commutativity if we want to make some progress. 

What could the simplest form of non-commutativity be ? Perhaps we want to start from the observables $x_j,p_j$ ($j$-th coordinates of $\bx$ and $\bp$, $j=1 \ldots d$) and construct them as operators on some suitable space satisfying
\begin{equation}\label{eq:Heisenberg}
 [x_j,p_j] = x_j p_j - p_j x_j = \im \hbar. 
\end{equation}
We demand the simplest possible form of non-commutativity: the commutator is a constant (purely imaginary for latter convenience, and proportional to a number called $\hbar$, soon to be set equal to $1$). This is the \index{Heisenberg uncertainty principle}{Heisenberg uncertainty principle} which, as we will see, prevents one from knowing exactly the position and momentum of a particle\footnote{Continuing in the line of the previous footnote, it becomes even more difficult, even in principle, to measure $\bx$ and $\bp$ simultaneously and add the findings.} at the same time.

The connection with the formalism quickly described in the introductory paragraph is as follows. Suppose\footnote{Pretending we are unaware of the actual historical route to these ideas.} we brutally decide that the state of the system is described by a function $\psi : \R^d \mapsto \C$. Suppose we decide that $x_j$ acts on $\psi (\bx)$ via multiplication by the first coordinate of the argument $\bx$, that $p_j$ acts on $\psi (\bx)$ by multiplication in the Fourier domain by the first coordinate of $\bp$, the dual variable of $\bx$. Then it is equivalent to demand that $p_j$ acts as $-\im \partial_{x_j}$ and we have constructed operators satisfying our simplest non-commutative requirement~\eqref{eq:Heisenberg}.  

This crash-course on quantization is of course rather short, but I hope it motivates the following choices. Our quantum phase-space for one particle shall be the Hilbert space $L^2 (\R^d,\C)$. Square-integrability is demanded because we want to interpret $|\psi (\bx)| ^2$ as the probability density of our particle in position-space and $|\widehat{\psi} (\bp)| ^2$ as its probability density in momentum space. Here 
$$
\widehat{\psi}(\bp) = \frac{1}{(2\pi)^{d/2}} \int_{\R^d} \psi (\bx) e^{-\im \, \bp \cdot \bx}\,\mathrm{d}\bx$$ 
is the Fourier transform, normalized so as to be a $L^2$-isometry. This implies that functions of $\bx$ or $\bp$ can be given a natural meaning as observables: the expected value of $f(\bx)$ (respectively, $f(\bp)$) in the state $\psi$ is, in similarity with~\eqref{eq:exp value}, given by 
\begin{equation}\label{eq:recipe 1}
 \int_{\R^d} f (\bx) |\psi (\bx)| ^2 \mathrm{d}\bx,
\end{equation}
respectively
\begin{equation}\label{eq:recipe 2}
 \int_{\R^d} f (\bp) \left|\widehat{\psi} (\bp)\right| ^2 \mathrm{d}\bp. 
\end{equation}
Here we see that~\eqref{eq:Heisenberg} has lead us to the fact that one cannot specify with certainty both the position and the momentum of a quantum particle: this would demand that both $|\psi|^2$ and $|\widehat{\psi}| ^2$ are (close to) Dirac masses, a notorious impossibility. The Sobolev inequality is a convenient way to quantify this.

Now, applying the above recipes~\eqref{eq:recipe 1}-\eqref{eq:recipe 2}, we can bluntly replace $\bx,\bp$ by the corresponding (vector-valued) multiplications in~\eqref{eq:1 class hamil} to turn $H(\bx,\bp)$ to an operator acting on $L^2 (\R^d,\C)$
$$ H = - \Delta_{\bx} + V (\bx)$$
where the Laplacian is identified with the multiplication by $|\bp| ^2$ on the Fourier side and $V(\bx)$ acts as multiplication. The expected value of the energy in the state $\psi$ is then 
\begin{align*}
 \left\langle \psi | H | \psi \right\rangle_{L^2} &=  \int_{\R ^d} \left|\bp \right| ^2 \left|\widehat{\psi} (\bp)\right| ^2 \mathrm{d}\bp +   \int_{\R^d} V (\bx) |\psi (\bx)| ^2 \mathrm{d}\bx\\
 &= \int_{\R ^d} \left( | \nabla \psi (\bx)| ^2 + V (\bx) |\psi (\bx)| ^2 \right) \mathrm{d}\bx.
\end{align*}
As dynamics it is natural to take the Hamiltonian flow associated with $H$, namely the Schr\"odinger equation 
$$ \im \partial_t \psi = H \psi.$$
There can now be plenty of stationary states: all the eigenfunctions of $H$, and we shall be particularly interested in those achieving the lowest possible eigenvalue.

\medskip

\noindent \textbf{Mixed states.} It will be useful to allow some statistical uncertainty on the state of a particle. In classical mechanics this is achieved by identifying a \index{Pure state, mixed state}{pure state}, i.e. a point $\bx,\bp$ in the phase-space, to the corresponding Dirac mass, and then to take the convex envelope (all statistical superpositions) of such, obtaining of course all probability measures as state space. Slightly more abstractly, the phase space is $\R^d \times \R^d$, the \index{States, observables}{observables} include all continuous bounded functions thereon. The state space is the dual thereof, the set of probability measures. This convex set contains all mixed states, and its extreme points\footnote{Those one cannot write as non-trivial convex superpositions.}, the Dirac masses, are the pure states. 

In quantum mechanics we shall identify a pure state, i.e. a wave-function $\psi \in L^2 (\R^d,\C)$ with the corresponding orthogonal projector, denoted $|\psi \rangle \langle \psi|$. Then we may form statistical (meaning convex) combinations thereof, obtaining (by the spectral theorem) all positive trace-class operators on $L^2 (\R^d,\C)$ with trace $1$. 

Another way of saying this is that the phase-space is now $L^2 (\R^d,\C)$, and that the observables include bounded operators thereon. There is a twist in that we do not consider as state-space the full dual of the bounded operators, but restrict\footnote{For an infinite-dimensional Hilbert space, the dual of the bounded operators acting on it strictly includes the trace-class, which is actually the pre-dual of the bounded operators.} this to the trace-class~\cite{Simon-79,Schatten-60}. This is because we want the extreme points of the space of mixed states, identified with pure states, to be rank-one orthogonal projections.

\medskip

\noindent \textbf{Many classical particles.} For a system of $N$ classical particles the generalization is straightforward. The phase-space becomes $\R ^{dN} \times \R ^{dN}$, the observable space $C ^0 (\R^{dN} \times \R ^{dN})$, the state space $\cP (\R^{dN}\times \R^{dN})$. A prototype Hamilton function is 
\begin{equation}\label{eq:N class hamil}
H_N (\bX_N;\bP_N) = \sum_{j=1} ^N \left( |\bp_j| ^2 + V (\bx_j) \right) + \sum_{1\leq i < j \leq N} w (\bx_i-\bx_j) 
\end{equation}
with the notation $\bX_N = (\bx_1,\ldots,\bx_N),\bP_N = (\bp_1,\ldots,\bp_N)$ and we have introduced a pair interaction potential. Newton's equations take the form 
$$
\begin{cases}
 \partial_t \bx_j = 2 \bp_j \\
 \partial_t \bp_j = - \nabla_{\bx_j} H_N (\bX_N;\bP_N)
\end{cases}
$$
and they specify the evolution of pure states identified with points of the phase space. More generally one might be interested in the evolution of mixed states $\mubf_N \in \cP (\R^{dN}\times \R ^{dN})$, given by the push-forward along the trajectories of Newton's equations. Zero-temperature equilibria can become much less trivial to describe, because we are now looking for the minimal points of 
$$ \bX_N \mapsto \sum_{j=1} ^N V (\bx_j) + \sum_{1\leq i < j \leq N} w (\bx_i-\bx_j).$$
Something of importance is that we always choose $H_N$ to be symmetric under the exchange of particle labels:
$$  H_N \left( \bx_{1}, \ldots, \bx_{N}; \bp_{1}, \ldots, \bp_{N} \right) = H_N \left( \bx_{\sigma (1)}, \ldots, \bx_{\sigma (N)}; \bp_{\sigma (1)}, \ldots, \bp_{\sigma (N)} \right)$$
for any permutation $\sigma$. This is because we are thinking of identical particles, that must be \index{Indistinguishability}{indistinguishable}. Accordingly, any reasonable equilibrium state of $H_N$ will be a probability measure symmetric under particle-label exchanges. Time-evolution also preserves this symmetry condition. Thus we are ultimately only interested in states invariant under particle exchanges: probability measures over the phase space satisfying
\begin{equation}\label{eq:class indis} 
\mubf_N \left( \bx_{1}, \ldots, \bx_{N}; \bp_{1}, \ldots, \bp_{N} \right) = \mubf_N \left( \bx_{\sigma (1)}, \ldots, \bx_{\sigma (N)}; \bp_{\sigma (1)}, \ldots, \bp_{\sigma (N)} \right) 
\end{equation}
for any permutation $\sigma$.

\medskip

\noindent \textbf{Many quantum particles.} The phase-space now becomes $L^2 (\R^{dN})$ and accordingly the observables (respectively states) include\footnote{Many observables of interest are unbounded.} all bounded self-adjoint operators (respectively positive trace-class operators) acting thereon. Thus a \index{Pure state, mixed state}{pure state} is a wave-function $\Psi_N \in L^2 (\R ^{dN})$ identified with the corresponding orthogonal projector $|\Psi_N \rangle \langle \Psi_N |$. A mixed state is a positive trace-class operator $\Gamma_N$ with trace $1$ that we see (via the spectral theorem) as a statistical superposition of pure states
$$ \Gamma_N = \sum_{j=1} ^\infty \lambda_{N,j} |\Psi_{N,j} \rangle \langle \Psi_{N,j} | $$
with $(\Psi_{N,j})_j$ an orthonormal basis of $L^2 (\R^{dN})$ and $\lambda_{N,j}$ positive numbers adding to~$1$. 

We obtain a quantum Hamiltonian from~\eqref{eq:N class hamil} as previously
\begin{equation}\label{eq:N quant hamil}
H_N  = \sum_{j=1} ^N \left( -\Delta_{\bx_j} + V (\bx_j) \right) + \sum_{1\leq i < j \leq N} w (\bx_i-\bx_j).
\end{equation}
This specifies an energy,
$$ \cE_N [\Psi_N] = \langle \Psi_N | H_N | \Psi_N \rangle $$
for a pure state, and 
$$ \cE_N [\Gamma_N] = \tr \left( H_N \Gamma_N \right) $$
for a mixed state. The time-evolution is the \index{Many-body Schr\"odinger equation}{many-body Schr\"odinger flow} 
$$ \im \partial_t \Psi_N = H_N \Psi_N$$
for a pure state, and the von Neumann equation
$$ \im \partial_t \Gamma_N = [H_N,\Gamma_N]$$
for a mixed state, with $[A , B] = AB-BA$ the commutator of two operators $A$ and $B$. Equilibrium states of this evolution shall be our chief concern in this review, in particular energy minimizers (under $L^2$ unit mass constraint).

\medskip

\noindent\textbf{Many bosons.} In quantum mechanics there is a twist as to how we implement \index{Indistinguishability}{indistinguishability} of particles. Certainly, in accordance with~\eqref{eq:class indis} and the interpretation of $|\Psi_N|^2,|\widehat{\Psi_N}| ^2$ as probability densities in position/momentum space, we would like to have 
\begin{align}\label{eq:quant indis 1} 
\left|\Psi_N \left( \bx_{1}, \ldots, \bx_{N}\right) \right| ^2 &= \left| \Psi_N \left( \bx_{\sigma (1)}, \ldots, \bx_{\sigma (N)}\right) \right|^2\nonumber \\
\left|\widehat{\Psi_N} \left( \bp_{1}, \ldots, \bp_{N}\right) \right| ^2 &= \left| \widehat{\Psi_N} \left( \bp_{\sigma (1)}, \ldots, \bp_{\sigma (N)}\right) \right|^2.
\end{align}
For various reasons we shall not delve into, this is not a sufficient requirement (for starters, quantum mechanics is settled in linear spaces, whereas~\eqref{eq:quant indis 1} are non-linear constraints). As indicated previously, we want to restrict further to fully symmetric or fully antisymmetric wave-functions. A possible rationale for this is that these seem the most simple choices guaranteeing~\eqref{eq:quant indis 1}. One can further argue that, denoting $U_\sigma$ the unitary operator permuting particle labels according to $\sigma$,
\begin{equation}\label{eq:perm unit}
(U_\sigma \Psi_N) \left( \bx_{1}, \ldots, \bx_{N}\right) = \Psi_N \left( \bx_{\sigma (1)}, \ldots, \bx_{\sigma (N)}\right),
\end{equation}
any reasonable operator of the form~\eqref{eq:N quant hamil} will commute with $U_\sigma$ for any $\sigma$. It thus makes sense to look for normalized eigenfunctions $\Psi_N$ that are also eigenfunctions of $U_\sigma$ for all $\sigma$. But if for all $\sigma$  
$$ 
\Psi_N \left( \bx_{\sigma (1)}, \ldots, \bx_{\sigma (N)}\right) = \epsilon(\sigma) \Psi_N\left( \bx_{1}, \ldots, \bx_{N}\right)
$$
the map $\sigma \mapsto \epsilon(\sigma)$ must be a one-dimensional representation of the permutation group, hence either $\epsilon(\sigma) \equiv 1$ or $\epsilon(\sigma) = (-1)^{\rm{sgn}(\sigma)}$. 

This means we consider the action of $H_N$ only on fully symmetric functions 
\begin{equation}\label{eq:intro sym bis}
\Psi_N (\bx_1,\ldots,\bx_N) = \Psi_N (\bx_{\sigma(1)},\ldots,\bx_{\sigma(N)}), \quad \forall \sigma 
\end{equation}
or on fully antisymmetric functions
\begin{equation}\label{eq:intro asym}
\Psi_N (\bx_1,\ldots,\bx_N) = (-1) ^{\mathrm{sgn} (\sigma)} \Psi_N (\bx_{\sigma(1)},\ldots,\bx_{\sigma(N)}), \quad \forall \sigma. 
\end{equation}
This dichotomy reflects the empirically observed division of fundamental particles into two types, bosons and fermions. Suggestions of exotic quasi-particles falling outside of this dichotomy shall not be discussed here, see~\cite[Chapter~7]{Rougerie-hdr} and references therein. The above choices are referred to as ``\index{Quantum statistics, bosons and fermions}{quantum statistics}'' because they determine the way particles in thermal equilibrium with no interactions will populate the energy levels of the Hamiltonian.

From now on we always restrict to Case~\eqref{eq:intro sym bis}. This means our many-body Hilbert space is actually 
$$ \gH_N := \bigotimes_{\sym} ^N L^2 (\R^d) \simeq L^2 _{\sym} (\R ^{dN}).$$
Observables are always assumed to act on this space, which is legitimate because even if extended to the full $L^2  (\R ^{dN})$ they commute with the unitaries $U_\sigma$. States are positive trace-class operators with trace $1$ acting on $\gH_N$. Note that a state $\Gamma_N$ on the full $L^2  (\R ^{dN})$ preserves $\gH_N$ (i.e. is bosonic) if it satisfies 
\begin{equation}\label{eq:bosons}
 U_\sigma \Gamma_N = \Gamma_N U_\sigma = \Gamma_N
\end{equation}
for all permutation $\sigma$. A weaker notion of \index{Indistinguishability}{indistinguishability}, closer to what is demanded~\eqref{eq:class indis} for classical particles (``boltzons'') would be 
\begin{equation}\label{eq:boltzons}
 U_\sigma \Gamma_N U_\sigma ^* = \Gamma_N. 
\end{equation}

\medskip

\noindent\textbf{\index{Reduced density matrix}{Reduced density matrices.}} The full state $\Gamma_N$ (or the full wave-function $\Psi_N$) in fact contains too much information. It is often convenient/necessary to forget some of this information to make rigorous statements. In classical mechanics we can form the \index{Marginal density}{marginal densities} of a state $\mubf_N$, i.e. integrate out some of the degrees of freedom. The $k$-particle reduced density, given by
\begin{multline*}
 \mubf_N ^{(k)} \left(\bx_1,\ldots,\bx_k ; \bp_1, \ldots, \bp_k \right) \\
 = \int_{\R^{2d(N-k)}} \mubf\left(\bx_1,\ldots,\bx_k, \bx_{k+1},\ldots,\bx_N ; \bp_1, \ldots, \bp_k, \bp_{k+1},\ldots \bp_N \right) \mathrm{d}\bx_{k+1} \ldots \mathrm{d}\bx_N \mathrm{d}\bp_{k+1} \ldots \mathrm{d}\bp_N, 
\end{multline*}
contains all the information we need to describe what a typical $k$-tuple of particles does. Usually, it is sufficient to retain this information only for small values of $k$ (compared to $N$). This is all we need and hope to record. Observe that $\mubf_N ^{(k)}$ is obtained from $\mubf_N$ by testing against observables that depend only on $2k$ coordinates. Let us generalize this to the quantum case.

Given a state (aka, density matrix) $\Gamma_N$ on $\gH_N$ ($\Gamma_N = |\Psi_N \rangle \langle \Psi_N |$ in case we are actually thinking of a wave-function $\Psi_N$) we define its \emph{$k$-particles reduced density matrix} $\Gamma_N ^{(k)}$, a positive trace-class operator on $\gH_k$ by demanding that 
\begin{equation}\label{eq:def red mat dual}
 \tr_{\gH_k} \left( A_k \Gamma_N ^{(k)} \right) = {N \choose k} \tr_{\gH_k} \left( A_k \otimes \1 ^{\otimes (N-k)} \Gamma_N \right) 
\end{equation}
for every bounded operator $A_k$ on $\gH_k$. Here $A_k \otimes \1 ^{\otimes (N-k)}$ acts as $A_k$ on the first $k$ coordinates and trivially on the $N-k$ others (the choice of which particles are ``first'' and which are ``others'' is irrelevant because we always consider states satisfying~\eqref{eq:boltzons}). The combinatorial factor ${N \choose k}$ is there for normalization reasons: it is natural to have the trace of $\Gamma_N ^{(k)}$ equal to the number of $k$-tuples of particles.

What we have just defined is in fact the partial trace
\begin{equation}\label{eq:def red mat}
 \Gamma_N ^{(k)} = {N \choose k} \tr_{k+1 \to N} \left( \Gamma_N \right) 
\end{equation}
where again the choice of which $N-k$ degrees of freedom to trace upon is irrelevant. We could also identify $\Gamma_N$ with its integral kernel~\cite[Section~VI.6]{ReeSim1} 
$$ 
\Gamma_N \Phi_N (\bX_N) = \int_{\R^{dN}} \Gamma_N (\bX_N;\bY_N) \Phi_N (\bY_N) d\bY_N,$$
i.e., for a pure state $\Gamma_N = |\Psi_N \rangle \langle \Psi_N |$,  
$$ \Gamma_N (\bX_N;\bY_N) = \Psi_N (\bX_N) \overline{\Psi_N (\bY_N)}.$$
Then the integral kernel of $\Gamma_N ^{(k)}$ is obtained as 
$$ 
\Gamma_N ^{(k)} (\bX_k;\bY_k)  = {N \choose k} \int_{\R ^{d(N-k)}} \Gamma_N \left( \bX_k, \bZ_{N-k}; \bY_k, \bZ_{N-k} \right) \mathrm{d}\bZ_{N-k}.
$$
Observe, as a first use of reduced density matrices, that the energies we are interested in depend only on the second density matrix. With $H_N$ as in~\eqref{eq:N quant hamil}, 
$$
 \langle \Psi_N |H_N| \Psi_N \rangle_{L^2} = \tr_{\gH} \left( \left( -\Delta + V \right) \Gamma_N ^{(1)} \right) + \tr_{\gH_2}\left( w (\bx-\by) \Gamma_N ^{(2)} \right).
$$

\section{\index{Second quantization}{Second-quantized formalism}}\label{sec:second quant}

Roughly speaking, we have so far been working in first quantization: observables are defined as operators acting on classical fields, i.e. wave-\emph{functions}. By second quantization one usually means that fields too become operators\footnote{Edward Nelson used to say that first quantization is a mystery but second quantization is a functor.}. This can range from a deep physical move (as in quantum electrodynamics, where the electromagnetic field is quantized, and the number of associated photonic particles may \emph{physically} fluctuate) to a basic calculational trick (as in the ordinary non-relativistic quantum mechanics that concerns us in this text). In the latter case, second quantization has a \index{Grand-canonical}{grand-canonical} flavor: we can let the particle number fluctuate (fixing the chemical potential instead) if it turns out to be computationally convenient. This comes with a set of algebraic tools that are also extremely handy in the canonical (fixed particle number) context of this text. It will often be handy to rely on this ``second quantization'' formalism, by which we mean introduce \index{Creation and annihilation operators}{creation and annihilation operators}. 

We will not use this formalism throughout the text, so some hasty readers\footnote{And those who are already familiarized, of course.} might skip this section for now. A word of warning though: physicists learn this formalism in kindergarten, and it is tremendously useful. I would urge less hasty readers to take the occasion of them reading this text anyway to get acquainted with these notions. See~\cite{DerGer-13,GusSig-06,Solovej-notes} for more details.

\medskip 

\noindent \textbf{Fock space.} Recall we consider only the case of bosons in this review, but fermions are treated very similarly. It is convenient to gather all $N$-particle spaces under one roof, the bosonic \index{Fock space}{Fock space} 
\begin{equation}\label{eq:Fock}
 \gF = \gF (\gH) := \C \oplus \gH \oplus \gH_2 \oplus \ldots \oplus \gH_N \oplus \ldots  
\end{equation}
One calls 
\begin{equation}\label{eq:vacuum}
|0\rangle = 1 \oplus 0 \oplus 0 \oplus \ldots 
\end{equation}
the \index{Vacuum vector}{\emph{vacuum vector}}, that has no particles at all. 

We shall call \index{Grand-canonical}{grand-canonical} state (by opposition to canonical, $N$-particle, states encountered previously) a positive trace-class operator on $\gF$ with trace $1$. We shall only be interested in ``diagonal states'' of the form 
$$ \Gamma = \Gamma_0 \oplus \Gamma_1 \oplus \ldots \oplus \Gamma_N \oplus \ldots$$
with 
$$ \sum_{n=0} ^{\infty} \tr_{\gH_n} \, \Gamma_n = 1.$$
A $N$-particle state is recovered if all the $\Gamma_n$ are $0$ except $\Gamma_N$. If that is not the case we interpret the \index{Number operator}{particle number} as a genuine observable/operator 
\begin{equation}\label{eq:number second}
 \cN = \bigoplus_{n=0} ^\infty \, n \1_{\gH_n} 
\end{equation}
whose expectation value in a state $\Gamma$ is 
$$ \tr \left( \cN \Gamma \right) = \sum_{n=0} ^{\infty} n \tr_{\gH_n} \, \Gamma_n.$$

\medskip 

\noindent \textbf{\index{Reduced density matrix}{Reduced density matrices}, again.} Given a $k$-particle observable $A_k$ acting on $\gH_k$ we can lift it to Fock space in the natural way 
$$ \bbA_k := \bigoplus_{n= k} ^\infty \sum_{1 \leq i_1 < \ldots < i_k \leq n} A_{i_1,\ldots,i_k}$$
where $A_{i_1,\ldots,i_k}$ denotes $A_k$ acting on variables $i_1,\ldots,i_k$ of a $n$-particle wave-function. Then, we define the $k$-particle reduced density matrix of a grand-canonical state $\Gamma$ by setting 
\begin{equation}\label{eq:dens mat GC}
 \tr_{\gH_k} \left( A_k \Gamma^{(k)} \right) = \tr_{\gF} \left( \bbA_k \Gamma \right). 
\end{equation}
For a diagonal state we have 
$$ \Gamma^{(k)} = \sum_{n\geq k} \Gamma_n ^{(k)} = \sum_{n\geq k} {n \choose k} \tr_{k+1 \to n} \left( \Gamma_n  \right) .$$
where the reduced density matrix $\Gamma_n ^{(k)}$ of a $n$-particle state is defined as in the previous section~\eqref{eq:def red mat}.  

\medskip

\noindent \textbf{\index{Creation and annihilation operators}{Creation and annihilation operators}.} Given a one-body wave-function $u\in L^2 (\R^d)$ we define the associated annihilation operator $a(u)$ acting on the Fock space as specified by 
$$ \left( a (u) \psi_n \right) (\bx_1,\ldots,\bx_{n-1}) := \sqrt{n} \int_{\R^d} \overline{u(\bx)} \psi_n (\bx,\bx_1,\ldots,\bx_{n-1}) \mathrm{d}\bx. $$
Here $\psi_n$ is a $n$-particles bosonic wave-function, and the action is extended by linearity to the whole Fock space. Note that any annihilation operator ``annihiliates the vacuum''
$$ a (u) |0\rangle = 0.$$
The formal adjoint $a^{\dagger} (u)$ of $a(u)$ is the creation operator
$$ \left( a^\dagger (u) \psi_n \right) (\bx_1,\ldots,\bx_{n+1}) := \frac{1}{\sqrt{n+1}} \sum_{j=1} ^{n+1} u (x_j) \psi_n \left( x_1,\ldots, x_{j-1},x_{j+1},\ldots,x_n\right).$$
Note that $a(u)$ sends the $n$-particles sector of Fock space to the $(n-1)$-particles sector, while $a^\dagger (u)$ sends the $n$-particles sector to to the $(n+1)$-particles sector, whence the names annihilation and creation operators. 

The main thing to remember about these operators, and basically the only one used in practice to calculate with them is that they satisfy the \emph{\index{Canonical commutation relations (CCR)}{canonical commutation relations}} (CCR). Namely, let $u,v \in L ^2 (\R^d)$, then 
\begin{equation}\label{eq:CCR}
[a(u),a(v) ] = [a^\dagger (u),a^{\dagger} (v)] = 0, \quad [a (u), a^\dagger (v) ] = \langle u | v \rangle_{L^2}.  
\end{equation}

It is sometimes convenient to consider operators $a_\bx, a^\dagger_\bx,$ annihilating/creating a particle at a point $x\in \R ^d$ rather than in a state $u\in L^2 (\R^d)$. These operator-valued distributions are defined by requiring that, for all $u\in L ^2 (\R ^d)$, 
$$ 
a (u) = \int  \overline{u(\bx)} a_\bx \mathrm{d}\bx, \quad a^{\dagger} (u) =  \int  u(\bx) a^\dagger _\bx \mathrm{d}\bx. 
$$
The CCR now takes the form 
\begin{equation}\label{eq:CCR x}
[a_\bx,a_\by ] = [a^\dagger_\bx, a^{\dagger} _\by] = 0, \quad [a _\bx, a^\dagger_\by ] = \delta ( \bx-\by)  
\end{equation}
with $\delta$ the Dirac mass at the origin.

\medskip 

\noindent \textbf{Relation with Hamiltonians.} Consider extending~\eqref{eq:N quant hamil} to the Fock space in the natural way
$$ \bH = \bigoplus_{n=0} ^\infty H_n.$$
There is a very useful expression for $\bH$ (or any other reasonable operator on the Fock space) in terms of annihilation and creation operators: Let $(u_j)_{j\in \N}$ be an orthonormal basis of $L^2 (\R ^d)$. Then 
\begin{multline}\label{eq:hamil second}
 \bH = \sum_{j,k} \langle u_j | -\Delta + V  | u_k \rangle_{L^2(\R ^d)} a^\dagger (u_j) a (u_k) \\ + \frac{1}{2}\sum_{i,j,k,\ell} \left\langle u_i \otimes u_j | w (\bx-\by) | u_k \otimes u_{\ell} \right\rangle_{L^2 (\R^{2d})} a^{\dagger} (u_i) a ^{\dagger} (u_j) a (u_k) a (u_\ell). 
\end{multline}
Observe that a one-particle operator corresponds to a quadratic operator in creation/annihilation operators while a two-particle operator corresponds to a quartic one.

Using $a_\bx,a^\dagger_\bx$ instead we also have 
\begin{equation}\label{eq:second hamil x} 
\bH = \int_{\R^d} a^\dagger _\bx  \left( - \Delta_\bx + V (\bx) \right) a_\bx \mathrm{d}\bx + \frac{1}{2} \iint_{\R^d \times \R ^d} w (\bx-\by) a^\dagger _{\bx} a ^{\dagger}_{\by} a_\bx a_\by \mathrm{d}\bx \mathrm{d}\by. 
\end{equation}

\medskip 

\noindent \textbf{Relation with density matrices.} For any state $\Gamma$ over the bosonic Fock space (in particular any $N$-particle state living on only one sector), its reduced density matrices can be characterized by 
\begin{multline}\label{eq:dens mat CCR}
 \left\langle v_1 \otimes_\sym \ldots \otimes_\sym v_k | \Gamma ^{(k)} | u_1 \otimes_\sym \ldots \otimes_\sym u_k \right\rangle_{L^2} \\
 = \tr_{\gF} \left( a^\dagger (u_1) \ldots a ^{\dagger} (u_k) a (v_1) \ldots a(v_k) \Gamma \right). 
\end{multline}
Alternatively, $\Gamma ^{(k)}$ being a $k$-particle trace-class operator, we can identify it with an integral kernel which satisfies 
$$\Gamma ^{(k)} \left( \bx_1,\ldots, \bx_k ; \by_1,\ldots, \by_k \right) = \tr_{\gF} \left( a^\dagger _{\by_1} \ldots a^\dagger_{\by_k} a_{\bx_1} \ldots a_{\by_k} \Gamma \right).$$

\section{\index{Mean-field approximation}{Mean-field approximation} and \index{Scaling limits}{scaling limits}}\label{sec:sca lim} 

We can now return to the main theme of the review. We consider the general \index{Many-body Schr\"odinger operator}{many-body Schr\"odinger operator}
\begin{equation}\label{eq:intro Schro op bis}
H_N := \sum_{j=1} ^N \left(-\im \nabla_{\bx_j} + \bA (\bx_j) \right) ^2 + V (\bx_j) + \sum_{1\leq i < j \leq N} w_N (\bx_i - \bx_j),
\end{equation}
re-introducing the vector potential $\bA$ of an external magnetic field $B = \mathrm{curl}\, \bA$. Note that the interaction potential got decorated with a $N$ to indicate that it will soon be chosen $N$-dependent for reasons we shall explain. We could change the kinetic energy operator $\left(-\im \nabla_{\bx} + \bA (\bx) \right) ^2$ to something pseudo-relativistic, but we refrain from doing so. The above operator essentially includes all the possible types of difficulty one might encounter when dealing with mean-field limits. We shall mainly work variationally by considering the energy functional
\begin{equation}\label{eq:N energy}
 \cE_N [\Psi_N] = \langle \Psi_N |H_N| \Psi_N \rangle_{L^2}.
\end{equation}

\medskip 

\noindent\textbf{Bosonic \index{Ground state}{ground states.}} We focus on the action of $H_N$ on the bosonic space (we often denote $\gH=L^2 (\R^d)$ our basic one-particle Hilbert space)
$$\gH_N = \gH = L^2_\sym (\R ^{dN},\C)$$
and consider the ground state energy\footnote{With $\langle \Psi_N |H_N| \Psi_N \rangle$ understood in the sense of quadratic forms and set equal to $+\infty$ is $\Psi_N$ is not in the quadratic form domain of $H_N$.}
\begin{equation}\label{eq:N GSE}
E (N) = \inf\left\{ \langle \Psi_N |H_N| \Psi_N \rangle, \: \Psi_N \in \gH_N,\: \int_{\R^{dN}} |\Psi_N| ^2 = 1 \right\},  
\end{equation}
implicitly assuming the Hamiltonian is bounded from below so that the infimum exists. Of course $E(N)$ is the bottom of the spectrum of $H_N$. Most of the time we will consider the case where $E(N)$ is actually a minimum, thus the lowest eigenvalue of $H_N$, and we will study the associated minimizer(s)/eigenfunction(s) $\Psi_N$. When there is no minimizer or when the method allows it, we will also consider sequences of quasi-minimizers satisfying 
$$ \langle \Psi_N | H_N | \Psi_N \rangle \leq E (N) (1 + o(1)) $$
in the limit $N\to \infty$. Our main goal is to argue that, for large $N$, minimizers or quasi-minimizers can be looked for/approximated in the form of a Hartree \index{Trial state, Hartree}{trial state}
\begin{equation}\label{eq:ansatz}
 \Psi_N (\bx_1,\ldots,\bx_N) = u ^{\otimes N} (\bx_1,\ldots,\bx_N) := u(\bx_1) \ldots u (\bx_N) 
\end{equation}
when $N\to \infty$. Ansatz~\eqref{eq:ansatz} is a \index{Bose-Einstein condensate (BEC)}{Bose-Einstein condensate}, the quantum analogue of independent and identically distributed (iid) particles.

\medskip 

\noindent\textbf{Scaling limits.} Of course there is no reason to believe that the ground state of a large bosonic system should necessarily factorize, even in the large $N$ limit. Certainly, many types of correlations can occur. What is however true, is that for \emph{weakly interacting} systems, the ground state essentially always factorize. Here ``weakly interacting'' is not meant in a perturbative fashion: it is in fact sufficient that one-particle energies and interaction energies be of the same order of magnitude, i.e. that the two sums in~\eqref{eq:intro Schro op bis} weigh roughly the same in the limit. Since there are $N$ terms in the first sum (one-particle energy) and $\sim \frac{N(N-1)}{2}$ terms in the second sum (interactions), we have to assume some $N$-dependence of $w_N$ to achieve that.   

We shall always take as fixed reference length scale that of the one-particle Schr\"odinger operator $\left(-\im \nabla_{\bx} + \bA (\bx) \right) ^2 + V$. Think of the particles living in a fixed box if you wish, although we typically prefer $V$ to be a soft confinement, i.e. a potential on $\R^d$ with trapping behavior at infinity 
$$
V(\bx) \underset{|\bx| \to \infty}{\to} + \infty.
$$
In particular, we \emph{do not} work in the \index{Thermodynamic limit}{thermodynamic limit} of large volume/fixed density, which is sometimes thought of as the only relevant one in condensed matter physics. The latter view has its reasons: typical real-life systems have gigantic numbers of particles, an extension extremely large compared to their size, and the density is the experimentally adjustable parameter. However, scaling limits give some insights that would be extremely hard to vindicate in the thermodynamic limit. Besides, \index{Cold atoms}{cold atoms experiments}~\cite{PetSmi-01,PitStr-03,BloDalZwe-08,DalGerJuzOhb-11,Fetter-09,Cooper-08,Viefers-08} are mostly in regimes different from condensed matter physics (particle numbers are much smaller, and the extent of the system is set by a trapping potential). Moreover, their parameters can be tuned to a large extent (in particular, interaction potentials can be accessed via Feshbach resonances).    

There are then $N$ particles in a fixed volume and (because they do not satisfy any exclusion principle), the total one-body energy scales like $N$. We would like the total interaction energy to do the same, which requires something of the sort
\begin{equation}\label{eq:interaction scale}
 N \times \left(\mbox{Range of interactions}\right) ^{d} \times \left(\mbox{Strength of interactions}\right) \underset{N\to \infty}{\sim} 1 
\end{equation}
for the first two terms yield the number of particles a given one typically interacts with. Thus we think of $w_N$ as being of the form 
$$ w_N (\bx) = \lambda_N w \left( \frac{\bx}{L_N} \right)$$
where $w$ is a fixed potential and $\lambda_N,L_N$ (energy and length scales of the interaction) are chosen to depend on $N$ in such a way that~\eqref{eq:interaction scale} holds. It is by now common practice to guarantee this by picking a fixed non-negative number $\beta$ and setting 
$$ L_N = N ^{-\beta}, \quad \lambda_N = N ^{d\beta -1},$$
thus (the choice of dividing by $N-1$ instead of $N$ is for notational convenience later)
\begin{equation}\label{eq:scaled potential}
\boxed{w_N (\bx) = \frac{1}{N-1} N ^{d\beta} w \left( N ^{\beta} \bx\right) = \frac{1}{N-1} w_{N,\beta} (\bx).} 
\end{equation}
Tuning $\beta$ gives some freedom as to the physics we are describing, as we know explain.

\medskip

\noindent\textbf{\index{Mean-field limit}{Mean-field limits.}} Let us start with the simplest case $\beta = 0$. The range of the interaction is then comparable to the extension of the full system, so each particle interacts with essentially all the others. To make the interaction energy per particle finite we brutally divide the interaction potential by $N$. This is perhaps the most sensible thing to do mathematically if we think of long range forces, say electrostatic-gravitational. In this situation we really expect some statistical averaging to take place, in the spirit of the law of large numbers.

Actually, the physics is not much different for $\beta < 1/d$: the range of the interaction can be much smaller than the full system size, but it stays much larger than the typical inter-particle distance $N^{-1/d}$. To fulfill~\eqref{eq:interaction scale} the interaction strength stays small, $\lambda_N \to 0$ in the $N\to \infty$ limit, so that each particle interacts weakly with many others at a time. 

One can think of the case $0 < \beta < 1/d$ as a $\beta = 0$ case but with $w$ (formally !) replaced by a Dirac mass 
\begin{equation}\label{eq:delta int}
\left( \int_{\R^d} w \right) \delta_0 
\end{equation}
at the origin, leading to local interactions in the limit even though they are fairly long-range in the original system.  

\medskip

\noindent\textbf{\index{Dilute limit}{Dilute limits.}} When $\beta$ crosses the crucial $1/d$ threshold upwards, one enters a rather different physical regime. We now have that the range of the interactions is much smaller than the typical inter-particle distance $L_N \ll N^{-1/d}$, but the strength of the interaction becomes large $\lambda_N\gg 1$: each particle interacts with very few particles at a time, but very strongly.  It is less intuitive now that the system should stay weakly correlated in such a situation, but it in fact does as long as $\beta$ is not too large. 

In fact, for 
\begin{equation}\label{eq:dilute regime}
1/d < \beta < \begin{cases}
                 +\infty \mbox{ if } d=1,2\\
                 1 \mbox{ if } d = 3
                \end{cases} 
\end{equation}
one hardly notices any difference in the asymptotics of the bosonic many-body problem as compared with the mean-field regime: we end up again with an effective potential~\eqref{eq:delta int}. This is a remarkable fact when one thinks of the rather different physical situations involved. As we shall see, even if statements look very similar, one does notice a difference when it comes to proofs. The dilute case is much more difficult than the mean-field case, which is only fair from an analysis standpoint: the interaction is much more singular.

\medskip 

\noindent\textbf{\index{Gross-Pitaevskii limit}{Gross-Pitaevskii limit.}} The most challenging analysis questions arise when we leave the dilute regime~\eqref{eq:dilute regime} upwards. We shall only discuss this in 3D. In 2D one would need~\cite{CarCenSch-20,JebLeoPic-16,LieYng-01,LieSeiYng-01,SchYng-07} to make the scaling of the interaction potential depend exponentially on $N$. In 1D I am not aware of any equivalent of the regime we shall now discuss.

Thus we focus on the case $\beta = 1$ in 3D. What is special with this ? To answer we need to introduce the \emph{scattering length} $a_w$ of an interaction potential $w$. It can be characterized as 
\begin{equation}\label{eq:def scat}
4\pi a_w = \inf\left\{ \int_{\R^3} |\nabla f |^2 + \frac{1}{2} w |f| ^2, \: f(\bx) \underset{|\bx| \to \infty}{\to} 1  \right\}. 
\end{equation}
One should think of the above problem as giving the minimal energy of a pair of particles interacting locally via the potential $w$ (locally because of the boundary condition we require). The wave-function $f$ corresponds to the relative motion of the particles (the center of mass coordinate is removed). Going back to~\eqref{eq:interaction scale} we see by scaling that if $\beta = 1,d=3$ 
$$ a_{w_N} = N ^{-1} a_w < \frac{1}{8\pi N} \int w $$
where the upper bound is obtained by using the trial state $f\equiv 1$ in~\eqref{eq:def scat}. On the other hand, if $\beta < 1$ 
$$ a_{w_N} \sim N ^{-1} \int w.$$
In full generality, one should always take the scattering length into account to model interactions when $\beta >0$, but this affects the leading order in the large $N$ limit only if $\beta = 1$. See Section~\ref{sec:sca length} for more details.

It is better appreciated why the above is a length by introducing the \index{Scattering solution}{zero-energy scattering solution} $f_w$, i.e. the solution of the variational problem~\eqref{eq:def scat}. It solves 
\begin{equation}\label{eq:scat eq}
-\Delta f_w + \frac{1}{2} w f_w = 0 \mbox{ in } \R^3, \: f_w (\bx) \underset{|\bx| \to \infty}{\to} 1  
\end{equation}
and behaves, for large $|\bx|$, as 
$$ f_w (\bx) = 1 - \frac{a_w}{|\bx|} (1 + o (1)).$$

Following the early works of Gross and Pitaevskii~\cite{Gross-61,Pitaevskii-61}, in most situations of interest for the dilute Bose gas one should replace $w_N\rightsquigarrow 8 \pi a_{w_N} \delta_0$ and, perhaps, only then take the limit $N\to \infty$. In other words, the interaction potential obtained when taking first the limit of short-range interactions has the \index{Scattering length}{scattering length} as its parameter. This limit does not make a lot of sense mathematically, because delta (or point) interactions are hard to define (see~\cite{CorAntFinMicTet-15,CorAntFinMicTet-12,CorFinTet-15,AntFigTet-94} and references therein for works in this direction). The Gross-Pitaevskii limit $N\to\infty,\beta = 1$ yields the same final result as one would get if one could take first the range of the interactions to $0$, and then only $N\to \infty$. This is the way physicists usually think of the mean-field approximation in a dilute gas.

The main difference with the cases discussed before is that the effective potential obtained in the limit is indeed 
\begin{equation}\label{eq:delta scat}
8 \pi a_{w} \delta_0 
\end{equation}
instead of~\eqref{eq:delta int}. The difficulty is that this is an effect of short-range correlations between the particles, so that one has to go beyond the iid ansatz~\eqref{eq:ansatz} in this regime. 

\section{Non-linear Schr\"odinger functionals}\label{sec:NLS}

Let us now turn to the limit objects we shall derive in the large $N$ limit. Essentially they are obtained by inserting the ansatz~\eqref{eq:ansatz} in the many-body energy functional~\eqref{eq:N energy} and computing what one gets for large $N$. There is a significant twist however as regards the \index{Gross-Pitaevskii limit}{Gross-Pitaevskii limit}.

\medskip 

\noindent\textbf{Pure mean-field case,} $\beta = 0$ in~\eqref{eq:interaction scale}. We simply insert~\eqref{eq:ansatz} in~\eqref{eq:N energy}. Modulo approximating ${N\choose 2}$ by $N^2 /2$ we then get 
$$ \cE_N [\Psi_N] = N\cEh [u]+ O(1)$$
with the non-local \index{Hartree energy functional}{Hartree energy functional}
\begin{equation}\label{eq:rev Hartree f}
\cEh[u] = \int_{\R^d} \left| \left(-\im \nabla + \bA \right) u \right| ^2 + V |u| ^2 + \frac{1}{2} \iint_{\R^d \times \R ^d} |u(\bx)| ^2 w (\bx - \by) |u (\by)| ^2 \mathrm{d}\bx \mathrm{d}\by.
\end{equation}
Under standard assumptions on the data $\bA,V,w$ there exist minimizers for the ground state energy 
$$ 
\Eh := \inf\left\{ \cEh [u], \: \int_{\R^d} |u|^2  = 1 \right\}.
$$
We denote one such minimizer by $\uh$ and by $\Mh$ the set of all minimizers. Non uniqueness occurs~\cite{AshFroGraSchTro-02,GuoSei-13,Seiringer-02,CorRinYng-07,CorPinRouYng-12} typically for non-zero $\bA$ and/or attractive interaction $w$. We do not insist on the required conditions on $\bA,V,w$, for they are anyway implied by what we need to assume for the many-body problem to make sense.

We shall mostly focus on the case of \index{Trapping potential}{trapping external potentials} 
\begin{equation}\label{eq:trapping}
V(\bx) \underset{|\bx| \to \infty}{\longrightarrow} \infty  
\end{equation}
in which case the direct method in the calculus of variations leads straightforwardly to existence of minimizers. Sometimes we make more specific assumptions to control the growth of $V$ at infinity
\begin{equation}\label{eq:trapping s}
V(\bx) \geq c_s |\bx| ^s - C_s      
\end{equation}
for some $s>0$ and constant $c_s,C_s >0$. In case $V$ does not grow at infinity (non-trapped case), existence of minimizers is a more subtle matter and requires \index{Concentration-compactness}{concentration-compactness} techniques. This case shall concern us only in Section~\ref{sec:non trap}.

\medskip 

\noindent\textbf{Non-linear Schr\"odinger case,} $\beta >0$ in~\eqref{eq:interaction scale}. As discussed above, there is a clear distinction between the regimes $\beta < 1/d$ and $\beta > 1/d$. It however does not show up in the form of the limiting functional, as long as $\beta$ satisfies~\eqref{eq:dilute regime}. As previously, inserting~\eqref{eq:ansatz} in~\eqref{eq:N energy} leads to the, now $N$-dependent, \index{Hartree energy functional}{Hartree functional}  
\begin{equation}\label{eq:Hartree func N}
\int_{\R^d} \left( \left| \left(-\im \nabla + \bA \right) u \right| ^2 + V |u| ^2 \right) + \frac{N^{d\beta}}{2} \iint_{\R^d \times \R ^d} |u(\bx)| ^2 w \left(N ^{\beta} \left(\bx - \by\right)\right) |u (\by)| ^2 \mathrm{d}\bx \mathrm{d}\by. 
\end{equation}
Since of course 
$$ N^{d\beta} w \left( N ^{\beta} \, \cdot \right) \underset{N\to \infty}{\wto} \left( \int w \right)\delta_0 $$
as distributions, the sensible thing to do is to replace, in the large $N$ limit, the above functional by a local version. For shortness we denote 
\begin{equation}\label{eq:convention int w}
b_w := \int_{\R^d} w.  
\end{equation}
The \index{Non-linear Schr\"odinger energy functional}{non-linear Schr\"odinger functional} is then 
\begin{equation}\label{eq:rev nls f}
\cEnls[u]:= \int_{\R^d} \left| \left(-\im \nabla + \bA \right) u \right| ^2 + V |u| ^2 + \frac{b_w}{2} \int_{\R^d} |u(\bx)| ^4  \mathrm{d}\bx.
\end{equation}
There is now some discussion to be had as to whether the above is bounded below. Contrarily to the pure mean-field case, this is not implied by the $N$-body problem being well-defined for fixed $N$. 

In the non-linear Schr\"odinger case we shall always assume~\eqref{eq:trapping}. Then the infimum 
$$ 
\Enls := \inf \left\{ \cEnls [u], \: \int_{\R^d} |u| ^2 = 1 \right\}
$$
exists and is attained by minimizers $\unls$ forming the set $\Mnls$ if one makes the following

\begin{assumption}[\textbf{Stability of NLS energy}]\label{asum:NLS stab}\mbox{}\\
Let $w$ be the unscaled potential $w$ in~\eqref{eq:scaled potential} and $b_w$ its integral. We assume:
 \begin{itemize}
 \item In 1D, nothing particular. 
 \item In 2D, $b_w$ not to be too negative. We demand 
 $$ b_w > -a^*$$
where $a^*$ is the best constant in the Gagliardo-Nirenberg inequality 
\begin{equation}\label{eq:GagNir}
\frac{a^*}{2} \int_{\R ^2} |u| ^4 \leq \left(\int_{\R^2} |u| ^2 \right) \left(\int_{\R^2} |\nabla u| ^2 \right).
\end{equation}
\item In 3D, that $b_w$ be non-negative $b_w \geq 0$. 
\end{itemize} 
\end{assumption}

These constraints are imposed by the respective scalings of the interaction $\int |u|^4$ and kinetic energy $\int |\nabla u| ^2$ terms under mass-preserving changes of functions
$$ u(\bx) \mapsto \frac{1}{L^{d/2}} u\left(\frac{\bx}{L}\right). $$
The enemy is indeed mass-concentration of a minimizing sequence at some point under the influence of attractive interactions ($L \to 0$ schematically). Under the above assumptions, such a scenario is always prevented by the coercive kinetic energy, for it would lead to the energy being $+\infty$.

\medskip 

\noindent\textbf{Gross-Pitaevskii case,} $\beta =1$ in~\eqref{eq:interaction scale}. We discuss this case only in 3D, in lower dimensions it would require a different scaling of the interactions~\cite[Chapters~3 and 6]{LieSeiSolYng-05}. The limit object is a NLS functional as above, but with a different coupling constant. Of course that does not change the mathematical properties of the limit, or the existence theory for its minimizers. It does however change tremendously the derivation from the many-body problem and we thus prefer to use a different notation to avoid confusions.

Let thus $\EGP,\uGP,\MGP$ respectively be the minimum, a minimizer, and the set of minimizers (all this under unit mass constraint of course) of the \index{Gross-Pitaevskii energy functional}{Gross-Pitaevskii functional}
\begin{equation}\label{eq:GP func}
\cEGP [u] = \int_{\R^3} \left| \left(-\im \nabla + \bA \right) u \right| ^2 + V |u| ^2 + 4\pi a_w \int_{\R^3} |u(\bx)| ^4  \mathrm{d}\bx.
\end{equation}
where $a_w$ is the scattering length~\eqref{eq:def scat} of the fixed potential $w$ appearing in~\eqref{eq:interaction scale}. We assume it is non-negative to ensure boundedness from below of the energy. 

Observe that~\eqref{eq:def scat} ensures that 
$$ 4 \pi a_w < \frac{b_w}{2}$$
by taking the trial state $f\equiv 1$ in the variational principle and observing this cannot be an exact minimizer (as per~\eqref{eq:scat eq}). Thus for $\beta = 1$ the energy is \emph{lower} than one might expect from the arguments discussed above, leading to the NLS functional. This comes about because short-range correlations must be inserted on top of the ansatz~\eqref{eq:ansatz}. 

We will be more precise later in this text, but, roughly speaking, the idea behind the emergence of the Gross-Pitaevskii functional is that the ground state of the system actually looks like 
\begin{equation}\label{eq:ansatz GP}
\Psi_N (\bx_1,\ldots,\bx_N) \approx c_N \prod_{j= 1} ^N u (\bx_j) \prod_{1\leq i < j \leq N} f_{w_N} \left( \bx_i - \bx_j\right) 
\end{equation}
with $f_{w_N}$ the solution of~\eqref{eq:def scat}-\eqref{eq:scat eq} for the potential~\eqref{eq:interaction scale} and $c_N$ a normalization constant. The extra pair correlations in the second factor are responsible for the reduction of the effective coupling constant. Note that for $\beta < 1$ we have   
$$ 8 \pi a_{w_N} \underset{N\to \infty}{\longrightarrow} b_w = \int w$$
so that the effect of pair correlations in~\eqref{eq:ansatz GP} is not seen at leading order.

\section{Main theorem}

We are now ready to state a ``meta-theorem'' embodying the kind of results the material reviewed herein aims at proving. In the rest of the review we will present several avatars of this general statement. As we shall see, only a handful of methods allow to prove it in its full glory, i.e. with minimal assumptions and maximal conclusions. 

\medskip

\noindent\textbf{Comments on assumptions.} As far as the assumptions are concerned, we shall not be picky as to the regularity of the data $\bA,V,w$. A minima we need to assume that the original many-body Hamiltonian~\eqref{eq:intro Schro op bis} is (essentially) self-adjoint, which is ensured under the

\begin{assumption}[\textbf{Potentials}]\label{asum:pots}\mbox{}\\
Let $\max (1,d/2) < p < \infty$. We assume that 
\begin{itemize}
 \item $V\in L^p_{\rm loc} (\R^d) +  L ^{\infty}_{\rm loc} (\R ^d)$ with $$V(\bx)\underset{|\bx|\to\infty}{\to} \infty$$ 
 \item $w\in L^p (\R^d) + L^\infty (\R^d)$ with 
 $$w(x) \underset{|\bx|\to\infty}{\to} 0$$ 
 \item $\bA \in L^2_{\rm loc} (\R^d,\R^d)$ 
\end{itemize}
\end{assumption}

See e.g.~\cite[Remark~3.2]{Rougerie-LMU} and~\cite{ReeSim2} for more background. The discussion of magnetic fields is in~\cite{AvrHerSim-78}. We will not try to stick to the above in our statements, and shall often be rather liberal as regards regularity. Our criteria of success are indeed more structural: 
\begin{itemize}
 \item does the method allow for a non-zero $\bA$ ? 
 \item does the method allow for (partially) attractive interactions ?
 \item how large a $\beta$ (how singular an interaction) can the method afford ?
\end{itemize}
The third point has been discussed at length above. In particular, in the NLS/dilute/GP limits we do not care\footnote{A remarkable aspect of some of the methods we discuss is however that they can accomodate hard-core potentials, see Section~\ref{sec:sca length}.} too much about the regularity of $w$ since it is scaled to converge to a Dirac delta function anyway. In the pure mean-field regime $\beta = 0$, the singularity of the potential accommodated by the method of proof is more relevant.

As regards the other two points, first recall that non-zero $\bA$ or interaction potentials with attractive parts in general lead to the minimizers of the limit problem to be non-unique. The inclusion of non-trivial gauge fields in addition forces one to make crucial use of bosonic symmetry. Indeed, for $\bA \equiv 0$, minimizers of the $N$-body problem \emph{without} symmetry constraint are automatically bosonic wave-functions (see~\cite[Section~3.2.4]{LieSei-09} or Theorem~\ref{thm:bos min} below). Thus one may drop the constraint if it turns out to be convenient. This simplification is not possible in the presence of a non-zero $\bA$: examples are known~\cite{Seiringer-03} of situations where bosonic ground states differ from ground states without symmetry.  

As for attractive interactions, the main issue is that in a NLS limit $\beta >0$ they render even the question of whether the energy is bounded below proportionally to $N$ highly non-trivial. We shall need to distinguish between several notions of stable potentials:

\begin{definition}[\textbf{Stability notions}]\label{def:stability}\mbox{}\\
We say that the unscaled interaction potential $w$ from~\eqref{eq:scaled potential} is 
\begin{enumerate}
 \item repulsive if $w\geq 0$.
 \item classically stable if there is a constant $C>0$ such that for all $\bx_1,\ldots,\bx_N \in\R^d$
 $$ \sum_{1\leq i < j \leq N} w (\bx_i - \bx_j) \geq - C N.$$
 \item Hartree-stable if 
 $$ \inf_{u\in H^1 (\R^d)} \left( \frac{\iint_{\R^d \times \R^d} |u(\bx)| ^2 w(\bx-\by) |u(\by)| ^2 \mathrm{d}\bx \mathrm{d}\by }{\int_{\R^d} |u|^2 \int_{\R^d} |\nabla u|^2}\right) > -1$$
 \item NLS-stable if it satisfies Assumption~\ref{asum:NLS stab}, or GP-stable if it has positive scattering length. 
\end{enumerate}
We say that the many-body Hamiltonian~\eqref{eq:intro Schro op bis} is stable
\begin{enumerate}
 \item of the first kind if $H_N \geq - C(N)$ as an operator, for some function $C(N) \geq 0$.
 \item of the second kind if $H_N \geq -C N$ as an operator, for some constant $C>0$.
\end{enumerate}
\end{definition}

In the above, ``bounded below as an operator'' has the usual meaning that $A\geq B$ if 
$$ \langle \psi, A \psi \rangle \geq \langle \psi, B \psi \rangle$$
for any vector $\psi$. In particular the spectrum of $H_N$ (and thus the ground state energy) is bounded below, i.e. no sequence of state can give an arbitrarily negative energy. 

The four different stability notions for the potential $w$ are listed from the most stringent to the less. ``Repulsive'' speaks for itself: any particle encounter costs energy and (recall that $w\to 0$ at infinity) particles prefer to be as far apart as can be. 

``Classically stable''~\cite{Ruelle} means the potential does not need the quantum kinetic energy/uncertainty principle to obtain a bound from below. Note that if $w$ is classically stable then the scaled $w_N$ in~\eqref{eq:scaled potential} also is provided $\beta < 1/d$. In the dilute regime $\beta > 1/d$ this is no longer the case. A typical example of a classically stable interaction which needs not be repulsive is a potential of positive type: $\widehat{w} \geq 0$. We will use this notion in 3D mostly, where (because of scaling, cf the comments after Assumption~\ref{asum:NLS stab}) we cannot hope that the quantum kinetic helps stabilizing the system.

``Hartree-stable'' is a refinement of NLS-stable: it guarantees that the $N$-dependent Hartree functional~\eqref{eq:Hartree func N} with potential $w_{N,\beta}$ is bounded below independently of $N$. It could indeed very well be that the limit NLS functional is stable, but not the intermediary Hartree functional one obtains in the first place by inserting a \index{Trial state, Hartree}{factorized ansatz} in the many-body energy. In 3D (because of scaling again) Hartree stability boils down to classical stability. We will use the notion in 2D mostly, where it is implied~\cite{LewNamRou-14c} by the more transparent 
\begin{equation}\label{eq:Hartree stable}
\int_{\R^2} |w_-| < a^*  
\end{equation}
where $a^*$ is the optimal constant in~\eqref{eq:GagNir} and $w_-$ the negative part of $w$.

We refer to~\cite{LieSei-09} for further discussion of the notions of stability for the many-body Hamiltonian. Stability of the first kind is our basic starting point: the many-body Hamiltonian makes sense and its spectrum has a lower bound. Stability of the second kind is part of our conclusions. As we shall see, depending on the stability of the interaction $w$ it ranges from ``trivial statement'' to ``the main thing we have to prove''.

\medskip

\noindent\textbf{Comments on conclusions.} The first conclusion we would like to obtain, which serves as a minimal requirement, is the convergence of the ground state energy 
$$ \frac{E(N)}{N} \underset{N\to\infty}{\to} E^{\rm MF}$$
where the \index{Mean-field energy functional}{mean-field energy} $E^{\rm MF}$ is the Hartree, non-linear Schr\"odinger or Gross-Pitaevskii energy, depending on the regime (value of $\beta$) under consideration. In this text we shall discuss only convergence and not error estimates, although their precision might also be considered a criterion of efficiency for the methods. We do not consider it explicitly however, for the error estimates given by the methods we discuss would most of the time be far from the expected optimal ones.  

The obvious next question after the convergence of the energy concerns the convergence of ground state themselves. Namely, how well can we expect an actual many-body ground state $\Psi_N$ to be approximated by a pure \index{Bose-Einstein condensate (BEC)}{Bose-Einstein condensate} (BEC)
$$ \Psi_N \approx u ^{\otimes N}$$
with $u$ minimizing the appropriate mean-field energy functional ? The first remark is that these objects live in a $N$-dependent space, so that one cannot expect a convergence. Both are $L^2$-normalized, so a meaningful estimate could be 
$$ \norm{\Psi_N - u ^{\otimes N}}_{L^2 (\R^{dN})} \underset{N\to \infty}{\to} 0.$$
This is unfortunately \emph{wrong} as soon as the interaction $w$ is non-zero (see Section~\ref{sec:next} below for further references). The reason is that we physically do not expect that all particles are in the same quantum state, merely a vast majority of them. But  
$$ u^{\otimes (N-1)} \otimes_{\sym} v \perp u^{\otimes N} $$
in $L^2 (\R^{dN})$ as soon as $v \perp u$ in $L^2 (\R^d)$. Thus, even a \emph{single} particle in a state orthogonal to $u$ would ruin a norm estimate as above. We thus definitely have to make sense of the idea that ``\emph{most} particles are in the same quantum state''.

This is precisely what \index{Reduced density matrix}{reduced density matrices} 
$$ \Gamma_N ^{(k)} = {N \choose k} \tr_{k+1 \to N} |\Psi_N \rangle \langle \Psi_N|$$
are good for. We aim at a statement of the form 
$$ \frac{1}{{N \choose k}} \Gamma_N ^{(k)} \underset{N\to \infty}{\to} | u^{\otimes k} \rangle \langle u ^{\otimes k} |, $$
say for $k$ fixed when $N\to \infty$. Modulo the ${N \choose k}^{-1}$ normalization this is saying that the reduced density matrix $\Gamma_N ^{(k)}$ of the full many-body state $\Gamma_N = |\Psi_N \rangle \langle \Psi_N |$ converges to that of a pure BEC. Colloquially this means that most particles $k$-tuples are in the state $u ^{\otimes k}$. The convergence above should be strong in trace-class norm (the natural topology for reduced density matrices).

The above implicitly assumes that there is a unique mean-field minimizer $u$ that the interacting particles would want to populate. If that is not the case, we should lower our expectations to a statistical statement: the reduced density matrix $\Gamma_N ^{(k)}$ is close to a convex combinations of projectors $| u^{\otimes k} \rangle \langle u ^{\otimes k} |$ on mean-field minimizers: 
$$ \frac{1}{{N \choose k}} \Gamma_N ^{(k)} \underset{N\to \infty}{\to} \int_{\cM ^{\rm MF}} | u^{\otimes k} \rangle \langle u ^{\otimes k} |\mathrm{d}\mu (u) $$
where $\mu$ is a probability measure (independent of $k$) over the set $\cM ^{\rm MF}$ of all mean-field minimizers (see the previous section). 

\medskip 

\noindent \textbf{The meta-statement.} Now that we have lucidly evaluated what our expectations should be, we are ready to state a main theorem. We shall be slightly informal, the goal being to put all the different situations that we shall discuss under a single standard roof that we will afterwards decorate in different manners:

\begin{theorem}[\textbf{Scaling limit of bosonic ground states, generic statement}]\label{thm:main}
Let $\bA,V,w$ be (say smooth) functions such that $H_N$, the many-body Hamiltonian~\eqref{eq:intro Schro op bis} makes sense as a self-adjoint operator on $L^2 (\R ^{dN},\C)$. Assume that $V$ is trapping 
\begin{equation}\label{eq:trapping bis} 
V(\bx) \underset{|\bx|\to \infty}\to + \infty
\end{equation}
so that $H_N$ has compact resolvent and discrete spectrum. Let then $E(N)$ be its lowest eigenvalue and $\Psi_N$ an associated eigenfunction (ground state energy and ground state, respectively). 

Let $\beta \geq 0$ be as large as we can afford (but $\beta \leq 1$ if $d= 3$). If $\beta >0$ and $d\geq 2$, further assume Hartree stability for $d=2$ or classical stability (for $d=3$) (see Definition~\ref{def:stability}).

Let $\cEMF,\EMF,\uMF,\MMF$ be respectively the mean-field functional, its ground-state energy, one of its ground states and the set of its ground states respectively. Here $\rm{MF}$ stands (see Section~\ref{sec:NLS}) for $\rm{H}$ if $\beta = 0$, for $\rm{NLS}$ if $\beta >0$ (and $\beta < 1$ when $d=3$), for $\rm{GP}$ if $\beta = 1$ and $d=3$.  

We have, in the limit $N \to + \infty$:

\smallskip 

\noindent\textbf{Convergence of the energy:} 
\begin{equation}\label{eq:conv ener}
\frac{E(N)}{N} \to \EMF. 
\end{equation}

\smallskip 

\noindent\textbf{Convergence of reduced density matrices:} let $\Gamma_N ^{(k)},k\geq 0$ be the reduced density matrices of a many-body ground state $\Psi_N$. There exists a Borel\footnote{We use the $L^2$ topology on wave-functions.} probability measure $\mu$ on $\MMF$ (independent of $k$) such that 
\begin{equation}\label{eq:conv states}
{N \choose k} ^{-1} \Gamma_N ^{(k)} \to \int_{\MMF} |u ^{\otimes k} \rangle \langle u ^{\otimes k} | \mathrm{d}\mu(u) 
\end{equation}
along a subsequence (independent of $k$).
\end{theorem}

\begin{proof}[Comments]\mbox{}\\
If the mean-field minimizer is unique,~\eqref{eq:conv states} for $k=1$ implies that the first reduced density matrix has an eigenvalue of order $N$, which is \index{Bose-Einstein condensate (BEC)}{Bose-Einstein condensation}. Indeed, the eigenvalue of the one-body density matrix should be identified as the number of particles occupying the corresponding eigenfunctions (the notion originates in~\cite{PenOns-56}).
\end{proof}

\section{Outline}

The rest of the text, devoted to different versions and proofs of the main meta-statement above, is organized in four chapters. The rationale is to offer as exhaustive as possible a panorama of the tools relevant to deal with the problems defined above. Many of those tools have a broader interest and range of application. I very much hope this review can serve as an introduction to those, even though their full power will not necessarily be revealed here.

For pedagogical reasons I wanted to introduce new methods one at a time, in (what I felt was the) order of increasing mathematical sophistication. A smoother presentation is hopefully achieved by separating the regimes of our interest in three families: mean-field limits, dilute limits, Gross-Pitaevskii limit. I also felt the need of separating in two parts the discussion of the basic mean-field limits. This leads to the following organization: 

\begin{itemize}
 \item Chapter~\ref{cha:MF 1} presents two classes of methods that deal with mean-field limits of bosonic ground states without really using the quantum character of the problem. They have their limitations, but also an appealing simplicity. 
 \item Chapter~\ref{cha:MF 2} presents the basic methods to be used throughout all the rest of the review (with the notable exception of Sections~\ref{sec:thermo} and~\ref{sec:LDA}), in the simpler context of mean-field limits. Here we fully use the fact that we are dealing with bosonic quantum mechanics.
 \item Chapter~\ref{cha:dilute} tackles the dilute limit. As discussed above, there is a qualitative change in the physics as compared with the mean-field limit. This is reflected in the mathematical attack on these problems: basically we supplement the tools of Chapter~\ref{cha:MF 2} with refined estimates allowing to tame the singular nature of interactions.
 \item Chapter~\ref{cha:GP} addresses the Gross-Pitaevskii limit, that we see as a special kind of dilute limit where two-particle correlations have to be extracted. The tools to achieve this are to a large extent superimposed on the techniques of the previous two chapters. 
\end{itemize}

Before embarking on this program, let me mention a few connections to topics not covered here.

\section{Connections and further topics}\label{sec:next}

The main body of these notes is limited in scope, a necessary consequence of my choice to be rather precise  on the topics I did choose to cover. The reader should however bear in mind that the ideas and tools we shall encounter belong to a broader context and have wider applications. 

Below is a brief review of the literature on topics very much related to these notes. I apologize for not being able to cite all the relevant material (I have favored review texts when they are available), and for connections I might be unaware of. References for quantum statistical mechanics in general are~\cite{Ruelle,BraRob1,BraRob2,GusSig-06}, and for Bose systems in particular~\cite{Verbeure-11,LieSeiSolYng-05,PitStr-03,PetSmi-01,Rougerie-spartacus,Rougerie-LMU,Seiringer-06}. 

\subsection{Bogoliubov theory}

Several hints on this topic are provided in Sections~\ref{sec:Bogoliubov}-\ref{sec:Bogoliubov low}, but for reasons of space I had to refrain from trying a full review. 

The question is ``what is the next-to-leading order in the large $N$ expansion, after the mean-field contribution''? The answer is: \index{Bogoliubov theory}{Bogoliubov's theory}, as formulated in 1947~\cite{Bogoliubov-47}, or maybe appropriately generalized (see~\cite{AngVerZag-92,CorDerZin-09,ZagBru-01} for review). A particularly appealing heuristics is that, since the minimizers of the mean-field functional gives the leading order, the next order ought to be given by perturbing around minimizers. By definition the first variation vanishes, and thus the Hessian of the mean-field functional at its minimum is the relevant object. This is a one-body object that one can second-quantize (cf Section~\ref{sec:second quant}) to obtain a quadratic (in annihilators/creators) many-body Hamiltonian describing the quantum fluctuations around the \index{Bose-Einstein condensate (BEC)}{Bose-Einstein condensate}. Many results are known on such bosonic Bogoliubov Hamiltonians~\cite{BacBru-16,BruDer-07,Derezinski-17,NamNapSol-16}.

In recent years, important progress has been made in this direction. It is now known rigorously that Bogoliubov's theory indeed describes the next-to-leading order of many-body minimizers, as well as the low-lying excitation spectrum (first few eigenvalues above the ground state energy). See~\cite{Seiringer-11,GreSei-13,DerNap-13,LewNamSerSol-13,NamSei-14} for results in the mean-field limit,~\cite{BocBreCenSch-17} for the dilute limit, and~\cite{BocBreCenSch-18} for the Gross-Pitaevskii limit. More recently, the expansion of the low-lying energy spectrum~\cite{BosPetSei-20} and of the ground state density matrix~\cite{NamNap-20} have been considered in the mean-field regime, beyond Bogoliubov theory. Another approach to such large $N$ expansions is in~\cite{Pizzo-15a,Pizzo-15b,Pizzo-15c}. 

Bogoliubov's theory also plays a role in the ground-state energy expansion of the extended homogeneous Bose gas, in the low-density regime (the Lee-Huang-Yang correction~\cite{ErdSchYau-08,YauYin-09}). By this we mean that the thermodynamic limit is taken first, and only then, ideally, the density sent to $0$ (see Section~\ref{sec:thermo}, this is even harder than the Gross-Pitaevskii limit). This has been justified first in simplified settings~\cite{GiuSei-09,BriSol-19} (``mean-field-like'' and ``dilute-like'' settings, respectively) before a full derivation of the Lee-Huang-Yang formula was given in~\cite{BriFouSol-19,FouSol-19} (see also~\cite{Fournais-20}). To the best of my knowledge, the first rigorous justification of Bogoliubov's theory, on a particular example, is to be found in~\cite{LieSol-01,LieSol-04,Solovej-06}. 

\subsection{Dynamics}

There is a huge literature devoted to the time-evolution of (approximately) Bose-condensed initial data along the \index{Many-body Schr\"odinger equation}{many-body Schr\"odinger flow}. This is the natural dynamical pendant of the ground-state theory I review below. Reviews are in~\cite{BenPorSch-15,Golse-13,Schlein-08,Spohn-12}. The message is that an initially factorized wave-function stays approximately factorized after time-evolution. There are by now many methods to prove this and I apologize for inevitable\footnote{A quick search in my (not exhaustive) folders reveals more than 120 papers on this topic.} omissions in the brief list below:
\begin{itemize}
 \item the BBGKY hierarchy approach~\cite{Spohn-80,AdaBarGolTet-04,ErdYau-01,ElgErdSchYau-06,ErdSchYau-07,ErdSchYau-09,ErdSchYau-10,BarErdGolMauYau-02,BarGolMau-00,CheHaiPavSei-13,CheHaiPavSei-14,CheHol-13,CheHol-15,CheHol-16,ChePav-11,ChePav-14}.
 \item Pickl's method of ``Gr\"onwalling'' the number of excited particles~\cite{Pickl-10,Pickl-11,Pickl-15,KnoPic-10,JebLeoPic-16,JebPic-17}.
 \item the \index{Coherent states}{coherent states} method~\cite{AmmBre-12,Hepp-74,GinVel-79,GinVel-79b,RodSch-09,BenOliSch-12,CheLeeSch-11,BreSch-19}.
 \item the ``Egorov method''~\cite{FroGraSch-07,FroKnoPiz-07,AnaSig-11}.
 \item the Wigner measure method~\cite{AmmNie-08,AmmNie-09,AmmNie-11,AmmNie-15,AmmFalPaw-16}.
 \item the Lieb-Robinson bounds approach~\cite{ErdSch-09}.
 \item the optimal transport method~\cite{GolPau-15,GolMouPau-16}.
\end{itemize}

More recently, the derivation of the dynamical counterpart of \index{Bogoliubov theory}{Bogoliubov's theory} has also attracted attention~\cite{BreNamNapSch-17,BocCenSch-15,LewNamSch-14,MitPetPic-16,GriMacMar-10,GriMacMar-11,GriMac-12,NamNap-15,NamNap-16,NamNap-17}. In addition to giving the next-to-leading order in the dynamical setting, it can also serve as a means to control quantum fluctuations, and thereby derive the mean-field equation quantitatively, an approach pioneered in~\cite{Hepp-74}. Mean-field dynamics beyond Bogoliubov theory are studied in~\cite{BosPetPicSof-19}.
 
\subsection{Positive temperature}

The scaling limit of positive temperature bosonic equilibria is not a topic as developed as that of zero-temperature equilibria and the time evolution thereof. Lots of things remain poorly understood, in particular the holy grail, a proof of \index{Bose-Einstein condensate (BEC)}{Bose-Einstein condensation} in the thermodynamic limit, seems way out of reach, except in very special cases~\cite{AizLieSeiSolYng-04,KenLieSha-88}. Here is a selection of papers:
\begin{itemize}
 \item free-energy estimates for dilute gases in the thermodynamic limit~\cite{DeuMaySei-19,Seiringer-06,Seiringer-08,Yin-10}.
 \item rigorous bounds on the critical temperature for BEC~\cite{SeiUel-09,BetUel-10}.
 \item proofs of BEC in scaling limits~\cite{DeuSeiYng-18,DeuSei-19}.
 \item derivation of classical field theory from Gibbs states in a mean-field limit~\cite{FroKnoSchSoh-16,FroKnoSchSoh-17,FroKnoSchSoh-20,LewNamRou-14d,LewNamRou-17,LewNamRou-18b,LewNamRou-18_2D,LewNamRou-20,Sohinger-19}.
\end{itemize}

The latter topic is a natural extension of the derivation of ground states and the time evolution of factorized data (see~\cite{FroKnoSchSoh-20b,Rougerie-xedp15,Rougerie-SMF18,LewNamRou-ICMP} for more informal accounts). ``Classical field theory'' here means a measure of the form $\exp(-\cEMF [u])du$ on one-body wave-functions $u$. This is the natural ``positive temperature equilibrium'' of mean-field theory, and indeed, once properly defined, an invariant of the non-linear Schr\"odinger evolution. One difficulty lies in actually defining this object when the one-body state space is infinite dimensional. In finite dimension the problem is simpler, see~\cite{Knowles-thesis,Gottlieb-05} and~\cite[Appendix~B]{Rougerie-LMU,Rougerie-cdf}.

\subsection{Bosons in special settings and other approximations}

Various generalizations of the setting we discussed are of interest. In particular, if another physically relevant limit is superimposed to the mean-field/dilute/GP one, it is of interest to investigate whether and how the limits commute. Here is a selection of topics: 

\begin{itemize}
\item Reduced dimensionalities~\cite{LieSeiYng-04,LieSeiYng-03,SchYng-07,SeiYin-08}.
 \item Multi-components gases~\cite{AnaHotHun-17,MicOlg-16,MicOlg-17,Olgiati-17,MicNamOlg-19,Nguyen-20}.
 \item Multiple well potentials~\cite{OlgRou-20,OlgRouSpe-20,RouSpe-16}.
 \item Large magnetic fields/rotation speeds~\cite{BruCorPicYng-08,LewSei-09,LieSeiYng-09}.
 \item Fragmented condensation~\cite{VdBLewPul-86,DimFalOlg-18}.
\end{itemize}

I also mention a few works on different types of approximations for the ground state of the Bose gas. There does not seem to be too many of those that one could turn into rigorous mathematics. The Lieb-Liniger model for 1D bosons with contact interactions is remarkable in being basically exactly soluble~\cite{LanHekMin-17,Girardeau-60,LieLin-63,Lieb-63}. An alternative (to mean-field/Gross-Pitaevskii/Bogoliubov) approximation scheme is studied in \cite{Lieb-63b,LieLin-64,LieSak-64,CarLieJau-19,CarLieLosJau-20,CarHolLieJau-20}. 

\subsection{Quasi-classical systems}

The topic of this review might be called ``semiclassical'': a macroscopic quantum system to some extent behaves classically\footnote{In that the non-commutativity of quantum fields is ignored.} at leading order. One could call ``quasi-classical'' a variant of this situation: a finite-size quantum system acquires an effective classical interaction by coupling to another, macroscopic system. The latter becomes classical in an appropriate limit, see~\cite{LieTho-97,DonVar-83,CorFal-18,CorFalOli-19,CorFalOli-19b,FraGan-17,FraSei-19,LeoRadSchSei-19,LeoMitSei-20,LeoPic-18} and references therein.  

\subsection{Classical mean-field and related limits}

The mean-field limit for Bose systems has a natural analogue in classical mechanics. For equilibrium states this is dealt with in~\cite{MesSpo-82,Kiessling-89,Kiessling-93,CagLioMarPul-92,KieSpo-99}, see~\cite[Chapter~2]{Rougerie-spartacus,Rougerie-LMU} for review and~\cite{Serfaty-15,Serfaty-17} for more advanced topics. As regards dynamics (derivation of Vlasov or Boltzmann equations, propagation of molecular chaos), see~\cite{Golse-13,Spohn-12,GalRayTex-14,Mischler-11,PulSim-16,Jabin-14,Spohn-80} for reviews. Note that in this context as well, mean-field (typically leading to Vlasov's equation) and low/density dilute limits (typically leading to Boltzmann-like equations) should be distinguished, cf the aforementioned references for discussion.

\subsection{Fermionic mean-field limits and beyond}

Finally, fermionic systems also have scaling limits of their own. Because of the Pauli principle, it turns out they are naturally coupled with semiclassical limits. This requires specific methods to couple the two types of limits. A selection of references is~\cite{Bach-92,BacBrePetPicTza-16,BenPorSch-14,ElgErdSchYau-04,LieSim-77b,PetPic-16,Thirring-81,FouLewSol-15,LieSolYng-94,LieSolYng-94b,LieSolYng-95,FouMad-19,LieYau-87,LewMadTri-19}. 

Because of the Pauli principle again, fermions have a tendency to avoid one another a lot more than bosons, which makes dilute limits (with few particle encounters) more tricky to study. If spin is taken into account however, dilute-type interactions between different spin components are particularly relevant and pose problems akin to some we encounter in this review~\cite{LieSeiSol-05,Seiringer-06b,FalGiaHaiPor-20}.

An interesting recent direction is to study rigorously the corrections to the energy of the homogeneous Fermi gas, in the mean-field limit~\cite{Benedikter-19,BenNapPorSchSei-19,BenNapPorSchSei-20,HaiPorRex-20}. It turns out that they are due to correlations of a special form that can be understood via \emph{bosonization}, i.e. the emergence of bosonic quasi-particles. The latter are effectively weakly correlated and hence described by a bosonic \index{Bogoliubov Hamiltonian}{(Bogoliubov) quadratic Hamiltonian}. 

Some exotic/hypothetical quasi-particles known as anyons interpolate between bosons and fermions. Deriving effective mean-field models for those is also of interest~\cite{LunRou-15,Girardot-19,GirRou-20}.

\chapter{Mean-field limits, classical mechanics methods}\label{cha:MF 1}

We start our grand tour of derivation of mean-field type results by reviewing methods whose inspiration is drawn from classical mechanics. They proceed by either manipulations of the interaction potential, seen as a classical energy (Section~\ref{sec:hamil}) or by recasting the problem as a \index{Classical statistical mechanics}{classical statistical mechanics} ensemble (Section~\ref{sec:class deF}). The common point is that we do not use the full quantumness of the problem (in particular, bosonic statistics) and thus will not get optimal results. The simplicity of the methods still makes them appealing.

\section{Hamiltonian-based methods}\label{sec:hamil}

By ``Hamiltonian-based'' we mean a method that uses one or several special assumptions on the basic Hamiltonian~\eqref{eq:intro Schro op bis}. There is simply not enough structure in~\eqref{eq:intro Schro op bis} as it stands to prove the general form of Theorem~\ref{thm:main} by simply manipulating the formal expression of the Hamiltonian. One must crucially rely on the structure of the space we act on, namely use bosonic symmetry in one way or another. 

By contrast, the methods of this section are insensitive to bosonic symmetry. They work only in restricted cases, but their relative simplicity makes them appealing. Without further ado, let us present the simplifying assumption we shall use:

\begin{assumption}[\textbf{\index{Positivity improving}{Positivity improving} case}]\label{asum:pos pre}\mbox{}\\
Pick $A\equiv 0$ and $V\geq 0$ in~\eqref{eq:intro Schro op bis}. In particular the one-body Hamiltonian $h = -\Delta + V$ satisfies
\begin{equation}\label{eq:pos prev}
\left\langle u | h | u \right\rangle_{L^2 (\R^d)} \geq \left\langle |u| \,, h \, |u| \right\rangle_{L^2 (\R^d)}
\end{equation}
and the associated heat flow $e^{-th}$ is positivity improving for any $t>0$. Namely it maps non-trivial non-negative functions to positive functions. 
\end{assumption}

\begin{proof}[Comments]
If we had included pseudo-relativistic effects, Laplacians would be replaced by fractional Laplacians. In the absence of magnetic fields~\eqref{eq:pos prev} holds also in this case, and the heat flow stays positivity preserving, so that the methods of this section generalize to this case. 

For the usual Laplacian, our case of concern,~\eqref{eq:pos prev} is essentially just the fact (used with $\bA \equiv 0$) that, writing $u = \sqrt{\rho} e^{i\varphi}$
\begin{equation}\label{eq:diamag}
 \int_{\R^d} |\left(-\im \nabla + \bA\right) u| ^2 = \int_{\R^d} |\nabla \sqrt{\rho}| ^2 + \int_{\R^d} \rho |\left(\nabla + \bA\right) \varphi| ^2 \geq \int_{\R^d} |\nabla \sqrt{\rho}| ^2. 
\end{equation}
See~\cite[Theorem~7.8]{LieLos-01}. 

As regards the positivity improving property, the \index{Trotter product formula}{Trotter product formula} roughly says the following (see~\cite{Simon-05} or~\cite{ReeSim4} for details): for self-adjoint operators $A$ and $B$
$$ \exp (A+B) = \lim_{n\to \infty} \left( \exp \left(\frac{A}{n}\right) \exp \left(\frac{B}{n}\right) \right)^n.$$
Applied to $e^{-th}$ this gives an approximate expression of its integral kernel:
\begin{multline}\label{eq:Trotter} 
e^{-th} (\bx;\by) = \lim_{n\to \infty} \\ \int_{\R^{d(n-1)}} e^{\frac{t}{n}\Delta} (\bx;\by_1) e^{-\frac{t}{n}V(\by_1)} e^{\frac{t}{n}\Delta} (\by_1;\by_2) \ldots  e^{\frac{t}{n}\Delta} (\by_{n-1};\by) e^{-\frac{t}{n}V(\by)} \mathrm{d}\by_1\ldots \mathrm{d}\by_{n-1}.
\end{multline}
Since the \index{Heat kernel}{heat kernel} 
\begin{equation}\label{eq:heat kernel}
e^{t\Delta} (\bx;\by) = \frac{1}{(4\pi t)^{d/2}} e^{-|\bx -\by|^2 /(4t)}
\end{equation}
is positive, it follows that also $e^{-th} (\bx;\by) > 0$, so that $e^{-th}$ maps non-negative functions to positive functions. 

One may recognize that~\eqref{eq:Trotter} at finite $n$ is a discretization of 
$$e^{-th} (\bx;\by) = \int e^{-\int_0^t V(\omega(t))dt}  \mathrm{d}\mu^t_{\bx,\by} (\omega) $$
where $\mathrm{d}\mu^t_{\bx,\by} (\omega)$ is the conditional Wiener measure, the probability density for a Brownian motion leaving $\by$ at time $0$ and reaching $\bx$ at time $t$ to follow the path $\omega:[0,t] \mapsto \R^d$. This is the \index{Feynman-Kac formula}{Feynman-Kac formula}, see~\cite{Simon-05} or~\cite{ReeSim4} again for precise statements.
\end{proof}

Assuming the above, one can give a relatively soft proof of part of our main Statement~\ref{thm:main}. Recall the notation from Section~\ref{sec:NLS}: the \index{Hartree energy functional}{Hartree functional} is 
$$
\cEh[u] = \int_{\R^d} \left| \nabla u \right| ^2 + V |u| ^2 + \frac{1}{2} \iint_{\R^d \times \R ^d} |u(\bx)| ^2 w (\bx - \by) |u (\by)| ^2 \mathrm{d}\bx \mathrm{d}\by
$$
with minimum (under unit $L^2$ mass constraint) $\Eh$ and minimizer $\uh$. The \index{Non-linear Schr\"odinger energy functional}{NLS functional} is 
$$
\cEnls[u]:= \int_{\R^d} \left| \nabla u \right| ^2 + V |u| ^2 + \frac{b_w}{2} \int_{\R^d} |u(\bx)| ^4  \mathrm{d}\bx.
$$
with minimum (under unit $L^2$ mass constraint) $\Enls$, minimizer $\unls$ and the notation 
$$
b_w := \int_{\R^d} w.  
$$

\begin{theorem}[\textbf{Mean-field limit, positivity improving case}]\label{thm:MF pos pre}\mbox{}\\
We make Assumptions~\ref{asum:pots} and~\ref{asum:pos pre} (i.e. $\bA\equiv 0$ in~\eqref{eq:intro Schro op bis}). We also assume Hartree stability (see Definition~\ref{def:stability}). Let $0 \leq \beta < 1/d$. 

If $\beta = 0$ let the mean-field energy and ground state(s) $\EMF,\uMF$ stand for the Hartree objects recalled above. If $\beta >0$ replace them by the corresponding NLS objects. 

We have the following, in the limit $N\to \infty$ 

\noindent\textbf{Convergence of the energy:} the lowest eigenvalue $E(N)$ of~\eqref{eq:intro Schro op bis} satisfies
$$ \frac{E(N)}{N} \underset{N\to \infty}{\to} \EMF $$

\smallskip 

\noindent\textbf{Convergence of reduced density matrices:} if in addition the mean-field ground state $\uMF$ is unique, 
\begin{equation}\label{eq:conv states pos pre}
{N \choose k} ^{-1} \Gamma_N ^{(k)} \to  \left|(\uMF) ^{\otimes k} \right\rangle \left\langle (\uMF) ^{\otimes k} \right| . 
\end{equation}
where $\Gamma_N ^{(k)},k\geq 0$ are the reduced density matrices of a many-body ground state $\Psi_N.$
\end{theorem}

In this chapter we shall follow the pioneering works~\cite{BenLie-83,LieYau-87}, with additions from~\cite{Lewin-ICMP} and~\cite{LewNamRou-18a}. See also~\cite{BauSei-01,Nguyen-19,Nguyen-20,SeiYngZag-12} where similar tools are used. The proof we reproduce below for general interaction potentials does not seem to have appeared before~\cite{Lewin-ICMP}. In~\cite{BenLie-83,LieYau-87} the interaction potential has positive Fourier transform (Coulomb potential) or negative Fourier transform (Newtonian potential) respectively. In fact, if the interaction has positive Fourier transform one does not need Lemma~\ref{lem:Levy} below and one can prove stronger results~\cite{SeiYngZag-12} with variants of the methods below.

\subsection{Toolbox}

We start with the main consequence of assuming a positivity-improving heat flow. The next statement is~\cite[Theorem~XIII.47]{ReeSim4}.

\begin{theorem}[\textbf{Ground states of Schr\"odinger operators}]\label{thm:GS schro}\mbox{}\\
Assume $V\in L^1_{\rm loc}(\R^n)$ is trapping
$$ 
V (\bx) \underset{|\bx|\to \infty}\to \infty. 
$$
Let $H = -\Delta_{\R^n} + V$, seen as a self-adjoint operator (Friedrichs extension of the associated quadratic form) on $L^2 (\R^n), n\geq 1$. Then $H$ has discrete spectrum. The eigenspace corresponding to the lowest eigenvalue is one-dimensional, of the form $\left\{ c \Phi \right\}_{c\in \C}$ with $\Phi$ a positive function.   
\end{theorem}

\begin{proof}[Comments]
\cite[Section~XIII.12]{ReeSim4} contains much more general results, in particular, the trapping assumption for the potential $V$ is not needed. It is just convenient to ensure the existence of the lowest eigenvalue. In our applications we will use locally bounded potentials, which can be assumed positive by just shifting the energy reference. 

Briefly, the idea of the proof is that, because of~\eqref{eq:pos prev}, if $\Psi$ is a ground state, so is $|\Psi|$. Then, either $|\Psi| - \Psi \equiv 0$ or it is also a ground state, and hence an eigenfunction of $e^{-tH}$. 

But $e^{-tH}$ is \index{Positivity improving}{positivity improving}. One can see this using the \index{Trotter product formula}{Trotter product formula} as above (or the \index{Feynman-Kac formula}{Feynman-Kac formula}, see~\cite[Theorem VIII.30]{ReeSim1} or~\cite[Theorem~1.1]{Simon-05}) to write the integral kernel of $e^{-tH}$ in terms of that of $e^{t\Delta}$, which is explicit and positive.  

A non-negative eigenfunction of a positivity improving operator may not vanish on a positive measure set (i.e. the heat flow spreads mass instantaneously, so a stationary state thereof is non-zero almost everywhere). Thus either $|\Psi| - \Psi \equiv 0$ or $|\Psi| - \Psi > 0$ almost everywhere. In the latter case we must have $|\Psi| = - \Psi$ almost everywhere. 

This shows that any candidate eigenfunction for the lowest eigenvalue can be chosen positive almost everywhere. Clearly, there cannot be two such orthogonal positive functions. 
\end{proof}

As far as the mean-field limit is concerned, the essence of Assumption~\eqref{asum:pos pre} is that it allows to ignore bosonic symmetry, as per the:

\begin{theorem}[\textbf{Unrestricted minimizers are bosonic}]\label{thm:bos min}\mbox{}\\
Let $\Psi_N$ be an absolute ground state for~\eqref{eq:intro Schro op bis}, namely
$$ \langle \Psi_N |H_N| \Psi_N \rangle = \min\left\{ \langle \Psi_N |H_N| \Psi_N \rangle, \Psi_N \in L^2 (\R^{dN}), \int_{\R^{dN}} |\Psi_N| ^2  = 1 \right\}.$$
Under Assumption~\eqref{eq:pos prev}, $\Psi_N$ must be bosonic, i.e. $\Psi_N \in L^2 _{\sym} (\R^{dN})$. 
\end{theorem}

\begin{proof}
A simple proof is in~\cite[Section~3.2.4]{LieSei-09}. Here is another, less simple (in that it uses~Theorem~\ref{thm:GS schro} above). It follows from~\eqref{eq:pos prev} that 
$$ \langle \Psi_N, H_N \Psi_N \rangle \geq \langle |\Psi_N|\,, H_N \,|\Psi_N| \rangle,$$
hence $|\Psi_N|$ must also be a minimizer and there exists a $c \in \C,|c|=1$ such that $\Psi_N = c |\Psi_N|$. 

On the other hand, $H_N$ commutes with all the unitaries $U_\sigma$ permuting particle labels~\eqref{eq:perm unit}. Thus, a ground state (unique modulo a constant phase factor) must satisfy for any permutation $\sigma$
$$ U_\sigma \Psi_N = e_{\sigma} \Psi_N$$
for numbers $e_\sigma \in \{-1,1\}$. Applying this to $|\Psi_N|\geq 0$ all the $e_\sigma$'s must be $1$, and thus $|\Psi_N|$ be symmetric under particle label exchange. As per the above, any minimizer $\Psi_N$ must also be.
\end{proof}

It is often useful (but not strictly needed) to know another \index{Hoffmann-Ostenhof$\,^2$ inequality}{consequence} of~\eqref{eq:pos prev}, first derived\footnote{The square in the attribution is the (by now standard) way to give credit to both {authors} of~\cite{Hof-77}.} in~\cite{Hof-77}.

\begin{lemma}[\textbf{Hoffmann-Ostenhof$\,^2$ inequality}]\label{lem:HO ineq}\mbox{}\\
Let $\Psi_N \in L^2 (\R^{dN})$ and $\rho_{\Psi_N}$ be the corresponding one-particle density
\begin{equation}\label{eq:dens unsym}
\rho_{\Psi_N} (\bx) := \sum_{j=1} ^N \int_{\R^{d(N-1)}} |\Psi_N (\bx_1,\ldots,\bx_{j-1}, \bx, \bx_{j+1},\ldots,\bx_N)|^2 \mathrm{d}\bx_1,\ldots,\mathrm{d}\bx_{j-1}, \mathrm{d}\bx_{j+1},\ldots,\mathrm{d}\bx_N. 
\end{equation}
If $h$ satisfies~\eqref{eq:pos prev} then 
\begin{equation}\label{eq:HofOst}
\left\langle \Psi_N \,, \sum_{j=1} ^N h_{j} \, \Psi_N \right\rangle_{L^2 (\R^{dN})} \geq \left\langle \sqrt{\rho_{\Psi_N}}\, , h \, \sqrt{\rho_{\Psi_N}} \right\rangle_{L^2 (\R^d)}. 
\end{equation}
\end{lemma}

\begin{proof}
The following proof is from~\cite{Lewin-ICMP}. We do not assume $\Psi_N$ to have any symmetry, so we extend the definition~\eqref{eq:def red mat} of the one-body density matrix: 
\begin{equation}\label{eq:red mat nonsym}
 \gamma_{\Psi_N} := \sum_{j=1} ^N \tr_{\neq j} \left( |\Psi_N \rangle \langle \Psi_N | \right) 
\end{equation}
where the symbol $\tr_{\neq j}$ means taking the partial trace with respect to all degrees of freedom but the $j$-th. Observe that in terms of integral kernels 
\begin{equation}\label{eq:red mat dens}
 \gamma_{\Psi_N} (\bx;\bx) = \rho_{\Psi_N} (\bx). 
\end{equation}
The desired inequality follows by noting that for any real-valued functions $v_1,\ldots,v_k$
$$ \sum_{k=1} ^K \langle v_k,  h  v_k \rangle \geq \left\langle \left(\sum_{k=1} ^K |v_k| ^2 \right) ^{1/2},  h \left(\sum_{k=1} ^K |v_k| ^2 \right)^{1/2} \right\rangle.$$
Indeed~\eqref{eq:pos prev} being valid for complex valued functions implies 
$$ \langle f | h | f \rangle + \langle g | h | g \rangle = \langle (f +\im g) | h | (f +  \im g) \rangle \geq \left\langle \sqrt{ |f| ^2 + |g|^2}  | h | \sqrt{ |f| ^2 + |g|^2} \right\rangle$$
for any real-valued functions $f,g$, and it suffices to iterate this inequality.

Then, using the spectral decomposition of $\gamma_{\Psi_N}$
$$ \gamma_{\Psi_N} = \sum n_k |u_k \rangle \langle u_k | $$
we have (the first equality is similar to~\eqref{eq:def red mat dual}) 
\begin{align*}
 \left\langle \Psi_N \,, \sum_{j=1} ^N h_{j} \, \Psi_N \right\rangle_{L^2 (\R^{dN})} &= \tr \left( h \gamma_{\Psi_N} \right) \\
 &= \sum_{k} n_k \langle u_k \,, h \, u_k \rangle \\
 &\geq \left\langle \sqrt{\sum_{k} n_k |u_k| ^2 }  \,, h \,  \sqrt{\sum_{k} n_k |u_k| ^2 } \right\rangle
\end{align*}
and we recognize that~\eqref{eq:red mat dens} says that  
$$\sum_{k} n_k |u_k| ^2 = \rho_{\Psi_N}.$$
\end{proof}

Now we introduce two tools that will allow us to bound the interaction from below using one-body terms. We first have a lower bound for repulsive interactions (more precisely, interactions with positive Fourier transform), originating in~\cite{Onsager-39}:

\begin{lemma}[\textbf{\index{Onsager's inequality}{Onsager's inequality}}]\label{lem:Onsager}\mbox{}\\
Assume $w\in L^{\infty} (\R^d)$ has a non-negative Fourier transform. Then, for any $\bx_1,\ldots,\bx_N \in \R ^d$ and $\rho \in L^1 (\R^d)$ 
\begin{equation}\label{eq:Onsager}
\sum_{1\leq i < j \leq N} w (\bx_i - \bx_j) \geq \sum_{j=1}^N w\star \rho (\bx_j) - \frac{1}{2} \iint_{\R^d\times \R^d} \rho (\bx) w(\bx-\by) \rho(\by) \mathrm{d}\bx d\by -\frac{N}{2} w (0). 
\end{equation}
\end{lemma}

\begin{proof}
Modulo a density argument, we assume that $w$ is smooth. For any $\eta: \R^d \mapsto \R$ such that the integrals make sense we have
$$ \iint_{\R^d\times \R^d} \eta (\bx) w(\bx-\by) \eta(\by) \mathrm{d}\bx \mathrm{d}\by = \int_{\R^d} \widehat{w} (\bk) |\widehat{\eta} (\bk)| ^2 \mathrm{d}\bk \geq 0.$$
In particular we can apply this (with a slight abuse of notation, or another regularization argument) to the Radon measure 
$$\eta = \sum_{j= 1} ^N \delta_{x_j} - \rho,$$ 
and this gives the inequality. 
\end{proof}

It turns out that one can bound an arbitrary (regular) interaction from below, combining Onsager's lemma with a trick due to L\'evy-Leblond~\cite{LevyLeblond-69}, whose use in the context of our interest originates in~\cite{LieThi-84,LieYau-87}. 

\begin{lemma}[\textbf{\index{L\'evy-Leblond's trick}{L\'evy-Leblond's trick}}]\label{lem:Levy}\mbox{}\\
Let $w\in L^{\infty} (\R^d)$, which we write in the form 
$$ w = w_1 - w_2.$$
Let then $H_N$ be the associated many-body Hamiltonian~\eqref{eq:intro Schro op bis}, with $\bA\equiv 0$ as in Assumption~\ref{asum:pos pre}. Assume $N = 2M$ is even for simplicity and let $\bZ_M = (\bz_1,\ldots,\bz_M) \in \R^{dM}$. Define a new Hamiltonian on $L^2 (\R^{dM})$, parametrized by $\bZ_M$,
\begin{multline}\label{eq:LevLebHam} 
\widetilde{H}_M (\bZ_M) := \sum_{j=1} ^M \left( -2\Delta_{\bx_j} + 2 V(\bx_j) - \frac{2(2M-1)}{M}\sum_{k=1} ^{M} w_2 (\bx_j-\bz_k) \right) \\ + \frac{2(2M-1)}{M-1}\sum_{1\leq i< j \leq M} w_1 (\bx_i - \bx_j)
\end{multline}
with associated bosonic ground-state energy $E(\bZ_M)$. We have,
\begin{equation}\label{eq:Levy}
E(N) \geq \inf_{\bZ_M \in \R^{dM}} \left( E(\bZ_M) + \frac{2(2M-1)}{M-1}\sum_{1\leq k < \ell \leq M} w_2 (\bz_k-\bz_\ell) \right). 
\end{equation}
where $E(N)$ is the ground state energy of $H_N$.
\end{lemma}

In applications we will write the splitting $w=w_1 - w_2$ with $w_1$ and $w_2$ having non-negative Fourier transforms, so that one can apply Lemma~\ref{lem:Onsager} to both parts. In words, the lemma means that one can bound the energy from below by artificially splitting the particles in two groups: $M$ quantum particles feeling the Hamiltonian $\widetilde{H}_M$ and $M$ classical particles at positions $\bz_1,\ldots,\bz_M$. The lower bound is attained by first optimizing the position of the quantum particles, given those of the classical ones, and then optimizing over the positions of the classical particles. The point is that all the attraction (mediated by $w_2$) is now inter-species, so that from the point of view of the quantum particles (in the first minimization), it is a one-body term. One may also split the particles in groups more elaborately (in particular, assuming $N$ even is not necessary), see the aforementioned references.

\begin{proof}
Denote 
$$ C_1 = \frac{2(2M-1)}{M-1}, \quad C_2 = \frac{2(2M-1)}{M}.$$
For ease of notation we split the particle labels in two groups 
$$\bY_M = (\by_1,\ldots,\by_M) := (\bx_1, \ldots, \bx_M)$$
and 
$$ \bZ_M = (\bz_1,\ldots,\bz_M) := (\bx_{M+1}, \ldots,\bx_{2M}).$$
Let then $\Psi_N$ be a bosonic $N$-particle wave-function. Using its symmetry
$$ \left\langle \Psi_N | \sum_{j=1} ^N h_{\bx_j}  | \Psi_N \right\rangle = 2 \left\langle \Psi_N | \sum_{j=1} ^M h_{\by_j}  | \Psi_N \right\rangle$$
and 
$$
\int_{\R^{Nd}} \left( \sum_{1\leq i< j \leq N} w(\bx_i-\bx_j) \right) \left|\Psi_N (\bx_1,\ldots,\bx_N)\right|^2 = \int_{\R^{2Md}} W (\bY_M,\bZ_M) \left|\Psi_N (\bx_1,\ldots,\bx_N)\right|^2
$$
with (using the handy prefactor $(N-1)^{-1}$ in front of the interaction)
\begin{align*} 
W (\bY_M,\bZ_M) &= C_1 \sum_{1\leq i < k\leq M} w_1 (\by_i-\by_j) - C_2 \sum_{i=1} ^M \sum_{j=1} ^M w_2 (\by_i-\bz_j) + C_1 \sum_{1\leq i <j\leq M} w_2 (\bz_i - \bz_j) 
\\&= W_1 (\bY_M,\bZ_M) + W_2 (\bZ_M).
\end{align*}
Denote $\Gamma_M$ the mixed $M$-particles state obtained from $\Gamma_N= |\Psi_N \rangle \langle \Psi_N|$ by tracing out the $\bz$ variables. Then using the above notation we have 
\begin{align*}
\langle \Psi_N | H_N | \Psi_N \rangle &= \tr\left( \sum_{j=1} ^M h_{\by_j} \Gamma_M \right) + \int_{\R^{2M}} \Gamma_N (\bY_M,\bZ_M;\bY_M,\bZ_M)\mathrm{d}\bY_M \mathrm{d}\bZ_M  \\
&+ \langle \Psi_N | W_2 (\bZ_M) | \Psi_N \rangle\\
&= \int_{\bZ_M \in \R^{dM}} \tr \left( \left(\sum_{j=1} ^M h_{\by_j} + W_1 (\bY_M,\bZ_M) + W_2 (\bZ_M) \right) \widetilde{\Gamma_{\bZ_M}}  \right) \mathrm{d}\bZ_M  
\end{align*}
where we identify density matrices with their integral kernels, denote 
$$ \widetilde{\Gamma_{\bZ_M}} (\bY_M;\bY'_M) = \Gamma_N (\bY_M,\bZ_M;\bY'_M,\bZ_M)$$
and observe that the latter, integrated over $\bZ_M$, yields $\Gamma_N (\bY_M;\bY'_M)$. There remains to use, at fixed $\bZ_M$, the operator lower bound 
$$ \sum_{j=1} ^M h_{\by_j} + W_1 (\bY_M,\bZ_M) + W_2 (\bZ_M) \geq \inf_{\bZ_M \in \R^{dM}} \left( E(\bZ_M) + \frac{2(2M-1)}{M-1}\sum_{1\leq k < \ell \leq M} w_2 (\bz_k-\bz_\ell) \right)$$
and note that 
$$ \int_{\bZ_M} \tr \left( \widetilde{\Gamma_{\bZ_M}}\right) \mathrm{d}\bZ_M = 1.$$
\end{proof}

The previous lemmas will allow us to prove energy convergence. To deduce convergence of states, we will rely on a very simple observation, the \index{Feynman-Hellmann principle}{Feynman-Hellmann principle}. This applies to variational problems whose dependence on an extra parameter is of interest. The statement is roughly that ``derivative of the minimum $=$ derivative of the functional, evaluated at the minimizer''. 

We state this as a lemma, but prefer to stay vague as to the actual formulation. We find it more convenient to decline the (very simple) proof as needed in specific cases rather than have too abstract a formulation. 

\begin{lemma}[\textbf{Feynman-Hellmann principle}]\label{lem:FH}\mbox{}\\
For $\eps \in [-\eps_0, \eps_0]$ we are given a variational principle 
$$ E( \eps) = \min \left\{ \cE_{\eps} [u], \: u \mbox{ in some space }  \right\}.$$
Assume that, at $\eps = 0$ there is a unique minimizer $u_0$. Further assume that
$$ \eps \mapsto \cE_{\eps} [u_0]$$
is continuous and 
$$ \eps \mapsto E (\eps)$$
is differentiable at $\eps = 0$. Then
\begin{equation}\label{eq:FH general}
\left(\partial_{\eps} E(\eps)\right)_{|\eps = 0} =  \left(\partial_{\eps} \cE_\eps \right) [u_0].
\end{equation}
\end{lemma}

In most of our applications we will have the special form
$$
\cE_{\eps} [u] = \cE_{0} [u] + \eps \cP[u]
$$
so that $\eps \mapsto E( \eps)$ is concave as an infimum over linear function. 

\begin{argument}
Taking $\eps >0$ and using the variational principle we have
$$
\frac{E(\eps) - E (0)}{\eps} \leq \frac{\cE_{\eps} [u_0] - \cE_0 [u_0]}{\eps} 
$$
and 
$$ \frac{E(0) - E(-\eps)}{\eps} \geq \frac{\cE_{0} [u_0] - \cE_{-\eps} [u_0]}{\eps}.$$
Letting $\eps \to 0$ gives the desired result.
\end{argument}

\subsection{Applications}\label{sec:MF Onsager}

We now explain how to use the above tools to prove Theorem~\ref{thm:MF pos pre}. For pedagogical reasons, redundancies in the argument are not tracked down. We assume that $w\in L^\infty (\R^d)$ in the sequel, in order to apply Lemmas~\ref{lem:Onsager} and~\ref{lem:Levy} as they stand. If the potential is more singular, a suitable regularization allows to adapt the proof. Our main sources~\cite{BenLie-83,LieYau-87} deal with Coulomb/Newton interactions. For the former, the \index{Lieb-Oxford inequality}{Lieb-Oxford inequality}~\cite{Lieb-79,LieOxf-80,LieSei-09} can be used.

\begin{proof}[Proof of Theorem~\ref{thm:MF pos pre}] We separate the more difficult issue of convergence of states from that of energy convergence. 

\medskip 

\noindent \textbf{Energy convergence.} Using a \index{Trial state, Hartree}{trial stat}e of the form $\Psi_N = u ^{\otimes N}$ as in~\eqref{eq:ansatz} we immediately get 
\begin{equation}\label{eq:Hartree N}
 \frac{E(N)}{N} \leq \cEh_N [u] = \langle u | h | u\rangle + \frac{1}{2} \iint_{\R^d \times \R^d} |u(\bx)| ^2 N ^{d\beta} w \left( N^{\beta} (\bx-\by) \right)  |u(\bx)| ^2 \mathrm{d}\bx \mathrm{d}\by. 
\end{equation}
For $\beta = 0$, the right-hand side does not depend on $N$ and is exactly the \index{Hartree energy functional}{Hartree functional}, so we get the inequality
$$ \frac{E(N)}{N} \leq \Eh.$$
For $\beta >0$ one needs to prove that the infimum of the $N$-dependent functional $\cEh_N$ converges to $\Enls$. This is a rather simple exercise, the details of which we shall skip. It yields, for all $\beta >0$, and under the stated conditions on the interaction potential (they guarantee that $\cEh_N$ is bounded below independently of $N$) 
$$ \limsup_{N\to \infty}\frac{E(N)}{N} \leq \Enls.$$
To get a corresponding lower bound we use \index{L\'evy-Leblond's trick}{L\'evy-Leblond's trick} Lemma~\ref{lem:Levy} with $w_1 $ the inverse Fourier transform of $\widehat{w_N}_+$ and $w_2$ the inverse Fourier transform of $\widehat{w_N}_-$, where  
$$ w_N (\bx) = \frac{N ^{d\beta}}{N-1} w (N^{\beta} \bx)$$
and $+,-$ subscripts indicate positive/negative parts.This way both $w_1$ and $w_2$ have non-negative Fourier transforms. 

We now bound $\widetilde{H}_M (\bZ_M)$, as  defined in Lemma~\ref{lem:Levy}, from below. Let $M= N/2$ and $\Phi_M$ be a $M$-body wave function. Using Assumption~\ref{asum:pos pre} it follows from the \index{Hoffmann-Ostenhof$\,^2$ inequality}{Hoffmann-Ostenhof$\,^2$ inequality}~\ref{lem:HO ineq} that 
\begin{multline}\label{eq:use HO}
\Big\langle \Phi_M \Big| \sum_{j=1} ^M \left( -2\Delta_{\bx_j} + 2 V(\bx_j) - \frac{2(2M-1)}{M}\sum_{k=1} ^{M} w_2 (\bx_j-\bz_k) \right) \Big| \Phi_M \Big\rangle \\ 
\geq 2 \int_{\R^d} \left( \left| \nabla \sqrt{\rho_{\Phi_M}} \right|^2 + V \rho_{\Phi_M} \right) - \frac{2(2M-1)}{M}\sum_{k=1} ^{M} \int_{\R^d} w_2 (\bx-\bz_k) \rho_{\Phi_M} (\bx) \mathrm{d}\bx
\end{multline}
with $\rho_{\Phi_M}$ the one-particle density of $\Phi_M$, defined as in~\eqref{eq:dens unsym}. Then, if we use \index{Onsager's inequality}{Onsager's inequality} Lemma~\ref{lem:Onsager} with $\rho = \rho_{\Phi_M}$ we get 
\begin{align*} 
\left\langle \Phi_M \right| \sum_{1\leq i< j \leq M} w_1 (\bx_i - \bx_j) \left| \Phi_M \right\rangle &\geq \left\langle \Phi_M \right| \sum_{j= 1} ^M w_1 \star \rho_{\Phi_M} (\bx_j) \left| \Phi_M \right\rangle \\
&- \frac{1}{2} \iint_{\R^d \times \R^d} \rho_{\Phi_M} (\bx) w_1 (\bx - \by)  \rho_{\Phi_M} (\by) \mathrm{d}\bx \mathrm{d}\by - \frac{M}{2} w_1 (0)\\
&= \frac{1}{2} \iint_{\R^d \times \R^d} \rho_{\Phi_M} (\bx) w_1 (\bx - \by)  \rho_{\Phi_M} (\by) \mathrm{d}\bx \mathrm{d}\by - \frac{M}{2} w_1 (0),
\end{align*}
and thus 
\begin{multline*}
 \left\langle \Phi_M \right| \widetilde{H} (\bZ_M) \left| \Phi_M \right\rangle \geq N \cEh_N \left[ \sqrt{\eta} \right] - \frac{2(2M-1)}{M} \sum_{k=1} ^M w_2 \star \rho_{\Phi_M} (\bz_k) \\ + \frac{2(N-1)}{N} \iint_{\R^d \times \R^d} \rho_{\Phi_M} (\bx) w_2 (\bx - \by)  \rho_{\Phi_M} (\by) \mathrm{d}\bx \mathrm{d}\by - \frac{M}{2} w_1 (0)  
\end{multline*}
where we denote 
$$\eta = M ^{-1} \rho_{\Phi_M}.$$
On the other hand, using Lemma~\ref{lem:Onsager} with $\rho = (M-1) M ^{-1} \rho_{\Phi_M}$ we obtain
\begin{multline*}
 \frac{2(2M-1)}{M-1}\sum_{1\leq k < \ell \leq M} w_2 (\bz_k-\bz_\ell) \geq \frac{4(N-1)}{N} \sum_{k=1} ^M w_2 \star \rho_{\Phi_M} (\bz_k) \\ - \frac{(N - 1)(M-1)}{M^2} \iint_{\R^d \times \R^d} \rho_{\Phi_M} (\bx) w_2 (\bx - \by)  \rho_{\Phi_M} (\by) \mathrm{d}\bx \mathrm{d}\by - C N w_2 (0). 
\end{multline*}
Putting the previous inequalities together yields 
\begin{multline*} 
\left\langle \Phi_M \right| \widetilde{H} (\bZ_M) \left| \Phi_M \right\rangle + \frac{2(2M-1)}{M-1}\sum_{1\leq k < \ell \leq M} w_2 (\bz_k-\bz_\ell)  \\ \geq N \cEh_N \left[ \sqrt{\eta} \right] - C N \left(w_1 (0) + w_2 (0) \right) \geq N \EMF - C N^{d\beta} + o (N)
\end{multline*}
where we bound $\cEh_N \left[ \sqrt{\eta} \right]$ from below by $\Eh_N$ and skip the proof that when $\beta >0$ this converges to $\Enls$ when $N\to \infty$. This holds for any $M-$body bosonic function $\Phi_M$ and thus, going back to~\eqref{eq:Levy}, concludes the proof of the energy lower bound. 

\medskip 

\noindent \textbf{Convergence of states.} Now we assume that there exists a unique mean-field minimizer and prove~\eqref{eq:conv states pos pre}. This is done by introducing a perturbed problem depending on a small parameter $\eps$, and proving that the corresponding ground state energy converges to the appropriate mean-field limit. As per Lemma~\ref{lem:FH}, the reduced density matrices will be accessed by differentiating the energy in $\eps$. There will then remain to argue that the mean-field limit and differentiation in $\eps$ can be commuted. 

Let thus $\eps $ be a (small) real number and $B$ be a bounded operator on $L^2 (\R^d)$. The perturbed $N$-body Hamiltonian we shall consider is
$$ H_{N,\eps} = H_N := \sum_{j=1} ^N \left(- \Delta_{\bx_j} + V (\bx_j) + \eps B_{\bx_j} \right)+ \sum_{1\leq i < j \leq N} w_N (\bx_i - \bx_j),
$$
where 
$$ B_{\bx_j} = \1 ^{\otimes j-1} \otimes B \otimes \1 ^{\otimes N-j}.$$
Denote $E (N,\eps)$ the corresponding ground state energy \emph{without bosonic symmetry constraint}, namely the infimum of 
$$ \langle \Psi_N | H_{N,\eps} | \Psi_N\rangle $$
amongst all $L^2$- normalized $N$-body wave-functions $\Psi_N$. As per Theorem~\ref{thm:GS schro}, Assumption~\ref{asum:pos pre} implies that $E(N,0) = E(N),$ our original bosonic ground state energy. The $E(N,\eps)$ minimum needs however not be attained by a bosonic state.

Let $\cEmix_\eps$ be a perturbed mean-field functional extended to \index{Pure state, mixed state}{mixed states} (density matrices): 
\begin{equation}\label{eq:MF mixed}
\cEmix_\eps [\gamma] = \tr\left( \left(h + \eps B \right) \gamma \right) + \frac{1}{2} \iint_{\R^d \times \R^d} \rho_\gamma (\bx) w(\bx-\by) \rho_{\gamma} (\by) \mathrm{d}\bx \mathrm{d}\by 
\end{equation}
where $w$ is replaced by $b\delta_0$ if $\beta >0$ and $\gamma$ is a positive trace-class operator on $L^2 (\R^d)$, with $\rho_\gamma$ its density: 
$$ \rho_\gamma (\bx) = \gamma (\bx;\bx) = \sum_j \lambda_j |u_j (\bx)| ^2.$$
As usual $\gamma$ is identified with its integral kernel, and we denote $\lambda_j,u_j$ its eigenvalues and eigenfunctions. Observe that if $\gamma = |u\rangle \langle u|$ is a pure state and $\eps = 0$, the above is nothing but our target mean-field functional $\cEMF [u]$. 

We claim that, for all $\eps$ (possibly we need it to be small enough) 
\begin{equation}\label{eq:con ener mixed}
 \lim_{N\to \infty} \frac{E(N,\eps)}{N} = \Emix_\eps 
\end{equation}
where $\Emix_\eps$ is the minimum of $\cEmix_\eps$ amongst all one-particle mixed states ($\tr \, \gamma = 1$). For an upper bound, we observe that we can extend the minimization of the perturbed $N$-body energy without changing the result : 
\begin{equation}\label{eq:N body ener mixed}
E (N,\eps) = \inf \left\{ \tr \left( H_{N,\eps} \Gamma_N \right), \: \Gamma_N \mbox{ positive operator on } L^2 (\R^{dN}), \: \tr\, \Gamma_N = 1 \right\}. 
\end{equation}
Indeed, the energy is linear in the operator $\Gamma_N$ and any such positive operator with trace $1$ is, by the spectral theorem, a convex combination of orthogonal projector.

Hence we have, for any one-body mixed state $\gamma$ 
\begin{equation}\label{eq:up bound mixed}
 \frac{E (N,\eps)}{N} \leq N^{-1} \tr \left( H_{N,\eps} \gamma^{\otimes N} \right). 
\end{equation}
This is the place where it is useful to have dropped the bosonic symmetry constraint: $\gamma^{\otimes N}$ is certainly a $N$-body state (it is even symmetric in the sense of~\eqref{eq:boltzons}) but it is not bosonic unless $\gamma$ is pure~\cite{HudMoo-75}.

The right-hand side of~\eqref{eq:up bound mixed} is a Hartree-like functional for the mixed state $\gamma$. If $\beta = 0$ we directly get the upper bound corresponding to~\eqref{eq:con ener mixed}. If $\beta >0$ we minimize in $\gamma$, then pass to the limit $N\to \infty$. It is again an exercise on mean-field functionals to prove that this gives 
\begin{equation}\label{eq:con ener mixed up}
 \limsup_{N\to \infty} \frac{E(N,\eps)}{N} \leq \Emix_\eps. 
\end{equation}
To obtain a corresponding lower bound, we use again that, as minimizer in~\eqref{eq:N body ener mixed} we may use a mixed state $\Gamma_N$ satisfying~\eqref{eq:boltzons}, for $H_{N,\eps}$ commutes with all the unitaries exchanging particle labels. Then we make \emph{three simple observations} concerning our previous proof of energy convergence.

\textbf{First}, in the proof of Lemma~\ref{lem:Levy} (\index{L\'evy-Leblond's trick}{{L\'evy-Leblond's trick}}), the bosonic symmetry assumption was not used. We can work just as well with a (mixed) minimizer for~\eqref{eq:N body ener mixed} satisfying~\eqref{eq:boltzons} and obtain 
\begin{equation}\label{eq:no HO pre}
E(N,\eps) \geq \inf_{\bZ_M \in \R^{dM}} \left( E_\eps(\bZ_M) + \frac{2(2M-1)}{M-1}\sum_{1\leq k < \ell \leq M} w_2 (\bz_k-\bz_\ell) \right) 
\end{equation}
where $E_\eps(\bZ_M)$ is the ground state energy \emph{without bosonic symmetry} of~\eqref{eq:LevLebHam} (with $V$ replaced by $V+\eps B$). We seek a lower bound to this quantity.

\textbf{Second}, the use of the \index{Hoffmann-Ostenhof$\,^2$ inequality}{Hoffmann-Ostenhof$\,^2$ inequality} Lemma~\ref{lem:HO ineq} was somewhat superfluous in the first part of the proof. Namely, if we return to~\eqref{eq:use HO} we can simply write, instead of using~\eqref{eq:HofOst}, that    
\begin{equation}\label{eq:no HO}
 \tr\left( \left(\sum_{j=1} ^M - \Delta_{\bx_j} + V (\bx_j)+\eps B_{\bx_j} \right) \Gamma_M \right) = \tr \left( \left(- \Delta + V +\eps B \right) \Gamma_M^{(1)} \right)
\end{equation}
for any $M-$body state, where $\Gamma_M ^{(1)}$ is the associated one-body density matrix of $\Gamma_M$, defined as in~\eqref{eq:red mat dens}. This allows to bound the kinetic energy terms of~\eqref{eq:no HO pre}. 

\textbf{Third} we can bound the interaction terms in~\eqref{eq:no HO pre} exactly as discussed above. We pick $\Phi_M$ a minimizer for $E_\eps(\bZ_M)$, $\Gamma_M = |\Phi_M\rangle \langle \Phi_M|$ and apply\index{Onsager's inequality}{Onsager's inequality} Lemma~\ref{lem:Onsager} as above, first with $\rho = \rho_M ^{(1)} (\bx) = \Gamma_M ^{(1)} (\bx;\bx)$ (for the $w_1$ part of the interaction) and then with $\rho = (M-1)M^{-1}\rho_M ^{(1)} (\bx)$ (for the $w_2$ part of the interaction). Inserting in~\eqref{eq:no HO pre}, combining with~\eqref{eq:no HO} and using the variational principle yields 
$$ \inf_{\bZ_M \in \R^{dM}} \left( E_\eps(\bZ_M) + \frac{2(2M-1)}{M-1}\sum_{1\leq k < \ell \leq M} w_2 (\bz_k-\bz_\ell) \right) \geq N \Emix_\eps - C N ^{d\beta - 1} + o(1).$$
Hence~\eqref{eq:con ener mixed} is proved, for any $\eps$ small enough. This implies the desired convergence of density matrices, as we now explain. Observe that the functions on both sides of~\eqref{eq:con ener mixed} are concave in $\eps$, as infima of linear functions of $\eps$. It follows that the convergence~\eqref{eq:con ener mixed} implies (this is often referred to as Griffith's lemma)
\begin{equation}\label{eq:con ener der mixed}
 \frac{\partial_\eps E(N,\eps)}{N} \underset{N\to \infty}{\to} \partial_\eps \Emix_\eps. 
\end{equation}
We claim that the left-hand side evaluated at $ \eps = 0$ is  
\begin{equation}\label{eq:FH 1}
 \left(\frac{\partial_\eps E(N,\eps)}{N}\right) _{|\eps = 0} = N ^{-1} \tr B \gamma_N ^{(1)} 
\end{equation}
with $\Gamma_N^{(1)}$ the reduced density matrix of the unique (as per Theorem~\ref{thm:GS schro}) bosonic ground state, while the right-hand side is 
\begin{equation}\label{eq:FH 2}
 \left(\partial_\eps \Emix_\eps \right) _{\eps = 0} = \langle \uMF | B | \uMF \rangle 
\end{equation}
with $\uMF$ the unique mean-field minimizer. This is a \index{Feynman-Hellmann principle}{Feynman-Hellmann argument} for which we skip some details, see Lemma~\ref{lem:FH}. The main thing we have to verify to obtain~\eqref{eq:FH 2} is that $\Emix_0$ is uniquely minimized by $|\uMF \rangle \langle \uMF |$ (the corresponding ingredient for~\eqref{eq:FH 1} is directly given by Theorem~\ref{thm:GS schro}). But a minimizer $\gamma_0$ for $\cEmix$ must also minimize the linearized 
$$ \sigma \mapsto \tr\left( \left( h  + w \star \rho_{0} \right) \sigma \right)  $$ 
where $w$ is the final potential (hence, a Dirac delta if $\beta >0$) and $\rho_{0} (\bx) = \gamma_0 (\bx;\bx)$. Thus the range of $\gamma_0$ lies within the ground eigenspace of the Schr\"odinger operator $h  + w \star \rho_{0} $. This is a nonlinear condition for $\gamma_0$, but all we need to know is that this mean-field operator has a unique ground state, as per Theorem~\ref{thm:GS schro} (here we use that $h  + w \star \rho_{0}$ is \index{Positivity improving}{positivity improving} in the sense of Assumption~\ref{asum:pos pre}). Hence $\gamma_0$ must be rank one, a pure state $|u_0\rangle \langle u_0|$. As noted previously its energy then boils down to the mean-field/NLS energy, which $u_0$ must minimize and thus be equal to $\uMF$. 

At this stage we have that~\eqref{eq:FH 1} converges to~\eqref{eq:FH 2} for any bounded operator $B$. This gives the strong trace-class convergence (first weak-$\star$ convergence, then preservation of the trace norm, hence strong convergence~\cite[Addendum~H]{Simon-79}) claimed in~\eqref{eq:conv states pos pre} for $k=1$. To obtain the corresponding statement at $k>1$, denote $\ada,a$ the \index{Creation and annihilation operators}{creation/annihilation operators} associated to $\uMF$ in Section~\ref{sec:second quant}. We have 
$$
\left\langle \Psi_N | \left( \ada a \right) ^k | \Psi_N \right\rangle \geq  \left( \left\langle \Psi_N | \ada a  | \Psi_N \right\rangle \right) ^k 
$$
by Jensen's inequality. The convergence of the \index{Reduced density matrix}{first density matrix}, together with~\eqref{eq:dens mat CCR} implies that the right-hand side, divided by ${N \choose k},$ converges to $1$. On the other hand, using the CCR~\eqref{eq:CCR} repeatedly and~\eqref{eq:dens mat CCR} again, we have   
$$ { N \choose k } ^{-1} \left\langle \Psi_N | \left( \ada a \right) ^k | \Psi_N \right\rangle = { N \choose k } ^{-1} \left\langle (\uMF) ^{\otimes k} |\Gamma_N ^{(k)}| (\uMF) ^{\otimes k} \right\rangle + o (1).$$
Hence, in operator norm, 
$${ N \choose k } ^{-1} \Gamma_N ^{(k)} \to  | (\uMF) ^{\otimes k} \rangle \langle (\uMF) ^{\otimes k} |,$$
but the right-hand side being rank one, it is not difficult to see that the convergence must actually hold in trace-class norm.
\end{proof}

\subsection{Remarks}

There are a couple of remarks to be made about the proof strategy detailed in the two previous sections:

\medskip

\noindent\textbf{1.} Lemma~\ref{lem:Onsager} is stated with smooth interactions in mind. As we saw, it is still useful when applied to potentials scaled with $N$ as in~\eqref{eq:scaled potential}, where singularities appear in the limit of large $N$. Similarly, if one is instead interested in a fixed singularity ($\beta = 0$ but a singular potential in~\eqref{eq:scaled potential}, e.g. the Coulomb/Newtonian one), there is not much difficulty in adapting the strategy: one may truncate the potential around the singularities. The estimates being quantitative, one then passes to the vanishing truncation limit at the same time as the mean-field limit. 

\medskip

\noindent\textbf{2.} It is conceivable that one may obtain convergence of the first reduced density matrix (in the form of~\eqref{eq:conv states}) from~\eqref{eq:con ener der mixed} even if one does not assume that the mean-field minimizer is unique. This should use a bit of non-trivial convex analysis but should be doable, in view of a similar approach discussed in~\cite{LieSei-06}. We do not pursue this for lack of space and motivation.

\medskip

\noindent\textbf{3.} Obtaining~\eqref{eq:conv states} for $k>1$ would however be highly non-trivial without using the \index{Quantum de Finetti theorem}{quantum de Finetti theorem} (see below). This is because one would need the equivalent of~\eqref{eq:con ener der mixed}, but with $B$ replaced by a general $k$-body operator. Here one would no longer be able to use the cornerstones of the proof: \index{Onsager's inequality}{Onsager's inequality} and \index{L\'evy-Leblond's trick}{L\'evy-Leblond's trick}, Lemmas~\ref{lem:Onsager} and~\ref{lem:Levy}. For the same reason, it does not seem obvious to take into account three-body interactions within this framework.

\medskip

\noindent\textbf{4.} Perhaps the most severe limitation of the method is that it deeply relies on the fact that $\bA \equiv 0$. Superficially this is because we use \index{Hoffmann-Ostenhof$\,^2$ inequality}{Hoffmann-Ostenhof$\,^2$ inequality} Lemma~\ref{lem:HO ineq}, but that can be bypassed, as we saw. More fundamentally, the method is too rough in that it does not really distinguish the bosonic from the ``boltzonic'' problem, and that it relies on uniqueness properties of ground states (Theorem~\ref{thm:GS schro}).   
 
 \medskip

\noindent\textbf{5.} Another drawback (perhaps this is the most severe actually) is that the method (more specifically, Lemma~\ref{lem:Onsager}) seems difficult to improve to deal with dilute ($\beta < 1/d$) limits.  

\medskip

The last three points of the above list may serve as motivation for the rest of the text: it seems that we have hit the limitations of methods based on structural properties of the Hamiltonian. To go further it is desirable to take into account the structure of the bosonic space it acts on. Most of the rest of this review is concerned with exploiting such properties.

\newpage

\section{Applying the classical de Finetti theorem}\label{sec:class deF}

A first approach exploiting the structure of the state-space relies on the \index{Classical de Finetti theorem}{classical de Finetti theorem} (also known as the \index{Hewitt-Savage theorem}{Hewitt-Savage theorem}). It is perhaps less natural, and certainly less powerful, than approaches we will discuss later, in particular those based on the quantum de Finetti theorem. 

The main idea in this section is to reformulate the bosonic many-body ground-state problem as a classical statistical mechanics ensemble. Namely, we want to minimize the classical part (external potential plus interaction) of the energy jointly with (minus) a sort of \index{Entropy}{entropy}, accounting for the quantum kinetic energy. This works again only under Assumption~\ref{asum:pos pre} (no magnetic field), and since we see the problem as an effective classical one, this only allows to access information of a classical nature: convergence of reduced densities\footnote{Probability densities in position space only.} instead of the full reduced density matrices. The method originates in~\cite{Kiessling-12}. A review of a variant of it is in~\cite[Appendix~A]{Rougerie-spartacus,Rougerie-LMU}. 

We shall sketch a proof of the following simplification of Theorem~\ref{thm:main}:

\begin{theorem}[\textbf{Derivation of Hartree's theory, simplified statement}]\label{thm:class quant confined}\mbox{}\\
We make Assumptions~\ref{asum:pots} and~\ref{asum:pos pre} (i.e. $\bA\equiv 0$ in~\eqref{eq:intro Schro op bis}). Set $\beta = 0$ and \index{Hartree energy functional}{let} 
$$
\cEh[u] = \int_{\R^d} \left| \nabla u \right| ^2 + V |u| ^2 + \frac{1}{2} \iint_{\R^d \times \R ^d} |u(\bx)| ^2 w (\bx - \by) |u (\by)| ^2 \mathrm{d}\bx \mathrm{d}\by
$$
with minimum (under unit $L^2$ mass constraint) $\Eh$ and minimizer $\uh$.

We have the convergence of the ground-state energy (lowest eigenvalue of~\eqref{eq:intro Schro op bis}):
$$\lim_{N\to\infty}\frac{E(N)}{N}=\Eh.$$
Let $\Psi_N \geq 0$ be a ground state of $H_N$ and
\begin{equation}\label{eq:red densities}
 \rho_N ^{(n)} (\bx_1,\ldots,\bx_n) := \int_{\R ^{d(N-n)}} \left|\Psi_N\left( \bx_1,\ldots,\bx_N \right)\right| ^2 \mathrm{d}\bx_{n+1}\ldots \mathrm{d}\bx_N  
\end{equation}
be its $n$-body \index{Marginal density}{marginal density}. There exists a probability measure $\mu$ on $\cM_{\rm H}$, the set of minimizers of $\cEh$ (modulo a phase), such that, along a subsequence
\begin{equation}\label{eq:app A def fort result}
\lim_{N\to\infty} \rho^{(n)}_{N}= \int_{\cM_{\rm H}} \left|u ^{\otimes n}\right| ^2 \mathrm{d}\mu(u) \mbox{ for all } n\in \N,
\end{equation}
strongly in $L ^1 \left(\R ^{dn}\right).$ In particular, if $\Eh$ has a unique minimizer (modulo a constant phase), then for the whole sequence
\begin{equation}
\lim_{N\to\infty} \rho^{(n)}_{N} = \left|\left(\uh\right)^{\otimes n}\right| ^2 \mbox{ for all } n\in \N.
\label{eq:app A BEC-confined} 
\end{equation}
\end{theorem}

The proof uses a compactness argument, whence the restriction to $\beta = 0$. The main idea is to formulate the problem only in terms of the probability density in position space $|\Psi_N|^2$. The quantum kinetic energy is in fact identical to the \index{Fisher information}{Fisher information} of the probability measure $|\Psi_N| ^2$. We thus want to minimize jointly the classical part of the energy and the Fisher information. Replacing the latter by minus the classical entropy of $|\Psi_N| ^2$ turns the problem into the classical \index{Boltzmann-Gibbs ensemble}{Boltzmann-Gibbs ensemble}, whose mean-field limit has been tackled in~\cite{MesSpo-82,CagLioMarPul-92,Kiessling-89,Kiessling-93,KieSpo-99}, based on the classical de Finetti theorem. The main idea of this section, originating in~\cite{Kiessling-12}, is to adapt the same strategy to the problem with the Fisher information replacing the classical entropy.  

\subsection{Classical reformulation}

We have already observed in Section~\ref{sec:non trap} that Assumption~\ref{asum:pos pre} implies
$$ \cE_N [\Psi_N] \geq \cE_N \left[| \Psi_N |\right].$$
The ground state energy can thus be calculated using only positive test functions
\begin{equation}\label{eq:app A formul class pre}
E(N) = \inf \left\{ \cE_N [\Psi_N], \Psi_N \in L_{\sym} ^2 (\R ^{dN})\right\} = \inf \left\{ \E_N [\Psi_N], \Psi_N \in L_{\sym} ^2 (\R ^{dN}), \Psi_N \geq 0 \right\}.
\end{equation}
In this section we actually forget all about bosonic statistics (we have observed that this is legitimate in Theorem~\ref{thm:bos min}) and write the infimum as 
\begin{equation}\label{eq:app A formul class}
E(N) =  \inf \left\{ \cE_N \left[\sqrt{\mubf_N}\right], \mubf_N \in \cP_{\sym}  (\R ^{dN})\right\}
\end{equation}
where $\mubf_N$ is a probability measure on $\R^{dN}$. It plays the role of $|\Psi_N| ^2$ so we retain the information that it is symmetric under label exchanges : 
\begin{equation}\label{eq:mubf sym}
\mubf_N (\bx_{\sigma(1)},\ldots,\bx_{\sigma(N)}) =  \mubf_N (\bx_{1},\ldots,\bx_{N})
\end{equation}
for all permutations $\sigma$. To further strengthen the analogy with \index{Classical statistical mechanics}{classical statistical mechanics} we rewrite the many-body energy as 
\begin{multline}\label{eq:ener class}
\cF_N [\mubf_N] := \int_{\R^{dN}} \left( \sum_{j=1} ^N V (\bx_j) + \frac{1}{N-1} \sum_{1\leq j < k \leq N} w (\bx_j-\bx_k) \right) \mathrm{d}\mubf_N (\bx_1,\ldots,\bx_N)\\ + \frac{1}{4} \int_{\R^{dN}} \sum_{j=1} ^N  \left| \nabla_j \log \mubf_N \right| ^2 \mathrm{d}\mubf_N
\end{multline}
where the $\cF$ stands for ``free-energy''. Indeed, we want to see this not as the quantum energy of $N$ bosons, but as the free-energy of classical particles whose positions are probabilistic. The first term is, as discussed previously, just the energy due to the potentials of the particles distributed according to $\mubf_N$. We want to see the second as a kind of entropy that prevents the latter to be a Dirac delta, i.e. forces some probabilistic uncertainty. This is why we wrote it as the \index{Fisher information}{\emph{Fisher information}}
\begin{equation}\label{eq:Fisher}
 \cI [\mubf_N] := \frac{1}{4} \sum_{j=1} ^N \int_{\R^{dN}}  \left| \nabla_j \log \mubf_N \right| ^2 \mathrm{d}\mubf_N = \sum_{j=1} ^N \int_{\R^{dN}}  \left| \nabla_j \sqrt{\mubf_N} \right| ^2. 
\end{equation}
If we replace this term by minus the temperature $T$ times the \index{Entropy}{entropy}
$$ \frac{1}{4} \int_{\R^{dN}} \sum_{j=1} ^N  \left| \nabla_j \log \mubf_N \right| ^2 \mathrm{d}\mubf_N \rightsquigarrow T \int_{\R^{dN}} \log \mubf_N \mathrm{d}\mubf_N $$
then~\eqref{eq:ener class} turns into the bona fide \index{Boltzmann-Gibbs ensemble}{Boltzmann-Gibbs} free energy. To study the large $N$ limit thereof, the strategy of~\cite{MesSpo-82,CagLioMarPul-92,Kiessling-89,Kiessling-93,KieSpo-99} is to 
\begin{enumerate}
 \item Pass to the $N\to\infty$ limit in the marginals of $\mubf_N$ to obtain a certain limit problem in terms of probability measures of \emph{infinitely many} variables, symmetric under their exchange. The de \index{Classical de Finetti theorem}{Finetti-Hewitt-Savage theorem} asserts that \emph{all} the latter are convex combinations of tensor powers, as in the right-hand side of~\eqref{eq:app A def fort result}.
 \item Observe that the limit problem is \emph{linear} in (the limit of) $\mubf_N$. Hence the infimum is attained at the extremal points of the (convex) variational set. The de Finetti-Hewitt-Savage theorem precisely says that the latter are tensor powers, so that we obtain the mean-field (free-) energy as the infimum of the limit problem.
\end{enumerate}

Note that the problem at fixed $N$ is NOT linear in $\mubf_N=|\Psi_N|^2$ but rather in $|\Psi_N \rangle \langle \Psi_N|$. We discuss this in the next section, but before that we state our main tool, which gives all \index{Hewitt-Savage theorem}{information} needed on the limits of symmetric probability measures of many variables.

\begin{theorem}[\textbf{Hewitt-Savage in large $N$ limit}]\mbox{}\\ \label{thm:Hewitt-Savage}
Let $(\mubf_N)_N$ be a sequence of symmetric probability measure over $\R^{dN}$. Assume that the marginals $\mubf_N ^{(k)}$ are tight:
\begin{equation} \label{eq:tightness_assumption}
\limsup_{R\to\infty}\sup_{N\in\mathbb{N}}\left(1-\mubf_N^{(k)}\left(B(0,R)^k\right) \right) =0.
\end{equation}
Extract a subsequence such that 
$\mubf_N ^{(1)} \wto \mu \in \cP (\R^d)$ 
as measures. There exists a unique probability measure $P \in \cP (\cP (\R^d))$ such that for any fixed $k\in\mathbb{N}$, and along the previously extracted subsequence,
\begin{equation} \label{eq:HS}
\mubf_N^{(k)}\to \mubf ^{(k)} = \int_{\mathcal{P}(\mathbb{R}^d)}\mu^{\otimes k}dP(\mu).
\end{equation}
in total variation norm.
\end{theorem}

\begin{proof}[Comments]
The marginals are the \index{Marginal density}{reduced densities}, defined as in~\eqref{eq:red densities}. When applying the above to trapped systems in the sequel, the tightness of the marginals will be essentially for free. There are two possible proof strategies for Theorem~\ref{thm:Hewitt-Savage}, see~\cite[Chapter~2]{Rougerie-spartacus,Rougerie-LMU} and~\cite{Mischler-11} for more details:
\begin{itemize}
 \item Pass to the limit in all the marginals to obtain a hierarchy $(\mubf^{(k)})_{k\in \N}$. Since obviously 
$$ \left(\mubf_N ^{(k+1)}\right)^{(k)} = \mubf_N ^{(k)}$$ 
one can deduce that the limit hierarchy is consistent in the sense that  
$$ \left(\mubf ^{(k+1)}\right)^{(k)} = \mubf ^{(k)}$$ 
for all $k$. Hence it defines a probability measure $\mubf \in \cP (\R ^{d\N})$ over sequences in $\R ^d$. This probability over infinitely many variables is still symmetric under variable exchanges. The symmetric probability measures over $\R ^{d\N}$ clearly form a convex set. Less clear is the crucial fact that the extremal points thereof are exactly the factorized probability measures. The proof~\cite{HewSav-55} is actually by contradiction. Accepting this fact, the existence of the measure $P$ in~\eqref{eq:HS} follows from the Choquet-Krein-Milman theorem~\cite{Simon-convexity}. The uniqueness is not hard to show.
\item Construct a measure $P_N \in \cP (\cP (\R ^{d}))$ at each fixed $N$ approximating the marginals in the manner 
\begin{equation}\label{eq:DF}
\mubf_N^{(k)} \approx \int_{\mathcal{P}(\mathbb{R}^d)}\mu^{\otimes k}dP_N (\mu) 
\end{equation}
and pass to the limit in $P_N$. The construction from~\cite{DiaFre-80} is very natural: since $\mubf_N$ does not see the order of the points in $\R^d$, it is actually a measure over \index{Empirical measure}{empirical measures} of the form 
$$ \mathrm{Emp} (\bX_N) = N^{-1} \sum_{j=1} ^N \delta_{\bx_j}, \quad \bX_N = (\bx_1,\ldots,\bx_N) \in \R^{dN}.$$
A simple but clever computation reveals that, if one defines $P_N$ to be 
$$ P_N = \int_{\R^{N}} \delta_{\rho =  \mathrm{Emp} (\bX_N)} \mathrm{d}\mubf_N (\bX_N)$$
then~\eqref{eq:DF} holds, and one has thus explicitly constructed an approximation of the target measure $P$.
\end{itemize}
\end{proof}

\subsection{Limit problem and use of the classical de Finetti theorem}

We start from~\eqref{eq:ener class} and pass to the limit $N\to \infty$, to obtain a problem posed directly in terms of a symmetric probability $\mubf \in \cP (\R^{dN})$ over infinitely many variables:
\begin{multline}\label{eq:app A prob lim}
\cF [\mubf]:= \limsup_{n\to \infty} \frac{1}{n} \cI [\mubf ^{(n)}] 
\\ + \int_{\R ^d} V(\bx) \mathrm{d}\mubf ^{(1)}(x) + \frac{1}{2}\iint_{\R ^d\times \R ^d } w(x-y) \mathrm{d}\mubf ^{(2)}(x,y), 
\end{multline}
where $\cI [\mubf ^{(n)}]$ is the \index{Fisher information}{Fisher information}~\eqref{eq:Fisher} and $\mubf^{(n)}$ stands for the $n$-th marginal of $\mubf$ (which can in fact be identified with the sequence $(\mubf ^{(n)})_{n\in \N}$, with $\mubf^{(n)} \in \cP_{\rm sym} (\R^{dn})$).

We have the following lemma:

\begin{lemma}[\textbf{Passing to the limit}]\label{lem:app A lim}\mbox{}\\
Let $\mubf_N \in \cP_{\sym}(\R ^{dN})$ achieve the infimum in~\eqref{eq:app A formul class}. Along a subsequence we have
$$ \mubf_N ^{(n)} \wto_* \mubf ^{(n)} \in \cP_{\rm sym} (\R ^{dn})$$
for all $n\in \N$, in the sense of measures. The sequence $\left( \mubf ^{(n)}\right)_{n\in \N}$ defines a probability measure $\mubf\in\cP_{\rm sym} (\R^{d\N})$ and we have
\begin{equation}\label{eq:app A lim inf}
\liminf_{N\to \infty} \frac{E(N)}{N} \geq \cF [\mubf]. 
\end{equation}
\end{lemma}

\begin{proof}[Comments]
Cf~\cite[Lemma~A.4]{Rougerie-spartacus,Rougerie-LMU}. Extracting convergent subsequences is straightforward, passing to the limit in the potential terms also is. Radon measures form the dual space of the continuous functions decaying at infinity, and we use the associated weak-$\star$ topology (i.e. weak convergence as measures). The limit is also a probability measure (i.e. does not lose mass) by a tightness argument (Prokhorov's theorem).  

To pass to the liminf in the Fisher information it is to recall that it is a strictly convex function of the density~\cite[Lemma~A.1]{LieSeiYng-00}. This allows~\cite{Rougerie-19b} to prove (see~\cite{Carlen-91b} or~\cite[Lemma~3.7]{HauMis-14} for other proofs) super-additivity of the Fisher information as used in~\cite{Kiessling-12}, in accord with the original approach of~\cite{MesSpo-82} of \index{Classical statistical mechanics}{classical statistical mechanics} equilibria (where one uses subadditivity of the entropy, cf~\cite[Section~3.2]{Rougerie-spartacus,Rougerie-LMU}). 
\end{proof}

Next we want to bring Theorem~\ref{thm:Hewitt-Savage} to bear on the limit problem~\eqref{eq:app A prob lim}. For this purpose one uses the next result, which has independent interest~\cite{HauMis-14,FouHauMis-14}, along with variants~\cite{Salem-19,Salem-19b,Rougerie-19b} dealing with fractional Fisher informations.

\begin{lemma}[\textbf{The mean Fisher information is linear}]\label{lem:app A kin lim}\mbox{}\\
The functional
$$ \mubf \mapsto \cF [\mubf]$$
defined in~\eqref{eq:app A prob lim} is affine on $\cP_{\rm sym} (\R ^{d\N})$. In particular the mean Fisher information
\begin{equation}\label{eq:mean Fisher}
\cI [\mubf]:= \limsup_{n\to \infty} \frac{1}{n} \cI \left[\mubf ^{(n)} \right] = \lim_{n\to \infty} \frac{1}{n} \cI \left[\mubf ^{(n)} \right]
\end{equation}
is affine.
\end{lemma}

\begin{proof}[Comments]
Taking a marginal is of course a linear operation, so it is obvious that the potential energy terms in~\eqref{eq:app A prob lim} are affine. The non-trivial part is that~\eqref{eq:mean Fisher} is affine (that the sup equals the limsup equals the lim is part of the statement). 

\medskip 

\noindent \textbf{First proof~\cite{Kiessling-12}.} It is fairly simple to see that the mean entropy 
$$
\limsup_{n\to \infty} - \frac{1}{n} \int_{ \R ^{dn}} \mubf ^{(n)} \log \mubf ^{(n)}  = \lim_{n\to \infty} - \frac{1}{n} \int_{ \R ^{dn}} \mubf ^{(n)} \log \mubf ^{(n)}
$$
is affine, and this has independent interest (in particular for the mean-field limit of the \index{Boltzmann-Gibbs ensemble}{Boltzmann-Gibbs ensemble}). See~\cite{RobRue-67} for the original reference and~\cite[Lemma~2.7]{Rougerie-spartacus,Rougerie-LMU}. Then the Fisher information is the derivative of the \index{Entropy}{entropy} along the heat flow. The heat equation is linear, and one thus deduces the linearity of the mean Fisher information from that of the mean entropy.

\medskip 

\noindent \textbf{Second proof.} This is part of a larger-scale investigation of questions related to kinetic theory and classical molecular chaos~\cite{HauMis-14}. The relevant statement is Theorem~5.7 therein. See also~\cite{Salem-19,Salem-19b}. Essentially one proves that $\cI$ is affine when restricted to simple measures  such as those of the form (one actually needs more)
$$ \mubf^{(n)} = \theta \rho_1 ^{\otimes n} + (1-\theta) \rho_2 ^{\otimes n}$$
by direct calculations and concludes via a general abstract argument~\cite[Lemma~5.6]{HauMis-14}. The direct calculation can be tedious. Briefly, one uses a ``orthogonality argument'': if $\rho_1 \neq \rho_2$, $\rho_1^{\otimes n}$ becomes very much alien to $\rho_2^{\otimes n}$ when $n\to \infty$:
$$ \int_{\R^{dn}} \sqrt{\rho_1^{\otimes n}} \sqrt{\rho_2^{\otimes n}} \to 0.$$
Any cross terms then drop from the calculation in the limit.

\medskip 

\noindent \textbf{Third proof.} If we go back to the quantum kinetic energy formulation, we can recall it is linear in $|\Psi_N \rangle \langle \Psi_N|$. Based on this fact, linearity as a function of $|\Psi_N|^2$ in the limit of large $N$ follows from the quantum de Finetti Theorem~\ref{thm:DeFinetti fort} stated below~\cite{Rougerie-19b}. This has interest for the other applications of the lemma. For bosonic mean-field limits however, if one is to use to \index{Quantum de Finetti theorem}{quantum de Finetti theorem}, it is more natural to do it as we shall describe shortly (without reformulating the problem as a classical ensemble).
\end{proof}

We can now briefly present a

\begin{proof}[Sketch of proof for Theorem~\ref{thm:class quant confined}]
The energy upper bound is again a simple trial state argument. For the energy lower bound one first passes to the liminf in the energy using Lemma~\ref{lem:app A lim}. This gives 
$$ \liminf_{N\to \infty} \frac{E(N)}{N} \geq \cF [\mubf]$$
where $\mubf \equiv (\mubf^{(k)})_{k\in \N}$ is the collection of the limits $N\to\infty$ of the marginals of a minimizer $\mubf_N = |\Psi_N| ^2$. Combining Theorem~\ref{thm:Hewitt-Savage} and Lemma~\ref{lem:app A kin lim} yields 
\begin{align*}
 \liminf_{N\to \infty} \frac{E(N)}{N} &\geq \cF [\mubf] = \int \cF \left[\rho ^{\otimes \infty}\right] dP (\rho)\\
 &= \int \cEh \left[ \sqrt{\rho}\right] dP (\rho)
\end{align*}
with $P$ the de Finetti measure associated to $\mubf$. The energy lower bound is clear from the above, so is the fact that $P$ must be concentrated on (squares of) Hartree minimizers. 
\end{proof}

\newpage

\chapter{Mean-field limits, quantum mechanics methods}\label{cha:MF 2}

Now we enter methods much more tailor-made for the bosonic mean-field problem. With the notable exception of the ``local density'' approach to dilute limits (to be discussed in Chapter~\ref{cha:GP}), the methods of this chapter will form the backbone of all of our proof strategies in the rest of the review.

\section{The quantum de Finetti theorem and applications}

The reader might be puzzled by our use of structure theorems from \index{Classical statistical mechanics}{classical statistical mechanics} (and/or probability theory) to deal with a quantum problem in Section~\ref{sec:class deF}. Is there a quantum variant of the de \index{Classical de Finetti theorem}{Finetti-Hewitt-Savage} \index{Hewitt-Savage theorem}{theorem} that one could use to avoid the detours we just made ? 

The answer is affirmative~\cite{Stormer-69,HudMoo-75}, and the quantum de Finetti theorem will allow us to give a first proof of the full statement of Theorem~\ref{thm:main} with a strategy adapted from~\cite{FanSpoVer-80,RagWer-89,Werner-92,PetRagVer-89}. We shall use compactness arguments, so we are still limited to the pure mean-field case $\beta = 0$, but we can handle external magnetic fields for the first time and obtain convergence of all \index{Reduced density matrix}{reduced density matrices}:

\begin{theorem}[\textbf{Mean-field limit of bosonic ground states, $\beta = 0$}]\label{thm:main beta=0}\mbox{}\\
Let the \index{Hartree energy functional}{Hartree functional} be
$$
\cEh[u] = \int_{\R^d} \left| \left(-\im \nabla + \bA \right) u \right| ^2 + V |u| ^2 + \frac{1}{2} \iint_{\R^d \times \R ^d} |u(\bx)| ^2 w (\bx - \by) |u (\by)| ^2 \mathrm{d}\bx \mathrm{d}\by
$$
with minimum $\Eh$ and minimizer $\uh$. We set $\beta = 0$ and work under Assumption~\ref{asum:pots}. We have

\smallskip

\noindent\textbf{Convergence of the energy:} 
$$
\frac{E(N)}{N} \to \Eh. 
$$

\smallskip 

\noindent\textbf{Convergence of reduced density matrices:} let $\Gamma_N ^{(k)},k\geq 0$ be the reduced density matrices of a many-body ground state $\Psi_N$. There exists a Borel probability measure $\mu$ on $\Mh$ (the set of Hartree minimizers) such that 
$$
{N \choose k} ^{-1} \Gamma_N ^{(k)} \to \int_{\Mh} |u ^{\otimes k} \rangle \langle u ^{\otimes k} | \mathrm{d}\mu(u) 
$$
strongly in trace-class norm, along a subsequence (independent of $k$).
\end{theorem}

\subsection{The \index{Quantum de Finetti theorem}{quantum de Finetti theorem}}

To prove Theorem~\ref{thm:main beta=0} we shall take full advantage of the structure of states of large bosonic systems, or more precisely of their density matrices. This generalizes Theorem~\ref{thm:Hewitt-Savage}, which dealt only with densities, i.e. classical objects. Here is the statement we shall rely on: 

\begin{theorem}[\bf Strong quantum de Finetti in large $N$ limit]\label{thm:DeFinetti fort}\mbox{}\\
Let $\gH$ be a separable Hilbert space and $(\Gamma_N)$ a sequence of (mixed) bosonic \index{Pure state, mixed state}{states} on $\gH_N := \bigotimes_{\rm sym} ^N \gH$. Let $\Gamma_N ^{(k)},k\in \N$ be the associated reduced density matrices, and assume that 
\begin{equation}\label{eq:deF CV DMs}
{N \choose k} ^{-1} \Gamma_N ^{(k)}  \underset{N\to \infty}{\to} \gamma^{(k)}
\end{equation}
strongly in trace-class norm, for all $k\in \N$.  

There exists a unique Borel probability measure $\mu \in \cP (S\gH)$ on the sphere $S\gH = \left\{ u \in\gH, \norm{u} = 1\right\}$ of $\gH$, invariant under the action\footnote{Multiplication by a constant phase $e ^{i\theta}, \theta \in \R$.} of $S ^1$, such that
\begin{equation}
\gamma^{(k)}=\int_{S\gH}|u^{\otimes k}\rangle\langle u^{\otimes k}| \, \mathrm{d}\mu(u)
\label{eq:melange}
\end{equation}
for all $k\geq0$.
\end{theorem}

\begin{proof}[Comments]
See~\cite{Stormer-69,HudMoo-75} for the original references~\cite{LewNamRou-14,AmmNie-08} for different proofs and~\cite{Ammari-hdr,Rougerie-spartacus,Rougerie-LMU} for more pedagogical accounts. 

\medskip 

\noindent\textbf{1.} The ``strong'' in the name of the theorem refers to the fact that we assume strong trace-class convergence in~\eqref{eq:deF CV DMs}. It is easy to see that such a convergence implies that ($\tr_{k+1}$ means partial trace over one factor of the tensor product $\gH_{k+1}$)
\begin{equation}\label{eq:consist quant}
 \tr_{k+1} \gamma ^{(k+1)} = \gamma^{(k)} 
\end{equation}
and one says that the sequence (hierarchy) of density matrices  $\left( \gamma^{(k)}\right)_{k\in \N}$ is consistent. The theorem of~\cite{HudMoo-75} applies to such a sequence, which defines an abstract state with infinitely many particles. The original version of the result~\cite{Stormer-69} applies directly to states with infinitely many particles. 

\smallskip

\noindent\textbf{2.} The original proof is not constructive, much in the spirit of the first proof of Theorem~\ref{thm:Hewitt-Savage} we alluded to. Constructive proofs were obtained later, see~\cite{Rougerie-spartacus,Rougerie-LMU} for more details. A really direct construction has so far been obtained only for finite dimensional Hilbert spaces $\gH$ (see Section~\ref{sec:deF semi}), but one can lift this construction to infinite dimensional spaces by Fock-space/geometric localization~\cite{AmmNie-08,LewNamRou-14}.

\smallskip

\noindent\textbf{3.} The original theorem applies not only to bosonic states as stated here. The classical symmetry assumption~\eqref{eq:class indis} is sufficient (boltzonic states), but then the measure lives over mixed one-body states $\gamma$ (positive trace-class operators) not only on pure states (rank-one projectors). 
\end{proof}

To use the above, it is necessary to obtain strong compactness of the density matrices as in~\eqref{eq:deF CV DMs}. This is easy for confined systems as mostly discussed in these notes, but it can be useful to relax the assumption (see Section~\ref{sec:non trap} below):

\begin{theorem}[\bf Weak quantum de Finetti theorem]\label{thm:DeFinetti faible}\mbox{}\\
Let $\gH$ be a separable Hilbert space and $(\Gamma_N)$ a sequence of (mixed) bosonic states on $\gH_N := \bigotimes_{\rm sym} ^N \gH$. Let $\Gamma_N ^{(k)},k\in \N$ be the associated \index{Reduced density matrix}{reduced density matrices}, and assume that 
\begin{equation}\label{eq:deF CV DMs faible}
{N \choose k} ^{-1} \Gamma_N ^{(k)}  \underset{N\to \infty}{\wto} \gamma^{(k)}
\end{equation}
weakly-$\star$ in the trace-class, for all $k\in \N$.  

There exists a unique Borel probability measure $\mu \in \cP (B\gH)$ on the ball $B\gH = \left\{ u \in\gH, \norm{u} \leq 1\right\}$ of $\gH$, invariant under multiplication by a constant phase, such that
\begin{equation}
\gamma^{(k)}=\int_{B\gH}|u^{\otimes k}\rangle\langle u^{\otimes k}| \, \mathrm{d}\mu(u)
\label{eq:melange faible}
\end{equation}
for all $k\geq0$.
\end{theorem}

\begin{proof}[Comments]\mbox{}\\
\noindent\textbf{1.} The weak-$\star$ convergence in~\eqref{eq:deF CV DMs faible} is the usual notion. The trace-class being the dual of the compact operators one demands that 
$$ {N\choose k} ^{-1} \tr_{\gH_k} \left( K_k \Gamma_N ^{(k)}\right) \to \tr_{\gH_k} \left( K_k \gamma ^{(k)}\right)$$
for any compact operator $K_k$. Since the convergence needs not hold for $K_k = \1_{\gH_k}$ there is generically a loss of mass in the limit 
\begin{equation}\label{eq:loss mass}
 \tr_{\gH_k} \gamma^{(k)} \leq 1 = \tr {N\choose k} ^{-1} \tr_{\gH_k} \Gamma_N ^{(k)}. 
\end{equation}
In fact, the convergence is strong if and only if there is equality in the above~\cite{dellAntonio-67,Simon-79}, which is not obvious since the trace-class is not reflexive.

\smallskip 

\noindent\textbf{2.} Modulo a diagonal extraction from any sequence of bosonic states $\Gamma_N$, one can always assume that~\eqref{eq:deF CV DMs faible} holds. The theorem is thus generic in this sense.

\smallskip 

\noindent\textbf{3.} In view of~\eqref{eq:loss mass}, it is fairly natural that the limit measure has to live over the unit ball instead of the unit sphere. A remarkable consequence of the theorem is that, if there is equality for one $k$ in~\eqref{eq:loss mass} (and hence strong convergence as in~\eqref{eq:deF CV DMs}), the measure must live on the sphere. Since it does not depend on $k$, there must be equality in~\eqref{eq:loss mass}, and thus strong trace-class convergence as in~\eqref{eq:deF CV DMs}, for all $k$.  
 
\smallskip 

\noindent\textbf{3.} For the same reason as discussed in point \textbf{1}, the limit hierarchy $(\gamma^{(k)})_k$ is in general not consistent. One can only get
$$ \tr_{k+1} \gamma ^{(k+1)} \leq \gamma ^{(k)}$$
as operators. This is quite insufficient to obtain~\eqref{eq:melange faible} as simple examples show, so one must retain the information that $\gamma^{(k)}$ is a limit of states with increasing particle numbers. 

\smallskip 

\noindent\textbf{4.} A semi-constructive proof of the result using \index{Fock space}\index{Fock space} \index{Localization}{localization} also yields useful corollaries regarding the part of the sequences not described by the weak-$\star$ limits~\cite{LewNamRou-14}.
\end{proof}

\subsection{Proof of the mean-field theorem}\label{sec:proof deF}

We now give the proof of Theorem~\ref{thm:main beta=0}, which originates in~\cite{FanSpoVer-80,RagWer-89,Werner-92,PetRagVer-89}. The argument we follow is from~\cite{LewNamRou-14}.

\begin{proof}
The energy upper bound is straightforward, using trial states of the form $u ^{\otimes N}$. We focus on the energy lower bound and the convergence of states. For all this proof we denote 
$$ h := \left(-\im \nabla + \bA \right)^2 + V $$
the one-body operator.

Consider a ground state $\Psi_N$ and denote $\Gamma_N = |\Psi_N \rangle \langle \Psi_N |$ the associated density matrix. Then
\begin{align}\label{eq:afort ener matrices}
\frac{E(N)}{N} &=  \frac1N \left\langle \Psi_N, H_N \Psi_N \right\rangle_{\gH ^N} = N^{-1}\tr_{\gH} \left( h\, \Gamma_N ^{(1)}\right) + {N \choose 2} ^{-1} \tr_{\gH_2} \left( w\, \Gamma_N ^{(2)} \right)
\end{align}
where $w$ is identified with the multiplication operator by $w(\bx-\by)$ on the two-particles space $\gH_2$. 

Denote 
$$ \gamma_N ^{(k)} := {N \choose k } ^{-1} \Gamma_N ^{(k)}.$$
This is (for any $k$) a bounded sequence of trace-class operators. Hence, modulo a diagonal procedure, we may extract a (not relabeled) subsequence such that~\eqref{eq:deF CV DMs faible} holds. We apply Theorem~\ref{thm:DeFinetti faible} and obtain a measure satisfying~\eqref{eq:melange faible}. We want to show that this measure lives over the unit sphere of $\gH$. In view of the comments after Theorem~\ref{thm:DeFinetti faible} we only need to show that 
$$ \tr \, \gamma ^{(1)} = 1. $$
Under our assumptions on $w$ we easily obtain the a priori bound 
$$ \tr_{\gH} \left( h\, \gamma_N ^{(1)}\right) \leq C,$$
independently of $N$. Thus, the positive operator (we can shift the energy reference to get $h \geq 0$)
$$ h ^{1/2} \gamma_N ^{(1)} h ^{1/2}$$
is bounded in trace-class. Modulo a further extraction we can assume that it converges weakly star in trace-class, and the limit must be $h ^{1/2} \gamma ^{(1)} h ^{1/2}$ Then 
$$ 1 = \tr \, \gamma_N ^{(1)} = \tr \left( h ^{-1} h ^{1/2} \gamma_N ^{(1)} h ^{1/2} \right) \underset{N\to \infty}{\to} \tr \, \gamma^{(1)} $$
because $h ^{-1}$ is compact (consequence of~\eqref{eq:trapping}, by the Sobolev compact embedding theorem directly for $\bA\equiv 0$ using~\cite{AvrHerSim-78} for non-trivial $\bA$). 

Then, along the previous subsequences we can apply Theorem~\ref{thm:DeFinetti fort}, which yields a measure $\mu \in \cP (S\gH)$ describing the limit density matrices as per~\eqref{eq:melange}. 

Next, Assumption~\ref{asum:pots} implies that there exists a constant $C_h$ such that, on the two-body space 
$$ H_2 := h_\bx + h_\by + w (\bx-\by) + C_h \geq 0.$$
As in~\eqref{eq:afort ener matrices}, 
$$
\tr_{\gH^2} \left( H_2 \gamma_N ^{(2)} \right) = N^{-1} \left\langle \Psi_N, H_N \Psi_N \right\rangle_{\gH ^N} + C_h
$$
and we deduce the a priori bound 
$$ \tr \left( H_2 \gamma_N ^{(2)} \right) \leq C. $$
We may thus assume that the positive operator $H_2 ^{1/2} \gamma_N ^{(2)} H_2 ^{1/2}$ converges weakly-$\star$ in trace-class (modulo a further extraction). The limit must be $H_2 ^{1/2} \gamma ^{(2)} H_2 ^{1/2}$ with the previously obtained $\gamma^{(2)}$. Using Fatou's lemma for operators (weak-$\star$ lower semi-continuity of the norm) and the cyclicity of the trace we deduce 
$$ \liminf_{N\to\infty} \tr \left( H_2 \gamma_N ^{(2)} \right) \geq \tr \left( H_2 \gamma ^{(2)} \right).$$
But by definition 
$$ \tr \left( H_2 \gamma_N ^{(2)} \right) = \frac{E(N)}{N} + C_h$$
and we have proved that $\gamma ^{(2)}$ has unit trace. We can thus subtract $C_h$ to both sides of the previous equation and get 
$$ \liminf_{N\to\infty} \frac{E(N)}{N} \geq \tr \left( \left(H_2 - C_h\right) \gamma ^{(2)} \right).$$
Inserting the representation~\eqref{eq:melange} of $\gamma^{(2)}$ gives 
$$ \liminf_{N\to\infty} \frac{E(N)}{N} \geq \int_{S\gH} \cEh [u] \mathrm{d}\mu (u) \geq \Eh $$ 
because $\mu$ is a probability measure. This is the sought-after energy lower bound. Combining with the upper bound shows that $\mu$ must be concentrated on Hartree minimizers, and thus the convergence of density matrices follows.  
\end{proof}

\subsection{Non-trapped case}\label{sec:non trap}

The weak quantum de Finetti theorem (together with some corollaries of its proof) is well-suited to deal with systems that lack compactness at infinity. Such a defect of compactness leads to the failure of~\eqref{eq:deF CV DMs} in the strong topology, but one can still rely on Theorem~\ref{thm:DeFinetti faible}. This is reviewed at length in~\cite{Rougerie-LMU,Rougerie-spartacus} so we will here only state the main result generalizing Theorem~\ref{thm:main beta=0}, obtained in~\cite{LewNamRou-14}. The main point is to consider one-body potentials that decay instead of grow at infinity, so that particles might escape from the trap.

\begin{theorem}[\textbf{Mean-field limit, non-trapped case}]\label{thm:non trap}\mbox{}\\
We set $\beta = 0$ and work under Assumption~\ref{asum:pots}, except for the fact that we take $V,|A|^2 \in L^p(\R^d) + L^{\infty} (\R^d)$ non-trapping: 
\begin{equation}\label{eq:non trapping} 
V(\bx),|\bA (\bx)|^2 \underset{|\bx|\to \infty}\to 0.
\end{equation}
Let then $E(N)$ be the bottom of the spectrum of $H_N$ (eigenvalue or bottom of essential spectrum) and $(\Psi_N)_N$ a sequence of approximate ground states: 
$$ \langle \Psi_N |H_N | \Psi_N \rangle \leq E (N) + o (N)$$
for $N\to \infty$.

Let the \index{Hartree energy functional}{Hartree functional} be
$$
\cEh[u] = \int_{\R^d} \left| \left(-\im \nabla + \bA \right) u \right| ^2 + V |u| ^2 + \frac{1}{2} \iint_{\R^d \times \R ^d} |u(\bx)| ^2 w (\bx - \by) |u (\by)| ^2 \mathrm{d}\bx \mathrm{d}\by
$$
with minimum $\Eh$ and minimizer(s) $\uh$. Let $\Mh$ be the set of weak limits of Hartree minimizing sequences:
\begin{multline}\label{eq:Hartree weak}
\Mh := \big\{ u \in L^2 (\R^d) \,|\, \exists (u_n)_n \in L^{2} (\R^d) \mbox{ such that } \, u_n \wto u \mbox{ weakly in } L^2 \\
 \norm{u_n}_L^2 = 1, \mbox{ and } \cEh[u_n] \to \Eh \big\} 
\end{multline}

We have, in the limit $N \to + \infty$:

\smallskip 

\noindent\textbf{Convergence of the energy:} 
\begin{equation}\label{eq:conv ener non}
\frac{E(N)}{N} \to \Eh. 
\end{equation}

\smallskip 

\noindent\textbf{Convergence of reduced density matrices:} let $\Gamma_N ^{(k)},k\geq 0$ be the reduced density matrices of an approximate ground state $\Psi_N$. There exists a Borel probability measure $\mu$ on $\Mh$ (independent of $k$) such that 
\begin{equation}\label{eq:conv states non}
{N \choose k} ^{-1} \Gamma_N ^{(k)} \wto_\star \int_{\Mh} |u ^{\otimes k} \rangle \langle u ^{\otimes k} | \mathrm{d}\mu(u) 
\end{equation}
weakly-$\star$ in the trace-class.
\end{theorem}

\begin{proof}[Comments] 
This is exactly saying that, in the mean-field limit, the binding of bosonic particles is described to leading order by Hartree theory. In general, the limit in~\eqref{eq:conv states non} is only weak-$\star$, for there may be a loss of mass at infinity. If it occurs, the limits describing the set $\Mh$ are only $L^2$-weak.
 
One can describe in more details the limit set $\Mh$ using binding inequalities and a standard concentration-compactness analysis~\cite{Lions-84,Lions-84b,Struwe}. Again, see~\cite{Rougerie-LMU,Rougerie-spartacus} for more comments. 
\end{proof}

\newpage

\section{\index{Localization}{Localization} plus quantum de Finetti}\label{sec:deF loc}

We turn to methods allowing to prove Theorem~\ref{thm:main} for $\beta >0$, i.e. when the limit object is a \index{Non-linear Schr\"odinger energy functional}{local non-linear Schr\"odinger functional}. In this case there are two limits to be taken at the same time: \begin{itemize}                                                                                                                                                                                                                  \item large particle number
\item interaction potential converging to a Dirac mass.                                                                                                                                                                                                                       \end{itemize}
This rules out compactness methods and calls for quantitative estimates, for it is not obvious that the two limits commute. The general philosophy however stays the same as in the previous section: generic many-particles bosonic states can be represented as statistical superpositions of factorized states. Inserting such a representation in the many-body energy directly leads to the mean-field energy. Our task is to control the error thus made, and now we turn quantitative.

Unfortunately, the known versions of Theorem~\ref{thm:DeFinetti fort} coming with quantitative estimates on the convergence~\eqref{eq:melange} are valid only in finite dimensional Hilbert spaces. Our strategy shall thus be to first localize the problem to finite dimensions. The dimension of the projected one-particle space must be chosen to minimize the sum of two errors: that due to the use of the \index{Quantum de Finetti theorem}{quantum de Finetti theorem} in the projected space, and the energy coming from particles living outside the projected space.

For pedagogical reasons we first mention a rough method for controlling the localization error. As we will see below, one can improve the result a lot, e.g. by relying on the variational PDE satisfied by minimizers to obtain a priori estimates. The following theorem can however be proved in a fully variational way, in the spirit of $\Gamma$-convergence~\cite{DalMaso-93,Braides-02}. Recall the NLS energy functional 
$$
\cEnls[u]:= \int_{\R^d} \left| \left( -\im \nabla + \bA \right) u \right| ^2 + V |u| ^2 + \frac{b_w}{2} \int_{\R^d} |u(\bx)| ^4  \mathrm{d}\bx.
$$
with minimum (under unit $L^2$ mass constraint) $\Enls$, minimizer(s) $\unls$ and the notation 
$$
b_w := \int_{\R^d} w.  
$$

\begin{theorem}[\textbf{NLS limit of bosonic ground states}]\label{thm:NLS}\mbox{}\\
Make Assumptions~\ref{asum:pots} and~\ref{asum:NLS stab} plus the more specific \index{Trapping potential}{trapping condition} 
\begin{equation}\label{eq:trapping s bis}
V(\bx) \geq c_s |\bx| ^s - C_s      
\end{equation}
for some $s>0$. Further assume that $(1+|\bx|) w (\bx) \in L^1 (\R^d)$ and $\widehat{w}\in L^1 (\R^d)$. Let 
\begin{equation}\label{eq:beta NLS}
0 < \beta < \frac{1}{2d}. 
\end{equation}
We have, in the limit $N \to + \infty$:

\smallskip 

\noindent\textbf{Convergence of the energy:} 
$$
\frac{E(N)}{N} \to \Enls. 
$$

\smallskip 

\noindent\textbf{Convergence of reduced density matrices:} let $\Gamma_N ^{(k)},k\geq 0$ be the reduced density matrices of a many-body ground state $\Psi_N$. There exists a Borel probability measure $\mu$ on the set $\Mnls$ of NLS minimizers such that, along a subsequence, 
$$
{N \choose k} ^{-1} \Gamma_N ^{(k)} \to \int_{\Mnls} |u ^{\otimes k} \rangle \langle u ^{\otimes k} | \mathrm{d}\mu(u) 
$$
strongly in trace-class norm.
\end{theorem}

\begin{proof}[Comments]\mbox{}\\
\noindent\textbf{1.} This is taken from~\cite{Rougerie-19}. A previous version with a smaller $\beta$ was obtained in~\cite{LewNamRou-14c}. The latter quantity controls the speed at which the interaction potential converges to a Dirac delta. The larger it is, the more singular the limit. The threshold obtained above is in the middle of the mean-field regime  (the crossover to the dilute regime occurs at $\beta = 1/d$). With additional tools one can do much better, see below.

\medskip 

\noindent\textbf{2.} The main novelty in the proof technique consists in the derivation of explicit, quantitative if far from optimal, estimates on the energy. For this we use quantitative versions of the quantum de Finetti theorem which apply to finite dimensional one-particle Hilbert spaces (call $D$ the dimension). The $k$-body density matrix of a $N$-body bosonic state is approximated by a statistical superposition of factorized states with an explicit error bound, function of $D,N$ and $k$. The two versions (Subsections~\ref{sec:deF semi} and~\ref{sec:deF info}) we shall give differ both in the error bound they provide (order of magnitude, norm in which it is expressed) and the construction of the approximating state. 

\medskip 

\noindent\textbf{3.} In Subsection~\ref{sec:localization} we explain how to project the original one-body Hilbert space to finite dimensions, and control the associated error. This is the part we shall optimize more carefully later when dealing with dilute limits.

\medskip 

\noindent\textbf{4.} Although we shall not explain it here (see~\cite{LewNamRou-14c}), the method below allows to deal with the dilute regime (in fact, any $\beta >0$) in 1D. This is because the interaction is sub-critical with respect to the kinetic energy in this case, more precisely we have the Gagliardo-Nirenberg-Sobolev inequality
$$ \norm{u}_{L^{\infty} (\R)} ^2 \leq \norm{u}_{L^{2} (\R)} \norm{\nabla u}_{L^{2} (\R)}.$$
This allows for a much more efficient control of the localization error.
\end{proof}

\subsection{\index{Quantum de Finetti theorem}{Quantum de Finetti}: \index{Semiclassical analysis}{semiclassical} version}\label{sec:deF semi}

The following result is taken from~\cite{ChrKonMitRen-07,Chiribella-11,LewNamRou-14b,Harrow-13}, and is reminiscent of ideas of~\cite{AmmNie-08,AmmNie-09}. It is extensively discussed in~\cite[Chapter~4]{Rougerie-LMU,Rougerie-spartacus}, to which we refer for more details.

\begin{theorem}[\textbf{Semiclassical quantum de Finetti}]\label{thm:deF semi}\mbox{}\\
Let $\gH$ be a complex Hilbert space of dimension $D<\infty$. Let $\Gamma_N$ be an associated $N$-particle bosonic state, a positive trace-class operator on $\gH_N:=\bigotimes_{\rm sym}^N \gH$ with unit trace. Let $\Gamma_N ^{(k)}$ be the associated reduced density matrices, defined as in~\eqref{eq:def red mat}.

Define a measure $\mu_N$ on $S\gH$, the unit sphere of $\gH$ in the manner
\begin{equation}\label{eq:CKMR}
 \mathrm{d}\mu_N (u) := D_N \left\langle u ^{\otimes N} | \Gamma_N | u^{\otimes N} \right\rangle\mathrm{d}u  
\end{equation}
with $du$ the normalized Lebesgue measure on $S\gH$ and $D_N$ the dimension of the bosonic space $\gH_N$.   
 
Then
\begin{equation}\label{eq:deF semi}
 \tr \left| {N \choose k} ^{-1} \Gamma_N ^{(k)} - \int_{S\gH} |u^{\otimes k} \rangle \langle u ^{\otimes k}| \mathrm{d}\mu_N (u)  \right| \leq C \frac{Dk}{N}
\end{equation}
for some universal constant $C>0$. 
\end{theorem}

\begin{proof}[Comments]\mbox{}\\
\noindent\textbf{1.} Note the particularly simple construction of the measure. In fact, by \index{Schur's lemma}{Schur's lemma} we have 
\begin{equation}\label{eq:Schur}
D_N \int_{S\gH} |u^{\otimes N} \rangle \langle u ^{\otimes N}| \mathrm{d} u = \1_{\gH_N} 
\end{equation}
and one may think of the above as a coherent state~\cite{ComRob-12,KlaSka-85} resolution of the identity. The statement~\eqref{eq:deF semi} on reduced density matrices would follow if we could represent our original state within this \index{Coherent states}{coherent state} basis in the form
$$ \Gamma_N = \int_{S\gH} |u^{\otimes N} \rangle \langle u ^{\otimes N}| \mathrm{d}\mu_N (u).$$
In the vocabulary of semiclassical analysis, this amounts to looking for an \index{Symbol, upper and lower}{upper symbol}. The object defined in~\eqref{eq:CKMR} on the other hand is the lower symbol\footnote{Upper and lower symbols go under various names in the literature, see the discussion after Definition~\ref{def:symbols} below.} of $\Gamma_N$. A rationale for the theorem is that, in a semiclassical regime, upper and lower symbols have a strong tendency to coincide~\cite{Lieb-73b,Simon-80} (see Section~\ref{sec:coherent} below). That the large $N$ limit of bosonic systems can be interpreted as a semiclassical limit is a fact extensively used for the dynamical mean-field problem~\cite{AmmNie-08,AmmNie-09,AmmNie-11,FroGraSch-07,FroKnoPiz-07,FroKno-11,FroKnoSch-09,GinVel-79,GinVel-79b,Hepp-74}. 

\medskip

\noindent\textbf{2.} A convenient way to prove the theorem (and bolster its semiclassical feel) is to realize that, denoting 
$$ \gamma_N ^{(k)} = \int_{S\gH} |u^{\otimes k} \rangle \langle u ^{\otimes k}| \mathrm{d}\mu_N (u),$$ 
we have 
\begin{equation}\label{eq:AWick}
 \langle v^{\otimes k} |  \gamma_N ^{(k)}| v^{\otimes k} \rangle = \frac{(N+D-1)!}{(N+k+D-1)!} \tr \left( a(v)^k a ^\dagger(v)^k \Gamma_N \right) 
\end{equation}
where $a,a^\dagger$ are \index{Creation and annihilation operators}{annihilation and creation operators}, see Section~\ref{sec:second quant}. This completely determines the $k$-body operator $\gamma_N ^{(k)}$, and Equation~\eqref{eq:deF semi} says that it is close to ${N \choose k} ^{-1} \Gamma_N^{(k)}$. The latter is fully determined by~\eqref{eq:dens mat CCR}, whose right-hand side is (apart from a constant factor $\sim 1$) that of~\eqref{eq:AWick} with the annihilators/creators in the reverse order. A proof of the theorem (that of~\cite{LewNamRou-14b}, see the aforementioned references for alternatives) then essentially consists  in a repeated application of the \index{Canonical commutation relations (CCR)}{CCR}~\eqref{eq:CCR} to commute annihilators and creators and thereby compare~\eqref{eq:AWick} to~\eqref{eq:dens mat CCR}. The crux of the proof is that commutators between annihilators and creators are bounded independently of $N$, which is much smaller than the ``typical value'' of a single such operator. When projected on the $N$-body space such a typical value is of order $\sqrt{N}$ as can be expected (cf~\eqref{eq:number second} and~\eqref{eq:hamil second}) from the expression
$$ \cN = \sum_{k= 1} ^D a^\dagger (u_k) a(u_k)$$
where $(u_k)_{k=1} ^D$ is an orthonormal basis of $\gH$. Each commutator should thus be thought as a remainder when its expectation value is taken in a state with a large number of particles.

\medskip

\noindent\textbf{3.} A variant of the theorem, used for example in~\cite{AmmNie-08,AmmNie-09,AmmNie-11,AmmNie-15,LewNamRou-14d} works on the bosonic \index{Fock space}{Fock space}~\eqref{eq:Fock} built on $\gH$. Namely, one uses the variant of~\eqref{eq:Schur} given by 
$$ \pi ^{-D} \int | \xi(u) \rangle  \langle \xi(u) | \mathrm{d} u= \1_{\gF}$$
where 
\begin{equation}\label{eq:coh u}
\xi(u) = e^{-\norm{u} ^2/2} \bigoplus_{N=0} ^\infty \frac{u^{\otimes N}}{\sqrt{N!}}  
\end{equation}
is now a genuine coherent state in that it is of the form~\cite[Chapter~1]{KlaSka-85} 
\begin{equation}\label{eq:coh u bis}
 e^{a^{\dagger} (u)- a(u)} |0\rangle 
\end{equation}
with $|0\rangle = 1 \oplus 0 \oplus \ldots$ the vacuum state of Fock space. The error in the theorem is then quantified in terms of the average particle number (expectation value of $\cN$) of a state on Fock space. The use of coherent states for dynamical mean-field problems has a long history~\cite{Hepp-74,GinVel-79,RodSch-09,Schlein-08,Schlein_ICMP09-10,BenPorSch-15}.

\medskip

\noindent\textbf{4.} In the applications to mean-field limits, the main limitation of the above theorem is the dependence of the error term on the dimension of the one-body Hilbert space, $D$. It leads to the necessity of projecting the full $L^2$-space to a very small subspace, and hence to a rather bad localization error. In the next section we give another version of the quantum de Finetti theorem that has a much better dependence on $D$.
\end{proof}

\subsection{\index{Quantum de Finetti theorem}{Quantum de Finetti}: information-theoretic version}\label{sec:deF info}

In~\cite{BraHar-12,LiSmi-15}, a variant of Theorem~\ref{thm:deF semi} has been obtained, where the error's dependence on $D$ is only logarithmic. There are several catches to be able to achieve this, which we will discuss after having stated the 

\begin{theorem}[\textbf{Information-theoretic quantum de Finetti}]\label{thm:deF info}\mbox{}\\
Let $\gH$ be a complex Hilbert space of dimension $D<\infty$. Let $\Gamma_N$ be an associated $N$-particle symmetric state, a positive trace-class operator on $\gH^{\otimes N}$ with unit trace satisfying~\eqref{eq:boltzons}. Let $\Gamma_N ^{(k)}$ be the associated reduced density matrices, defined as in~\eqref{eq:def red mat}.

For each $k\in \N$ there exists a probability measure $\mu_N^{(k)}$ on the set of one-body states 
\begin{equation}\label{eq:mixed 1body}
 \cS := \left\{ \gamma \in \gS^1 (\gH), \gamma = \gamma^\dagger \geq 0, \tr\, \gamma  = 1\right\} 
\end{equation}
such that, for all self-adjoint operators $A_1,\ldots,A_k$ on $\gH$ 
\begin{equation}\label{eq:deF info}
 \tr \left| A_1 \otimes \ldots \otimes A_k \left( {N \choose k} ^{-1} \Gamma_N ^{(k)} - \int_{\cS} \gamma^{\otimes k} \mathrm{d}\mu_N ^{(k)}(\gamma) \right) \right| \leq C k \sqrt{\frac{\log D}{N}} \prod_{j=1} ^k \norm{A_j}
\end{equation}
for some universal constant $C>0$, where $\norm{A_j}$ stands for the operator norm.  
\end{theorem}

\begin{proof}[Comments]\mbox{}\\
\noindent \textbf{1.} The formulation we give is less powerful than the originals~\cite{BraHar-12,LiSmi-15} where in particular the error is expressed in the LOCC norm (local operations and classical communication). See~\cite{Rougerie-19,Girardot-19} for the reformulation testing against tensorized operators. Pedagogical discussions of the original proof are in~\cite{BraChrHarWal-16,Rougerie-19}. The first obvious drawback of this result, as compared with Theorem~\ref{thm:deF semi} is the weaker way in which the error is measured. Another point, worthy of note but not really annoying, is that the constructed measures a priori depend on $k$.  We usually only need $k=2$.

\medskip 

\noindent \textbf{2.} Theorem~\ref{thm:deF semi} can also be extended from bosonic to general symmetric states~\cite{ChrKonMitRen-07}. In Theorem~\ref{thm:deF info} above we state this generalization explicitly. The reason is that, even if we start from a bosonic state (satisfying the stronger~\eqref{eq:bosons} on top of~\eqref{eq:boltzons}) as needs be done for the topics of this review, the constructed measure does not charge only bosonic states. Namely, it a priori lives on the full one-body state space~\eqref{eq:mixed 1body} instead of just pure states $\gamma = | u \rangle \langle u|$. This is a nuisance that has to be taken care of when using the above for bosonic mean-field limits. 


\medskip 

\noindent \textbf{3.} I chose to refer to the statement as an ``information-theoretic'' version of the quantum de Finetti theorem because of its proof. Most of it proceeds by quantifying the errors made using information-based quantities: relative entropies, mutual informations etc... Pinsker's inequality is then used at the very end of the proof to recover trace-norm-based measures of the error.

\medskip 

\noindent \textbf{4.} The $N$-dependence of the bound is worse than that from Theorem~\ref{thm:deF semi}, but for the applications we target this is more than made-up for by the excellent dependence on~$D$.  
\end{proof}

Here is a glimpse of the construction that leads to the result, which is only semi-explicit. See the original references~\cite{BraHar-12,LiSmi-15} or~\cite[Appendix~A]{Rougerie-19} for more details. More background on the notions and heuristics below are in the lecture notes~\cite{BraChrHarWal-16}.

One formalizes the idea of a measurement of the last $N-k$ subsystems of our $N$ particles by associating to $\Gamma = \Gamma_N$ the family of states 
\begin{equation}\label{eq:measure}
\Gamma_{\mu} := \frac{\tr_{k+1\to N} \left( \1^{\otimes ^k} \otimes M_\mu \Gamma \right) \otimes |e_\mu \rangle \langle e_\mu |}{p_\mu},  \quad p_\mu := \tr \left( \1^{\otimes k} \otimes M_\mu \Gamma \right) 
\end{equation} 
where $(e_\mu)_\mu$ is an orthogonal basis of $\gH ^{\otimes (N-k)}$ and the $M_\mu$'s are positive matrices such that
$$ \sum_\mu M_\mu = \1 ^{\otimes (N-k)}.$$
Roughly, in quantum mechanics a measurement of an observable (self-adjoint operator) with spectral decomposition 
$$ \sum_\mu a_\mu |e_\mu \rangle \langle e_\mu |$$
in a state $\gamma$ leads to the value $a_\mu$ with probability $\langle e_\mu | \gamma | e_\mu \rangle$. After the measurement the system is in the pure state $|e_\mu \rangle \langle e_\mu|$. For various reasons (in particular, if the measurement is done over a subsystem only) one generalizes this by allowing the probability to end up in the state $|e_\mu \rangle \langle e_\mu|$ to be of the form $\tr \left( M_\mu \gamma\right)$ where, as above the positive matrices $M_\mu$ add to the identity.

From this point of view, we associate the states $\Gamma_\mu$ in~\eqref{eq:measure} to the original $\Gamma$ by performing a generalized measurement over $N-k$ particles only. The out-coming state is $\Gamma_\mu$ with probability $p_\mu$. Now we can form a statistical superposition of factorized states as follows
\begin{equation}\label{eq:post measure}
 \widetilde{\Gamma} = \sum_\mu p_\mu \left( N ^{-1} \Gamma_\mu ^{(1)}\right)^{\otimes N} 
\end{equation}
and hope it will accurately approximate the original $\Gamma$. Namely, we are trying to guess a good de Finetti representation of $\Gamma$ by (fictitiously) making measurements on $N-k$ subsystems and using the so-obtained information to construct a measure over one-body states. 

Now, for each measurement (orthonormal basis $(e_\mu)_\mu$ and positive matrices $(M_\mu)_\mu$ adding to the identity) we can evaluate the error between the density matrices of the associated $\widetilde{\Gamma}$ and those of the original $\Gamma$. Clearly, for any choice of measurement 
\begin{multline*}
 \inf_{\nu \in \cP (\cS)}\tr \left| A_1 \otimes \ldots \otimes A_k \left( {N \choose k} ^{-1} \Gamma_N ^{(k)} - \int_{\cS} \gamma^{\otimes k} \mathrm{d}\nu (\gamma) \right) \right| \\ \leq
 \tr \left| A_1 \otimes \ldots \otimes A_k {N \choose k} ^{-1} \left(  \Gamma_N ^{(k)} - \widetilde{\Gamma} ^{(k)} \right) \right|. 
\end{multline*}
So if we can construct a trial measurement such that the right-hand side of the above is bounded by the right-hand side of~\eqref{eq:deF info} there must exist a measure over one-body states such that~\eqref{eq:deF info} holds. See the above references for details on this procedure. The place where the construction ceases to be explicit is when the \emph{minimum} error over all measurements is bounded above by the \emph{maximum} error within a certain sub-class of factorized measurements.

\subsection{\index{Localization}{Localization} method}\label{sec:localization}

Now we sketch the proof of Theorem~\ref{thm:GS schro}, mostly by providing the localization (to finite dimensional spaces) needed to put Theorems~\ref{thm:deF semi} and~\ref{thm:deF info} to good use. See also~\cite[Chapter~7]{Rougerie-LMU,Rougerie-spartacus}

\bigskip 

\noindent \textbf{Localizing the Hamiltonian.} Denote 
\begin{equation}\label{eq:one body h}
 h = \left( -\im \nabla + \bA \right) ^2 + V \geq 0 
\end{equation}
the one-body Hamiltonian (assuming it is positive is just a shift of the energy reference). Let $\Lambda \geq 0$ be an energy cut-off that will ultimately be optimized over. Let 
\begin{equation}\label{eq:projectors}
 P = \1_{h \leq \Lambda}, \quad Q= \1 - P 
\end{equation}
be spectral projectors associated to $h$, which we use as localizers in energy-space. We refer to $Q\gH$ as the subspace of excited particles.

Our $h$ has compact resolvent, hence 
$$ N_\Lambda := \mathrm{dim} (P L^2 (\R^d)),$$
the number of energy levels below the cut-off $\Lambda$, is finite. In fact we have
\begin{equation}\label{eq:dim}
N_\Lambda := \dim (P) \leq C \Lambda ^{\frac{d}{s} + \frac{d}{2}}  
\end{equation}
with $s$ the exponent in~\eqref{eq:trapping s}. This goes under the name of a \index{Lieb-Thirring inequality}{Cwikel-Lieb-Rosenblum bound} and is a \index{Cwikel-Lieb-Rosenblum bound}{particular} case ($\delta = 0$) of  

\begin{lemma}[\textbf{A Lieb-Thirring inequality}]\label{lem:LT}\mbox{}\\
Let $\lambda_1,\ldots,\lambda_j,\ldots$ be the eigenvalues of~\eqref{eq:one body h}, counted with multiplicity. We have, for any $\delta \geq 0$, 
\begin{equation}\label{eq:LT}
\sum_{j, \lambda_j \leq \Lambda} \lambda_j ^\delta \leq C  \Lambda^{\delta + d/s + d/2}
\end{equation}
for some constant $C= C (d,s,\delta) >0$ with $s$ the exponent in~\eqref{eq:trapping s} and $d$ the spatial dimension.
\end{lemma}

\begin{proof}[Comments]
Statements of this type are particularly important in rigorous many-body quantum mechanics. See~\cite[Lemma~3.3]{LewNamRou-14c} and references therein for the proof of the particular version above, and~\cite[Chapter~4]{LieSei-09} or~\cite[Chapters~3 and~5]{Simon-05} for general background.  

The right-hand side of~\eqref{eq:LT} is obtained as the large $\Lambda$ asymptotics for the \index{Semiclassical analysis}{semiclassical} analogue of the left-hand side, namely the phase-space integral
$$
\iint_{\bx,\bp \in \R^d \times \R^d}  \left( |\bp|^2 + |\bx| ^s \right)^\delta \1_{\left\{|\bp|^2 + |\bx| ^s \leq \Lambda\right\}} \mathrm{d} \bx \mathrm{d}\bp.
$$
\end{proof}

The expectation value of the many-body energy per particle associated with~\eqref{eq:intro Schro op bis} in the state vector $\Psi_N$ is given by
\begin{equation}\label{eq:to two body}
N^{-1} \left\langle \Psi_N | H_N | \Psi_N \right\rangle = {N \choose 2} ^{-1} \frac{1}{2} \tr \left( H_2 \Gamma_N ^{(2)} \right) 
\end{equation}
with the two-body \index{Reduced density matrix}{density matrix} $\Gamma_N ^{(2)}$ as in~\eqref{eq:def red mat}. The two-body Hamiltonian is
\begin{equation}\label{eq:two body hamil} 
H_2 := h_{\bx_1} + h_{\bx_2} + w_{N,\beta} (\bx_1 - \bx_2) 
\end{equation}
with 
$$w_{N,\beta} (\bx) = N^{d\beta} w (N^{\beta} \bx).$$
To obtain a lower bound to the energy we localize the two-body Hamiltonian using the following simple lemma (this is~\cite[Lemma~3.6]{LewNamRou-14c}):

\begin{lemma}[\textbf{Localized two-body Hamiltonian}]\label{lem:localize-energy}\mbox{}\\
Assume that $\Lambda \ge C \eps^{-1} N^{d\beta}$ for $0< \eps < 1$ and a large enough constant $C>0$. Then we have, as operators on $L^2 (\R^{2d})$,   
\begin{align} \label{eq:GP H2-localized-error}
H_2 \geq  P^{\otimes 2} H_{2} ^\eps P^{\otimes 2} + \frac{\Lambda}{2} \left( Q \otimes \1 + \1 \otimes Q \right)
\end{align}
where 
\begin{equation}\label{eq:eps two body}
H_2 ^\eps = H_2 - \eps \left|w_{N,\beta}  \left(\bx-\by\right)\right|. 
\end{equation}
\end{lemma}

This says that if the cut-off is chosen large enough, the kinetic energy of excited particles outweighs their interaction energy (both the interaction between excited particles and the interaction between excited and non-excited particles).

\bigskip

\noindent\textbf{Localizing the state.} The main term to be bounded from below is now 
$$ {N \choose 2} ^{-1} \tr \left( \left(P^{\otimes 2} H_{2} ^\eps P^{\otimes 2} \right) \Gamma_N ^{(2)} \right) = {N \choose 2} ^{-1} \tr \left(  H_{2} ^\eps \left(P^{\otimes 2} \Gamma_N ^{(2)} P^{\otimes 2} \right)\right)$$
and the next trick is to view it not as the \emph{expectation of a localized Hamiltonian in the original state}, but as the \emph{expectation of the original Hamiltonian in a localized state}. The idea has a long history, recalled and distillated in~\cite{Lewin-11} to yield the method we now sketch. See also~\cite[Chapter~5]{Rougerie-spartacus,Rougerie-LMU} for more details and references.

We want to see the projected two-body density matrix $P^{\otimes 2} \Gamma_N ^{(2)} P^{\otimes 2}$ as the genuine density matrix of a projected state. It is proved in~\cite{Lewin-11} that this is doable, provided the latter state is looked for on the Fock space~\eqref{eq:Fock}. More precisely, there exists a unique state $\Gamma_N ^P$ on the projected \index{Fock space}{Fock space} 
$$ \gF (P\gH) = \C \oplus \gH \oplus \gH_2 \oplus \ldots \oplus \gH_n \oplus \ldots$$
of the form 
$$ \Gamma_N ^P = \Gamma_{N,0} ^P \oplus \Gamma_{N,1} ^P \oplus \ldots \oplus \Gamma_{N,N} ^P \oplus 0 \oplus \ldots  $$
such that, for all $k\leq N$ 
\begin{equation}\label{eq:loc DM}
\left( \Gamma_N ^P \right) ^{(k)} := \sum_{n\geq k} \left( \Gamma_{N,n} ^P \right)^{(k)} = P^{\otimes k} \Gamma_N ^{(k)} P ^{\otimes k}.
\end{equation}
Thus we write the quantity to be bounded below as  
\begin{equation}\label{eq:ener P localized}
 {N \choose 2} ^{-1} \sum_{n\geq 2}\tr \left(  H_{2} ^\eps \left( \Gamma_{N,n} ^P\right) ^{(2)}\right). 
\end{equation}
We can similarly consider a $Q$-localized state, with $Q= \1 - P$ and estimate the second term of~\eqref{eq:GP H2-localized-error} in a similar fashion. Very importantly, the $P$-localized and $Q$-localized states are related by the equality
\begin{equation}\label{eq:loc relation}
\tr\, \Gamma_{N,n} ^P = \tr\, \Gamma_{N,N-n} ^Q 
\end{equation}
which basically means that the probability of having $n$ particles out of $N$ $P$-localized is the same as the probability of having $N-n$ particles out of $N$ $(\1-P)$-localized. Certainly this is very reasonable to expect and the construction of these states does reflect that the two events just mentioned really are the same. We shall only need that they have the same probability.

\bigskip

\noindent\textbf{Proof of Theorem~\ref{thm:NLS}.} This is pretty much that given in Section~\ref{sec:proof deF}, made quantitative by inserting the above tools. We only sketch it.

\medskip

\noindent\textbf{1}. The energy upper bound is the easy part. Take a \index{Trial state, Hartree}{factorized Hartree ansatz}~\eqref{eq:ansatz}, and prove that the so-obtained Hartree energy with $w_{N,\beta}$ potential converges to 
the NLS one when $N\to \infty$. A short exercise in nonlinear analysis.

\medskip

\noindent\textbf{2}. For the lower bound, we insert the localized states in the energy lower bound expressed by~Lemma~\ref{lem:localize-energy}, e.g.~\eqref{eq:ener P localized} expresses the contribution of the first term in~\eqref{eq:GP H2-localized-error}. 
 
 \medskip

\noindent\textbf{3}. Now is the time to use the quantitative versions of the quantum de Finetti theorem for the $P$-localized state. In view of~\eqref{eq:loc DM} we can do this for the projection $\Gamma_{N,n}^P$ of $\Gamma_N^P$ to each $n$-particles sector of Fock space, and obtain an approximation for $P^{\otimes 2} \Gamma_N ^{(2)} P ^{\otimes 2}$ in the de Finetti form by summing the contributions of each $n$.
 
 \medskip

\noindent\textbf{4}. You might be worried that in the previous step we apply Theorems~\ref{thm:deF semi} or~\ref{thm:deF info} to states with $n\leq N$ particles, a number that needs not be large. Then the error estimates are not very efficient. But one can use the second term in~\eqref{eq:GP H2-localized-error}. For large $\Lambda$ it will say that the $Q$-localized state does not want to have too many particles. As per~\eqref{eq:loc relation} this means that the $P$-localized state has many particles, so that the main contributions in the above step come from sectors $n\sim N$, and give a good error.
 
 \medskip

\noindent\textbf{5}. The energy cut-off $\Lambda$ needs to be chosen $\propto N^{d\beta}$ as in Lemma~\ref{lem:localize-energy} to control the contributions from $Q$ terms. The dimension of the one-body Hilbert space in the \index{Quantum de Finetti theorem}{quantum de Finetti theorem} is then handled via~\eqref{eq:dim}. The final error term depending on $N$ and $\beta$ and tells us how big we can afford the latter to be. The parameter $\eps$ from Lemma~\ref{lem:localize-energy} is sent to zero at the very end in order for the expectation of~\eqref{eq:eps two body} in a factorized state to converge to the NLS energy.
 
 \medskip

\noindent\textbf{6}. If Theorem~\ref{thm:deF semi} has been used, as in~\cite{LewNamRou-14c}, the dependence on $s$ (growth of the \index{Trapping potential}{confining potential}) is pretty bad, as seen by combining~\eqref{eq:deF semi} and~\eqref{eq:dim}. This leads to a rather small, $s$-dependent $\beta$ for which the energy can be proved to converge.
 
 \medskip

\noindent\textbf{7}. Using Theorem~\ref{thm:deF info} instead, as in~\cite{Rougerie-19}, gives access to much better values of $\beta$ (those stated in the Theorem), for now the number of states below the energy cut-off does not matter much. Of course one has then to decompose the interactions potential in a form that permits the use of~\eqref{eq:deF info}. This can be done via the Fourier transform 
 \begin{align}\label{eq:Fourier}
 W (\bx-\by) &= \int_{\R ^d} \widehat{W} (\bp) e^{i \bp\cdot \bx} e^{-i \bp \cdot \by} \mathrm{d}\bp \nonumber \\
 &= \int_{\R ^d} \widehat{W} (\bp) \left(\cos (\bp \cdot \bx) \cos (\bp \cdot \by) + \sin (\bp \cdot \bx) \sin (\bp \cdot \by)\right) \mathrm{d}\bp
 \end{align}
Each term in the integral can be handled separately, for the multiplication operators in the integral indeed are of the form $A\otimes B$. 
 
 \medskip

\noindent\textbf{8}. Also, when using Theorem~\ref{thm:deF info} one has to pass to the limit a bit carefully to recover a measure that lives only on pure states in the end, as per Theorem~\ref{thm:DeFinetti fort}.  
 
 \medskip

\noindent\textbf{9}. Finally, a corollary of the energy convergence is that the final de Finetti measure must be concentrated on minimizers of the mean-field functional, which gives convergence of density matrices.

\bigskip

We shall not enter into more details. The main limitation of the method is in the \index{Localization}{localization} method, Lemma~\ref{lem:localize-energy}. It forces us to take the energy cut-off $\Lambda \propto N^{d\beta}$ to control the projection error. Then the final error in the energy is roughly that given by applying the quantitative de Finetti theorem in the low-energy subspace:
\begin{itemize}
 \item Theorem~\ref{thm:deF semi} gives an error $\sim N_\Lambda /N$ in trace-class norm. This we must multiply by the operator norm of the projected Hamiltonian, which is $\sim N^{d\beta}$ because the one-body operator is projected to values $\lesssim \Lambda$ and the interaction potential has $L^{\infty}$ norm $\sim N^{d\beta}$ too. The final error in the energy per particle is then, using~\eqref{eq:dim}, of order $N^{\alpha}$ with
 $$ \alpha = d\beta (1 + d/s + d/2) - 1.$$
 We obtain energy convergence  for $\alpha < 0$, which puts a severe, $s$-dependent, constraint on $\beta$.
 \item When applying Theorem~\ref{thm:deF info}, we decompose the interaction as in~\eqref{eq:Fourier}. At fixed $\bp$ we have an error $\propto N^{-1/2} (\log (N_\Lambda))^{1/2}$. Integrating over $\bp$ we multiply this by 
 $$ \int_{\bp \in \R^d}  |\widehat{w_{N,\beta} }(\bp) |  \mathrm{d}\bp \propto N^{d\beta},$$
 whence a final error $N^{d\beta -1/2} (\log (N_\Lambda))^{1/2}$ which is small if $d\beta < 1/2$ (the logarithmic dependence on $N_\Lambda$ plays no role).
\end{itemize}

\newpage

\section{\index{Coherent states}{Coherent states} method}\label{sec:coherent}

Now we turn to a different approach to the bosonic mean-field limit, initiated in~\cite{LieSeiYng-05,LieSei-06} with~\cite{LieTho-97} and earlier papers~\cite{HepLie-73,Hepp-74} as sources of inspiration. Contrarily to the de Finetti based method discussed in the two previous sections, the one we consider now proceeds by manipulating the Hamiltonian, not the state of the system. The reasons which made me decide the method deserved a section of its own instead of being included in Section~\ref{sec:hamil} are two-fold: 
\begin{itemize}
 \item the method is very general, it does not rely on any particular property of the Hamiltonian, as methods in Section~\ref{sec:hamil} did.
 \item the method, as the de Finetti-based one, lends itself to generalizations to deal with dilute and GP limits, see below.
\end{itemize}
As regards the second point, a bibliographical remark is in order: I have for pedagogical reasons separated the material of~\cite{LieSei-06} in several parts. The original source deals directly with the more difficult GP limit, without trying the MF limit first. Most of the ``generalizations'' mentioned above are thus contained in the original paper. Several other tools from~\cite{LieSei-06} will be introduced later in the review. Variants and refinements of the coherent state method have been used for different problems~\cite{LieSeiYng-09,Seiringer-06,Seiringer-08,Seiringer-ICMP10}.

I have not tried to give a complete proof here, and in particular I have not computed the value\footnote{The original reference can deal with any $\beta \leq 1$, but this requires the additional use of more sophisticated tools to be introduced later in the text.} of $\beta$  it yields (it will depend on $s$ rather badly, as in~\cite{LewNamRou-14c}). The interested reader should however have no difficulty filling in the gaps of the proof sketch we will provide for the 

\begin{theorem}[\textbf{NLS limit of bosonic ground states, restatement}]\label{thm:NLS bis}\mbox{}\\
Make Assumption~\ref{asum:pots} plus the more specific trapping condition~\eqref{eq:trapping s bis} for some $s>0$. Assume stability for the NLS functional as in Assumption~\ref{asum:NLS stab}. Further assume that $w\in L^{\infty} (\R^d)$. There exists a $\beta_0 (s)$ such that, for $0<\beta < \beta_0(s)$, the following holds in the limit $N \to + \infty$:

\smallskip 

\noindent\textbf{Convergence of the energy:} 
$$
\frac{E(N)}{N} \to \Enls. 
$$

\smallskip 

\noindent\textbf{Convergence of the one-body reduced density matrix:} let $\Gamma_N ^{(1)}$ be the one-body reduced density matrix of a many-body ground state $\Psi_N$. There exists a Borel probability measure $\mu$ on the set $\Mnls$ of NLS minimizers such that, along a subsequence, 
$$
{N} ^{-1} \Gamma_N ^{(1)} \to \int_{\Mnls} |u \rangle \langle u | \mathrm{d}\mu(u) 
$$
strongly in trace-class norm. 
\end{theorem}

\begin{proof}[Comments]\mbox{}\\
\noindent\textbf{1.} Again, the statement above does not reflect the full power of the method as it was introduced in~\cite{LieSei-06}. See Sections~\ref{sec:better loc} and~\ref{sec:coh reloaded} below for this. Needless to say, the method also applies at $\beta = 0$.

\medskip

\noindent\textbf{2.} The only thing that was not explicit in~\cite{LieSei-06} is the treatment of attractive interactions. It is also conceivable that the method could be improved to give convergence of higher density matrices, but we do not pursue this.

\medskip

\noindent\textbf{3.} To some extent, the de Finetti based method and the coherent state method are two sides of a same coin. More precisely they are somehow dual to one another. We shall discuss this in Subsection~\ref{sec:compare deF coh}.
\end{proof}

\subsection{\index{Coherent states}{Coherent states} formalism}\label{sec:coh form}

Here we follow mainly~\cite[Section~1.3]{KlaSka-85}, but see also~\cite{ComRob-12}, in particular Section 10.3 therein. We start from $\gH$, a finite-dimensional complex Hilbert space with dimension $D$ and an orthonormal basis $u_1,\ldots,u_D$. Think of $\gH$ as a subspace of $L^2 (\R^d)$. 

We work on the bosonic Fock space $\gF = \gF (\gH)$ based on $\gH$, cf~\eqref{eq:Fock}. Denote $|0\rangle$ its vacuum vector, i.e. 
$$ | 0 \rangle = 1 \oplus 0 \oplus \ldots.$$

\begin{definition}[\textbf{Bosonic coherent states}]\label{def:coh}\mbox{}\\
Let $Z = (z_1,\ldots,z_D) \in \C ^D$ and define the associated \emph{coherent state} $\Psi_Z\in\gF$ as 
\begin{align}\label{eq:def coh}
\Psi_Z &= \exp \left( \sum_{j=1} ^D z_j a^\dagger (u_j) - \overline{z_j} a (u_j) \right) |0\rangle \nonumber\\
&= e^{-\frac{1}{2}\sum_{j=1} ^D |z_j| ^2} e^{\sum_{j=1} ^D z_j a^\dagger (u_j)} |0\rangle \nonumber \\
&= e^{-\norm{u_Z} ^2 /2} \bigoplus_{n=0} ^\infty \frac{u_Z^{\otimes n}}{\sqrt{n!}}
\end{align}
where the annihilation/creation operators are defined as in Section~\ref{sec:second quant} and 
$$ u_Z = \sum_{j=1} ^D z_j u_j.$$
\end{definition}

That the first two definitions are equivalent follows from the Baker-Campbell-Hausdorff formula ($a^\dagger$ and $a$ commute with their commutator). The third definition makes contact with the form we have already encountered in~\eqref{eq:coh u}-\eqref{eq:coh u bis}. To see it is equivalent to the first two, note that
$$ \sum_{j=1} ^D z_j a^\dagger (u_j) = a^\dagger (u_Z).$$
Observe the crucial fact that coherent states are eigenvectors of annihilation operators:
\begin{equation}\label{eq:coh anni}
a (u_j) \Psi_Z = z_j \Psi_Z 
\end{equation}
and more generally, for any $v\in\gH$, 
$$ a (v) \Psi_Z = \langle v | u_Z \rangle \Psi_Z.$$
This will be used to perform $c$-number substitutions in the Hamiltonian: replacing creators/annihilators by numbers in the second quantized form~\eqref{eq:hamil second}.

One use of the coherent states is that they form an overcomplete basis of the Fock space:

\begin{lemma}[\textbf{Coherent state partition of unity}]\label{lem:coh clos}\mbox{}\\
Denoting $\mathrm{d}Z$ the Lebesque measure on $\C^D \simeq \R^{2D}$ we have the closure relation (\index{Schur's lemma}{Schur lemma})
\begin{equation}\label{eq:coh basis}
\pi ^{-D} \int_{\C^D} |\Psi_Z \rangle \langle \Psi_Z | \mathrm{d}Z = \1_{\gF} 
\end{equation}
and the overlap formula
\begin{equation}\label{eq:coh over}
 \langle \Psi_{Z'} | \Psi_Z  \rangle_{\gF} = \exp \left( -\frac{1}{2} \sum_{j=1} ^D \left( |z_j| ^2 + |z'_j| ^2 - 2 \overline{z'_j} z_j \right)\right).
\end{equation}
\end{lemma}

\begin{proof}
A well-known direct computation, for which it is useful to recall the factorization property of Fock space 
$$ \gF (\gH) = \gF\left( \sspan (u_1) \oplus \ldots \oplus \sspan (u_D) \right)\simeq \gF \left(\sspan (u_1)\right) \otimes \ldots \otimes \gF\left(\sspan(u_D)\right)$$
in the sense of unitary equivalence, see e.g.~\cite[Appendix~A]{HaiLewSol_2-09}. Computations can then be reduced to the case where $\gH$ is one-dimensional. Details for this case are in~\cite[Section~1.3]{KlaSka-85}
\end{proof}

We can represent an operator $\bbA$ on $\gF$ in the coherent state basis as 
$$ \bbA = \pi ^{-2D} \int_{\C^D \times \C ^D}  \langle \Psi_{Z'} |\bbA |\Psi_{Z} \rangle |\Psi_{Z'} \rangle \langle \Psi_Z | \mathrm{d}Z \mathrm{d}Z'.$$
In fact, the basis is overcomplete enough that $\bbA$ is fully characterized by the diagonal elements $\langle \Psi_{Z} |\bbA |\Psi_{Z} \rangle$. Then a desirable and, perhaps, not too unreasonable thing to look for is a diagonal representation of $\bbA$ in the coherent state basis~\cite{Berezin-72,Lieb-73b,Simon-80}. 

\begin{definition}[\textbf{Symbols}]\label{def:symbols}\mbox{}\\
If an operator $\bbA$ can be put in the form 
\begin{equation}\label{eq:up symb}
\bbA = \pi ^{-D} \int_{\C^D}   \bbA^{\rm up} (Z) \, |\Psi_{Z} \rangle \langle \Psi_Z | \mathrm{d}Z 
\end{equation}
we call the map
$$ \C^D \ni Z \mapsto \bbA^{\rm up} (Z) \in \C$$
\index{Symbol, upper and lower}{its} \emph{upper symbol}. 

For any operator $\bbA$ we call the map
\begin{equation}\label{eq:low symb}
\C^D \ni Z \mapsto \bbA ^{\rm low} (Z) := \langle \Psi_Z |\bbA| \Psi_Z \rangle_\gF \in \R^+.  
\end{equation}
its \emph{lower symbol}.
\end{definition}

Symbol is meant in the usual sense of semiclassical/microlocal analysis as a representation of a quantum (perhaps pseudo-differential) operator as a function on some classical (perhaps symplectic) phase-space. The words ``upper'' and ``lower'' refer~\cite{Berezin-72,Lieb-73b,Simon-80} to the fact that the former give upper bounds on quantum partition functions, while the latter give lower bounds. This is the content of the \index{Berezin-Lieb inequalities}{Berezin-Lieb inequalities}~\cite[Appendix~B]{Rougerie-LMU,Rougerie-spartacus}. Other \index{Wick and anti-Wick quantization}{names} used in the literature are: lower symbol $\simeq$ covariant symbol $\simeq$ Husimi function $\simeq$ anti-Wick quantization $\simeq$ Toeplitz quantization, and upper symbol $\simeq$ contravariant symbol $\simeq$ Wigner measure $\simeq$ Wick quantization.  

What is the use of introducing two concepts with similar-looking names if they are not closely related ? We have the 

\begin{lemma}[\textbf{Relation between symbols}]\label{lem:symb}\mbox{}\\
Let $\bbA$ be an operator on $\gF$ with upper symbol $\bbA ^{\rm up}$ (we assume it exists). Then its lower symbol is given by 
\begin{equation}\label{eq:up to low}
\bbA ^{\rm low} (Z) = e^{\partial_Z \cdot \partial_{\overline{Z}}} \bbA ^{\rm up} (Z)  
\end{equation}
where 
$$\partial_Z \cdot \partial_{\overline{Z}} := \sum_{j=1} ^D \partial_{z_j} \partial_{\overline{z_j}}$$
and $\partial_z = \frac{1}{2}\left(\partial_x - \im \partial_y\right),\, \partial_{\overline{z}} = \frac{1}{2}\left(\partial_x + \im \partial_y\right) $ for a complex number $z= x +\im y$.

In addition, an operator $\bbA$ (and thus its \index{Symbol, upper and lower}{upper symbol}, if it has one) is uniquely determined by its lower symbol.
\end{lemma}

\begin{proof}
See~\cite[Section~1.3]{KlaSka-85}. The first claim is a consequence of the overlap expression~\eqref{eq:coh over}, which gives  
$$  \left|\langle \Psi_{Z'} | \Psi_Z  \rangle_{\gF} \right|^2= \exp \left( -\frac{1}{2} \sum_{j=1} ^D |z_j-z_j'| ^2 \right).
$$
A convolution with the above can be identified by a Fourier-side multiplication and related to the heat flow $e^{-\partial_Z \cdot \partial_{\overline{Z}}}$. The second claim is a unique analytic continuation argument from the lower symbol to all matrix elements $\langle\Psi_{Z'} |\bbA |\Psi_{Z} \rangle$. 
\end{proof}

Of course~\eqref{eq:up to low} goes in the wrong direction. We know the lower symbol exists, and how to compute it. We want to infer that an upper symbol exists, and compute it. In view of~\eqref{eq:up to low} this is tantamount to solving the heat flow \emph{backwards in time}, a dangerous undertaking (think of the regularizing properties of the forward heat flow). Fortunately, for the operators we shall be interested in (recall~\eqref{eq:hamil second}), this is doable (in~\eqref{eq:up poly} below we only consider the action of the backwards heat flow on analytic functions):

\begin{lemma}[\textbf{Symbols of polynomial operators}]\label{lem:sym pol}\mbox{}\\
Pick any normal-ordered monomial in creation/annihilation operators, i.e. for any $i_1,\ldots,i_k,j_1,\ldots,j_\ell\in \{1,\ldots,D\}$ denote 
$$ \bbA = a^\dagger (u_{i_1}) \ldots a^\dagger (u_{i_k}) a (u_{j_1}) \ldots a (u_{j_\ell}).$$
We have  
\begin{equation}\label{eq:low poly}
\bbA ^{\rm low} (Z) = \overline{z_{i_1} \ldots z_{i_k}} z_{j_1} \ldots z_{j_\ell} 
\end{equation}
and 
\begin{equation}\label{eq:up poly}
\bbA ^{\rm up} (Z) = e^{-\partial_Z \cdot \partial_{\overline{Z}}} \bbA ^{\rm low} (Z).  
\end{equation}
\end{lemma}

\begin{proof}
Note first than in~\eqref{eq:up poly} the exponential in fact acts as a polynomial of finite degree, for the higher terms applied to the lower symbol give $0$. Thus~\eqref{eq:up poly} is just~\eqref{eq:up to low} in a case where it is legitimate to invert the relation and thus obtain the existence of an upper symbol.
 
The expression~\eqref{eq:low poly} is a straightforward consequence of~\eqref{eq:coh anni}, of the fact that coherent states are normalized and of $a^\dagger$ being the adjoint of $a$. 

We give the seed of the computation leading to~\eqref{eq:up poly}. Consider the case where 
$$ \bbA = a^\dagger (u_1) a (u_1).$$
Then define
$$\widetilde{\bbA} := \pi ^{-D} \int_{\C^D} \bbA ^{\rm up} (Z) |\Psi_Z \rangle \langle \Psi_Z|$$   
with $\bbA ^{\rm up} (Z)$ as in~\eqref{eq:up poly}. We do not know it is an upper symbol for $\bbA$ yet. To confirm this we must prove $\bbA = \widetilde{\bbA}$ and in view of Lemma~\ref{lem:symb} it suffices to compare the lower symbols of these operators. In the case at hand 
$$ \bbA ^{\rm up} = |z_1|^2 - 1,$$
thus 
\begin{align*}
\langle \Psi_Z | \widetilde{\bbA} | \Psi_Z \rangle &= \pi ^{-D} \int_{\C^D} \left| \langle \Psi_{Z} |\Psi_{Z'} \rangle \right|^2 |z'_1| ^2 \mathrm{d}Z' - 1  
\\&= \pi ^{-D}  \left\langle \Psi_{Z} \Big| a (u_1) \int_{\C^D} |\Psi_{Z'} \rangle \langle \Psi_{Z'}| \mathrm{d}Z' a^\dagger (u_1) \Big| \Psi_Z \right\rangle  - 1\\
&= \langle \Psi_Z | a (u_1) a^\dagger (u_1) | \Psi_Z \rangle - 1 \\
&= \langle \Psi_Z | \bbA | \Psi_Z \rangle
\end{align*}
where we used~\eqref{eq:coh anni}, the closure relation~\eqref{eq:coh basis} and the CCR~\eqref{eq:CCR}. The general case follows from similar considerations. 
\end{proof}

\subsection{Sketch of proof for the \index{Mean-field limit}{mean-field limit}}\label{sec:coh proof}

To simplify the approach as compared with the original~\cite{LieSei-06} we start again from Lemma~\ref{lem:localize-energy} to obtain a lower bound to the energy. We focalize on the $P$-localized part (first term in the right-hand side of~\eqref{eq:GP H2-localized-error}) and use again the $Q$-localized part to ensure that most particles in the original state are $P$-localized.

\medskip

\noindent\textbf{Energy lower bound in terms of an upper symbol.} Let $\Gamma_N^P$ be the $P$-localization of a many-body ground state, as defined in Section~\ref{sec:localization}. We seek a lower bound to 
\begin{align}\label{eq:coh low bound}
\tr \left( P^{\otimes 2} H_{2}^{\eps} P^{\otimes 2} \Gamma_N ^{(2)} \right) &=  \tr \left( H_{2}^{\eps}  \left( \Gamma_{N} ^P\right) ^{(2)} \right)\nonumber\\
&= \tr_{\gF (P \gH)} \left( \bH^P \Gamma_N^P \right)
\end{align}
where the last trace is over the \index{Fock space}{Fock space} generated from the $P$-projection of $L^2 (\R^d)$. We have denoted (compare with~\eqref{eq:hamil second})  
\begin{multline}\label{eq:P Hamil Fock}
\bH^P = \sum_{j=1} ^{D} \langle u_j | h  | u_j \rangle_{L^2(\R ^d)} a^\dagger (u_j) a (u_j) \\ + \frac{1}{N-1} \sum_{1\leq i,j,k,\ell \leq D} \left\langle u_i \otimes u_j | w_{N,\beta} (\bx-\by) | u_k \otimes u_{\ell} \right\rangle_{L^2 (\R^{2d})} a^{\dagger} (u_i) a ^{\dagger} (u_j) a (u_k) a (u_\ell). 
\end{multline}
where the $u_j$'s form a basis of $L^2 (\R^d)$ made of eigenfunctions of the one-body Hamiltonian~\eqref{eq:one body h}. Again, everything is now localized to the subspace $h\leq \Lambda$ and we denote $D= N_\Lambda$ the number of eigenvalues of $h$ below the cut-off. 

Now we are working on the Fock space, but we are not at liberty to use the lowest eigenvalue of $\bH^P$ as a lower bound to the energy. For repulsive interactions ($w_{N,\beta}\geq 0$) the latter is $0$ and attained by a state that has no particle at all. Instead we follow a trick of~\cite{LieSei-06}. Suppose we know $\Gamma_N ^P$ has exactly $N$ particles, namely it equals its projection on the $N$-particles sector of Fock space. Then~\eqref{eq:coh low bound} would,  for any constant $K>0$, equal 
$$ \tr_{\gF (P \gH)} \left( \bH^P \Gamma_N^P \right) + \frac{K}{N} \tr_{\gF (P\gH)} \left( (\cN - N) ^2 \Gamma_N^P \right) $$
where $\cN$ is the particle \index{Number operator}{number operator}~\eqref{eq:number second}. Of course our projected state $\Gamma_N ^P$ needs not have (and in fact, will not have) exactly $N$ particles. However, we can ensure its projection on Fock-space sectors where the particle number is not $\sim N$ is small by using the second, $Q$-localized term in~\eqref{eq:GP H2-localized-error} and~\eqref{eq:loc relation} exactly as sketched in the previous section. 

With apologies for this lack of details, we will from now on take for granted that $\Gamma_N^P$ is almost a $N$-body state and continue our proof sketch by seeking a lower bound to the modified Hamiltonian
\begin{equation}\label{eq:coh hamil}
\bHt = \bH^P + \frac{K}{N} (\cN - N) ^2 
\end{equation}
acting on the projected Fock space $\gF (P \gH)$. Now we introduce coherent states as discussed in the previous subsection. In view of~\eqref{eq:P Hamil Fock}, the above is a polynomial in annihilation and creation operators. Using Lemma~\ref{lem:sym pol} it thus has an upper symbol $\bH^{\rm up}$ in the \index{Coherent states}{coherent state} basis built from the eigenfunctions $u_1,\ldots,u_D$ of $h$ and we can write 
$$ \bHt = \pi ^{-D} \int_{Z\in \C ^D}  \bH^{\rm up} (Z) |\Psi_Z \rangle \langle \Psi_Z| \mathrm{d}Z \geq \inf_{Z\in\C^D} \bH^{\rm up} (Z) $$
in the notation introduced above, and using~\eqref{eq:coh basis}. We can thus bound the lowest eigenvalue of $\bHt$ from below by the minimum value of the upper symbol $\bH^{\rm up}$. This we shall estimate using Lemma~\ref{lem:sym pol}.

\medskip

\noindent\textbf{Difference between \index{Symbol, upper and lower}{upper and lower symbols}.} Indeed, using~\eqref{eq:up poly} and~\eqref{eq:low poly} we have that 
\begin{align}\label{eq:hamil symb}
\bH ^{\rm up} (Z) &= \left( 1 - \partial_Z \cdot \partial_{\overline{Z}} + \frac{1}{2} \left(\partial_Z \cdot \partial_{\overline{Z}}\right)^2\right) \bH^{\rm low} (Z)\nonumber\\
\bH ^{\rm low} (Z) &= \sum_{j=1} ^{D} h_j \overline{z_j} z_j  + \frac{1}{N-1} \sum_{1\leq i,j,k,\ell \leq D} w_{ijk\ell} \overline{z_i z_j} z_k z_l &+ \frac{K}{N} \left(\sum_{j=1} ^D  |z_j| ^2 - N \right)^2.
\end{align}
In the first expression we have used that $\bH ^{\rm low}$ is a quartic polynomial in the components of $Z$ to expand the exponential and discard higher order terms. We simplified the notation in a hopefully transparent way in the second expression (compare with~\eqref{eq:P Hamil Fock}), setting 
$$ w_{ijk\ell} = \left\langle u_i \otimes u_j | w_{N,\beta} (\bx-\by) | u_k \otimes u_{\ell} \right\rangle_{L^2 (\R^{2d})}.$$

The crux of the energy lower bound is that the main contribution comes from the first term in $\bH^{\rm up}$, namely we expect that to compute the minimum in $Z$ one can approximate 
\begin{equation}\label{eq:symb equiv}
 \bH^{\rm up} \approx \bH^{\rm low}. 
\end{equation}
This yields what we aim at, for a simple computation gives 
\begin{multline}\label{eq:low symb MF}
 \bH^{\rm low} (Z) = \langle u_Z | h | u_Z \rangle_{L^2 (\R^d)} + \frac{1}{N-1} \iint_{\R^d \times \R^d} |u_Z (\bx)| ^2 w_{N,\beta} (\bx - \by) |u_{Z} (\by)| ^2 \mathrm{d}\bx \mathrm{d}\by \\ + \frac{K}{N} \left( \norm{u_Z} ^2 _{L^2 (\R^d)} - N \right)^2  
\end{multline}
where $u_Z$ is as in Definition~\ref{def:coh}. We leave the reader convince himself/herself that if we minimize the above with respect to $Z$, letting $K$ be very large in the limit $N\to \infty$ (this is part of the fine tuning of all parameters in the proof, which we do not pursue), the minimum is attained for $\norm{u_Z}^2 \sim N$ and the infimum converges to the desired mean-field energy (note the scaling properties of the functional to extract the needed factor of $N$).

\medskip

\noindent\textbf{Errors in the energy estimate.} Thus what is left is to vindicate~\eqref{eq:symb equiv}. What we learn from looking at the lower symbol is that the minimum of $\bH^{\rm up}$ is likely to be attained where $\sum_{j=1} ^D |z_j| ^2 \sim N$ is large. But clearly, because of the derivatives in $Z$ it is made of, the difference 
\begin{equation}\label{eq:coh symb diff}
 \left| \bH^{\rm up} - \bH^{\rm low} \right| \ll \bH^{\rm low} 
\end{equation}
contains terms at most quadratic in $Z$ and thus ought to be smaller than the leading term $\bH^{\rm low}$. Let us have a look at the different terms one needs to estimate to confirm this expectation. We have 
\begin{equation}\label{eq:coh symb diff bis}
 \bH^{\rm up} - \bH^{\rm low} = - \partial_{Z} \cdot \partial_{\overline{Z}} \bH^{\rm low} + \frac{1}{2} \left( \partial_{Z} \cdot \partial_{\overline{Z}} \right) ^2 \bH^{\rm low}. 
\end{equation}
We list the contributions of the different terms below:

\medskip

\noindent$\bullet$ In the first term of~\eqref{eq:coh symb diff bis} we have the contribution 
$$ \sum_{j=1} ^D h_j \leq C \Lambda^{1+d/s+d/2}$$
from the quadratic term of $\bH^{\rm low}$, i.e. the sum of the $D$ first eigenvalues of $h$. This we estimate using the \index{Lieb-Thirring inequality}{Lieb-Thirring inequality} from Lemma~\ref{lem:LT}. Note that we need a bound $\ll N$, the total energy's order of magnitude. In view of the choice of $\Lambda$ in Lemma~\ref{lem:localize-energy} we are already limited to $d\beta < (1+d/s + d/2) ^{-1}$.

\medskip

\noindent$\bullet$ The quartic part of $\bH ^{\rm low}$ coming from the interaction potential contributes to the first term of~\eqref{eq:coh symb diff bis} a $-(N-1)^{-1}$ times
\begin{multline*} 
\sum_{1 \leq i,j,k \leq D} \overline{z_j} z_k \left( w_{ijik} + w_{ijki} + w_{jiik} + w_{jiki} \right) = \\
 \sum_{1\leq i \leq D} \left\langle u_i \otimes u_Z + u_Z \otimes u_i | w_{N,\beta} (\bx-\by) | u_i \otimes u_Z + u_Z \otimes u_i \right\rangle_{L^2 (\R^{2d})}
\end{multline*}
with $w_{N,\beta} ( \,.\, ) = N ^{d\beta} w (N^{\beta} \,.\, ) $
This can be bounded in absolute value e.g. by 
$$ \frac{C}{N} \norm{ w_{N,\beta}}_{L^\infty} \norm{u_Z}_{L^2} ^2 \sum_{1\leq i \leq D} \norm{u_i}_{L^2} ^2 = C N_\Lambda N^{d\beta - 1} \norm{u_Z}_{L^2} ^2 .$$

\medskip

\noindent$\bullet$ The quartic part of $\bH ^{\rm low}$ coming from the interaction potential contributes to the second term of~\eqref{eq:coh symb diff bis} a $(N-1)^{-1} /2$ times
$$\sum_{1 \leq i,j \leq D} \left( w_{ijij} + w_{ijji} + w_{jiij} + w_{jiji} \right) \leq N_\Lambda ^2 N^{d\beta}.$$

\medskip

\noindent$\bullet$ The term we introduced to control the particle number contributes 
$$ \frac{2K}{N} \norm{u_Z}_{L^2} ^2 - 2C$$
when hit by $- \partial_{Z} \cdot \partial_{\overline{Z}}$ and $2 K/N$ when hit by $\left(\partial_{Z} \cdot \partial_{\overline{Z}} \right) ^2/2.$

\medskip

The bottom line is that, if $\beta$ is not too big, all the above error terms are either $ \ll N$ independently of $Z$, or can be absorbed in the main term $\bH^{\rm low}(Z)$ without changing the asymptotics of the minimum thereof. A bit of work gives energy convergence for small values of $\beta$.

\medskip

\noindent\textbf{Convergence of \index{Reduced density matrix}{reduced density matrices} when the ground state is unique.} This is more tricky with this method (even controlling just the first one as stated in Theorem~\ref{thm:NLS bis}). We perturb the problem and rely on the \index{Feynman-Hellmann principle}{Feynman-Hellmann} Lemma~\ref{lem:FH} as used already in Section~\ref{sec:MF Onsager}. Consider adding a small multiple of a arbitrary bounded self-adjoint $k-$particles operator to $H_N$:
\begin{equation}\label{eq:perturb hamil many}
 H_{N,\eta} := H_N + \eta N { N\choose k} \sum_{1\leq i_1 \neq \ldots\neq i_k \leq N} \bA_{i_1\ldots i_k}. 
\end{equation}
Here $\bA$ acts on $\gH_k$ and $\bA_{i_1\ldots i_k}$ acts on the  $i_1, \ldots, i_k$ factors of $\gH_N$. The Feynman-Hellmann principle of~Lemma~\ref{lem:FH} tells us that, \textbf{if} $H_N$ has a unique minimizer $\Psi_N$ with reduced density matrices $\Gamma_N ^{(k)}$, then 
\begin{equation}\label{eq:FH LS}
 \tr\left( \bA {N \choose k} ^{-1} \Gamma_N ^{(k)} \right) = N^{-1}\partial_{\eta} E(N,\eta)_{|\eta = 0} 
\end{equation}
where $E(N,\eta)$ is the lowest eigenvalue of $H_{N,\eta}$. 

The method above directly applies to the perturbed $H_{N,\eta}$ when $k=1,2$ and gives the convergence of $N^{-1}E(N,\eta)$ to a perturbed mean-field energy. Both functions of $\eta$ are concave as infima over linear functions. The derivative in $\eta$ of $N^{-1}E(N,\eta)$ also converges to the derivative of the mean-field energy. An analog of~\eqref{eq:FH LS} for the mean-field functional gives convergence of the $2$-body density matrix. With a bit more sweat one could perhaps obtain higher density matrices also.

\medskip

\noindent\textbf{Convergence of reduced density matrices when the ground state is not unique.} There was an \textbf{if} above, namely we assumed the ground state to be unique. If this does not hold, one can still proceed with a bit of non-trivial convex analysis. The convergences of perturbed energies sketched above give some information on the structure of the (convex) set of limits of density matrices. In fact one can show that its extreme points are projectors onto mean-field minimizers, and conclude using the Choquet-Krein-Milman theorem~\cite{Simon-convexity}. This is done in details in~\cite[Section~3]{LieSei-06} for the one-body density matrix, and can probably be adapted to the $k$-body density matrix (provided one can first show the corresponding perturbed energies converge). \qed 

\medskip

We have indicated only the crudest of bounds in the above sketch. We leave it to the interested reader to figure out what combination of H\"older/Young/Sobolev/Gagliardo-Nirenberg/\index{Lieb-Thirring inequality}{Lieb-Thirring inequalities} yields the best estimate. But, if reaching large values of $\beta$ is the main concern, one should either rely on Section~\ref{sec:deF info} (whose tools lead to the claimed $\beta < 1/(2d)$ in Theorem~\ref{thm:NLS}). Or, better, couple the present techniques with those we shall describe later.  

\subsection{\index{Coherent states}{Coherent states} versus \index{Quantum de Finetti theorem}{de Finetti}}\label{sec:compare deF coh}

In Sections~\ref{sec:deF loc} and~\ref{sec:coherent} we have described two complementary methods that allow to treat very general Hamiltonians in the mean-field regime. Now we informally explain how they are related. Actually, it is the de Finetti method based on Theorem~\ref{thm:deF semi} that is closely related to the coherent states method. The variant based on Theorem~\ref{thm:deF info} stays somewhat on its own. 

Let us assume that, via localization as sketched previously, we are reduced to a lower bound on some Hamiltonian $\bH\equiv \bH^P$ acting on the Fock space $\gF \equiv \gF (P\gH)$, $P$ a dimension-$D$ orthogonal projector. Also assume that $\bH^P$ already contains a term penalizing the particle number so that its ground state is likely to be concentrated around the sector with $N$ particles. We thus seek a lower bound to 
\begin{equation}\label{eq:coh abs}
\inf_{\Psi \in \gF, \norm{\Psi} = 1} \left\{ \langle \Psi |\bH | \Psi\rangle \right\} = \inf_{\Gamma \in \cS (\gF), \tr\, \Gamma = 1} \left\{ \tr\left(\bH \Gamma\right)\right\}
\end{equation}
where the second  infimum is over (mixed) states on the Fock space and the equality is obtained via the correspondence $|\Psi \rangle \langle \Psi| = \Gamma$ for pure states 

We introduce coherent states $\Psi_Z,Z \in \C^D$ with the notation of Section~\ref{sec:coh form} (and emphasize the similarity with what has been discussed in the comments to Theorem~\ref{thm:deF semi}). We can now associate upper and lower symbols to essentially any operator, using Lemmas~\ref{lem:symb} and~\ref{lem:sym pol}. When estimating $ \tr \left(\bH \Gamma\right)$ we can thus either 
\begin{enumerate}
 \item Use the upper symbol of $\bH$ to write 
\begin{align*}
\tr\left(\bH \Gamma\right) &= \tr \left( \pi ^{-D} \int_{\C^D} \bH^{\rm up} (Z) |\Psi_Z \rangle \langle \Psi_Z | \mathrm{d}Z \: \Gamma \right) \\
&= \pi ^{-D} \int_{\C^D} \bH^{\rm up} (Z) \Gamma^{\rm low} (Z)\mathrm{d}Z
\end{align*}
Then we observe that $\bH^{\rm low} (Z)$ gives the mean-field energy and that $\Gamma^{\rm low} (Z)$ is a probability measure. Thus if we can replace $\bH^{\rm up} \rightsquigarrow \bH^{\rm low}$ in the above, we have won. This is what we did in Section~\ref{sec:coh proof}.
\item Use the upper symbol of $\Gamma$ to write  
$$ \tr\left( \bH \Gamma \right) = \pi ^{-D} \int_{\C^D} \bH^{\rm low} (Z) \Gamma^{\rm up} (Z)\mathrm{d}Z.$$
Here we have made the mean-field energy $\bH^{\rm low} (Z)$ appear already, but $\Gamma^{\rm up} (Z)$ needs not have a sign. We would like to approximate it by a probability measure, which essentially amounts to saying that its negative part is small. This is in some sense the outcome of the quantum de Finetti theorem in Section~\ref{sec:deF semi}.  
\end{enumerate}
In both cases above, what we really want is to approximate
$$ \tr\left( \bH \Gamma \right) \approx \pi ^{-D} \int_{\C^D} \bH^{\rm low} (Z) \Gamma^{\rm low} (Z)\mathrm{d}Z,$$
which gives the desired result. In case (1) we do it by approximating the Hamiltonian, using its simple expression as a polynomial in \index{Creation and annihilation operators}{annihilators/creators} and Lemma~\ref{lem:sym pol}. In case (2) we approximate the state instead, but the algebra is very much related. Observe indeed that the crucial step in the proof of~\eqref{eq:up poly} consists in normal-ordering a polynomial in creators/annihilators. This is also the crucial step in a proof of Theorem~\ref{thm:deF semi}, see~\eqref{eq:AWick}. 

Thus the semiclassical approach to mean-field limits boils down to comparing \index{Normal order, anti-normal order}{normal ordered} and anti-normal ordered polynomials in creators/annihilators (normal order means all creators on the left, anti-normal order means all creators on the right). Looking at things backwards, this is like comparing two different quantization procedures leading from the mean-field functional to the many-body Hamiltonian: the \index{Wick and anti-Wick quantization}{Wick quantization} (normal order) and the anti-Wick quantization (anti-normal order). We refer in particular to~\cite{AmmNie-08,AmmNie-09,AmmNie-11,Ammari-hdr} for more details on this point of view. Let us make a hint in this direction by stating the

\begin{definition}[\textbf{Wick and Anti-Wick quantizations}]\label{def:quantiz}\mbox{}\\
Let $\gH$ be a complex Hilbert space of dimension $D<\infty$. Let $h$ be a self-adjoint operator on $\gH$ with spectral decomposition
$$ h = \sum_{j=1} ^D h_j |u_j \rangle \langle u_j|$$
and $w$ be a self-adjoint operator on $\gH ^{\otimes 2}$. Define the polynomial of $Z=(z_1,\ldots,z_D)\in \C^D$
\begin{align}\label{eq:MF poly}
\cE (Z) &:= \langle u_Z | h | u_Z \rangle_{\gH} + \frac{1}{2} \langle u_Z ^{\otimes 2} | w| u_Z ^{\otimes 2}\rangle_{\gH_2} \\
&= \sum_{j=1} ^D h_j \overline{z_j} z_j + \frac{1}{N-1} \sum_{1\leq i,j,k,\ell \leq D} w_{ijk\ell} \overline{z_i z_j} z_k z_\ell
\end{align}
where $u_Z$ is as in Definition~\ref{def:coh} and 
$$ w_{ijk\ell} = \langle u_i \otimes u_j | w | u_k \otimes u_{\ell}\rangle_{\gH^{\otimes 2}}.$$
Define two operators on $\gF (\gH)$: the \emph{Wick quantization} of $\cE$
\begin{equation}\label{eq:Wick}
\cE^{\rm W} = \cE (a^\dagger,a):=  \sum_{j=1} ^D h_j a^\dagger_j a_j + \frac{1}{N-1} \sum_{1\leq i,j,k,\ell \leq D} w_{ijk\ell} a^\dagger_i a^\dagger_j a_k a_\ell
\end{equation}
and its \emph{anti-Wick quantization}
\begin{equation}\label{eq:AWick bis}
\cE^{\rm AW} = \cE (a,a^\dagger):=  \sum_{j=1} ^D h_j  a_j a^\dagger_j + \frac{1}{N-1} \sum_{1\leq i,j,k,\ell \leq D} w_{ijk\ell}  a_k a_\ell a^\dagger_i a^\dagger_j
\end{equation}
where the annihilation/creation operators are those defined by $u_1,\ldots,u_D$. 
\end{definition}

The idea is to replace complex numbers by annihilation/creation operators in formal expressions. Since the latter objects do not commute, a choice has to be made as regards the order in which to put them, leading to the two cases~\eqref{eq:Wick} and~\eqref{eq:AWick bis}. The former is what we should start from to do many-body quantum mechanics (recall~\eqref{eq:hamil second}).

The relation with what we have discussed previously is 

\begin{lemma}[\textbf{Wick and Anti-Wick quantizations}]\label{lem:quantiz}\mbox{}\\
We use the notation of the previous definition and the concepts of Definition~\ref{def:symbols}. Then 
\begin{itemize}
 \item $\cE^{\rm W}$ has $Z\mapsto \cE (Z)$ for lower symbol.
 \item $\cE^{\rm AW}$ has $Z\mapsto \cE (Z)$ for upper symbol.
\end{itemize}
\end{lemma}

\begin{proof}
Only the second statement has not been discussed so far, but this is a variant of~\eqref{eq:up poly}. We hint at the proof by considering the case $D=1$ and $\cE (Z) = |Z|^2$. Then we have to prove that 
$$ \pi ^{-1} \int_\C \overline{Z} Z |\Psi_Z \rangle \langle \Psi_Z| \mathrm{d}Z = a a^{\dagger}.$$
But 
\begin{align*}
 \langle u | a a^{\dagger} | v \rangle &= \langle  a^{\dagger} u | a^{\dagger} v \rangle \\
 &= \pi ^{-1} \int_\C   \langle a^\dagger u |  \Psi_Z \rangle \langle \Psi_Z | a^\dagger v \rangle  \mathrm{d}Z \\
&=\pi ^{-1} \int_\C  \langle  u |  a \Psi_Z \rangle \langle a \Psi_Z | v \rangle  \mathrm{d}Z \\
&= \pi ^{-1} \int_\C \overline{Z} Z  \langle  u |  \Psi_Z \rangle \langle \Psi_Z | v \rangle  \mathrm{d}Z
\end{align*}
using the coherent state closure relation~\eqref{eq:coh basis} and~\eqref{eq:coh anni}.
\end{proof}

As a final remark regarding semiclassics we mention the \index{Berezin-Lieb inequalities}{Berezin-Lieb inequalities}~\cite{Berezin-72,Lieb-73b,Simon-80} that give bounds on free-energies/partition functions instead of ground-state energies, i.e. on the problem with temperature instead of that at zero temperature.  See~\cite[Appendix~B]{Rougerie-spartacus,Rougerie-LMU} and~\cite{LewNamRou-14d,LewNamRou-17,LewNamRou-18b} for further discussion of this topic. 

\chapter{\index{Dilute limit}{Dilute limits}}\label{cha:dilute}

Now we move one step further as regards physical relevance (with application to \index{Cold atoms}{cold atomic gases} in mind) and mathematical sophistication. Namely, we attack the dilute regime, $\beta > 1/d$ (see the discussion in Section~\ref{sec:sca lim} where $\beta$ enters in~\eqref{eq:scaled potential}). Remark that in most of our previous discussion we had not quite covered the full range $\beta < 1/d$. The best we could do in generality was $\beta < 1/(2d)$ in~Theorem~\ref{thm:NLS}.  Also recall the last comment following Theorem~\ref{thm:NLS} (in 1D, any $\beta >0$ can be covered with its method of proof plus the use of a Sobolev inequality). Thus only the cases $d=2,3$ still require our attention. 

This chapter presents extensions of the techniques introduced in~Sections~\ref{sec:deF loc} and~\ref{sec:coherent}. 

\begin{itemize}
 \item In Section~\ref{sec:better loc} we follow a method of~\cite{LieSei-06} to deal with repulsive interactions. This couples the coherent state method with a much better localization technique than that of Lemma~\ref{lem:localize-energy}. 
 \item In Section~\ref{sec:mom} we introduce a set of a priori estimates derived from the variational \index{Many-body Schr\"odinger equation}{many-body Schr\"odinger equation} satisfied by energy minimizers. This is the first time we depart from a purely variational treatment.
 \item In Section~\ref{sec:mom rep} we couple the \index{Moments estimates}{moments estimates} to the de Finetti-based method in the  case of repulsive interactions. The results will be improved later when we discuss the Gross-Pitaevskii regime, but we continue with our motto of introducing new tools one at a time.
 \item In Section~\ref{sec:mom att} we state the best results known to date regarding interactions with an attractive part. In 2D we thus obtain genuinely attractive (focusing) mass-critical NLS functionals. In 3D the limit functionals need to be repulsive (defocusing) but we can work with potentials that are only classically stable in the sense of Definition~\ref{def:stability}, instead of purely repulsive.
\end{itemize}

\section{Better \index{Localization}{localization} in the \index{Coherent states}{coherent states} method}\label{sec:better loc}

For purely repulsive interactions it turns out one can still work variationally by being much more careful about the localization method, i.e. with the process of separating particles between low and high momenta and throwing away some interaction energy for a lower bound. In fact, by retaining part of the interaction between low and high momenta we will be able to prove the following using methods\footnote{We have not yet exhausted the full content of this paper, cf the bibliographical comment at the beginning of Section~\ref{sec:coherent}. The tools allowing to reach $\beta = 1$ will be presented below.} from~\cite{LieSei-06}. Recall the \index{Non-linear Schr\"odinger energy functional}{NLS energy functional} 
$$
\cEnls[u]:= \int_{\R^d} \left| \left( -\im \nabla + \bA \right) u \right| ^2 + V |u| ^2 + \frac{b_w}{2} \int_{\R^d} |u(\bx)| ^4  \mathrm{d}\bx.
$$
with minimum (under unit $L^2$ mass constraint) $\Enls$, minimizer(s) $\unls$ and the notation 
$$
b_w := \int_{\R^d} w.  
$$

\begin{theorem}[\textbf{Dilute limit in 3D, repulsive case}]\label{thm:dilute 3D}\mbox{}\\
Let $d=3$ Make Assumptions~\ref{asum:pots} plus the more specific trapping condition~\eqref{eq:trapping s bis} for some $s>0$. Let
\begin{equation}\label{eq:beta dilute}
0 < \beta < 2/3
\end{equation}
and assume a purely repulsive interaction, $0 \leq w \in L^\infty (\R^3)$. We have 

\smallskip 

\noindent\textbf{Convergence of the energy:} 
$$
\frac{E(N)}{N} \to \Enls. 
$$

\smallskip 

\noindent\textbf{Convergence of the one-body reduced density matrix:} let $\Gamma_N ^{(1)}$ be the one-body reduced density matrix of a many-body ground state $\Psi_N$. There exists a Borel probability measure $\mu$ on the set $\Mnls$ of NLS minimizers such that, along a subsequence, 
$$
{N} ^{-1} \Gamma_N ^{(1)} \to \int_{\Mnls} |u \rangle \langle u | \mathrm{d}\mu(u) 
$$
strongly in trace-class norm. 
\end{theorem}

\begin{proof}[Comments]\mbox{}\\
\noindent\textbf{1.} The technique we shall expose certainly also works in 2D. We leave it to the reader to adapt the Sobolev exponents everywhere in~\cite{LieSei-06} to figure out which $\beta$ it allows to reach.
 
\medskip 
 
\noindent\textbf{2.} We shall use the positivity of the interaction in the proof below. Most likely we thus may not hope to access higher density matrices with the method below. Indeed, as sketched in Section~\ref{sec:coh proof} we would need to perturb the Hamiltonian with arbitrary (not necessarily positive) $k$-body terms, $k\geq 2$. 

\end{proof}

In the rest of this section we sketch the proof of the above, following~\cite{LieSei-06}. 

\medskip

\noindent\textbf{Extension to \index{Fock space}{Fock space}.} As in Section~\ref{sec:coh proof}, we want to use coherent states, and thus extend the original $N-$body Hamiltonian to the Fock space. We trade the sharp value of the particle number for a penalization thereof by writing 
$$
E(N) \geq \inf \mathrm{spec}_{\gF} \left( \bH_N + \frac{K}{N} \left( \cN - N\right) ^2 \right)
$$
for any $K>0$ ($E(N)$ is actually the supremum over $K>0$ of the right-hand side). Here $E(N)$ is the ground state energy of the original Hamiltonian acting on $\gH_N$. On the right-hand side we have the bottom of the spectrum of the Fock-space Hamiltonian inside the parenthesis: $\bH_N$ is the original $H_N$ extended to Fock space using~\eqref{eq:hamil second} (with $w = N^{-1} w_{N,\beta}$) and $\cN$ is the particle \index{Number operator}{number operator}. 

From now on we denote 
\begin{equation}
 \bG := \bH_N + \frac{K}{N} \left( \cN - N\right)^2 
\end{equation}
the Fock-space operator we shall concentrate on.

\medskip

\noindent\textbf{Coherent states, again.} The first observation is that it is not necessary to first project the Hamiltonian to finite dimensions before introducing coherent states. It is still desirable to introduce coherent states only for finitely many modes of the one-body Hamiltonian because it is not obvious what becomes of~\eqref{eq:coh basis} when $D= \infty$. 

We carry on with the notation of Section~\ref{sec:localization}, and split the one-body Hilbert space $\gH = L^2 (\R^d)$ between low kinetic energy modes ($P\gH$ with $P = \1_{h\leq \Lambda}$) and high kinetic energy modes ($Q\gH$ with $Q= \1 - P$). Then, as previously mentioned, the Fock space $\gF (\gH)$ tensorizes 
$$ \gF (P \gH \oplus Q \gH) \simeq \gF (P\gH) \otimes \gF (Q\gH).$$
This precisely means that there is a unitary operator
$$ \cU : \gF (P \gH \oplus Q \gH) \mapsto \gF (P\gH) \otimes \gF (Q\gH)$$
defined by its action on creation operators (with a similar formula for annihilation operators)
$$ \cU \, a^\dagger (f) \,\cU = a^{\dagger} (Pf)\otimes \1 + \1 \otimes a^{\dagger } (Qf).$$
See e.g.~\cite[Appendix~A]{HaiLewSol_2-09} or~\cite{Ammari-04,Lewin-11} for more details. Any operator $\bbA$ acting on $\gF(\gH)$ (or at least, and this is all we need, any polynomial in annihilation/creation operators) is customarily identified with its action $\cU \, \bbA \, \cU^*$ on $\gF (P\gH) \otimes \gF (Q\gH)$. 

Now we introduce coherent states for the modes $u_1,\ldots,u_D$ spanning $P\gH$ as described in Section~\ref{sec:coh form}. From~\eqref{eq:coh basis} we have the closure relation/\index{Schur's lemma}{Schur's lemma}
$$ \1_{\gF (P\gH) \otimes \gF (Q\gH)} = \pi ^{-D} \int_{\C^D} |\Psi_Z \rangle \langle \Psi_Z | \otimes \1_{\gF (Q\gH)} \mathrm{d}Z.$$
To an operator $\bbA$ on $\gF (\gH)$ we can then associate a \index{Symbol, upper and lower}{\emph{lower symbol}}
\begin{align}\label{eq:low symb ext}
\bbA^{\rm low} (Z) &= \langle \Psi_Z | \bbA |\Psi_Z\rangle \nonumber \\
&= \tr_{\gF(P\gH)} \left( |\Psi_Z \rangle \langle \Psi_Z | \otimes \1_{\gF (Q\gH)} \,\cU\, \bbA \, \cU^* \right) 
\end{align}
where the first line is the usual notation and the second its actual meaning. This is now a function from $\C^D$ (or, equivalently, from $P\gH$) with values in operators on $\gF (Q \gH)$. If there exists another  function $\bbA^{\rm up}$ satisfying
\begin{equation}\label{eq:up symb ext}
\cU\, \bbA \, \cU^* = \pi ^{-D} \int_{\C ^D} \bbA ^{\rm up} (Z) |\Psi_Z \rangle \langle \Psi_Z | \otimes \1_{\gH} \mathrm{d}Z 
\end{equation}
we call it the \emph{upper symbol} of $\bbA$. The following consequences of Lemmas~\ref{lem:symb}, \ref{lem:sym pol} and~\ref{lem:quantiz} will be of use. First, for operators that one can express as polynomials in annihilators/creators
\begin{equation}\label{eq:symbs ext}
\bbA ^{\rm up} (Z) = e^{-\partial_{Z} \cdot \partial_{Z}} \bbA^{\rm low} (Z).  
\end{equation}
Next, somewhat schematically,
\begin{equation}\label{eq:symbs ext bis}
\bbA ^{\rm low} (Z) = \mathrm{Cnum}_{\gF(P\gH)} \left( \mathrm{Nord}\, \bbA \right). 
\end{equation}
Here $\mathrm{Nord}$ denotes \index{Normal order, anti-normal order}{normal} ordering: $\bbA$ is put in a form with all creation operators on the left and all annihilation operators on the right\footnote{We will always use~\eqref{eq:symbs ext bis} with already normal-ordered expressions such as~\eqref{eq:hamil second} anyway.}. Then we perform the \index{Classical number substitution}{classical number substitution} $\mathrm{Cnum}$, by which we mean that the creation operator $a^{\dagger} (u_j)$ is replaced by $\overline{z_j}$ and the annihilation operator $a(u_j)$ by $z_j$, this for all $1\leq j \leq D$.

\medskip 

\noindent\textbf{\index{Localization}{Localization} of symbols.} Now we replace the use of Lemma~\ref{lem:localize-energy} by lower bounds to the symbols defined above, which are now operators on $\gF(Q\gH)$. This is our way to keep track of the excited particles. In view of~\eqref{eq:up symb ext} what we really want is a lower bound to the upper symbol, for it gives direct access to the original operator we are interested in. We however start with the lower symbol. Hereafter we lighten notation by setting, for $u\in L^2 (\R^d)$ 
\begin{align}\label{eq:WNbeta}
 I [u] &= \iint_{\R^d \times \R^d} |u (\bx)| ^2 W_{N,\beta} (\bx - \by) |u (\by)|^2 \mathrm{d}\bx \mathrm{d}\by\nonumber\\
 W_{N,\beta} (\bx) &= N^{3\beta - 1} w (N^{\beta} \bx) \geq 0.
\end{align}

\begin{lemma}[\textbf{Lower bound to lower symbol}]\label{lem:bound low}\mbox{}\\
Denote $h_j$ the $j$-th eigenvalue of $h$ (associated to the eigenfunction $u_j$) and 
\begin{equation}\label{eq:kin excit coh}
T = \sum_{j = D+1} h_j a^\dagger (u_j) a (u_j) 
\end{equation}
the kinetic energy operator restricted to $\gF (Q\gH)$. Let $u_Z$ be as in Definition~\ref{def:coh}. For any $\delta >0$ we have, as operators on $\gF (Q\gH)$, 
\begin{align}\label{eq:low low}
\bG ^{\rm low} (Z) &\geq \langle u_Z | h | u_Z \rangle + I [u_Z] \left(1 - \delta - h_D ^{-1/4} N^{-1} T \right) + \frac{K}{N} \left( \norm{u_Z}_{L^2} ^2 - N \right)^2 \nonumber \\
&- C \left( N^{\frac{\beta - 1}{2}} + N^{-1/2} h_D ^{-1/4} \sqrt{T} + N^{-1} \sqrt{T}\right)\left( \langle u_Z | h | u_Z \rangle + I [u_Z] \right) - \delta^{-1} N^{3\beta - 1}.
\end{align}
\end{lemma}

\begin{proof}[Comments]
This is an adaptation of~\cite[Equation (61)]{LieSei-06}. The positivity of the interaction is used by writing that, as an operator on the two-body space,
$$
W_{N,\beta} \geq P^{\otimes 2} W_{N,\beta} P^{\otimes 2} + (1-P^{\otimes 2}) W_{N,\beta} P^{\otimes 2} + P^{\otimes 2} W_{N,\beta} (1-P^{\otimes 2})
$$
for the difference between left and right sides is 
\begin{equation}\label{eq:complete square}
(1-P^{\otimes 2}) W_{N,\beta} (1-P^{\otimes 2})\geq 0. 
\end{equation}
Compared to Lemma~\ref{lem:localize-energy} we keep track of the interaction between pairs of low energy particles and pairs of high-low (or high-high) energy particles. This means that the \index{Second quantization}{second quantized} interaction $\sum_{ijk\ell} W_{ijk\ell} a^\dagger_i a^\dagger_j a_k a_\ell$ is bounded below by retaining terms with at most two indices $\geq D$. This leads, after using~\eqref{eq:symbs ext bis}, to a lower bound to the lower symbol in terms of a \index{Bogoliubov Hamiltonian}{quadratic} (Bogoliubov) Hamiltonian (in the $a,a^\dagger$ operators of the excited space) that one can then, with some effort, control\footnote{One might also use exact expressions for the ground state energy of bosonic quadratic Hamiltonians, but this is not the approach pursued in~\cite{LieSei-06}. See Sections~\ref{sec:Bogoliubov} and~\ref{sec:Bogoliubov low} below for comments on quadratic Hamiltonians.} using the kinetic energy of the excited modes (which is also quadratic). 

The virtue of~\eqref{eq:low low} is that it allows to pass to the limit first in $N\to \infty$ and then only $D\to \infty$, provided $\beta < 2/3$. Indeed, the main terms on the first line will be proved to be of order $N$, and we should expect the kinetic energy $T$ to be at most of this order as well (namely, its expectation value in a ground state will be at most of this order). Dividing the whole inequality by $N$ we see that the errors are all either of the form $f(D) o_N (1)$, a function of $D$ times something becoming small when $N\to\infty$ or the form $f(N) o_D (1)$, a \emph{bounded} function of $N$ times something becoming small when $D\to \infty$. The limiting term is the last one, which is $\ll N$ only provided $\beta < 2/3$. Under this condition we can take successively the limits $N\to \infty,D\to \infty,\delta\to 0$ and isolate the leading order.  
\end{proof}

Next we turn to the upper symbol. It follows from the considerations in Section~\ref{sec:coh form} that $\bG$ does have one, and that it is related to the lower symbol via~\eqref{eq:symbs ext}. The main part will come from the lower symbol itself, but we still have to bound the terms coming from the $Z,\overline{Z}$  derivatives in~\eqref{eq:symbs ext}. 

\begin{lemma}[\textbf{Lower bound to upper symbol}]\label{lem:bound up}\mbox{}\\
Same notation as above, with in addition
$$\cN ^Q = \sum_{j=D+1} ^\infty a^\dagger (u_j) a (u_j)$$
the number (operator) of excited particles. We have, as operators on $\gF (Q\gH)$, 
\begin{align}\label{eq:low up}
\bG^{\rm up} (Z) - \bG ^{\rm low} (Z) &\gtrsim - \sum_{j=1} ^D h_j - \frac{KD}{N} \left( \cN_Q + \norm{u_Z}_{L^2} ^2 + 1 \right) \nonumber\\
&-N^{-1}\left( \langle u_Z | h| u_Z \rangle  + T \right) \sum_{j\leq D} h_j ^{1/2} \nonumber\\
&- N^{\frac{\beta -2}{2}}  \left( \langle u_Z | h| u_Z \rangle + \cN_Q + 1\right) \sum_{j\leq D} h_j ^{3/4}
\end{align}
where $\gtrsim e $ means ``larger than a universal constant times $e$''.
\end{lemma}

\begin{proof}[Comments]
In view of~\eqref{eq:symbs ext} and since $\bG$ is quartic in \index{Creation and annihilation operators}{annihilators/creators}, this is about bounding from below 
$$ -\partial_Z \cdot \partial_{\overline{Z}} \, \bG^{\rm low} + \frac{1}{2} \left( \partial_Z \cdot \partial_{\overline{Z}}\right) ^2 \bG^{\rm low}.$$
The contribution of the interaction to the second term above is non-negative and can thus be dropped from the lower bound. Essentially this is a rephrasing of the convexity of the mean-field interaction energy and is based on $w$ being non-negative. 

Once this has been observed, the task left is to bound the term $\partial_Z \cdot \partial_{\overline{Z}} \bG^{\rm low}$ from above. The end result we rephrased in the lemma is~\cite[Equation~(71)]{LieSei-06}. The value of~\eqref{eq:low up} is that it is again possible to pass to the limit therein, first $N\to \infty$ and next $D\to \infty$, to obtain a $o(N)$.
\end{proof}

\noindent\textbf{Conclusion.} Using the two above lemmas and~\eqref{eq:up symb ext} yields an operator lower bound on $\bG$ (that we identify with $\cU \, \bG \, \cU^*$ since they are unitarily equivalent) of the form 
\begin{equation}\label{eq:coh final}
\bG \geq \pi^{-D} \int_{\C^D} \widetilde{\bG} (Z) |\Psi_Z \rangle \langle \Psi_Z | \otimes \1_{\gF (G\gH)} \mathrm{d}Z - \delta^{-1} N^{3\beta - 1} - \sum_{j=1} ^D h_j 
\end{equation}
with 
\begin{align*}
\widetilde{\bG} (Z) &- \langle u_Z | h | u_Z \rangle - I [u_Z] \left(1 - \delta - h_D ^{-1/4} N^{-1} T \right) - \frac{K}{N} \left( \norm{u_Z}_{L^2} ^2 - N \right)^2\nonumber \\
&\gtrsim - \left( \langle u_Z | h | u_Z \rangle + I [u_Z] \right) \left( N^{\frac{\beta - 1}{2}} + N^{-1/2} h_D ^{-1/4} \sqrt{T} + N^{-1} \sqrt{T}\right) \\
&- \frac{KD}{N} \left( \cN_Q + \norm{u_Z}_{L^2} ^2 + 1 \right) -N^{-1}\left( \langle u_Z | h| u_Z \rangle  + T \right) \sum_{j\leq D} h_j ^{1/2}\\
&- N^{\frac{\beta -2}{2}}  \left( \langle u_Z | h| u_Z \rangle + \cN_Q + 1\right) \sum_{j\leq D} h_j ^{3/4} - C \delta^{-1} N^{3\beta - 1}.
\end{align*}
To conclude we need to take the expectation value of both sides of~\eqref{eq:coh final} in a ground state of $\bG$, divide by $N$, and take the limits first $N\to \infty$ then $D\to \infty$, then $\delta \to 0$ and finally $K\to \infty$. This gives the needed energy lower bound in terms of the NLS energy (recall that the most stringent condition comes from the error $\delta^{-1} N^{3\beta - 1}$, which is $\ll N$ only for $\beta <  2/3$), provided we know some simple a priori bounds on the ground state. But since the Hamiltonian $\bG$ is made only of positive terms and its ground-state energy can easily be bounded above by a multiple of $N$, there is no difficulty in obtaining 
$$ \langle \Psi_{\bG} | T | \Psi_{\bG} \rangle_{\gF} \leq C N  $$
for any ground state vector $\Psi_\bG$, with $T$ the kinetic energy of the excited modes~\eqref{eq:kin excit coh}. By concavity we also have 
$$ \langle \Psi_\bG | \sqrt{T} | \Psi_\bG \rangle_{\gF}  \leq C \sqrt{N} $$
and since 
$$ \cN_Q \leq h_D ^{-1} T $$
we also have 
$$ \langle \Psi_\bG | \cN_Q | \Psi_\bG \rangle_{\gF} \leq C h_D ^{-1} N.$$
These are all the estimates needed to close the proof of the energy lower bound.

To deduce convergence of the one-body density matrix, the argument is the same (\index{Feynman-Hellmann principle}{Feynman-Hellmann-like}) as that sketched in Section~\ref{sec:coh proof}. The case of a non-unique ground state (fairly common with non-trivial magnetic field, $\bA \neq 0$) requires refined arguments, in particular some convex analysis. We do not reproduce the details.

\medskip

\section{\index{Moments estimates}{Moments estimates}}\label{sec:mom}

We saw in the previous section that the dilute regime could be reached by passing to the limit first in the particle number $N\to \infty$ and then only in the kinetic energy cut-off $\Lambda$. This is the key to bypass the bad dependence of error estimates on the dimension of the low-energy one-particle state-space in the semiclassical de \index{Quantum de Finetti theorem}{Finetti} method and the \index{Coherent states}{coherent states} method. 

Next we turn to another set of tools that allows to take limits in this order or, at least, to take the energy cut-off to infinity much more slowly than needed in Lemma~\ref{lem:localize-energy}. The main virtue of these tools, compared to those of the previous section, is that they will allow us to deal with attractive interactions. For the first time in this review we will exploit the variational \index{Many-body Schr\"odinger equation}{many-body Schr\"odinger equation}  satisfied by minimizers. The main idea is that, whereas control of the many-body Hamiltonian itself gives bounds on the kinetic energy
$$ \tr \left( h \Gamma_N^{(1)}\right)$$
with $h$ the one-particle Hamiltonian and $\Gamma_N^{(1)}$ the one-particle reduced density matrix of a ground state, control of higher powers of the Hamiltonian gives access to moments of the kinetic energy, say the second 
\begin{equation}\label{eq:2 moment}
\tr \left( h\otimes h \Gamma_N^{(2)}\right). 
\end{equation}
This provides a much better control of the localization error when one projects the problem to low kinetic energy modes. 

For a ground state vector $\Psi_N$ we have the variational equation 
$$ H_N \Psi_N = E(N) \Psi_N$$
and thus 
$$ H_N ^2 \Psi_N = E(N) ^2 \Psi_N.$$
If we can show that $H_N^2$ controls the non-interacting $\left(\sum_{j=1} ^N h_j\right) ^2$ then we can deduce a bound on~\eqref{eq:2 moment}. Our main task shall thus be such a control, which is non-trivial because of the well-known fact that squaring (in fact~\cite[Chapter~5]{Bhatia} taking any power $t>1$) is not an operator/matrix monotone operation. Even in the case of repulsive interactions where $H_N \geq \sum_{j=1} ^N h_j$ it certainly does not follow that $H_N^2 \geq \left(\sum_{j=1} ^N h_j\right)^2$.

The control we shall need is provided by a set of simple and much-less-simple inequalities bearing on the interaction potential. The much-less-simple ones have their origin in~\cite{ErdYau-01} where the control of higher moments of the Hamiltonian seems to have been used for the first time, in a dynamical setting (see also~\cite{ErdSchYau-07,ErdSchYau-09,ErdSchYau-10}). We state the inequalities in 2D and 3D only for these are the cases we are interested in. The statement uses a smooth interaction potential but of course, by density, each inequality can be extended to potentials for which the right side makes sense. 

\begin{lemma}[\textbf{Operator inequalities for pair interactions}]\label{lem:opineq}\mbox{}\\
Let $W:\R^d \mapsto \R$ be a smooth decaying function. Let $W(\bx-\by)$ be the associated multiplication operator on $L^2 (\R^{2d})$. Let 
$$
\begin{cases}
p\geq 3/2,\, \alpha > 3/4 \mbox{ if } d=3 \\
p>1, \alpha > 1/2 \mbox{ if } d=2.
\end{cases}
$$
We have that, as operators,
\begin{align}\label{eq:Sob}
|W(\bx-\by)| &\leq C_p \norm{W}_{L^p (\R^d)} (-\Delta_{\bx})\\
|W(\bx-\by)| &\leq C \norm{W}_{L^1 (\R^d)} (-\Delta_{\bx})^{\alpha} (-\Delta_{\by})^{\alpha} \label{eq:ErdYau}\\
(-\Delta_{\bx}) W(\bx-\by) +  W(\bx-\by) (-\Delta_{\bx}) &\geq \nonumber\\-C_{p} &\left( \norm{W}_{L^p} + \norm{W}_{L^2} \right) (1-\Delta_{\bx}) (1-\Delta_{\by}).\label{eq:Nam}
\end{align}
Moreover, with $\bA \in L^2_{\rm loc} (\R^d)$ and $h = \left(-\im \nabla + \bA \right)^2$
\begin{multline}\label{eq:Nam mag}
h_{\bx} W(\bx-\by) +  W(\bx-\by) h_{\bx} \geq \\ -C_{p} \left( \norm{W}_{L^p} + \norm{W}_{L^2} \right) \left( (1-\Delta_{\bx}) (1-\Delta_{\by}) + h_{\bx} (1-\Delta_{\by})\right)
\end{multline}
\end{lemma}

\begin{proof}[Comments]\mbox{}
The simple part is~\eqref{eq:Sob} which is just the Sobolev inequality. The much-less-simple part starts in~\eqref{eq:ErdYau}, a version of which first appeared in~\cite[Lemma~5.3]{ErdYau-01}. Another is in~\cite{LieSei-06}. The full statement is in~\cite[Lemma~3.2]{NamRouSei-15} and~\cite[Lemma~6]{LewNamRou-15}. We refer to~\cite[Lemma~3.2]{NamRouSei-15} for the proof\footnote{The astute reader will notice a small gap in the proof of~\eqref{eq:Nam mag} in case $W$ is allowed to have a negative part. It is easily fixed by an additional use of~\eqref{eq:Sob}.}. 

The virtue of~\eqref{eq:ErdYau} is that, when used with potentials scaled in the manner~\eqref{eq:scaled potential}, the $L^1$ norm is fixed (whereas the higher $L^p$ norms blow up when $N\to \infty$). Thus the pair interaction energy is controlled by a power $<1$ of the second moment. 

Likewise, in~\eqref{eq:Nam} if $\beta$ is not too large one can use the inequality with potentials scaled as in~\eqref{eq:WNbeta} and see that the operator on the left-hand side is controlled by $(1-\Delta_{\bx})(1-\Delta_\by )$.
\end{proof}

The second moments estimates are as follows. We state first a version for purely repulsive potentials from~\cite{NamRouSei-15} (this is a simpler version of Lemma~3.1 therein). A version for potentials with no sign from~\cite[Lemma~5]{LewNamRou-15} will be given in Section~\ref{sec:mom att} below. From now on we shall assume 
\begin{equation}\label{eq:h dom Lap}
h = \left( -\im \nabla + \bA \right) ^2 + V \geq c \left( -\Delta + V \right) - C 
\end{equation}
for two positive constants $c,C >0$. This is a mild but non-trivial assumption we shall comment on below.

\begin{lemma}[\textbf{Second \index{Moments estimates}{moment estimate}, repulsive case}]\label{lem:mom rep}\mbox{}\\
Assume that the pair interaction potential is repulsive, $w\geq 0$. Let $\beta < 2/3$ if $d=3$ or $\beta < 1$ if $d=2$. Let $\Psi_N$ be a ground state for~\eqref{eq:intro Schro op bis} and $\Gamma_N^{(2)}$ the associated two-particles reduced density matrix. Denoting $h$ the one-particle Hamiltonian in~\eqref{eq:intro Schro op bis} we have 
\begin{equation}\label{eq:2mom rep}
{N \choose 2} ^{-1} \tr \left( h_1 \otimes h_2 \, \Gamma_N ^{(2)} \right) = {N \choose 2} ^{-1} \left\langle \Psi_N \Big| \sum_{1\leq i <j \leq N} h_i h_j \Big| \Psi_N \right\rangle \leq C  
\end{equation}
for a constant $C>0$ independent of $N$.
\end{lemma}

\begin{proof}[Comments]Assumption~\eqref{eq:h dom Lap} allows to relate the magnetic Laplacian to the usual one which appears in most inequalities of Lemma~\ref{lem:opineq}. A convenient way to ensure its validity is to assume that for all $\bx\in \R^d$ (we are mostly concerned with the behavior at infinity)
\begin{equation}\label{eq:V dom A} 
|\bA (\bx)|^2 \leq \eps V(\bx) + C_\eps
\end{equation}
for some $0<\eps<1$ and $C_\eps\geq 0$. It is desirable to be able to dispense with this, in particular because in the emblematic case of a uniform magnetic field one wants to take $|\bA (\bx)|^2 = B^2|\bx|^2 /4$ and then~\eqref{eq:V dom A} puts a unnecessary constraint on the growth of $V$ at infinity. We do not pursue the removal of~\eqref{eq:h dom Lap} in this review, see~\cite[Step~2~in~Section~4.2]{NamRouSei-15} for this.
\end{proof}

\begin{proof}[Proof of Lemma~\ref{lem:mom rep}]
Using the \index{Many-body Schr\"odinger equation}{variational equation} we have 
$$ H_N ^2 \Psi_N = E (N) ^2 \Psi_N$$
and thus, using the usual trial state argument to bound $E(N)$, 
\begin{equation}\label{eq:bound ener squared}
  \langle \Psi_N | H_N^2 | \Psi_N \rangle \leq C N^2. 
\end{equation}
On the other hand 
\begin{multline*}
H_N ^2 = \sum_{1\leq i,j \leq N} h_i h_j + \sum_{1 \leq i,j,k \leq N} \left(h_i w_N(\bx_j - \bx_k) +  w_N(\bx_j - \bx_k) h_i \right) \\ 
+ \sum_{1\leq i,j,k,\ell\leq N} w_N(\bx_i - \bx_j) w_N(\bx_k - \bx_\ell). 
\end{multline*}
Here we assume $w_N \geq 0$, hence all the terms on the second line are positive operators (multiplication operators by positive functions). Also, for all distinct indices $i,j,k$ 
$$ h_i w_N(\bx_j - \bx_k) +  w_N(\bx_j - \bx_k) h_i \geq 0$$
because $h_i \geq 0$ and  $w_N(\bx_j - \bx_k) \geq 0$ are commuting operators when they act on different variables. Hence, using that $w_N$ is even, 
$$ H_N ^2 \geq \sum_{1\leq i,j \leq N} h_i h_j + 2 \sum_{1 \leq i,j \leq N} \left(h_i w_N(\bx_i - \bx_j) +  w_N(\bx_i - \bx_j) h_i \right).$$
Inserting~\eqref{eq:Nam mag} and using~\eqref{eq:h dom Lap} we deduce 
$$ H_N ^2 \geq \sum_{1\leq i,j \leq N} h_i h_j \left( 1- C N ^{d\beta/2 - 1} \right)$$
when $w_N$ is chosen as in~\eqref{eq:scaled potential}. When $\beta<2/d$ the coefficient in the parenthesis is positive for large $N$, and combining with~\eqref{eq:bound ener squared} concludes the proof.
\end{proof}

\section{\index{Dilute limit}{Dilute limit} with \index{Moments estimates}{moments estimates}, repulsive case}\label{sec:mom rep}

Inserting the moment estimates in the basic scheme of Section~\ref{sec:deF loc} leads to notable improvements of the main result. The method to deal with purely repulsive potentials is from~\cite{NamRouSei-15}, although not explicitly formulated that way (see Remark~3.3 and Section~4.2 therein). 

\begin{theorem}[\textbf{Dilute limit again, repulsive case}]\label{thm:dilute}\mbox{}\\
Make Assumptions~\ref{asum:pots} as in Theorem~\ref{thm:main} plus~\eqref{eq:trapping s} with some $s>0$ and 
\begin{equation}\label{eq:growth A}
\left| \bA (\bx)\right| \leq c e^{c|\bx|} 
\end{equation}
We work in dimensions $d=2,3$ and assume $0 \leq w \in L^1 (\R^d) \cap L^2 (\R^d)$. Let 
\begin{equation}\label{eq:beta dilute bis}
\begin{cases}
0 < \beta < 2/3 \mbox{ if } d = 3\\ 
0 < \beta < 1 \mbox{ if }  d = 2.
\end{cases}
\end{equation}
We have, in the limit $N \to + \infty$:

\smallskip 

\noindent\textbf{Convergence of the energy:} 
$$
\frac{E(N)}{N} \to \Enls. 
$$

\smallskip 

\noindent\textbf{Convergence of reduced density matrices:} let $\Gamma_N ^{(k)},k\geq 0$ be the reduced density matrices of a many-body ground state $\Psi_N$. There exists a Borel probability measure $\mu$ on the set $\Mnls$ of NLS minimizers such that, along a subsequence, 
$$
{N \choose k} ^{-1} \Gamma_N ^{(k)} \to \int_{\Mnls} |u ^{\otimes k} \rangle \langle u ^{\otimes k} | \mathrm{d}\mu(u) 
$$
strongly in trace-class norm.
Let $\rm{MF}$ stand for $\rm{NLS}$ in Theorem~\ref{thm:main}. 
\end{theorem}

\begin{proof}[Comments]\mbox{}\\
\noindent\textbf{1.} We have already obtained the 3D statement in Section~\ref{sec:better loc}, except for the convergence of higher density matrices. The latter is particularly transparent using the \index{Quantum de Finetti theorem}{quantum de Finetti theorem}.

\medskip

\noindent\textbf{2.} The attractive case is more involved because of possible stability issues. We discuss it in the next section.
\end{proof}

We give some elements of the proof. A first way to take advantage of the moment estimates from Lemma~\ref{lem:mom rep} is to control the localization to low one-body energy modes much more efficiently than in~Lemma~\ref{lem:localize-energy}. The following, which is essentially a restatement of~\cite[Equation (46)]{LewNamRou-15}, goes in a different direction than what we presented in Section~\ref{sec:better loc}. In particular, we no longer need the interaction potential to be repulsive (which will be useful in the next section). 

\begin{lemma}[\textbf{\index{Localization}{Localizing} the Hamiltonian, again}]\label{lem:loc}\mbox{}\\
Let $\Gamma_N$ be a $N$-particle state with normalized reduced density matrices (compare with~\eqref{eq:def red mat})
\begin{equation}\label{eq:normalized DM}
\gamma_N^{(2)} = \tr_{3\to N} \Gamma_N ,\quad \gamma_N^{(1)} = \tr_{2\to N} \Gamma_N
\end{equation}
With the above notation, for any 
\begin{equation}\label{eq:loc delta}
\begin{cases}
\delta > 1/2 \mbox{ in 2D} \\
\delta > 3/4 \mbox{ in 3D}
   \end{cases}
\end{equation}
there exists a $C_\delta >0$ such that  
\begin{multline}\label{eq:loc}
\tr \left( \left( H_2 - P^{\otimes 2} H_2 P^{\otimes 2} \right) \gamma_N ^{(2)} \right) \geq \\- C_{\delta} \Lambda ^{(\delta -1)/2} \left( \tr \left(h \gamma_N ^{(1)}\right) \right) ^{(1-\delta)/2}  \left( \tr \left(h\otimes h \gamma_N ^{(2)}\right) \right) ^{\delta}
\end{multline}
where $H_2$ is the two-body Hamiltonian~\eqref{eq:two body hamil}, $h$ the one-body Hamiltonian~\eqref{eq:one body h} and $P$ the associated projector~\eqref{eq:projectors} below the energy cut-off $\Lambda$.
\end{lemma}

\begin{proof}
Since $h\geq 0$, $PQ = 0$ and $P$ commutes with $h$ we certainly have 
$$ h_1 + h_2 \geq  P^{\otimes 2} \left( h_1 + h_2 \right) P^{\otimes 2}.$$
The interaction term is then our only concern. Denote 
$$\Pi = \1 ^{\otimes 2} - P^{\otimes 2}$$
and write ($W := w_{N,\beta}$ is hereafter identified with the multiplication by $w_{N,\beta}(\bx_1-\bx_2)$ on the two-body space)
$$ W - P^{\otimes 2} W P^{\otimes 2} = \frac{1}{2} \left( \Pi W + \Pi W P^{\otimes 2} + W\Pi + P^{\otimes 2}W \Pi \right)$$
Separating $W = W_+ - W_-$ and using the Cauchy-Schwarz inequality for operators we have 
\begin{equation}\label{eq:split W}
 W - P^{\otimes 2} W P^{\otimes 2} \geq - \frac{1}{4} \left(\eps |W| + \eps^{-1} \Pi |W| \Pi + \eps P^{\otimes 2} |W| P^{\otimes 2} + \eps^{-1} \Pi |W| \Pi  \right). 
\end{equation}
Then, for any $\delta$ as in~\eqref{eq:loc delta} 
\begin{equation}\label{eq:loc P}
\tr \left( \left( |W| + P^{\otimes 2} |W| P^{\otimes 2} \right) \gamma_N^{(2)} \right) \leq C_{\delta} \left(\tr\left( h\otimes h \gamma_N^{(2)} \right)\right)^{\delta}. 
\end{equation}
Indeed, from Lemma~\ref{lem:opineq} and Young's inequality, for any $\delta$ as in~\eqref{eq:loc delta} and $\eta >0$, 
$$ |W| \leq C_\delta h ^{\delta} \otimes h ^{\delta} \leq C_{\delta} \left( \eta ^{-1} h \otimes h + \eta ^{\delta/(1-\delta)}\right).$$
Taking the trace against $\gamma_N ^{(2)}$ and optimizing over $\eta$ yields~\eqref{eq:loc P} (recall $\gamma_N ^{(2)}$ has trace one, and observe that the $P^{\otimes 2} |W| P^{\otimes 2}$ term can be treated on the same footing). 

On the other hand, using Lemma~\ref{lem:opineq} again and the definition of $\Pi$, for any $\delta$ as in~\eqref{eq:loc delta}   
\begin{equation}\label{eq:loc Pi}
\Pi |W| \Pi \leq C_{\delta} \Pi h ^{\delta} \otimes h^{\delta} \Pi \leq C{\delta} \Lambda ^{\delta-1} \left(h^{\delta} \otimes h + h\otimes h^{\delta} \right)
\end{equation}
Using Young's inequality again we have, for any $\eta >0$,
$$ h ^{\delta} \leq C_\delta\left( \eta^{-1} h + \eta ^{\delta/(1-\delta)}\right).$$
Thus, taking the trace of~\eqref{eq:loc Pi} against $\gamma_N ^{(2)}$, optimizing over $\eta$ gives 
\begin{equation}\label{eq:loc Pi bis}
\tr \left( \Pi |W| \Pi \gamma_N^{(2)} \right) \leq C_{\delta} \tr\left( h\otimes h \gamma_N^{(2)} \right)^{\delta} \tr\left( h \gamma_N^{(1)} \right)^{1-\delta} 
\end{equation}
upon recalling that $\gamma_N ^{(1)}$ is the partial trace of $\gamma_N^{(2)}$. 

There remains to take the trace of~\eqref{eq:split W} against $\gamma_N^{(2)}$, insert~\eqref{eq:loc P} and~\eqref{eq:loc Pi bis}, optimize over $\eps$ and the result is proved.
\end{proof}

Now we can give a 

\medskip

\noindent\textbf{Sketch of proof for Theorem~\ref{thm:dilute}, repulsive case.} As in Section~\ref{sec:better loc} the goal is to take the limits first $N\to \infty$ and then only $\Lambda \to \infty$. This requires a rather fine control of the errors involved.  

\medskip

We assume $w\geq 0$ and then may use Lemma~\ref{lem:mom rep}. Inserting these estimates in~\eqref{eq:loc} we obtain that 
\begin{equation}\label{eq:loc use}
{N\choose 2} ^{-1} \tr \left( \left( H_2 - P^{\otimes 2} H_2 P^{\otimes 2} \right) \Gamma_N ^{(2)} \right) \geq - C_{\delta} \Lambda ^{(\delta -1)/2} 
\end{equation}
where $\Gamma_N^{(2)}$ is the two-particle reduced density matrix of a ground state of $H_N$ and $\delta$ may be chosen strictly smaller than $1$. The virtue of the above is that the localization error is now small for $\Lambda \to \infty$, independently of $N$. 

We use Theorem~\ref{thm:deF semi} to control the $P^{\otimes 2}$-localized part of the interaction. Observe that the operator norm of the projected two-body Hamiltonian satisfies 
\begin{equation}\label{eq:norm proj}
\norm{P^{\otimes 2} H_2 P ^{\otimes 2}} \leq C_\Lambda 
\end{equation}
where the constant on the right-hand side depends only on $\Lambda$, not on $N$. For the one-body term this is of course obvious, for the interaction term this follows from Inequality~\eqref{eq:ErdYau}. 

Using Theorem~\ref{thm:deF semi} and the localization method of Section~\ref{sec:localization} we construct a de Finetti measure $\mu_N$ for the projected states $\Gamma_N ^P$ associated to $\Gamma_N =|\Psi_N \rangle\langle \Psi_N |$, the orthogonal projector onto a ground state. We skip some details here but observe that, as per~\eqref{eq:deF semi} and~\eqref{eq:norm proj} we will essentially get  
\begin{equation}\label{eq:deF use}
{N\choose 2} ^{-1} \tr \left( P^{\otimes 2} H_2 P^{\otimes 2} \, \Gamma_N ^{(2)} \right) \geq \int \cEnls [u] \mathrm{d}\mu_N (u) - C_\Lambda o_N (1) \geq  \Enls  - C_\Lambda o_N (1)
\end{equation}
where $o_N (1) \to 0$ when $N\to \infty$ and is independent of $\Lambda$. 

We can now pass to the limit first as $N\to \infty$ to make the error in~\eqref{eq:deF use} vanish and then $\Lambda\to\infty$ to make the error in~\eqref{eq:loc use} vanish, and deduce the desired energy lower bound 
$$ \liminf_{N\to \infty} {N\choose 2} ^{-1} \tr \left( H_2  \, \Gamma_N ^{(2)} \right) \geq \Enls.$$
The corresponding energy upper bound is obtained as usual with a factorized (Hartree) \index{Trial state, Hartree}{trial state}. Note then that in~\eqref{eq:deF use} we have sandwiched 
$$ \int \cEnls [u] \mathrm{d}\mu_N (u)$$
in between the energy upper and lower bounds. With a bit extra effort we can prove that the measure $\mu_N$ (obtained as sketched in Section~\ref{sec:deF loc}) converges for large $N$ to $\mu$, the measure associated to the sequence of states $\Gamma_N$ via Theorem~\ref{thm:DeFinetti fort} (given the bounds we have already collected so far it is easy to see that one can apply this theorem). Moreover we will deduce from the energy bounds that, with $\cEnls$ and $\Enls$ the energy functional~\eqref{eq:rev nls f} and its infimum respectively,
$$ \int \cEnls [u] \mathrm{d}\mu (u) = \Enls,$$
which implies that the limit measure is concentrated on NLS minimizers, and hence the desired convergence of reduced density matrices. Observe that the only limitation on $\beta$ in this proof is that inherited from Lemma~\ref{lem:mom rep}.

\hfill \qed

\section{\index{Dilute limit}{Dilute limit} with \index{Moments estimates}{moments estimates}, attractive case}\label{sec:mom att}

We conclude our discussion of dilute limits by turning to the attractive case. In 2D this allows to obtain focusing mass-critical energy functionals in the limit, the result is from~\cite{LewNamRou-15,NamRou-19}. In 3D this allows to assume only classical stability instead of a purely repulsive potential. The result is implicit in the more general study~\cite{Triay-17} of Bose gases with dipole-dipole interactions.

\begin{theorem}[\textbf{Dilute limit again, attractive case}]\label{thm:dilute att}\mbox{}\\
Make Assumptions~\ref{asum:pots} as in Theorem~\ref{thm:main} plus~\eqref{eq:trapping s} with some $s>0$ and~\eqref{eq:growth A}. We work in dimensions $d=2,3$ and assume Hartree stability in the former case, classical stability in the latter (cf Definition~\ref{def:stability}) for some $w \in L^1 (\R^d) \cap L^2 (\R^d)$. Let 
\begin{equation}\label{eq:beta dilute ter}
\begin{cases}
0 < \beta < \frac{1}{3} + \frac{s}{42 s + 45} \mbox{ if } d = 3\\ 
0 < \beta < 1 \mbox{ if }  d = 2.
\end{cases}
\end{equation}
We have, in the limit $N \to + \infty$:

\smallskip 

\noindent\textbf{Convergence of the energy:} 
$$
\frac{E(N)}{N} \to \Enls. 
$$

\smallskip 

\noindent\textbf{Convergence of reduced density matrices:} let $\Gamma_N ^{(k)},k\geq 0$ be the reduced density matrices of a many-body ground state $\Psi_N$. There exists a Borel probability measure $\mu$ on the set $\Mnls$ of NLS minimizers such that, along a subsequence, 
$$
{N \choose k} ^{-1} \Gamma_N ^{(k)} \to \int_{\Mnls} |u ^{\otimes k} \rangle \langle u ^{\otimes k} | \mathrm{d}\mu(u) 
$$
strongly in trace-class norm.
\end{theorem}

\begin{proof}[Comments]\mbox{}\\
In 2D a natural threshold is reached in~\cite{NamRou-19}, in that the result holds as long as second moment estimates are available (Lemma~\ref{lem:mom att} below). In 3D the values of $\beta$ achieved so far are not much larger than the diluteness threshold $1/3$. This shows that allowing even a small attractive part in the interaction potential makes proofs much harder. See~\cite{Lee-09,Yin-10} for results in this direction. 
\end{proof}

We do not reproduce the proof, whose main ingredient is the following adaptation of Lemma~\ref{lem:mom rep}:

\begin{lemma}[\textbf{Second moment estimate, attractive case}]\label{lem:mom att}\mbox{}\\
Let $d=2$ and $\beta < 1$ or $d=3$ and $\beta < 2/3$. Let $\Psi_N$ be a ground state for~\eqref{eq:intro Schro op bis} and $\Gamma_N^{(2)}$ the associated two-particles reduced density matrix. Let $e_{N,\eps}$ be the ground state energy per particle ($N^{-1}$ times the lowest eigenvalue) of 
$$ H_{N,\eps} = H_N - \eps \sum_{j=1} ^N h_j.$$
We have, for all $0<\eps < 1$, 
\begin{equation}\label{eq:2mom att}
{N \choose 2} ^{-1} \tr \left( h_1 \otimes h_2 \Gamma_N ^{(2)} \right) \leq C_\eps \left(\frac{1+|e_{N,\eps}|}{\eps}\right) ^2  
\end{equation}
for a constant $C_\eps>0$ independent of $N$. 
\end{lemma}

\begin{proof}[Comments]\mbox{}\\
 \noindent\textbf{1.} This is stated as a ``non-necessarily repulsive case''. The potential needs not be (partially) attractive for the result to hold, but if it is not, one should rely on Lemma~\ref{lem:mom rep} instead.  

\medskip

\noindent\textbf{2.} Lemma~\ref{lem:mom rep} was a true a priori information on ground states and can be employed directly to estimate error terms in the mean-field limit. Lemma~\ref{lem:mom att} by contrast is a conditional statement: it becomes most useful if we happen to know that $|e_{N,\eps}|$ is of order $1$. In the attractive case this is not quite obvious, for it is precisely saying that the system is stable (of the second kind,~cf Definition~\ref{def:stability}) under our assumptions. This is part of what we are aiming to prove, namely that the stability of the mean-field problem implies that of the many-body one.  

\medskip 

\noindent\textbf{3.} The proof of Lemma~\ref{lem:mom att} follows from similar considerations as that of Lemma~\ref{lem:mom rep}, see~\cite{LewNamRou-15} and~\cite{Triay-17}.

\medskip 

\noindent\textbf{4.} To prove Theorem~\ref{thm:dilute att} one combines moments estimates with Lemma~\ref{lem:loc} again. The conclusion is more subtle: the bounds in Lemma~\ref{lem:mom att} not being truly a priori we have to perform a bootstrap argument to improve energy estimates progressively, and we cannot take limits $N\to \infty$ and $\Lambda \to \infty$ one after the other.
\end{proof}

\chapter{The \index{Gross-Pitaevskii limit}{Gross-Pitaevskii limit}}\label{cha:GP}

So far we have justified, in situations of increasing generality and mathematical difficulty, the absence of inter-particle correlations in the leading order contributions to the ground state of the Bose gas. The increase in mathematical difficulty was motivated by and correlated with an increase in physical relevance for dilute \index{Cold atoms}{atomic gases}. As explained in Section~\ref{sec:sca lim}, in the scaling limit corresponding to a dilute gas the interaction becomes rather singular. If the gas is made very dilute, it in fact becomes singular enough that inter-particle correlations can no longer be neglected, even at leading order. This is the problem we now tackle, for repulsive interaction potentials and in 3D. A review of this problem as of 2005 (and related topics) is in~\cite{LieSeiSolYng-05} (see also~\cite{Seiringer-06c}). The presentation below overlaps this text to some extent. 

The Gross-Pitaevskii limit is that of largest physical relevance for dilute systems. This is because the true \index{Scattering length}{scattering length} of the interaction potential appears as effective coupling constant, not just its first \index{Born approximation}{Born approximation}. In fact, physicists argue that in a dilute system, the interactions are only via $s$-wave scattering, and replace the interaction potential (recall the notation of Section~\ref{sec:sca lim}) by a Dirac-like $4\pi a_w \delta$ in the \index{Many-body Schr\"odinger Hamiltonian}{many-body Schr\"odinger Hamiltonian} \emph{before} making any other sort of approximation, mean-field or otherwise. This manipulation is hardly legitimate mathematically, for delta interactions are seldom Kato-bounded perturbations of the kinetic energy operator (except in 1D~\cite{LieLin-63,SeiYngZag-12,AmmBre-12,Rosenzweig-19} or when projected in special spaces~\cite{LewSei-09,LieSeiYng-09,SeiYng-20}). What the GP limit does is to provide a parameter regime where the final result of the approximation, the GP energy functional, is recovered. 

It will be important to realize that the effective GP interaction with the scattering length in front is actually made of interaction energy \emph{and} part of the (high frequency) kinetic energy.  Here is the plan:
\begin{itemize}
 \item We start by discussing the two-body scattering process that dictates the short-range pair correlations in the gas: Section~\ref{sec:sca length}.
 \item It is already rather non-trivial to come up with a good trial state, one that looks physically relevant, does the job, and can be handled in a mathematically rigorous fashion. See Sections~\ref{sec:Jas-Dyson} and~\ref{sec:Bogoliubov} respectively for two versions of the trial state, both yielding the sought-after energy upper bound.
 \item Perhaps the hardest part of the analysis reviewed in this chapter is to extract the effect of pair correlations in order to derive an energy lower bound. The tools to do that go under the name of \index{Dyson lemma}{Dyson lemmas}: Section~\ref{sec:Dyson}.
 \item Sections~\ref{sec:thermo} and~\ref{sec:LDA} review a method which is specific to the Gross-Pitaevskii limit\footnote{It could probably be of use also for the dilute limit of the previous chapter, but less probably for the mean-field limit.}. It works in two steps: first one proves a formula for the ground state energy of the \index{Homogeneous Bose gas}{homogeneous Bose gas} in a \index{Thermodynamic limit}{thermodynamic}/dilute limit. Then one uses this formula to deal with the inhomogeneous gas by a \index{Local density approximation (LDA)}{local density approximation} method. 
 \item In Section~\ref{sec:coh reloaded} we show how to combine the Dyson lemmas and the tools of the previous chapter to provide different proofs and extend the results to the case of magnetic fields.
 \item Finally in Section~\ref{sec:Bogoliubov low} we review other means than the Dyson lemma to extract pair correlations in energy lower bounds. The methods we review allow to derive GP ground states only for unscaled potentials $w$ with small enough scattering length (see below for more comments in this direction). 
\end{itemize}

For extensions to 2D (where, again, the scaling of the interaction is different) of the material presented below we refer to~\cite{CarCenSch-20,LieYng-01,LieSeiYng-01,SchYng-07}.

\section{Pair correlations and the scattering length}\label{sec:sca length}

For simplicity we henceforth work under the following assumptions: 

\begin{assumption}[\textbf{The interaction potential in the GP limit}]\label{asum:GP pot}\mbox{}\\
The unscaled interaction potential $w:\R^3 \mapsto \R$ is smooth, non-negative, radial, and with compact support in the ball of radius $ 0 < R_w < \infty$. 
\end{assumption}

The above can be relaxed to a large extent. Probably the most annoying assumption we make is that the potential be purely repulsive (see Section~\ref{sec:mom att}). It is pretty hard to remove it, although some results are known~\cite{Lee-09,Yin-10b}. What is definitely not needed is the smoothness. The potential could even have a hard-core, formally $w = +\infty$ inside a ball of radius $R_w^{\rm hard}\leq R_w$, which we would materialize by changing the configuration space from $\R^{3N}$ to $\R^{3N}\setminus \left\{ |\bx_i-\bx_j| \leq R_w^{\rm hard} \mbox{ for some } i\neq j \right\}.$ 

As hinted at in Sections~\ref{sec:sca lim}-\ref{sec:NLS} we should think that the state of our Bose gas contains pair correlations to reduce repulsive interactions. The GP limit is\footnote{By definition !} the regime where this has a leading order effect in the large $N$ limit, but we do not need that scaling to guess what the correlations should be. 

We look for a wave-function $f:\R^d \mapsto \C$ describing the motion of the relative coordinate of a pair of particles. This motion will happen on the length scale of the interaction potential. In a dilute gas, this is much shorter than the macroscopic extent of the full system. Thus $f$ should converge at infinity to some value prescribed by the macroscopic variations of the system. We shall take this value to be $1$, without loss of generality (see the next section where we  connect the short-range pair correlations to the macroscopic behavior of the system). A slightly different way to formulate what we are looking for is that we consider a infinite homogeneous system, with reference density $1$ and ask how it gets modified if we insert a repulsive potential $w$. 

This discussion leads us to the following variational problem (the $1/2$ in front of $w$ is because this is the energy in the relative coordinate of a pair of particles)
\begin{equation}\label{eq:GP def scat}
4\pi a_w = \inf\left\{ \int_{\R^3} |\nabla f |^2 + \frac{1}{2} w |f| ^2, \: f(\bx) \underset{|\bx| \to \infty}{\to} 1  \right\} 
\end{equation}
with the associated variational \index{Scattering solution}{equation} 
\begin{equation}\label{eq:GP scat eq}
-\Delta f + \frac{1}{2} w f = 0, \quad f(\bx) \underset{|\bx| \to \infty}{\to} 1.  
\end{equation}
We record the main properties of this problem in the following

\begin{theorem}[\textbf{The scattering problem}]\label{thm:scat leng}\mbox{}\\
The variational problem~\eqref{eq:GP def scat} and the PDE~\eqref{eq:GP scat eq} have a unique solution $f$. It is non-negative, radial and satisfies 
\begin{equation}\label{eq:GP sca leng}
f(\bx) = 1 - \frac{a_w}{|\bx|} \mbox{ for } |\bx| > R_w 
\end{equation}
where $R_w$ is the radius of the support of $w$. Here $a_w$ is the \index{Scattering length}{\emph{scattering length}} of the potential $w$ and it is connected to the minimal scattering energy as indicated in~\eqref{eq:GP def scat}. 

Moreover, let $f_R$ be the unique minimizer of 
\begin{equation}\label{eq:GP scat R}
\cE^{\rm scat}_R [f] := \int_{B(0,R)} |\nabla f |^2 + \frac{1}{2} w |f| ^2
\end{equation}
with Dirichlet boundary condition $f_R = 1$ on the circle of radius $R$. If $R> R_w$ we have 
\begin{equation}\label{eq:GP scat ener R}
E^{\rm scat}_R = \cE^{\rm scat}_R [f_R] = 4\pi a_w \left(1 - \frac{a_w}{R} \right)^{-1}
\end{equation}
and, with $f$ the solution to~\eqref{eq:GP scat eq}, 
\begin{equation}\label{eq:GP scat sol}
f_R (\bx) = \frac{f(\bx)}{f(R)}  
\end{equation}
for all $|\bx| \leq R$.
\end{theorem}

\begin{proof}[Comments]
See~\cite[Appendix~C]{LieSeiSolYng-05} for proofs. There the variational problem is set first in a ball of radius $R>R_w$ as in~\eqref{eq:GP scat R} but this is really a matter of convenience, see in particular the first remark following~\cite[Theorem~C.1]{LieSeiSolYng-05}. In 2D however (a case we are not concerned with) the restriction to a finite ball is necessary, for the limit $R\to \infty$ is trickier. 

The origin of~\eqref{eq:GP sca leng} is clear: for $|\bx| > R_w$ we have 
$$ -\Delta f = 0$$
and the right-hand side is the general form of a solution tending to $1$ at infinity. The scattering length is related to the value of $f (R_w)$ by 
$$
a_w = R_w (f (R_w) - 1).
$$
To see that the two characterizations~\eqref{eq:GP def scat} and~\eqref{eq:GP sca leng} of the scattering length are indeed equivalent, multiply~\eqref{eq:GP scat eq} by $f$ and integrate by parts on a ball of radius $R > R_w$. This yields 
\begin{align}\label{eq:calc scat} 
\int_{B(0,R)} |\nabla f| ^2 + \frac{1}{2} w |f| ^2  &= \oint_{\partial B(0,R)} f \partial_r f \nonumber \\
&= \oint_{\partial B(0,R)} \left(1 - \frac{a_w}{R}\right) \frac{a_w}{R^2} = 4\pi a_w - 4\pi \frac{a_w^2}{R} 
\end{align}
and it suffices to let $R\to\infty$ to see that the infimum in~\eqref{eq:GP def scat} is indeed $4\pi a_w$ with $a_w$ as in~\eqref{eq:GP sca leng}. 

A further useful characterization is to define 
$$ g:= w f$$
and observe that, integrating~\eqref{eq:GP scat eq} by parts on $B(0,R)$ 
$$ \frac{1}{2} \int_{B(0,R)} g = \oint_{\partial B(0,R)} \partial_r f = \oint_{\partial B(0,R)} f \partial_r f + o_R (1)$$
with $o_R (1) \to 0$ when $M\to \infty$. By the previous computation and taking the limit $R\to \infty$ we deduce 
\begin{equation}\label{eq:GP def scat bis}
8 \pi a_w = \int_{\R^3} g. 
\end{equation}
The quantity $a_w$ can thus be interpreted\footnote{In our (absence of) units.} as the minimal energy of a two-particles scattering process~\eqref{eq:GP def scat}, as the typical length scale thereof~\eqref{eq:GP sca leng} or as the integral~\eqref{eq:GP def scat bis} of an effective interaction potential $g$ incorporating the fine details of the process. 

We finally discuss~\eqref{eq:GP scat sol}. Consider the function defined on $\R^3$ by 
$$ 
\widetilde{f} (\bx) = \begin{cases} 
                       f (R) f_R (\bx) \mbox{ if } |\bx| \leq R\\
                       f(\bx) \mbox{ if } |\bx| > R
                      \end{cases}
$$
with $f$ the minimizer of the scattering energy in the full space and $f_R$ the minimizer in the ball of radius $R$. Certainly $\widetilde{f}$ is a valid trial state for the energy in the full space, and by definition it minimizes the energy on the ball $B(0,R)$ with boundary condition $\widetilde{f} (R) = f (R)$, while having the same energy as $f$ in the exterior of $B(0,R)$. It thus must be a minimizer over the full space, and by uniqueness one obtains $\widetilde{f} = f $, which is~\eqref{eq:GP scat sol}. Then~\eqref{eq:GP scat ener R} is just another version of the calculation we did in~\eqref{eq:calc scat} above (replace $f(\bx)$ by $f_R (\bx)= f(\bx) f(R)^{-1}$).
\end{proof}

In the notation of Section~\ref{sec:sca lim}, the \index{Gross-Pitaevskii limit}{Gross-Pitaevskii limit} is $d=3,\beta = 1$. Hence our scaled interaction potential is 
\begin{equation}\label{eq:GP sca pot}
w_N (\bx) = N^{2} w (N \bx). 
\end{equation}
By scaling it is easy to see from~\eqref{eq:GP def scat} that 
\begin{equation}\label{eq:GP scale scat}
a_{w_N} = N^{-1} a_{w}.  
\end{equation}
The rationale for the scaling in the GP limit can thus be rephrased by saying that the scattering length of $w_N$ measures the strength of interactions. This number we want to be of order $N^{-1}$ for a nice energy balance as in Section~\ref{sec:sca lim}. As per~\eqref{eq:GP scale scat} this is clearly achieved by taking a fixed $w$ and scaling as in~\eqref{eq:GP sca pot}. 

To see that the cases $\beta < 1$ really correspond to a degenerate version of the GP limit we state the 

\begin{theorem}[\textbf{\index{Born approximation}{Born series} for the scattering length}]\label{thm:Born}\mbox{}\\
Let 
$$ w_N (\bx) = N^{3\beta - 1} w (N^{\beta} \bx) $$
with $0<\beta < 1$. Let $\mathcal{L}_N$ be the operator defined by
$$ \mathcal{L}_N (g) (\bx) = \frac{1}{8\pi}w_N (\bx) \int_{\R^3} \frac{1}{|\bx - \by|} g (\by) \mathrm{d}\by.$$
Then the scattering length of $w_N$ satisfies, for any $K \in \N$
\begin{equation}\label{eq:Born}
a_{w_N} = \frac{1}{8\pi} \left(\int_{\R^3} w_N - \sum_{k=2} ^{K-1} \int_{\R^3} \mathcal{L}_N ^{k-1} (w_N) -  \int_{\R^3} \mathcal{L}_N ^{K-1} (w_N) \right)(1+o_N(1)).
\end{equation}
\end{theorem}

\begin{proof}[Comments]
This is considered folklore in the mathematical physics community, based on the derivation of physics textbooks. See~\cite[Remarks after Theorem~1.1]{BriSol-19} for more details. The Born series~\eqref{eq:Born} can also be obtained from an explicit representation of the scattering length in terms of a resolvent, see~\cite[Definition~2 and Appendix~A]{HaiSei-08}. 

The rationale is that, when $\beta <1$, the solution of the scattering equation associated with $w_N$ converges to $1$ uniformly. This gives the first term straightforwardly, which is just $N^{-1}$ times the integral of $w$, i.e. the effective coupling constant we get in the NLS limit discussed previously. More generally the $k$-th term in the series above is of order $N^{-1} N^{(\beta-1) (k-1)}$ with our scaled interaction potential.
\end{proof}

A posteriori, the above result explains why we never saw the full scattering length appear in the limits $\beta < 1$ considered before. The short-range correlations via the scattering process may well be present, and lead to the scattering length being the effective coupling constant. But its leading order when $N\to \infty$ is just the integral of the potential. Instead of checking whether our estimates are refined enough to capture the corrections and see more terms of the Born series, we prefer to work in a regime (the GP limit) where all terms in the series weigh the same. Approaches where more terms of the Born series are captured for $\beta < 1$ are e.g. in~\cite{GiuSei-09,BriSol-19}. This requires a discussion of another subleading correction, the \index{Bogoliubov theory}{Bogoliubov energy}.

\section{Jastrow-Dyson trial states}\label{sec:Jas-Dyson} 

In the previous section we defined the high-energy/short-range process that dictates the pair correlations in our systems. Now we turn to the construction of trial states that effectively incorporate those in order to modify the coupling constant from the naive guess (integral of the unscaled potential) to the actual, smaller, value set by the scattering length. 

This is the first subtlety we encounter in dealing with the GP limit: one cannot obtain the correct energy by just using factorized states of the form $u^{\otimes N}$. Instead we introduce the

\begin{definition}[\textbf{Jastrow-Dyson trial states}]\mbox{}\\
Let $R>0$ be a given radius and $u^{\rm GP}$ a minimizer of the Gross-Pitaevskii functional~\eqref{eq:GP func}, under unit mass constraint. Let $f_{R,N}$ be the minimizer of the scattering energy~\eqref{eq:GP scat R} in the ball of radius $R$, associated with the potential $w_N = N ^2 w (N \cdot)$. Extend it to $1$ outside of $B(0,R)$ and define
\begin{itemize}
 \item The \index{Trial state, Jastrow}{Jastrow trial state} 
\begin{equation}\label{eq:Jastrow trial}
\Psi_N ^{\rm Jas} (\bx_1,\ldots,\bx_N) = c_{\rm Jas} \prod_{j=1} ^N u^{\rm GP} (\bx_j) \prod_{1\leq i < j \leq N} f_{R,N} (\bx_i - \bx_j)
\end{equation}
with $c_{\rm Jas}$ a $L^2$ normalization constant.  
\item The \index{Trial state, Dyson}{Dyson} trial state  
\begin{equation}\label{eq:Dyson trial}
\Psi_N ^{\rm Dys} (\bx_1,\ldots,\bx_N) = c_{\rm Dys} \prod_{j=1} ^N u^{\rm GP} (\bx_j) \prod_{j=1} ^N F_{j} (\bx_1,\ldots,\bx_j)
\end{equation}
with $c_{\rm Dys}$ a $L^2$ normalization constant and 
$$ F_{j} (\bx_1,\ldots,\bx_j):= f_{R,N} \left( \min_{ 1 \leq k \leq j-1 } |\bx_j-\bx_k| \right).$$
By convention $F_{1} \equiv 1.$
\end{itemize}
\end{definition}

\begin{proof}[Comments]
Trial states of the form~\eqref{eq:Jastrow trial} originate in~\cite{Dingle-49,Jastrow-55} and are ubiquitous for they seem the most natural way to incorporate pair correlations, the next simplest thing after an uncorrelated trial state. Proving rigorous bounds with them is sometimes tricky however, which is why Dyson~\cite{Dyson-57} came up with the more subtle~\eqref{eq:Dyson trial}. The physical interpretation of this function is that particles are inserted one at a time in the trap, adapting their wave-function to the particles that are already there. More specifically, they correlate with their nearest neighbor amongst the previously inserted particles. This does not capture all pair-correlations, but turns out to be sufficient for leading order asymptotics.

In practice it is sufficient to take $R \gg N^{-1}$, the range of the potential $w_N$, by a small amount. It could seem natural to impose $R\ll N^{-1/3}$, the typical inter-particle distance, but that is not necessary, for the scattering solution quickly converges to $1$ on length scales $\gg N^{-1}$ anyway. 
\end{proof}

The above trial states seem like good guesses for the true ground state of the system. This is confirmed by evaluating their energies:

\begin{theorem}[\textbf{Energy of Jastrow-Dyson trial states}]\label{thm:GP up bound}\mbox{}\\
Take $d=3$ and $\beta = 1$ in the many-body Hamiltonian~\eqref{eq:intro Schro op bis}-\eqref{eq:scaled potential} and $ N^{-1} \ll R \ll 1$ in the previous definition. Then, in the limit $N\to \infty$, 
\begin{align*}
 \left\langle \Psi_N ^{\rm Jas} | H_N | \Psi_N^{\rm Jas} \right\rangle \leq N \EGP (1+o(1))\\
\left\langle \Psi_N ^{\rm Dys} | H_N | \Psi_N^{\rm Dys} \right\rangle \leq N \EGP (1+o(1)) 
\end{align*}
with $\EGP$ the GP ground state energy, infimum of~\eqref{eq:GP func}. Consequently, the many-body ground state energy for $d=3, \beta = 1$ satisfies
\begin{equation}\label{eq:GP up bound}
E(N) \leq N \EGP (1+o(1)).
\end{equation}
\end{theorem}

\begin{proof}[Comments]
We do not try to control the error precisely, but this is certainly doable, although one cannot hope for an optimal bound using only the trial states we discussed. Rigorous estimates as above originate in~\cite{Dyson-57}, using the trial function~\eqref{eq:Dyson trial}. This was improved and generalized much later in~\cite{LieSeiYng-00,Seiringer-03}. Later still it was realized that the trial function~\eqref{eq:Jastrow trial} actually does the job, with somewhat simpler computations~\cite{LieSeiYng-04,MicNamOlg-19}. We shall sketch only this latter estimate, and remark that the Dyson trial state giving the same energy as the more natural Jastrow one is remarkable, for it contains only special correlations. 

Before actually sketching a computation, we note that the Dyson trial state is not an obviously valid trial state when magnetic fields are present, $\bA \neq 0$. Indeed~\eqref{eq:Dyson trial} is not invariant under exchange of particles, thus it is not in our bosonic variational set. Without magnetic fields this is of no concern: we are at liberty to use a non-symmetric trial state, for Theorem~\ref{thm:bos min} tells us that the bosonic ground state energy and absolute ground state energy coincide. This is \emph{wrong} for $\bA \neq 0$, as exemplified e.g. in~\cite{Seiringer-03}. However, the infimum of the many-body energy over trial states of the form 
$$ \prod_{j=1} ^N u^ {\rm GP} (\bx_j) F (\bx_1,\ldots,\bx_N)$$
with $F$ \emph{real-valued} does coincide with the infimum over general real-valued $F$. This is proved in~\cite[Section~4.3]{Seiringer-03} and allows us to use the Dyson trial state even when there is a magnetic field. The point is that a real-valued $F$ does not really see the magnetic field, so that one can extend Theorem~\ref{thm:bos min} (following~\cite[Section~3.2.4]{LieSei-09} rather than the proof sketch we provided). 
\end{proof}

\begin{proof}
We sketch the calculation with the \index{Trial state, Jastrow}{Jastrow} trial state, which is clearly bosonic. The details are in~\cite[Section~3.2]{MicNamOlg-19}. As there we set $\bA \equiv 0$ for simplicity but the generalization is straightforward.

We need the following facts about the scattering solution: there are constants $c,C>0$ such that  
\begin{equation}\label{eq:GP bound scat}
c \leq f_{R,N} \leq 1,\: f_{R,N} (\bx) \geq 1 - \frac{C}{N|\bx|}, \: |\nabla f_{R,N}| \leq \frac{C}{N|\bx|^2}.   
\end{equation}
The upper bound $f_{R,N} \leq 1$ comes from the fact that, since $\Delta f_{R,N} \geq 0$ as per~\eqref{eq:GP scat eq}, $f_{R,N}$ must take its maximum on the boundary. The lower bound $c \leq f_{R,N}$ is a consequence of Harnack's inequality.

For the other two bounds observe that, by scaling, $f_{R,N} (N ^{-1} \, \cdot )$ must minimize the scattering energy associated to $w$ in the ball of radius $RN$. Thus, using~\eqref{eq:GP scat sol},  
$$ f_{R,N} (\bx) = \frac{f_w (N \bx)}{f_w (RN)} $$
where $f_w$ is the solution of the scattering solution for the unscaled $w$, in the full space. The desired estimates then follow because $f_w$ is a nice fixed function, to which one can apply Theorem~\ref{thm:scat leng}.

The main point in the calculation is that the kinetic energy $-\Delta_{\bx_j}$ acting on the trial state produces, in addition to the kinetic energy of $u^{\rm GP}$ and cross terms that one can bound using~\eqref{eq:GP bound scat}, the terms 
$$ \sum_{k\neq j} \int_{\R^{3N}} |\nabla_{\bx_j} f_{jk}| ^2 \frac{|\Psi_N^{\rm Jas}|^2}{|f_{jk}|^2} $$
with $f_{jk} = f_{R,N} (\bx_j - \bx_k)$. Grouping those with the interaction terms produces combinations of the form 
\begin{multline*}
 \int_{\R^{3N}} |\nabla_{\bx_j} f_{jk}| ^2 \frac{|\Psi_N^{\rm Jas}|^2}{|f_{jk}|^2} + \frac{1}{2} w_N (\bx_j - \bx_k) |\Psi_N^{\rm Jas}|^2 \\ 
 \leq \int_{\R^{3N}} \left(|\nabla_{\bx_j} f_{jk}| ^2 + \frac{1}{2} w_N (\bx_j - \bx_k) |f_{jk}|^2\right) |u ^{\rm GP} (\bx_j) | ^2 |u ^{\rm GP} (\bx_k) |^2 |\Psi_{N,jk}^{\rm Jas}|^2  
\end{multline*}
using $f_{R,N} \leq 1$ and with 
$$ \Psi_{N,jk}^{\rm Jas} := \frac{\Psi_{N}^{\rm Jas}}{u^{\rm GP} (\bx_j) u^{\rm GP} (\bx_k) f_{jk}\prod_{\ell\neq j,k}f_{j\ell}f_{k\ell}},$$
which is independent of $\bx_j$ and $\bx_k$ and can be shown, using~\eqref{eq:GP bound scat}, to satisfy
$$ \norm{\Psi_{N,jk}^{\rm Jas}}_{L^2 (\R^{3(N-2)})} = 1 + o(1).$$
Hence the integration in $\bx_j,\bx_k$ being independent from the rest, we obtain an effective interaction energy term
$$ 
\int_{\R^{6}} \left(|\nabla_{\bx_j} f_{jk}| ^2 + \frac{1}{2} w_N (\bx_j - \bx_k) |f_{jk}|^2\right) |u ^{\rm GP} (\bx_j) | ^2 |u ^{\rm GP} (\bx_k) |^2.
$$
There is a separation of scales in the above, since the scattering solution lives over a length scale $\sim N^{-1}$ and $u^{\rm GP}$ on the macroscopic length scale. This implies that essentially the integral is located where $\bx_j \simeq \bx_k$ and leads to it being asymptote to 
$$
\left(\int_{\R^{3}} |\nabla f_{R,N}| ^2 + \frac{1}{2} w_N |f_{R,N}|^2\right) \left(\int_{\R^3}|u ^{\rm GP} | ^4\right).
$$
This is the desired quartic interaction energy of $u^{\rm GP}$ and the prefactor is essentially $4 N^{-1} \pi a_w$ because of~\eqref{eq:GP scat ener R}. There are $\sim N^2$ such terms in the computation, summing them leads to the correct GP interaction energy. All other terms can be estimated similarly. 
\end{proof}

\section{\index{Trial state, Bogoliubov}{Bogoliubov-like} trial states}\label{sec:Bogoliubov}

Let us now discuss an alternative trial state, also giving the Gross-Pitaevskii energy in the limit. The precise computation is rather trickier than what we saw in the previous section, but it is also more systematic. The trial state indeed lends itself to modifications allowing to capture also the next-to-leading order in the energy, given by a modified \index{Bogoliubov Hamiltonian}{Bogoliubov-like Hamiltonian} (see Section~\ref{sec:next} for a brief discussion). This remark only scratches the surface of new important developments in the field~\cite{BocBreCenSch-17,BocBreCenSch-18,BocBreCenSch-18b}, namely the study of fluctuations around \index{Bose-Einstein condensate (BEC)}{Bose-Einstein condensation} and the derivation of the \index{Bogoliubov theory}{Bogoliubov excitation spectrum} in the GP limit (see also~\cite{BriFouSol-19,BriSol-19,FouSol-19}). We hint at these developments by discussing an alternative way of enforcing pair correlations in a trial state. The rigorous application of this idea originates in~\cite{ErdSchYau-08}. We present two more recent, closely related, constructions in two subsections. 

\subsection{Take 1}

The following discussion is a summary of~\cite[Appendix~A]{BenPorSch-15}. The trial state constructed there does not have a fixed particle number, i.e. it lives over the \index{Fock space}{Fock space}~\eqref{eq:Fock} (see however~\cite{BocBreCenSch-18,BocBreCenSch-18b,BocBreCenSch-18c,BocBreCenSch-17} for refinements). It is a trial state for the \index{Second quantization}{second-quantized Hamiltonian}
\begin{equation}\label{eq:GP second quant}
\bH_N := \bigoplus H_{N,n}, \quad H_{N,n}:= \sum_{j=1}^n h_{\bx_j} + \sum_{1\leq i < j \leq n} w_N (\bx_i - \bx_j)  
\end{equation}
with the one-particle Hamiltonian
$$ h = \left( -\im \nabla + \bA \right)^2 + V.$$
Strictly speaking it thus does not give a variational upper bound to the $N$-body energy we are concerned with. However, the state is very much concentrated around the $N$-particles sector of Fock space, and such an upper bound could be obtained by either projecting on the $N$-particle sector or considering modified energies as in Section~\ref{sec:coherent}: 
$$ \left\langle \Psi | \bH_N + C \left( \cN - N \right) ^2 | \Psi \right\rangle$$
where $\Psi$ is a normalized vector on the Fock space and $\cN$ the particle number operator~\eqref{eq:number second}. Adjusting $C$ as a function of $N$ yields a ground state very much concentrated around particle number $N$, and this can be used to deduce an upper bound to the $N$-body energy. 

\medskip 

Basically we are trying to implement pair correlations by building a state whose \index{Reduced density matrix}{first and second density matrices}~\eqref{eq:dens mat GC} are related by (identifying them with their kernels)
\begin{equation}\label{eq:GP correl inf}
 \Gamma^{(2)} (\bx_1,\bx_2;\by_1,\by_2) \simeq f_N(\bx_1-\bx_2) f_N (\by_1-\by_2)\Gamma^{(1)} (\bx_1;\by_1) \Gamma^{(1)} (\bx_2;\by_2) 
\end{equation}
with $f_N$ the solution of the scattering equation~\eqref{eq:scat eq} associated with $w_N$. There are two aspects to this equation: 
\begin{itemize}
\item The state is uncorrelated on macroscopic length scales, because $f_N (\bx)\simeq 1$ for $|\bx| \gg N^{-1}$ thus
\begin{equation}\label{eq:GP factor}
 \Gamma_N^{(2)} \simeq \Gamma_N ^{(1)} \otimes \Gamma_N ^{(1)} 
\end{equation}
e.g. in the trace-class topology.
\item On microscopic length scales, we recover the two-body scattering process. This is a singular perturbation of~\eqref{eq:GP factor} that will show up in any norm involving derivatives. In particular this will modify the kinetic energy drastically. 
\end{itemize}
We will implement the two aspects separately by unitary Fock-space operators. We use the notation of Section~\ref{sec:second quant} throughout. First, most particles will be in a condensed state, generated using the

\begin{definition}[\textbf{\index{Weyl operator}{Weyl operators}}]\label{def:Weyl}\mbox{}\\
For $u\in L^2 (\R^d)$ let the associated Weyl operator 
\begin{equation}\label{eq:GP Weyl}
W (u) = e^{\ada (u) - a (u)} = e ^{-\norm{u}_{L^2} ^2}\, e ^{\ada (u)} e^{-a(u)}.
\end{equation}
It is unitary with $W(u)^{-1} = W(u) ^* = W (-u)$  and generates a shift of creation and annihilations operators: 
\begin{align}\label{eq:Weyl shift} 
W (u) ^* \ada (v) W (u) &= \ada(v) + \langle u | v \rangle\nonumber\\
W (u) ^* a (v) W (u) &= a(v) + \langle v | u \rangle.
\end{align}
\end{definition}

You should compare with Definition~\ref{def:coh} and in particular observe that a coherent state is created by applying the Weyl operator to the \index{Vacuum vector}{vacuum vector}
\begin{equation}\label{eq:GP coherent}
\xi (u) := e^{-\norm{u}^2} \bigoplus_{n\geq 0} \frac{u^{\otimes n}}{\sqrt{n!}} = W(u) |0 \rangle 
\end{equation}
To create a \index{Coherent states}{coherent state} with $N$ particles in the normalized wave-function $\varphi$ one would apply $W \left(\sqrt{N} \varphi \right)$ to the vacuum. This is the \index{Grand-canonical}{grand-canonical} analogue of a \index{Bose-Einstein condensate (BEC)}{Bose-Einstein condensate}. One easily computes from the CCR~\eqref{eq:CCR} that the expected number of particles is $N$, with a much smaller variance $\sqrt{N}$:
\begin{align}\label{eq:GP coh num}
\left\langle \xi\left(\sqrt{N} \varphi\right) |\, \cN \, |  \xi\left(\sqrt{N} \varphi\right) \right\rangle &= N\nonumber \\
\left\langle \xi\left(\sqrt{N} \varphi\right) |\left(\cN - N \right)^2|  \xi\left(\sqrt{N} \varphi\right) \right\rangle &=N.
\end{align}

The strategy to generate correlations is to apply a Bogoliubov transformation to the vacuum before applying the Weyl operator. We stay basic on that matter and refer to~\cite[Chapter~9]{Solovej-notes} and~\cite{BacLieSol-94} for more details.

\begin{definition}[\textbf{\index{Bogoliubov transformation}{Bogoliubov transformation}}]\label{def:Bog trans}\mbox{}\\
Let $\gH$ be a separable complex Hilbert space. A Bogoliubov transformation is a unitary operator $\cT$ on the Fock space $\gF (\gH)$ built from $\gH$ as in~\eqref{eq:Fock} such that for all $f\in \gH$
\begin{equation}\label{eq:GP bog}
\cT  \ada (f) \cT ^* = \ada (Uf) + a (V f) 
\end{equation}
and a similar relation for \index{Creation and annihilation operators}{annihilation} operators, where $U$ is a linear map and $V$ an antilinear map\footnote{$V ( \lambda  f + g) = \overline{\lambda} V f + Vg$} on $\gH$.  

Because $\cT$ is unitary, the rotated annihilation and creation operators $\cT a (f) \cT^*$ and $\cT \ada (f) \cT^*$ still satisfy the \index{Canonical commutation relations (CCR)}{CCR}~\eqref{eq:CCR}. This implies that the maps $U,V$ are such that the operator
$$ 
\left(\begin{matrix}
U & V \\
\overline{V} & \overline{U}
\end{matrix}\right)
$$
is unitary on $\gH \oplus \overline{\gH}$.
\end{definition}

\begin{proof}[Comments]
We state a few ``folkloric'' facts about this concept, which might not be exactly right but that you can bear in mind as basic rules of thumb (see~\cite{BacLieSol-94,NamNapSol-16} and references therein for more rigor). 

The origin of this concept is in~\cite{Bogoliubov-47,Valatin-58} where it is used to explicitly diagonalize Hamiltonians that are quadratic polynomials in creation and annihilation operators. Any expression quadratic in $\ada (u),a(u)$ for $u\in \gH$ can be cast in the standard, exactly soluble, form 
$$ \sum c_j \ada (u_j) a (u_j)$$
where the $u_j$ form an orthonormal basis of $\gH$, by conjugation with a Bogoliubov transformation. This is particularly useful in \index{Bogoliubov theory}{Bogoliubov's approach} of the imperfect Bose gas. The main point is to transform expressions that have non particle number conserving contributions such as $\ada (u) \ada(v)$ or $a(u) a(v)$.

A closely related concept is that of \index{Quasi-free state}{quasi-free state} (Definition~\ref{def:quasi free} below), meaning a state that has its correlations in a standard ``gaussian'' form. This means that the expectation in a quasi-free state of any monomial in annihilators/creators can be explicitly computed from the expectations of monomials of order 2. In other words, all the higher density matrices can be computed from the first one by an explicit formula known as the (quantum) \index{Wick's theorem}{Wick theorem}. These particular states play a crucial, ubiquitous role, for they exhaust all the \index{Boltzmann-Gibbs ensemble}{Gibbs states}
$$ \Gamma_{\beta,H} = \frac{e^{-\beta H}}{\tr \left( e^{-\beta H}\right)}$$
of quadratic (non-interacting or weakly interacting, basically) Hamiltonians $H$. Here $\beta >0$ is an inverse temperature and $\Gamma_{\beta,H}$ is a positive-temperature equilibrium, minimizing a free energy (energy minus temperature $\times$ entropy).

In fact, there is an equivalence between quasi-free states, Bogoliubov transformations applied to the vacuum, and Gibbs states of quadratic Hamiltonians.
\end{proof}

We shall use a well-chosen Bogoliubov transformation to define our trial state. Keep in mind our last comments that Gibbs states of quadratic Hamiltonians and Bogoliubov transformations are essentially the same.  

\begin{definition}[\textbf{Bogoliubov-like \index{Trial state, Bogoliubov}{trial state}}]\label{def:bog trial}\mbox{}\\
Let $u^{\rm GP}$ be a minimizer of the GP functional~\eqref{eq:GP func}, $f_N$ be the solution to the scattering equation~\eqref{eq:scat eq} associated with $w_N$. Define the correlation function 
\begin{equation}\label{eq:GP correl func}
k (\bx;\by) := N (f_N(\bx-\by) - 1) \uGP (\bx) \uGP(\by) 
\end{equation}
and the associated \index{Bogoliubov transformation}{Bogoliubov transformation} 
\begin{equation}\label{eq:GP bog correl}
\cT := \exp\left( \frac{1}{2}\iint_{\R^3 \times \R^3} \left(k(\bx;\by) \ada_{\bx} \ada_{\by} - \overline{k(\bx;\by)} a_{\bx} a_\by \right)\mathrm{d}\bx \mathrm{d}\by \right) 
\end{equation}
with the creation and annihilation operators in configuration space as in~\eqref{eq:CCR x}. The Bogoliubov trial state is 
\begin{equation}\label{eq:GP bog trial}
\Psi_N ^{\rm Bog} := W \left( \sqrt{N} \uGP \right) \cT |0\rangle 
\end{equation}
with $W(\cdot)$ the Weyl operator of Definition~\ref{def:Weyl} and $|0\rangle$ the vacuum vector of the Fock space~\eqref{eq:Fock}.
\end{definition}

\begin{proof}[Comments]
The action of the Bogoliubov transformation on creators/annihilators can be explicitly computed~\cite{BenOliSch-12}:
\begin{align}\label{eq:GP Bog rot}
\cT^* a(g) \cT &= a\left( \cosh_k g \right) + \ada (\sinh_k g)\nonumber\\
\cT^* \ada(g) \cT &= \ada\left( \cosh_k g \right) + a (\sinh_k g)
\end{align}
with the operators (here the function $k$ is identified with the operator of which it is the convolution kernel and products are operator compositions)
\begin{equation}\label{eq:GP cosh}
 \cosh_k = \sum_{n\geq 0} \frac{1}{(2n)!} \left( k \overline{k} \right)^n ,\quad \sinh_k = \sum_{n\geq 0} \frac{1}{(2n+1)!} \left( k \overline{k} \right)^n k. 
\end{equation}
This gives the maps $U,V$ associated to $\cT$ in Definition~\ref{def:Bog trans}. See also~\cite[Chapter~5]{BenPorSch-15} for more discussion.

It is not entirely obvious why this construction implements the desired correlations on top of a mostly Bose-condensed state. A tentative rationale is to interpret the operator being exponentiated in $\cT$ as removing pairs of particles in the state $\uGP (\bx) \uGP(\by)$ and replacing them by pairs in the state $f_N (\bx-\by) \uGP(\bx)\uGP(\by)$. This enforces~\eqref{eq:GP correl inf}. The exponentiation makes calculations tractable while keeping this essential building block.  
\end{proof}

We may now state the energy estimate obtained with the above. Again, it works without fixing the particle number but there should be no difficulty in deducing an upper bound on $E(N)$ from this\footnote{In~\eqref{eq:GP ener Bog} I have stated a $O (\sqrt{N})$ remainder as in~\cite[Appendix~A]{BenPorSch-15}, but the constructed trial state actually yields a $O(1)$ remainder, which is optimal (see below).}.

\begin{theorem}[\textbf{Energy of a Bogoliubov-like trial state}]\label{thm:GP ener Bog}\mbox{}\\
With the Fock-space Hamiltonian as in~\eqref{eq:GP second quant} and $\Psi_N ^{\rm Bog}$ as in the previous definition we have 
\begin{equation}\label{eq:GP ener Bog}
\left\langle \Psi_N ^{\rm Bog} | \bH_N | \Psi_N ^{\rm Bog} \right\rangle \leq N \EGP + O (\sqrt{N}).  
\end{equation}
\end{theorem}

\begin{proof}
We present a very brief sketch. A less brief one can be found in~\cite[Appendix~A]{BenPorSch-15}. Supplemented with tools from~\cite[Chapter~5]{BenPorSch-15} and references therein, it can be turned into a complete proof. 

One starts by writing~\eqref{eq:GP second quant} using position \index{Creation and annihilation operators}{creators/annihilators}, as in ~\eqref{eq:second hamil x} and neglecting the magnetic field for simplicity
$$
\bH = \int_{\R^d} a^\dagger _\bx  \left( - \Delta_\bx + V (\bx) \right) a_\bx \mathrm{d}\bx + \iint_{\R^d \times \R ^d} w_N (\bx-\by) a^\dagger _{\bx} a ^{\dagger}_{\by} a_\bx a_\by \mathrm{d}\bx \mathrm{d}\by. 
$$
With a repeated use of~\eqref{eq:Weyl shift} and~\eqref{eq:GP Bog rot} together with the \index{Canonical commutation relations (CCR)}{CCR}~\eqref{eq:CCR x} one can compute 
$$  \cT^* W \left( \sqrt{N} \uGP \right) ^* \bH_N W \left( \sqrt{N} \uGP \right) \cT $$
and put it into \index{Normal order, anti-normal order}{normal order}, with annihilators on the right and creators on the left. In view of the trial state's definition one then needs to take the expectation of the above in the \index{Vacuum vector}{vacuum vector}
$$ |0\rangle = 1 \oplus 0 \oplus \cdots.$$
After normal ordering, all terms that still contain annihilation or creation operators will give a $0$ expectation in the vacuum. Only the constants produced by the normal ordering survive and yield (after a lengthy computation) 
\begin{align}\label{eq:GP mess}
&N \int_{\R^3} |\nabla \uGP| ^2 + \int_{\R^3} \norm{\nabla_\bx \sinh_\bx}_{L^2} ^2  \nonumber \\
\quad \quad &+ N \int_{\R^3} V |\uGP| ^2 + \int_{\R^3} V \norm{\sinh_\bx}_{L^2} ^2  \nonumber\\
\quad \quad &+ \frac{1}{2} \iint_{\R^3 \times \R^3} w_N (\bx - \by) \left| \left\langle \cosh_\by | \sinh_\bx \right\rangle \right|^2  \nonumber\\
\quad \quad &+ \frac{1}{2} \iint_{\R^3 \times \R^3} N w_N (\bx - \by) \left( \left\langle \sinh_\bx | \cosh_\by \right\rangle \uGP (\bx) \uGP (\by) + c.c. \right)\nonumber\\
\quad \quad &+\frac{N}{2} \iint_{\R^3 \times \R^3} N w_N (\bx - \by) |\uGP (\bx)| ^2 |\uGP (\by)| ^2
\end{align}
where the terms on the first line come from the kinetic energy, those on the second line from the external potential and the rest from the interaction. We have denoted 
$$ 
\cosh_\bx (\bz) := \cosh_k (\bx;\bz), \quad \sinh_\bx (\bz) := \sinh_k (\bx;\bz)
$$
the integral kernels of the operators appearing in~\eqref{eq:GP cosh}. Because of~\eqref{eq:GP sca leng} one should think that 
$$k (\bx;\by) \sim \frac{1}{|\bx-\by| + N ^{-1}}$$
and thus approximate 
\begin{equation}\label{eq:GP doigt mouille}
\cosh_\bx (\bz) \simeq \delta (\bx-\bz), \quad \sinh_\bx (\bz) \simeq k(\bx;\bz), 
\end{equation}
which allows to show that the second term of the second line of~\eqref{eq:GP mess} is negligible. All the action lies in combining the second term of the first line (contributed by the kinetic energy) with the interaction to reconstruct the GP interaction energy.

Using~\eqref{eq:GP doigt mouille} repeatedly and neglecting derivatives falling on $\uGP$ (much less singular that those acting on the \index{Scattering solution}{scattering solution}) one finds 
\begin{align*}
 \int_{\R^3} \norm{\nabla_\bx \sinh_\bx} ^2 &= \int_{\R^3} \left\langle \sinh_\bx | \Delta_\bx \sinh_\bx\right\rangle \\
 &\simeq N^4 \iint_{\R^3 \times \R^3} (1- f_N(\bx-\by)) |\uGP (\bx)| ^2 |\uGP (\by)| ^2 \left( - \Delta f (\bx - \by)\right)
\end{align*}
and there remains to use the scattering equation~\eqref{eq:GP scat eq} for $f_N$ (recall it is associated with the scaled potential $w_N$) and insert~\eqref{eq:GP doigt mouille} in the interaction terms of~\eqref{eq:GP mess} to obtain as recombination the effective interaction term 
$$ 
\frac{N}{2}\iint_{\R^3 \times \R ^3} N^3 w (N(\bx - \by)) f_N(\bx-\by) |\uGP (\bx)|^2 |\uGP (\by)|^2. 
$$
Observe now that as $N\to \infty$
$$ N^3 w (N(\bx - \by)) f_N(\bx-\by) = N^3 w (N(\bx - \by)) f(N(\bx-\by)) \wto \left( \int f w\right) \delta_0 $$
and use~\eqref{eq:GP def scat bis} to obtain the desired interaction energy.
\end{proof}

\subsection{Take 2} 

A variant of the above construction is in~\cite{NamNapRicTri-20}. It has the advantages of working directly at fixed particle number, and to allow us to encounter other tools of general interest, such as

\begin{definition}[\textbf{The \index{Excitation map}{excitation map}}]\label{def:excit map}\mbox{}\\
Let $\gH$ be a complex separable Hilbert space. Fix some $u\in \gH$ and denote $\gH^\perp$ its orthogonal in $\gH$. Uniquely write a generic $N$-particle bosonic vector $\Psi_N \in \gH_N$ in the manner 
$$ \Psi_N = \sum_{j=0} ^N \varphi_k \otimes_{\rm sym} u ^{\otimes (N-k)}$$
with bosonic $k-$particles vectors $\varphi_k \in \gH^\perp_k.$ The map 
\begin{equation}\label{eq:exc map}
\cU_N : \begin{cases}
	  \gH_N \mapsto \gF^{\leq N} \left( \gH^{\perp} \right)   \\
	  \Psi_N \mapsto \bigoplus_{k= 0} ^N \varphi_k
        \end{cases}
\end{equation}
is unitary from the $N$-particles space $\gH_N$ to the truncated \index{Fock space}{Fock space} 
$$ \gF^{\leq N} \left( \gH^{\perp} \right) = \bigoplus_{k= 0} ^N \gH^\perp_k.$$
\end{definition}

The definition is from~\cite{LewNamSerSol-13}. The one-body state vector $u$ is thought of as a reference low-energy state in which most particles reside, the orthogonal Hilbert space then represents excited states. The idea is that $\Psi_N \in \gH_N$ is ``close to'' $u^{\otimes N}$ if $\cU_N \Psi_N \in \gF^{\leq N} \left( \gH^{\perp} \right)$ has a low particle number expectation.

We will use $\cU_N$ (with $u = \uGP$ a GP minimizer) instead of the \index{Weyl operator}{Weyl operator} of the previous section. It is then important to be able to conjugate typical Hamiltonians with $\cU_N$, i.e. know the analogue of~\eqref{eq:Weyl shift}. We use the \index{Second quantization}{second quantized} formulation~\eqref{eq:hamil second}, and since $\cU_N$ is unitary, it is sufficient to know how to conjugate a single\footnote{Write $\cU_N a^\dagger a\, \cU_N^* = \cU_N a^\dagger \cU_{N-1} ^{*}  \cU_{N-1}  a\, \cU_N^*$, note that $  \cU_{N-1}  a\, \cU_N^* = \left(\cU_{N}     a^\dagger \, \cU_{N-1} ^{*}\right)^*$ etc...} \index{Creation and annihilation operators}{annihilator/creator}. Denote $\cN^{\perp}$ the \index{Number operator}{number operator} of $\gF^{\leq N} \left( \gH^{\perp} \right)$. The rule is then 
\begin{equation}\label{eq:excit act 1}
 \cU_N  a^\dagger(u) \cU_{N-1} ^{*} = \sqrt{N-\cN^{\perp}} 
\end{equation}
and, for all $v\in \gH^\perp$ 
\begin{equation}\label{eq:excit act 2}
 \cU_N  a^\dagger(v) \cU_{N-1} ^{*} = a^\dagger (v). 
\end{equation}
See~\cite[Section~4]{LewNamSerSol-13} for more details.

We construct a trial state by conjugating a \index{Pure state, mixed state}{mixed state} on $\gF^{\leq N} \left( \gH^{\perp} \right)$ by $\cU_N$. The rationale is that (a) if the latter has few excitations, then the conjugated state is essentially $(\uGP)^{\otimes N}$ and (b) if said particles are energetic enough, they modify the Hamiltonian felt by the non-excited particles. Calculations are made tractable by choosing an excited state of a special form (which we have alluded to already):

\begin{definition}[\textbf{\index{Quasi-free state}{Quasi-free/gaussian states}}]\label{def:quasi free}\mbox{}\\
Let $\Gamma$ be a (mixed) state on the Fock space $\gF (\gH)$ of a separable Hilbert space $\gH$ with finite particle number expectation~\eqref{eq:number second}. It is said to be \emph{quasi-free} if, for any monomial in annihilation/creation operators $(a^\sharp_j)_{1\leq j \leq 2J}$ the \index{Wick's theorem}{Wick rule} 
\begin{align}\label{eq:Wick rule}
\tr \left( a^\sharp_1 \ldots a^\sharp_{2J} \Gamma \right) &= \sum_{\sigma} \prod_{j=1} ^J \tr\left(a^{\sharp}_{\sigma (2j-1)} a^{\sharp}_{\sigma (2j)}\right)\nonumber\\
\tr \left( a^\sharp_1 \ldots a^\sharp_{2J-1} \Gamma \right) &= 0
\end{align}
holds, where the sum is over all pairings, i.e. permutations of the $2J$ indices such that $\sigma(2j-1) < \min \left\{ \sigma (2j), \sigma (2j+1)\right\}$ for all $j$. 
\end{definition}

\begin{proof}[Comments]
The Wick rule is the quantum generalization of the rule for computing higher moments of a gaussian random  variable as a function of its first moment, whence the name ``gaussian'' states. The name ``quasi-free'' comes from the fact that essentially any equilibrium of a weakly interacting Hamiltonian (i.e. quadratic in annihilators/creators) is quasi-free.

The definition says that a quasi-free state is fully determined (because all its density matrices are) by its one-body density matrix $\gamma:\gH \mapsto \gH$ and its pairing matrix $\alpha:\gH \mapsto \overline{\gH}$ defined by 
\begin{align}\label{eq:DM QF}
 \left\langle f | \gamma | g \right\rangle &= \tr \left(a^\dagger (g) a (f)  \Gamma \right)\nonumber \\
 \left\langle \overline{f} | \alpha | g \right\rangle &= \tr \left(a^\dagger (g) a^\dagger (f)  \Gamma \right)
\end{align}
for all $f,g\in \gH$. We have encountered $\gamma$ before, this is just the one-particle reduced density matrix~\eqref{eq:def red mat}. The \index{Pairing matrix}{pairing matrix} $\alpha$ looks at how $\Gamma$ couples different sectors of Fock space (it is zero for states with fixed particle number). The funny convention that $\alpha:\gH \mapsto \overline{\gH}$ should not concern you too much, it is because one might prefer linear to anti-linear operators (in contrast with the convention in Definition~\ref{def:Bog trans}).  
\end{proof}

Again, see~\cite{BacLieSol-94,Solovej-notes,Nam-thesis} for a more complete discussion, in particular, for the proof of the next lemma (Theorem~3.2 in~\cite{Nam-thesis}, see also~\cite[Appendix~A]{LewNamSerSol-13}). 

\begin{lemma}[\textbf{Quasi-free reduced density matrices}]\label{lem:QF DM}\mbox{}\\
Let $\gamma:\gH \mapsto \gH$ and $\alpha:\gH \mapsto \overline{\gH}$. There exists a unique mixed quasi-free state with $\Gamma$ these one-body and pairing density matrices if and only if 
$$ 
\gamma \geq 0, \quad \tr\, \gamma < \infty, \quad \overline{\alpha} = \alpha ^*$$
and
\begin{equation}\label{eq:cond DM QF}
\begin{pmatrix}
 \gamma & \alpha^* \\
\alpha & 1 + \overline{\gamma}
\end{pmatrix} \geq 0 \mbox{ on } \gH \oplus \overline{\gH} 
\end{equation}
Here we denote $\overline{A} = J A J$ with $J$ the complex conjugation. If (and only if) moreover
$$\alpha \alpha ^* = \gamma (1 + J\gamma J^*) \mbox{ and } \gamma \alpha = \alpha J \gamma J^*$$
then $\Gamma$ is \index{Pure state, mixed state}{pure} (a orthogonal projector).
\end{lemma}

This suggests an appealing construction. First pick one-body and pairing density matrices $\gamma,\alpha$ giving rise to the desired correlations (i.e. higher density matrices) via the Wick rule~\eqref{eq:Wick rule}. Check that they satisfy~\eqref{eq:cond DM QF}. Then there exists a state for which one can calculate everything, with the desired correlations built-in. This leads to the following 

\begin{definition}[\textbf{Bogoliubov-like \index{Trial state, Bogoliubov}{trial state}, again}]\label{def:bog trial 2}\mbox{}\\
Let $\uGP$ be a GP minimizer and $f_N$ the zero-energy scattering solution from Theorem~\ref{thm:scat leng} associated with $w_N (\bx) = N^2 w (N \bx)$. Let $k$ be the operator on $L^2 (\R^3)$ with integral kernel (cf Definition~\ref{def:bog trial})
\begin{equation}\label{eq:kernel cor}
k(\bx,\by):= \uGP (\bx) N (1-f_N (\bx-\by))  \uGP(\by) 
\end{equation}
and $Q$ be the orthogonal projector on $\gH^\perp$, the orthogonal of $\mathrm{span} (\uGP)$. Let 
$$ \gamma = Q k ^2 Q, \quad \alpha = \overline{Q} k Q$$
where $k^2$ is meant as an operator square. The above operators satisfy the requirements of~Lemma~\ref{lem:QF DM}. Let $\Gamma$ be the unique associated quasi-free state on $\gF (\gH^\perp)$ and 
$$ \Gamma_N := \cU_N ^* \1_{\cN^\perp \leq N} \Gamma \1_{\cN^\perp\leq N} \cU_N$$
with $\cN^\perp$ the number operator on $\gF (\gH^\perp)$.
\end{definition}

The result, from~\cite{NamNapRicTri-20}, is 

\begin{theorem}[\textbf{Energy of a Bogoliubov-like trial state, again}]\label{thm:GP ener Bog 2}\mbox{}\\
With $\Gamma_N$ as in the previous definition we have 
\begin{equation}\label{eq:up bound NNRT}
E (N) \leq \frac{\tr \left(H_N \Gamma_N \right)}{\tr \, \Gamma_N} \leq N\EGP + C 
\end{equation}
for a constant $C>0$ uniformly bounded in $N$.
\end{theorem}

\begin{proof}[Comments]
The trial state is mixed, but that is of no concern, the first inequality holds because the energy is linear in the state $\Gamma_N = |\Psi_N \rangle \langle \Psi_N |.$ Note that the order of the remainder in the energy upper bound has been made explicit, and is in fact optimal~\cite{BocBreCenSch-18}. 

Details of the calculation are lengthy, and shall not be reproduced. The point is that everything is rather explicit: one must conjugate the original Hamiltonian with $\cU_N$ from Definition~\ref{def:excit map}. This is conveniently done using the second-quantized expression~\eqref{eq:hamil second} and~\eqref{eq:excit act 1}-~\eqref{eq:excit act 2}. There are cancellations because $\uGP$ satisfies the GP variational equation, and almost-cancellations because $\Gamma$ has a bounded number of particles (replace $N-\cN^\perp \rightsquigarrow N$ each time it occurs.) Once this is done, what is left is to compute the expectation of the conjugated Hamiltonian in $\Gamma$. The latter being quasi-free, with explicit one-body and \index{Pairing matrix}{pairing DMs}, the \index{Wick's theorem}{Wick rule}~\eqref{eq:Wick rule} takes care of the computation. 

Now, why (as opposed to how) does this work ? To answer this, recall the correlations~\eqref{eq:GP correl inf} we want to enforce. In particular, to reduce the interaction energy we would like to have
\begin{equation}\label{eq:correl loc}
 \Gamma_N ^{(2)}(\bx,\by;\bx,\by) \simeq N^2 |\uGP (\bx)|^2 |\uGP(\by)|^2 |f_N (\bx-\by)| ^2.
\end{equation}
From the action of $\cU_N$ and the Wick rule one can compute exactly for our trial state. Let me mention only the salient points, which can be guessed by following~\eqref{eq:GP sca leng} and replacing 
\begin{equation}\label{eq:doigt 2}
 f_N (\bx) \rightsquigarrow 1 - \frac{a}{N|\bx|} 
\end{equation}
wherever it occurs. 

Heuristically (we neglect all occurrences of $Q$)
\begin{align*}
 \Gamma_N ^{(2)}(\bx,\by;\bx,\by) &\simeq N^2 |\uGP (\bx)|^2 |\uGP(\by)|^2 + \Gamma^{(2)} (\bx,\by;\bx,\by) \\
 &= N^2 |\uGP (\bx)|^2 |\uGP(\by)|^2 + \gamma (\bx;\bx) \gamma (\by;\by) + |\gamma (\bx;\by)| ^2 + |\alpha (\bx;\by)| ^2\\
 &\simeq N^2 |\uGP (\bx)|^2 |\uGP(\by)|^2 + |\alpha (\bx;\by)| ^2.
\end{align*}
The first line has simply the density matrix of the condensate plus that of $\Gamma$ on its right-hand side. Normally there would be cross-terms, but we neglect them for the following reasons. (a) $\gamma$ is fairly regular and does not contribute to the leading order. (b) There is a singularity in $\alpha$, and hence $|\alpha(\bx;\by)|^2$ is typically much bigger that $\alpha (\bx;\by).$ Thus any linear term in $\alpha,\gamma$ does not contribute to the leading order. In the second line we have expressed $\Gamma ^{(2)}$ using the Wick rule. The three terms are known as ``direct'', ``exchange'' and ``pairing''. In the third line we use again that $\gamma$ is regular.

Inserting the expression for $\alpha$ we find 
$$ \Gamma_N ^{(2)}(\bx,\by;\bx,\by) \simeq N^2 |\uGP (\bx)|^2 |\uGP(\by)|^2 \left( 1 + (1-f_N (\bx-\by)) ^2 \right) $$
and if we use~\eqref{eq:doigt 2} we have that 
$$ 1 + (1-f_N (\bx-\by)) ^2 \simeq 1 + \frac{a^2}{N^2 |\bx-\by|^2} \simeq f_N (\bx-\by) ^2$$
as desired for~\eqref{eq:correl loc}. In the last approximation we use 
$$ 1 + \frac{a^2}{N^2 |\bx-\by|^2} - f_N (\bx-\by) ^2 \simeq - 2 \frac{a}{N|\bx-\by|}.$$
The right-hand side is much smaller than the main terms for $N|\bx-\by|$ very small or very large, and this is all we care about. 

The above explains, I hope, how one can reproduce a \index{Trial state, Jastrow}{Jastrow-like} factor by using \index{Quasi-free state}{quasi-free states} with singular pairing density matrices. To obtain the final energy estimates one must be careful that $\gamma$ and $\alpha$ are negligible against $|\alpha|^2$ but \emph{their derivatives are not}: there are few excitations, but they are very energetic. Including their high kinetic energy (due to them being dealing with short length scales) in the calculation leads to the final estimate.
\end{proof}

\section{\index{Dyson lemma}{Dyson lemmas}}\label{sec:Dyson}

In the previous two sections we have seen how to extract the \index{Scattering length}{scattering length} from suitable trial states, and thus obtain the GP energy as an upper bound to the true ground state energy in the GP limit. It is much harder to obtain the GP energy as a lower bound, i.e. prove that the trial states just constructed are optimal. This is indeed tantamount to finding a universal way of extracting short-range pair correlations from a generic wave-function, in order for the original interaction to combine neatly with part of the kinetic energy and reproduce the two-body scattering process. Our weapon of choice to achieve this originates in~\cite{Dyson-57} and goes under the name of Dyson lemma. 

\medskip

The first version of the lemma bounds from below the kinetic and interaction energies with a possibly very singular potential in terms of a potential energy in a much softer new potential.  

\begin{lemma}[\textbf{Dyson's lemma}]\label{lem:Dyson 1}\mbox{}\\
Let $w$ be as in Assumption~\ref{asum:GP pot}, with finite range $R_w$. Let $U:\R^+ \mapsto \R$ be a function with 
$$ \int_{\R^+} U(r) r^2 \mathrm{d}r = a_w $$
where $a_w$ is the scattering length of $w$, as in Theorem~\ref{thm:scat leng}. Assume that the support of $U$ is disjoint from that of $w$:
$$
U(r) = 0 \mbox{ for all } r< R_w.
$$
Then, for any differentiable $f:\R^3 \mapsto \C$ and any convex domain $\Omega$ containing the origin
\begin{equation}\label{eq:Dyson lem 1}
\int_\Omega \left(|\nabla f| ^2 + \frac{1}{2} w |f|^2 \right) \geq \int_{\Omega} U |f|^2
\end{equation}

\end{lemma}

\begin{proof}[Comments]
Again, you should think of $f$ as describing the relative motion of a pair of particles. The whole point is that when we apply this to a scaled or singular potential (e.g. $w \rightsquigarrow w_N$ as in the \index{Gross-Pitaevskii limit}) with fixed integral, $R_w$ is typically very small and hence $w$ very large on its support. The (radial) function $U$ we replace it with however lives \emph{outside} the small support of $w$. We still want to fix its integral, but that no longer implies that $U$ needs be very large on its support. Thus we have obtained a lower bound in terms of a softer potential whose integral gives the scattering energy. Indeed, as a 3D function 
$$ \int_{\R^3} U = 4\pi a_w.$$
For an energy lower bound we can then think of the GP limit with potential $w_N (x) = N^{2} w (N x)$ as a \index{Mean-field limit}{mean-field limit} with the potential $U$. In practice we will not be at liberty to take $U$ as soft as we like, and the limit shall rather be a \index{Dilute limit}{dilute} one.
\end{proof}

\begin{proof}
We follow~\cite[Lemma~2.5]{LieSeiSolYng-05}. Let first $R\geq R_w$ and, for all $\sigma$ on the unit sphere $\mathbb{S}^2$, $R(\sigma)$ be the length of the radial segment starting at the origin and included in $\Omega$. We have that 
$$\int_\Omega \left(|\nabla f| ^2 + \frac{1}{2} w |f|^2 \right) \geq \oint_{\sigma \in \mathbb{S}^2} \int_{0}^{R(\sigma)} \left(|\partial_r f(r,\sigma)| ^2 + \frac{1}{2} w |f(r,\sigma)|^2  \right) r^2 \mathrm{d}r \mathrm{d}\sigma$$
and we bound the integral over each radial segment at fixed $\sigma$ in the manner
$$ \int_{0}^{R(\sigma)} \left(|\partial_r f(r,\sigma)| ^2 + \frac{1}{2} w |f(r,\sigma)|^2  \right) r^2 \mathrm{d}r \geq \begin{cases}                                                                                                              0 \mbox{ if } R(\sigma) < R\\
a_w |f(R,\sigma))| ^2 \mbox{ if } R(\sigma) \geq R.                                                                                                              \end{cases}
$$
Indeed, in the first case there is nothing to prove, for the integrand is non-negative. In the second case, since the scattering problems considered in Theorem~\ref{thm:scat leng} have radial solutions, we can bound from below by the ground state energy of $(4\pi)^{-1}$ times~\eqref{eq:GP scat R}, with Dirichlet boundary condition $f(R,\sigma)$. By calculations similar to those in the proof of Theorem~\ref{thm:scat leng} we obtain that for $R(\sigma) \geq R > R_w$, this lower bound is exactly as in the right-hand side of the above. 

Thus we have proved that, for any $R> R_w$ 
$$ \int_\Omega \left(|\nabla f| ^2 + \frac{1}{2} w |f|^2 \right) \geq \oint_{\sigma \in \mathbb{S}^2}  a_w |f(R,\sigma))|^2 \1_{R(\sigma) \geq R} \mathrm{d}\sigma.$$
Let now $\widetilde{U} = U a_w$ with $U$ as in the statement. Multiplying the above by $\widetilde{U} (R)$ and integrating with respect to $R^2 dR$ proves the lemma (the left-hand side does not depend on $R$).
\end{proof}

The previous lemma will not suffice for all our applications. A major drawback is that it gives away all the kinetic energy to obtain a lower bound on the interaction. One might fear that it thus also sacrifices the kinetic energy due to the gas' density varying at the macroscopic scale of the full system, which should however enter the final energy. Sometimes this is harmless: if the system is homogeneous for instance, there is no macroscopic kinetic energy to be recovered. Sometimes (see Section~\ref{sec:LDA} below) one can get away by applying Lemma~\ref{lem:Dyson 1} after having extracted the macroscopic kinetic energy by some neat trick. 

In the general case however, since interactions happen on a short length scale, only the high-frequency part of the kinetic energy should be used to control the scattering process. The low-frequency part should be left untouched and used to reconstruct the macroscopic variations of the density profile. The next lemma, from~\cite{LieSeiSol-05}, does just that\footnote{Originally it was introduced to deal with the low-density unpolarized Fermi gas. There the density does not vary on the macroscopic scale, but the kinetic energy of the free system is non-trivial (Pauli principle, Fermi sphere ...) and has to be recovered.}. Other variants are in~\cite{LewSei-09,SeiYng-20}.

\begin{lemma}[\textbf{Generalized \index{Dyson lemma}{Dyson lemma}}]\label{lem:Dyson 2}\mbox{}\\
Let $\chi :\R^+ \mapsto \R^+$ be a smooth function with $0 \leq \chi \leq 1$, $\chi (r) = 0$ for $r \leq 1$ and $\chi (r)= 1$ for $r\geq 2$. Define 
$$ \chi_{K} (x) := \chi (K x)$$
and let $\chi_K (\bp) (-\Delta) \chi_K (\bp)$ denote the operator acting as the multiplication by $\chi_K(p)^2 |p|^2$ in the Fourier domain:
$$ 
\chi_K (\bp) (-\Delta) \chi_K (\bp) \psi = \cF^{-1} \left( \chi_K(\bp)^2 |\bp|^2 \widehat{\psi} (\bp) \right). 
$$
Let $U:\R^+ \mapsto \R^+$ be as in the previous lemma, with in addition $U(r) \equiv 0$ for some $R>R_w$. For any $\eps >0$ we have that
\begin{equation}\label{eq:Dyson lem 2}
\chi_K (\bp) (-\Delta) \chi_K (\bp) + \frac{1}{2} w(x) \geq (1-\eps) U - C\frac{a_w R^2 K^5}{\eps}.
\end{equation}
In fact, for any differentiable function $f$, 
\begin{equation}\label{eq:Dyson lem 2 bis}
\int_{|\bx|\leq R}  \left| \chi_K (\bp) \nabla f \right|^2 + \frac{1}{2} w(x) |f|^2\geq \int_{\R^3} \left((1-\eps) U - C\frac{a_w R^2 K^5}{\eps} \right) |f|^2
\end{equation}
where $\chi_K (\bp) \nabla$ acts as $-\im \chi_K (\bp)\bp \cdot$ on the Fourier side.
\end{lemma}

\begin{proof}[Comments]
The proof is a variation on that we just discussed, which is the case $\chi \equiv 1$. See~\cite{LieSeiSol-05} for details. The remaining kinetic energy is $(1 - \chi_K^2 (p)) (-\Delta)$, so if we take $K\to \infty$ we have indeed consumed only its high-frequency part to reconstruct the scattering length via the integral of $U$. In applications one can  afford to take $K\to \infty$ (and $\eps \to 0$) at the very end of the proof, which is a quite strong indication of the separation of scales at work in the problem. 

In applications, $a_w \propto N^{-1}$ will be a small number in the limit $N\to \infty$. The range $R_w \propto N^{-1}$ is also very small. What we want to do is replace $w$ with a softer potential $U$ with fixed integral, which is possible if we can take $R$ much larger than $R_w$. We also want to take the frequency cut-off $K$ to be large. The lemma tells us how to tune $R$ and $K$ in order to make an affordable error.
\end{proof}

Next we turn to applying the above lemmas to the many-body problem. Let us start with a consequence of Lemma~\ref{lem:Dyson 1}:

\begin{corollary}[\textbf{Dyson's lower bound to the full Hamiltonian}]\label{cor:Dyson 1}\mbox{}\\
Consider the many-body Hamiltonian~\eqref{eq:intro Schro op bis}, possibly restricted to a finite domain, and with $w_N\geq 0$. Let 
$$ \bt_j := \min \left\{ |\bx_j - \bx_k|, 1 \leq k \leq N, k\neq j \right\}$$
be the distance from $\bx_j$ to its nearest neighbor. 

Let $U$ be associated to $w_N$ as in Lemma~\ref{lem:Dyson 1}. Then, as operators,
\begin{equation}\label{eq:Dyson 1 app} 
H_N \geq \sum_{j= 1} ^N U (\bt_j).
\end{equation}
\end{corollary}

\begin{proof}
By the diamagnetic inequality~\cite[Theorem~7.21]{LieLos-01} we can assume for this lower bound that the magnetic field is $0$. We also drop the one-body term. Consider then the part acting on the first particle, 
$$ -\Delta_{\bx_1} + \frac{1}{2} \sum_{1 < j \leq N} w_N (\bx_1-\bx_j).$$
Split the domain ($\R^d$ or a finite region) into Voronoi cells ($V_k$ is the set where  $\bx_1$ is closer to $\bx_k$ than to any other point in the collection)
$$ V_k := \left\{ \bt_1 = |\bx_1 - \bx_k| \right\}.$$
These are convex sets and we may thus apply Lemma~\ref{lem:Dyson 1} in them, with $\bx_1$ playing the role of the origin. This yields, for any $N$-particle wave-function $\Psi_N$ 
$$ \int_{V_k} |\nabla_{\bx_1} \Psi_N|^2 + \frac{1}{2} \sum_{1 < j \leq N} w_N (\bx_1-\bx_j)|\Psi_N|^2 \geq  
\int_{V_k} U (\bt_1) |\Psi_N|^2$$
where we use $w_N\geq 0$ to keep only the contribution of $w_N (\bx_1 - \bx_k)$ in $V_k$. Adding the contribution of all Voronoi cells, and then the parts acting on particles $2,\ldots N$ we get the statement.
\end{proof}

A drawback of the above, besides that we have used all the kinetic energy, is that the bound from below is in terms of a nearest neighbor potential instead of a genuine pair interaction. This can be handled in the dilute regime, for three-particles encounters are rare anyway.

\medskip

Next, Lemma~\ref{lem:Dyson 2} leads to 

\begin{corollary}[\textbf{Generalized Dyson lower bound}]\label{cor:Dyson 2}\mbox{}\\
Consider the many-body Hamiltonian~\eqref{eq:intro Schro op bis}, in GP scaling $w_N (\bx) = N^2 w(N\bx)$, with $w$ satisfying Assumption~\ref{asum:GP pot}. 

Let $R \geq 2 N^{-1} R_w $ and $U_R$ be associated to $w_N$ via Lemma~\ref{lem:Dyson 2}, with $U_R (r) = 0$ if $ N^{-1} R_w \leq r \leq R$. Let $\chi_K$ also be as in Lemma~\ref{lem:Dyson 2}. Then, for all $\eps >0$ and $K>0$ 
\begin{equation}\label{eq:Dyson 2 app}
H_N \geq \sum_{j=1} ^N \left(h_{\bx_j} - (1-\eps) \chi_K (\bp_j) (-\Delta_{\bx_j}) \chi_K (\bp_j) \right) +  (1-\eps)^2 W_N - C \frac{N^2 R^2 K^5}{\eps} 
\end{equation}
with 
$$ 
W_N (\bx_1,\ldots,\bx_N) := \sum_{i\neq j } U_R (\bx_i - \bx_j) \prod_{k \neq i, j} \Theta_{2R} (\bx_j-\bx_k) 
$$
with $\Theta_{2R}$ a radial smoothened Heaviside step-function:
$$ 0 \leq \Theta_{2R} \leq 1, \quad \Theta_{2R} (r) = 0 \mbox{ for } r \leq 2R, \quad \Theta_{2R} = 1 \mbox{ for } r \geq 4R.$$
Moreover, 
\begin{equation}\label{eq:Dyson 2 app 2}
W_N \geq \sum_{i\neq j } U_R (\bx_i - \bx_j) - \sum_{k\neq i \neq j \neq k} U_R (\bx_i - \bx_j) \left( 1- \Theta_{2R} (\bx_j-\bx_k) \right).
\end{equation}
\end{corollary}

\begin{proof}
We work on the part of the Hamiltonian acting on $\bx_1$. Consider fixing $\bx_2, \ldots, \bx_N \in \R^d$ with $|\bx_1-\bx_j| \geq 2R$ for all $j=2\ldots N$. Then there can only be one such point at a time with $|\bx_1 - \bx_j| \leq R$. Therefore, for any function $\psi$ of $\bx_1$  
$$ \int_{\R^d} |\chi_K (\bp_1) \nabla_{\bx_1} \psi|^2 + \frac{1}{2} \sum_{j=2} ^N w_N (\bx_1-\bx_j)|\psi|^2 \geq \sum_{j=2}^N \int_{|\bx_1-\bx_j|\leq R} |\nabla_{\bx_1} \psi|^2 + \frac{1}{2} w_N (\bx_1-\bx_j)|\psi|^2.$$
Applying Lemma~\ref{lem:Dyson 2} and integrating over the set $|\bx_1-\bx_j| \geq 2R$ for all $j=2\ldots N$ we find, for any $L^2$ normalized function $\Psi_N$ of the $N$ coordinates 
\begin{multline*}
 \int_{\R^{dN}} |\chi_K (\bp_1) \nabla_{\bx_1} \psi|^2 + \frac{1}{2} \sum_{j=2} ^N w_N (\bx_1-\bx_j)|\Psi_N|^2 \\
 \geq (1-\eps) \int_{\R^{dN}} \sum_{j\leq 2 \leq N} U_R (\bx_1-\bx_j) \prod_{2 \leq k\neq j \leq N} \Theta_{2R} (\bx_j-\bx_k) - \frac{C N^2 R^2 K^5}{\eps}. 
\end{multline*}
Multiplying by $(1-\eps)$ and adding the contributions of the part of the Hamiltonian acting on the $N-1$ other particles leads to 
\begin{multline*}
H_N \geq \sum_{j=1} ^N \left(h_{\bx_j} - (1-\eps) \chi_K (\bp_j) (-\Delta_{\bx_j}) \chi_K (\bp_j) \right) + 
\\ (1-\eps)^2 \left( \sum_{i\neq j } U_R (\bx_i - \bx_j) \prod_{k \neq i,j} \Theta_{2R} (\bx_j-\bx_k) \right)- C \frac{N^2 R^2 K^5}{\eps}. 
\end{multline*}
To conclude we note that 
$$ \prod_{k,j\neq i} \Theta_{2R} (\bx_j-\bx_k)  = \prod_{k,j\neq i } \left(1 - (1-\Theta_{2R} (\bx_j-\bx_k))\right) \geq 1 - \sum_{k,j\neq i } \left(1 -\Theta_{2R} (\bx_j-\bx_k)\right)$$
because, for any choice of numbers $0 \leq s_j \leq 1,j=1\ldots J$ we have 
$$ \prod_{j=1} ^J (1-s_j) \geq 1 - \sum_{j=1} ^J s_j$$
which can be proven by induction over $J$. 
\end{proof}

Combining~\eqref{eq:Dyson 2 app} and~\eqref{eq:Dyson 2 app 2} we have a lower bound with an unwanted three-body term
$$ \sum_{k\neq i \neq j \neq k} U_R (\bx_i - \bx_j) \left( 1- \Theta_{2R} (\bx_j-\bx_k)\right).$$
But the summand is non-zero only for $|\bx_i - \bx_j| \leq R$ and $|\bx_j - \bx_k| \leq 4R$, i.e. when three particles are at distance $\lessapprox R$ from one another. To discard this term we shall choose $R\ll N^{-1/3}$, the typical inter-particle distance, and the latter event will intuitively have very small probability. Proving it rigorously is not quite easy however. Having to choose $R\ll N^{-1/3}$ means bounding from below the Gross-Pitaevskii scaling Hamiltonian by a new Hamiltonian with dilute scaling: we cannot quite soften the potential as much as we would like to.

It will turn out useful to recall~\eqref{eq:Dyson 2 app}, even though the combination with~\eqref{eq:Dyson 2 app 2} is simpler. The reason is that in~\eqref{eq:Dyson 2 app} we keep memory of the fact that the original interaction was positive, an information which is lost in~\eqref{eq:Dyson 2 app 2}.  

\section{Thermodynamic energy of the homogeneous dilute gas}\label{sec:thermo}

We start our study of GP limits of bosonic ground states by considering the case of a homogeneous gas in a cubic box, that is the one-body potential is set to $0$ and the Euclidean space $\R^{dN}$ is replaced by $\Lambda ^{N}$ where $\Lambda$ is a cube of side length $L$. The reason this is simpler is that we do not have to worry about a non-trivial macroscopic density profile: the minimizer of the GP energy functional is the constant function. For later purposes (and mostly because this is an important result in itself) we shall however be concerned with bounds whose error terms are uniform in the side length $L$ of the box, provided the total density $\varrho = N / L ^3$ satisfies the diluteness condition 
\begin{equation}\label{eq:dilute thermo}
\varrho a^3 \ll 1 
\end{equation}
where $a$ is the scattering length of the interaction (which you may think of as setting its effective range). 

In other words, we are interested in the \index{Thermodynamic limit}{thermodynamic limit} of the energy, and want to be able to take this limit before we eventually play with the potential or the particle number/density to achieve~\eqref{eq:dilute thermo}: 

\begin{theorem}[\textbf{Energy of the \index{Homogeneous Bose gas}{homogeneous dilute Bose gas}}]\label{thm:therm}\mbox{}\\
Consider 
$$ H_N := \sum_{j=1} ^N -\Delta_{\bx_j} + \sum_{1\leq i < j \leq N} w_a (\bx_i - \bx_j)$$
acting on $L^2_{\rm sym} (\Lambda ^N)$, $\Lambda = [0,L]^3$ with periodic boundary conditions.  Here 
$$ w_a (\bx) = a^{-2} w ( a ^{-1} \bx )$$
with $w$ fixed satisfying Assumption~\ref{asum:GP pot} (hence $w_a$ has \index{Scattering length}{scattering length} $a$). Let $E(N,L)$ be the lowest eigenvalue of this operator, fix the density $\varrho = NL^{-3}$ and define 
\begin{equation}\label{eq:def therm ener}
e (\varrho) := \lim_{N,L \to \infty} \frac{E(N,L)}{N} 
\end{equation}
the energy density in the thermodynamic limit.

In the limit $\varrho a^3 \to 0$ we have 
\begin{equation}\label{eq:therm ener dilute}
e(\varrho) = 4\pi a \varrho (1 + o (1)). 
\end{equation}
More precisely, for any $N,L$ large enough, the lower bound
\begin{equation}\label{eq:therm ener bounds}
\frac{E(N,L)}{N} \geq 4\pi \varrho a \left( 1 - o (1)\right)
\end{equation}
holds, where the $o(1)$ goes to zero when $N a^{3} L^{-3} \to 0$. 
\end{theorem}

\begin{proof}[Comments]
An energy upper bound was derived in~\cite{Dyson-57}, but the corresponding lower bound was proved only some 40 year later~\cite{LieYng-98}. More explicit estimates of the error terms than what we state are known, and an upper bound matching~\eqref{eq:therm ener bounds}. All of this is already reviewed in~\cite{LieSeiSolYng-05} and we do not give further details. Note that one can achieve the \index{Dilute limit}{dilute limit} either by scaling the potential with some $a \to 0$ or by lowering the density $\varrho \to 0$. The latter ways is usually preferred physically: one considers the potential as given by the physics, and the density as the adjustable parameter. 

The reason we emphasize the lower bound in~\eqref{eq:therm ener bounds} is that it will be used in the next section, when we deal with the inhomogeneous gas via \index{Local density approximation (LDA)}{local density approximation}.
\end{proof}

We do not comment on the energy upper bound, whose proof resembles that sketched in Section~\ref{sec:Jas-Dyson}. A noteworthy difference is that one has to use Dyson's \index{Trial state, Dyson}{trial state} instead of Jastrow's to handle the thermodynamic limit. The excluded volume induced by including all pair correlations in the Jastrow state would be too large when taking the thermodynamic limit first. 

We focus on sketching the proof of~\eqref{eq:therm ener bounds}, referring to~\cite{LieSeiSolYng-05} for a more detailed exposition and references to original sources. Generalizations are in~\cite{Lee-09,Yin-10b,LieYng-01} Without any further comment we always ignore bosonic symmetry in this section, as per Theorem~\ref{thm:bos min}.

\medskip

The proof works in two steps that we present as two lemmas:

\begin{lemma}[\textbf{Energy in a small box}]\label{lem:small box}\mbox{}\\
Assume that $L \leq C (a \varrho ^2)^{-1/5}$ for a sufficiently small constant $C>0$. We have for any $0< \eps < 1$ and $R < L/2$,
\begin{multline}\label{eq:small box}
\frac{E(N,L)}{N} \geq 4\pi a \varrho (1-\eps) \left(1 - \frac{2R}{L} \right)^3 \left(1 + \frac{4\pi}{3} (R^3 - a^3 R_w ^3)  \right)^{-1} \\ 
\times \left( 1 - \frac{3}{\pi} \frac{aN}{(R^3 - a^3 R_w ^3)\left( \pi \eps L ^{-2} - 4 a L^{-3} N(N-1) \right)} \right). 
\end{multline}
\end{lemma}

\begin{proof}[Comments]
Here you should really think that $a$ and $\varrho$ are of order $1$ (they are in the thermodynamic limit, and one of them is taken to $0$ only afterwards). The condition $L \leq C (a \varrho ^2)^{-1/5}$ is necessary for the denominator in the last factor to be non-negative, for $a L^{-3} N(N-1)$ is essentially $a L^3 \varrho^2$. 

The rationale behind this expression is that we apply Corollary~\ref{cor:Dyson 1} with 
$$ U_R (r) = \begin{cases} 
              \frac{a}{3(R^3 - (aR_w)^3)^{-1}} \mbox{ for } aR_w < r < R\\
              0 \mbox{ otherwise,}
             \end{cases}
$$
the (appropriately normalized) indicator function of an annulus with inner radius the range of the interaction potential $w_a$ (as required in Lemma~\ref{lem:Dyson 1}) and an outer radius we are free to choose. The result is that we need to bound from below 
$$ E(N,L) \geq \int_{\Lambda^N} \sum_{j=1} ^N U(\bt_j) |\Psi_N| ^2.$$
If we could replace the true interacting ground state $\Psi_N$ by the (normalized) constant wave-function (true ground state of the non-interacting problem) in the above, we would get as lower bound 
\begin{equation}\label{eq:geometry}
 E(N,L) \geq 4\pi a \varrho \left(1 - \frac{2R}{L} \right)^3 \left(1 + \frac{4\pi}{3} (R^3 - a^3 R_w ^3)  \right) ^{-1} 
\end{equation}
without any further correction, and we would be in extremely good shape. The extra correcting factors come about when coping with the replacement $\Psi_N \rightsquigarrow$ non-interacting ground state.

Let us explain where~\eqref{eq:geometry} comes from. We bound from below 
\begin{equation}\label{eq:use Dyson}
 U_R(\bt_1) \geq \frac{a}{3(R^3 - (aR_w)^3)^{-1}} \1_{\bx_1\in \Lambda_{L-2R}} \1_{\exists k \geq 2, aR_w \leq |\bx_k - \bx_1| \leq R}
\end{equation}
where $\bx_1$ is in the subcube at least at distance $2 R$ from the boundary of the original cube. We integrate this first in $\bx_2, \ldots, \bx_N$. We have made sure that the ball of radius $R$ around $\bx_1$ always fits in the box so we have to estimate the probability (in the constant wave-function) that there is a point within the annulus $r\in [aR_w, R]$ centered on $\bx_1$:
\begin{align*}
 \mathbb{P} \left( \exists k \geq 2, aR_w \leq |\bx_k - \bx_1| \leq R \right) &= 1 - \mathbb{P} \left( \forall k \geq 2,  |\bx_k - \bx_1| \leq aR_w \mbox{ or } |\bx_k - \bx_1| \geq R \right)\\
 &= 1 - \mathbb{P} \left( |\bx_2 - \bx_1| \leq aR_w \mbox{ or } |\bx_2 - \bx_1| \geq R \right) ^{N-1}\\
 &= 1 - \left(1 - \frac{4\pi}{3L^3} \left(R^3 - (aR_w)^3 \right)\right) ^{N-1}\\
 &\geq 1 - \frac{1}{1+(N-1)\frac{4\pi}{3L^3} \left(R^3 - (aR_w)^3 \right)}
\end{align*}
using $(1-x) ^{N} \leq (1+Nx)^{-1}$ for $ 0 \leq x\leq 1$ (prove by induction on $N$). Insert in~\eqref{eq:use Dyson} to estimate the $\bx_2, \ldots ,\bx_N$ integration, obtain a factor $L^3(1-2R/L)^3$ from the $\bx_1$ integration over the subcube, multiply by $N$ for the contribution of the other particles. This gives~\eqref{eq:geometry}. 

We have not explained how one can afford to estimate using the constant function, minimizing the kinetic energy. The intuition is simple: the gap in kinetic energy (in the spectrum of $-\Delta$ restricted to the box) above the constant wave-function is $\sim L^{-2}$ (cost for putting one particle in a non-constant wave-function). If the cube is small enough, this is larger than the typical interaction energy, and one can use perturbation theory to control the discrepancy between the constant wave-function and the true, interacting one. We do not give details: one uses Temple's inequality, and the error is encoded in the last factor of the second line of~\eqref{eq:small box}. Note however that this can clearly work only in a very small cube, with the constraint that it still be much larger than the scattering length and range of the potential (physically this is what sets the length scale over which the ground state varies). 
\end{proof}

Clearly we cannot take the \index{Thermodynamic limit}{thermodynamic limit} using only Lemma~\ref{lem:small box}. Nor can we fix the size of the box and take the \index{Gross-Pitaevskii limit}{GP limit} $a= N^{-1} \to 0$. The next step is thus to split a big box into many sub-boxes where one can apply~\eqref{eq:small box} efficiently. We use 

\begin{lemma}[\textbf{\index{Cell method}{Cell method}}]\label{lem:cell}\mbox{}\\
Let $\ell \leq L$ such that $L/\ell$ is an integer. Then 
\begin{equation}\label{eq:cell method}
E(N,L) \geq \min \left\{ \sum_{n\geq 0} c_n E (n,\ell) \, \big|\, c_n\geq 0, \sum_{n\geq 0} c_n = \frac{L^3}{\ell^3}, \sum_{n\geq 0} n c_n = N\right\}.  
\end{equation}
\end{lemma}

\begin{proof}
Split $\Lambda = [0,L]^3$ into $J$ cubes (cells) $\Lambda_j,j=1\ldots J$ of side-length $\ell$. Then 
$$ \Lambda ^N = \cup_{n_1, \ldots, n_J } \Lambda_1 ^{n_1} \times \ldots \times \Lambda_j ^{n_j}$$
with $\sum_{j=1}^J n_j = N$ and the union is disjoint. Denote generically $\alpha$ the multi-index $(n_1,\ldots,n_J)$ and 
$$\Lambda_\alpha ^N = \Lambda_1 ^{n_1} \times \ldots \times \Lambda_j ^{n_j}.$$
Then, with $\Psi_N$ the ground-state wave-function, $X_N = (\bx_1,\ldots,\bx_N)$, and $W$ denoting the full interaction
\begin{align*} 
E(N,L) &= \sum_{\alpha}\int_{\Lambda_\alpha ^N}  \left(|\nabla_{X_N} \Psi_N| ^2 + W |\Psi_N|^2 \right)\\
&\geq \min_{\alpha} E (\Lambda_\alpha ^N) \sum_{\alpha} \int_{\Lambda_\alpha ^N}|\Psi_N|^2  = \min_{\alpha} E (\Lambda_\alpha ^N)
\end{align*}
where we denoted $E (\Lambda_\alpha ^N)$ the lowest eigenvalue of our $N$-body Schr\"odinger operator restricted to $\Lambda_\alpha ^N$ (with no boundary conditions, hence Neumann boundary conditions, and no bosonic symmetry). Since the interaction potential is non-negative, on $\Lambda_\alpha ^N,$ 
$$ -\Delta_{X_N} \geq \sum_{j=1} ^J \left( - \Delta_{X_j} + \frac{1}{2} \sum_{ \bx_k \neq \bx_\ell \in X_j } w_a (\bx_k - \bx_\ell) \right) $$
with $X_j$ denoting the $n_j$ coordinates in $\Lambda_j$ and we dropped the interaction across cells. Hence 
$$ E (\Lambda_\alpha ^N) \geq \sum_{j=1} ^J E (n_j,\ell) $$
and we reorganize the sum  
$$ E (\Lambda_\alpha ^N) \geq \sum_{n=1} ^J c_n E (n,\ell) $$
where $c_n$ is the number of cells $\Lambda_j$ containing exactly $n$ particles (i.e. having $n_j= n$). Certainly these numbers satisfy the constraints stated in~\eqref{eq:cell method}, for the total number of particles is always $N$ and the total volume occupied by the cells, $\sum_n c_n \ell ^3$ must equal the volume of the big cube.  
\end{proof}

The advantage is that we are now free to choose the size of the cell in which we apply~\eqref{eq:small box}. Roughly, what we want is 
$$ a \ll R \ll \varrho ^{-1/3} \ll \ell \ll (\rho a) ^{-1/2}$$
for the Dyson potential to have a range much larger than the \index{Scattering length}{scattering length}, but still correspond to a dilute interaction; for each cell to contain a macroscopic number of particles, and for the gap of the kinetic energy in each cell to control the typical interaction energy within the cell. There is bit of fine tuning to be done here, the details of which we do not provide.

Note however that to apply Lemma~\ref{lem:small box} with a good error term in the small boxes of side length $\ell$, we have to make sure that essentially $n \sim \varrho \ell^3$ for the configuration minimizing the right-hand side of~\eqref{eq:cell method}, where $\varrho$ is the density in the big box. This is not for free, but can be (approximately) guaranteed using the subadditivity of the ground state energy.

\section{\index{Local density approximation (LDA)}{Local density approximation} method}\label{sec:LDA}

The estimates of the previous section are the basic input to the first derivation of Gross-Pitaevskii ground states that we present, which is also~\cite{LieSeiYng-00,LieSei-02} the first to have been obtained. This will work only in the absence of magnetic fields, for the method ignores bosonic symmetry (see~\cite{Seiringer-03} for an application of the method to systems with magnetic fields but no bosonic symmetry). 

Now we know that 
$$ e(\varrho) \sim 4\pi a \varrho$$
for the ground state energy density of an infinitely extended \index{Homogeneous Bose gas}{homogeneous Bose gas} with density $\varrho$ and scattering length $a$. The GP functional for an inhomogeneous Bose gas can then be seen as a local density approximation (LDA)
$$ \cEGP[u] \approx \int_{\R^3} \left( |\nabla u | ^2 + V |u|^2 \right) + \int_{\R^{3}} |u(\bx)|^2 e \left(|u(\bx)|^2\right)\mathrm{d}\bx. $$
In the first term we isolate the macroscopic kinetic energy and trapping energy, responsible for the overall profile of the gas. In the second term we approximate the interaction energy using the energy of the homogeneous gas locally in space, obtaining an energy density $\varrho (\bx) e (\varrho (\bx)).$

Two remarks are in order:
\begin{itemize}
 \item This is more subtle than it looks, for the energy of the homogeneous gas uses some\footnote{When the interaction is a hard-core, it is even \emph{only} made of kinetic energy.} kinetic energy. As we understood previously, this is a high-frequency component, and we have to understand how it decouples from the rest.
 \item Such an approximation is most relevant if the gas is dilute, i.e. $\beta > 1/3$ in our convention. Then one can neglect the non-local nature of the original interaction and make the LDA work.
\end{itemize}

As for the second point, we leave it to the reader to check that, indeed, the proof of the theorem below could also be used to treat the dilute regime, without magnetic field. These considerations lead us again to the Gross-Pitaevskii energy functional
$$ \cEGP[u] = \int_{\R^3} \left( |\nabla u | ^2 + V |u|^2 \right) + 4 \pi a_w\int_{\R^{3}} |u(\bx)|^4 \mathrm{d}\bx $$
with $a_w$ the scattering length of the unscaled interaction potential $w$.

\begin{theorem}[\textbf{\index{Gross-Pitaevskii limit}{Gross-Pitaevskii limit}, first statement}]\label{thm:GP first}\mbox{}\\
Let $H_N$ be the many-body Hamiltonian~\eqref{eq:intro Schro op bis} in GP scaling, 
$$ w_N (\bx) = N^2 w (N \bx)$$
with $w$ fixed, satisfying Assumption~\ref{asum:GP pot} and with scattering length $a$ (cf Theorem~\ref{thm:scat leng}). Assume~\eqref{eq:trapping s bis} and that there is no magnetic field, $\bA \equiv 0$ in~\eqref{eq:intro Schro op bis}. We have, in the $N\to \infty$ limit,

\noindent\textbf{Convergence of the energy:} 
$$
\frac{E(N)}{N} \to \EGP. 
$$

\smallskip 

\noindent\textbf{Convergence of \index{Reduced density matrix}{reduced density matrices}:} let $\Gamma_N ^{(k)},k\geq 0$ be the reduced density matrices of a many-body ground state $\Psi_N$ and $u = \uGP$ the (unique) Gross-Pitaevskii minimizer (cf Section) We have
$$
{N \choose k} ^{-1} \Gamma_N ^{(k)} \to |u ^{\otimes k} \rangle \langle u ^{\otimes k} | 
$$
strongly in trace-class norm.
Let $\rm{MF}$ stand for $\rm{NLS}$ in Theorem~\ref{thm:main}. 
\end{theorem}

\begin{proof}[Sketch of proof]
Again, a more complete proof is already reviewed in~\cite{LieSeiSolYng-05}. This combines tools from~\cite{LieSeiYng-00} (energy convergence) and~\cite{LieSei-02,LieSeiYng-02b,LieSeiYng-03b} (convergence of states). See also~\cite{LieSeiYng-01} for the 2D case. 

The energy upper bound is derived as in Section~\ref{sec:Jas-Dyson}. We turn to the 

\medskip 

\noindent\textbf{Energy lower bound.} We shall rely on Lemma~\ref{lem:Dyson 1}, and since we now deal with an inhomogeneous gas it is important to first extract the kinetic energy of the Gross-Pitaevskii minimizer. Take a normalized ground-state wave-function and write it as 
\begin{equation}\label{eq:extract}
\Psi_N (\bx_1,\ldots,\bx_N) = \prod_{j=1}^N \uGP (\bx_j) F(\bx_1,\ldots,\bx_N) 
\end{equation}
where $\uGP$ is the GP minimizer (without magnetic field, it is unique, and can be chosen strictly positive). Note that for three functions related by $\Psi = U F$, partial integration gives 
$$ \int |\nabla \Psi| ^2 = \int |U|^2 |\nabla F| ^2 + |F| ^2 |\nabla U|^2 + 2 \int U F \nabla U \cdot \nabla F = \int |U|^2 |\nabla F| ^2 - \int |F| ^2 \Delta U.$$
From the GP variational equation 
\begin{equation}\label{eq:GP eq}
 - \Delta \uGP + V \uGP + 8\pi a |\uGP|^2 \uGP = \left(\EGP + 4\pi a \int_{\R^d} |\uGP| ^4 \right) \uGP 
\end{equation}
and the above identity\footnote{This trick is extensively used when studying vortex patterns in Ginzburg-Landau and Gross-Pitaevskii theory~\cite{Aftalion-06,CorPinRouYng-11b}. In this context it originates in~\cite{LasMir-99}.} we get 
\begin{equation}\label{eq:LasMir}
\left\langle \Psi_N | H_N |\Psi_N \right\rangle = N \EGP + 4\pi a N \int_{\R^d} |\uGP| ^4 + Q (F) 
\end{equation}
with the quadratic form 
$$ 
Q (F) := \sum_{j=1} ^N \int_{\R^{3N}} \prod_{k=1}^N \left| \uGP (\bx_k)\right| ^2 \left(  |\nabla_{\bx_j} F| ^2 - 8\pi a  |\uGP(\bx_j)| ^2 |F|^2 + \frac{1}{2}\sum_{i \neq j } w(\bx_i - \bx_j) |F| ^2  \right).
$$
Now the rationale is that the weights $\left| \uGP (\bx_j)\right| ^2$ vary over a much larger length scale than the content of the big parenthesis above, for we have gotten rid of the external potential $V$ that sets the GP length scale. Thus, locally in space, $F$ will look like the minimizing configuration of the homogeneous Bose gas and then, roughly
$$ |\nabla_{\bx_j} F| ^2  + \frac{1}{2}\sum_{i \neq j} w(\bx_i - \bx_j) |F| ^2 \approx 4\pi a N^{-2} \rho_N (\bx_j)^2 $$
where $\rho_N$ is the \index{Marginal density}{marginal density} of a single particle
$$ 
\rho_N (\bx) = N \int_{\R^{3(N-1)}} |\Psi_N (\bx,\bx_2,\ldots,\bx_N)|^2 \mathrm{d}\bx_2\ldots \mathrm{d}\bx_N.
$$
Indeed we have seen in the previous section that the energy density of the homogeneous gas is $4\pi a  \rho$ and we have to apply this with our potential that has scattering length $a/N$.  If we can justify the above local density approximation, then we find  
\begin{multline*}
 Q (F) \approx \sum_{j=1} ^N 4\pi a \int_{\R^{3}} \left( N^{-2} \rho_N (\bx_j)^2 - 2 N^{-1} \rho_N (\bx_j) |\uGP(\bx_j)| ^2\right) \\ 
 \geq - \sum_{j=1} ^N 4\pi a \int_{\R^{3}} |\uGP(\bx_j)| ^4 = - 4\pi aN \int_{\R^{3}} |\uGP| ^4  
\end{multline*}
by completing the square. Inserting in~\eqref{eq:LasMir} completes the lower bound.

To make the above rigorous, one applies a variant of the \index{Cell method}{cell method} of Lemma~\ref{lem:cell} to the quadratic form $Q(F)$, reducing to this functional with variables in small boxes~\cite[Appendix~B]{Seiringer-diploma}. In these one can approximate $|\uGP|^2$ by a constant (from elliptic PDE techniques the smoothness and decay at infinity of $\uGP$ are under control), in which case bounding $Q$ from below reduces to applying~\eqref{eq:therm ener bounds}. One needs to ensure there are sufficiently many particles in each box to do that, but that can be proven to be the case, at least in a configuration optimizing the distribution of particles amongst the cells. 

\medskip 

\noindent\textbf{Convergence of states.} Since the Gross-Pitaevskii minimizer is unique, it is sufficient to prove convergence of the first reduced density matrix, as explained at the end of Section~\ref{sec:MF Onsager}. Roughly, a detailed inspection of the proof of energy convergence shows that all the kinetic energy used to control the scattering process is contained in small sets close to particle encounters. Thus we expect that $F$ (defined in~\eqref{eq:extract} by extracting the GP profile from the ground state) is almost constant. 

In the approach of~\cite{LieSei-02} (see~\cite{Seiringer-06c} for a variant) this is proved by considering the auxiliary function 
$$ f(\bx,\bX_{N-1}):= \frac{\Psi(\bx,\bX_{N-1})}{\uGP (\bx)} $$
where $\bX_{N-1} = \left( \bx_2,\ldots,\bx_N\right)$. Because we have extracted from $\Psi(\bx,\bX_{N-1})$ the GP profile, the kinetic energy of $f$ is located where $\bx$ is close to one of the points in $\bX_{N-1}$. Hence $\nabla_\bx f$ is small in $L^2$ sense outside of a set of small measure. Using a Poincar\'e inequality one can hence hope to control the deviation of $f$ from its mean. A special inequality~\cite{LieSeiYng-03b} is needed because the set outside of which we control the gradient can be irregular and disconnected. If we accept that $f(\bx,\bX_{N-1})$ is roughly constant in $\bx$ we find ($\uGP$ and $\Psi_N$ are real-valued)
\begin{align*}
 \frac{1}{N} \left\langle \uGP | \gamma_N^{(1)} | \uGP \right\rangle &= \int \uGP (\bx) \uGP (\by) \Psi_N(\bx,\bX_{N-1}) \Psi_N(\by,\bX_{N-1}) \mathrm{d}\bx \mathrm{d}\by \mathrm{d}\bX_{N-1}\\
 &= \int |\uGP (\bx)|^2 |\uGP (\by)|^2 f(\bx,\bX_{N-1}) f(\by,\bX_{N-1}) \mathrm{d}\bx \mathrm{d}\by \mathrm{d}\bX_{N-1}\\
 &\approx \int  |\uGP (\bx)|^2 |\uGP (\by)|^2 f(\bx,\bX_{N-1}) ^2  \mathrm{d}\bx \mathrm{d}\by \mathrm{d}\bX_{N-1} \\
 &= \int  |\uGP (\by)|^2 \mathrm{d}\by \int |\Psi_N (X_N)| ^2  \mathrm{d}\bX_{N} = 1,
\end{align*}
which implies 
$$ N^{-1} \gamma_N^{(1)} \to |\uGP \rangle \langle \uGP | $$ 
in operator norm, and hence in trace norm because the limit is rank one.
\end{proof}

\section{\index{Coherent states}{Coherent states}/de \index{Quantum de Finetti theorem}{Finetti} method reloaded}\label{sec:coh reloaded}

In the previous section we have derived GP ground states from many-body quantum mechanics with a special method (\index{Local density approximation (LDA)}{local density approximation}) based on the diluteness of the gas. Now we connect to more general methods, namely those of Sections~\ref{sec:deF loc} and~\ref{sec:coherent}. We have already explained in Chapter~\ref{cha:dilute} how to extend these approaches to the dilute regime. The \index{Dyson lemma}{Dyson lemmas} of Section~\ref{sec:Dyson} will turn the GP limit into a ``dilute regime plus extra error terms'', and what we need now is explain how to control those error terms, mostly due to spurious three-body terms from the application of Lemma~\ref{lem:Dyson 2}. Unlike the previous two sections, we keep track of bosonic symmetry all along, and thus provide a derivation that works even when external magnetic fields are turned on. Recall the expression of the \index{Gross-Pitaevskii energy functional}{Gross-Pitaevskii energy functional}
$$ \cEGP[u] = \int_{\R^3} \left( \left|\left(-\im \nabla + \bA \right)u \right| ^2 + V |u|^2 \right) + 4 \pi a\int_{\R^{3}} |u(\bx)|^4 \mathrm{d}\bx $$

\begin{theorem}[\textbf{\index{Gross-Pitaevskii limit}{Gross-Pitaevskii limit}, full statement}]\label{thm:GP second}\mbox{}\\
Let $H_N$ be the many-body Hamiltonian~\eqref{eq:intro Schro op bis} in GP scaling, 
$$ w_N (\bx) = N^2 w (N \bx)$$
with $w$ fixed, satisfying Assumption~\ref{asum:GP pot} and with \index{Scattering length}{scattering length} $a$. Assume~\eqref{eq:trapping s bis} and~\eqref{eq:growth A}. We have, in the limit $N \to + \infty$:

\smallskip 

\noindent\textbf{Convergence of the energy:} 
$$
\frac{E(N)}{N} \to \EGP. 
$$

\smallskip 

\noindent\textbf{Convergence of reduced density matrices:} let $\Gamma_N ^{(k)},k\geq 0$ be the reduced density matrices of a many-body ground state $\Psi_N$. There exists a Borel probability measure $\mu$ on $\MGP$ (the set of Gross-Pitaevskii ground states) such that, along a subsequence, 
$$
{N \choose k} ^{-1} \Gamma_N ^{(k)} \to \int_{\MGP} |u ^{\otimes k} \rangle \langle u ^{\otimes k} | \mathrm{d}\mu(u) 
$$
strongly in trace-class norm.
\end{theorem}

We are mostly after an efficient lower bound to the ground state energy. The idea is to use Corollary~\ref{cor:Dyson 2} to turn the GP limit in a dilute limit with a potential whose integral encodes the scattering length, and which is in a scaling $R \sim N^{-\beta} \gg N^{-2/3}$. This permits the use of the methods of Chapter~\ref{cha:dilute} on this problem. We will have to discard unwanted terms from the \index{Dyson lemma}{Dyson-like lower bound} and, as hinted at at the end of Section~\ref{sec:Dyson}, this requires tacking $R \sim N^{-\beta} \ll N^{-1/3}$ i.e. we cannot reduce the GP singularity all the way down to a mean-field one.  

For convenience we always assume~\eqref{eq:h dom Lap} in the rest of this section, as we did in Section~\ref{sec:mom rep}. This can be bypassed by ad-hoc arguments, see the original references we will cite. 

\subsection{Take 1}

Let us explain first how the three body term in~\eqref{eq:Dyson 2 app 2} is dealt with in~\cite{LieSei-06}. Applying Corollary~\ref{cor:Dyson 2} we find that the energy is bounded from below as 
\begin{equation}\label{eq:LieSei low}
E (N) \geq E^{\rm Dys}_1 (N) - C N^3 \left\langle \Psi_N | U_R (\bx_1 - \bx_2) \left( 1- \Theta_{2R} (\bx_2-\bx_3) \right) | \Psi_N \right\rangle  - C \frac{N^2 R^2 K^5}{\eps} 
\end{equation}
where $E^{\rm Dys}_1 (N)$ is the lowest eigenvalue of the ``\index{Dyson Hamiltonian}{Dyson Hamiltonian}''
\begin{equation}\label{eq:Dyson hamil 1}
\sum_{j=1} ^N \left(h_{\bx_j} - (1-\eps) \chi_K (\bp_j) (-\Delta_{\bx_j}) \chi_K (\bp_j) \right) + (1-\eps)^2 \sum_{i\neq j } U_R (\bx_i - \bx_j) 
\end{equation}
acting on bosonic functions and $\Psi_N$ is any ground state of the original Hamiltonian $H_N$. The new potential $U_R$ has integral $4 \pi a N^{-1}$ and a range $R$ which will play the role of $N^{-\beta}$ in the notation of the rest of the notes. The point of the Dyson lemma is that $R\gg N^{-1}$, i.e. we have turned the GP limit in a dilute one.

An adaptation of the methods of Chapter~\ref{cha:dilute} yields the lower bound 
$$ \liminf_{K\to \infty} \liminf_{\eps \to 0}\liminf_{N\to \infty} \frac{E^{\rm Dys}_1 (N)}{N} \geq \EGP$$
provided $R\gg N^{-2/3}$ (which is $\beta < 2/3$ in the notation of the rest of the notes, the range of applicability of Theorems~\ref{thm:dilute 3D} and~\ref{thm:dilute}). We do not comment about how the limits $\eps \to 0$ and $K\to \infty$ are disposed of. This is in any event much simpler than what we discussed so far, and we can afford to take these limits \emph{after} $N\to \infty$. Observe then that the last term in~\eqref{eq:LieSei low} is $o(N)$ provided $R \ll N^{-1/2}$, certainly compatible with $R \gg N^{-2/3}$. 

The crucial point is now to control the second term of~\eqref{eq:LieSei low} and prove it is $o(N)$ provided $R$ is not too large. The key lemma is 

\begin{lemma}[\textbf{Three particle expectations in ground states}]\label{lem:LieSei 3}\mbox{}\\
Let $P_N$ be the orthogonal projector onto the ground eigenspace of $H_N$. Let $\xi : \R^9 \mapsto \R^+$ be an arbitrary positive function, identified with the corresponding multiplication operator on the three-body space~$L^2 (\R^9)$ . Then, for any $\alpha >0$ 
\begin{multline}\label{eq:Lie Sei 3}
\tr \left( \xi (\bx_1,\bx_2,\bx_3) P_N \right) \leq\\
C \alpha^{-6/2} \sup_{\bx_1 \in \R^3} \left(\int_{\R^6} \xi (\bx_1,\bx_2,\bx_3) \mathrm{d}\bx_2 \mathrm{d}\bx_3 \right) \exp\left(\alpha E(N-3) - E (N) \right).  
\end{multline}
\end{lemma}

\begin{proof}[Comments]
We do not discuss the proof of the above, which heavily relies on path integrals ideas (the \index{Trotter product formula}{Trotter} and the \index{Feynman-Kac formula}{Feynman-Kac-It\^o} formulas). It can be found in~\cite[Lemma~2]{LieSei-06}. The reason why we get $P_N$ instead of $|\Psi_N \rangle \langle \Psi_N|$ for a single ground state in the right-hand side of~\eqref{eq:Lie Sei 3} is that one starts from estimates in positive temperature equilibria of $H_N$, and then lets the temperature go to zero. The \index{Boltzmann-Gibbs ensemble}{Gibbs state} then converges to $P_N$ (all ground states are equally likely in this limit). Of course if the ground state is unique, this is irrelevant. 

The original inequality uses the \index{Heat kernel}{heat kernel} at time (or inverse temperature) $\alpha$ of ($V$ being the one-body potential in $H_N$)
$$ h_3 = \sum_{j=1} ^3 \left( - \Delta_{\bx_j} + V (\bx_j) \right),$$
whose kernel we identify with $e^{-\alpha h_3}$. Then the first term (involving $\xi$) in the right-hand side of~\eqref{eq:Lie Sei 3} is replaced by $\Lambda_{\alpha}$, the operator norm of the map $L^2 (\R^9) \mapsto L^2 (\R^9)$ whose integral kernel is 
$$K_{\alpha,\xi} \left( \bX;\bY \right) := \sqrt{\xi\left( \bX \right)} \left( e^{-\alpha h_3} \left( \bX; \bY \right) \right) \sqrt{\xi(\bY)}$$
Thus $\Lambda_\alpha$ is in fact the norm of the  heat flow 
$$ \dd_t u_3 = h_ 3 u_3$$
at time $\alpha$ (whose Green function is $e^{-\alpha h_3}$), seen as an operator from $L^2 ( \R^9, \sqrt{\xi (\bX)}^{-1} \mathrm{d}\bx)$ to itself. 

To see that the lemma we stated follows from this original formulation, observe that the Feynman-Kac-Trotter formula~\cite{Simon-05} implies (similarly as in the discussion around~\eqref{eq:Trotter}) that for $V\geq 0$ 
$$ e^{-\alpha h_3} \left( \bx_1,\bx_2,\bx_3; \by_1,\by_2,\by_3 \right) \leq k_\alpha(\bx_1;\by_1) k_\alpha(\bx_2;\by_2) k_\alpha(\bx_3;\by_3)$$
with $k_\alpha$ the heat kernel at time $\alpha$ of the free Laplacian on $\R^3$. The latter is explicitly known~\eqref{eq:heat kernel}, but we only use that $k_\alpha(\bx;\by) \leq C \alpha^{-3/2}.$ Inserting in the above we find 
$$
\left| \int_{\bX,\bY \in \R^9} \overline{f (\bX)} K_{\alpha,\xi} \left( \bX;\bY\right) f (\bY) \mathrm{d}\bY \mathrm{d}\bx\right| \leq C \alpha^{-6/2} \int_{\bx,\by \in \R^3} |F_\xi (\bx)| k_\alpha (\bx;\by) |F_\xi (\by)| \mathrm{d}\bx \mathrm{d}\by
$$
with 
$$
F_\xi (\bx) := \int_{\bx_2,\bx_3 \in \R^2} f(\bx,\bx_2,\bx_3) \sqrt{\xi(\bx,\bx_2,\bx_3)}.
$$
But the heat kernel $k_\alpha$ is bounded from $L^2$ to $L^2$, with bound $1$, so 
\begin{multline*}
\left|\int_{\bX,\bY \in \R^9} \overline{f (\bX)} K_{\alpha,\xi} \left( \bX;\bY\right) f (\bY) d\bY \mathrm{d}\bx\right| \leq C \alpha^{-6/2} \int_{\R^3} F_\xi (\bx) ^2 \mathrm{d}\bx \\
\leq C \alpha^{-6/2} \int\left(\int \xi(\bx_1,\bx_2,\bx_3) \mathrm{d}\bx_2 \mathrm{d}\bx_3 \right) \left( \int f^2(\bx_1,\bx_2,\bx_3) \mathrm{d}\bx_2 \mathrm{d}\bx_3 \right)  \mathrm{d}\bx_1 
\end{multline*}
where we used Cauchy-Schwarz to bound $F_\xi$ from above pointwise. This being so for any $f\in L^2 (\R^9)$ proves that 
$$ \Lambda_\alpha \leq C \alpha^{-6/2} \sup_{\bx_1 \in \R^3} \int_{\R^6} \xi (\bx_1,\bx_2,\bx_3) \mathrm{d}\bx_2 \mathrm{d}\bx_3,$$
as used in~\eqref{eq:Lie Sei 3}.
\end{proof}

To conclude the proof of the energy lower bound, one first has to ensure that $E(N-3) - E (N)$, which appears in the right-hand side of~\eqref{eq:Lie Sei 3}, stays finite when $N\to \infty$. This is intuitively true (this is the energy gain for removing three particles from the system). An efficient a priori bound is not obvious but can be circumvented by appropriate arguments that we do not reproduce. Let us thus assume that $E(N-3) - E (N)$ is bounded independently of $N$ and see how the above lemma allows to dispose of the three-body term in~\eqref{eq:LieSei low}.  

Averaging~\eqref{eq:LieSei low} with respect to $|\Psi_N\rangle \langle \Psi_N|$ we can freely replace 
$$ \left\langle \Psi_N | U_R (\bx_1 - \bx_2) \left( 1- \Theta_{2R} (\bx_2-\bx_3) \right) | \Psi_N \right\rangle$$
by 
$$ 
\tr\left( U_R (\bx_1 - \bx_2) \left( 1- \Theta_{2R} (\bx_2-\bx_3) \right) P_N \right), 
$$
which is, in view of the above, bounded by a constant (depending on $\alpha >0$) times 
$$ \sup_{\bx_1} \int_{\R^6} U_R (\bx_1- \bx_2 ) \left( 1- \Theta_{2R} (\bx_2-\bx_3) \right)  \mathrm{d}\bx_2 \mathrm{d}\bx_3 \leq C R^3 N^{-1}.$$
Here we performed first the integral in $\bx_3$ to obtain the factor $R^{3}$ (independently of $\bx_1,\bx_2$), then the integral in $\bx_2$, using that $\int_{\R^3} U_R$ is bounded by $N^{-1}$. 

Thus the second term in~\eqref{eq:LieSei low} is bounded by $C N^2 R^3 \ll N$ if we choose $R \ll N^{-1/3}$, the typical inter-particle distance. We are at liberty to do so, for this is compatible with the requirement $\beta < 2/3$ in Chapter~\ref{cha:dilute}. Hence our sketch of the energy lower bound is complete. 

To deduce convergence of density matrices, one can apply all this machinery to a perturbed Hamiltonian and use a \index{Feynman-Hellmann principle}{Feynman-Hellmann} argument as mentioned in Section~\ref{sec:coh proof}. This is a bit long and strictly speaking has been considered only for the first density matrix, see~\cite{LieSei-06}. The convergence of density matrices is more transparent with the de Finetti method, as we shall discuss at the end of the next section. 

\subsection{Take 2}

Our second way of dealing with the three-body term is taken from~\cite{NamRouSei-15}. We replace~\eqref{eq:LieSei low} by   
\begin{equation}\label{eq:NamRouSei low}
E (N) \geq E^{\rm Dys}_2 (N) - C \frac{N^2 R^2 K^5}{\eps} 
\end{equation}
where $E^{\rm Dys}_2 (N)$ is the lowest eigenvalue of the ``\index{Dyson Hamiltonian}{Dyson Hamiltonian}''
\begin{multline}\label{eq:Dyson hamil 2}
\Htilde_N := \sum_{j= 1} ^N \left(h_{\bx_j} - (1-\eps) \chi_K (\bp_j) (-\Delta_{\bx_j}) \chi_K (\bp_j) \right) \\+ (1-\eps)^2 \sum_{i\neq j } U_R (\bx_i - \bx_j) \prod_{k \neq i, j} \Theta_{2R} (\bx_j-\bx_k) 
\end{multline}
acting on bosonic functions. Compared with the previous section we have not yet used~\eqref{eq:Dyson 2 app 2}. We do it now: 
\begin{equation}\label{eq:NamRouSei low 2}
E ^{\rm Dys}_2 (N) \geq E ^{\rm Dys}_1 (N) - C N^3 \left\langle \Psit_N | U_R (\bx_1 - \bx_2) \left( 1- \Theta_{2R} (\bx_2-\bx_3) \right) | \Psit_N \right\rangle
\end{equation}
where $E ^{\rm Dys}_1 (N)$ is the lowest eigenvalue of~\eqref{eq:Dyson hamil 1} and $\Psit_N$ a ground state of~\eqref{eq:Dyson hamil 2}. We can estimate $E ^{\rm Dys}_1 (N)$ as in Chapter~\ref{cha:dilute} provided $R \gg N^{-2/3}$, as already explained. The main difference with the approach in~\eqref{eq:LieSei low} is that we have to bound the expectation of the three-body term in a ground state $\Psit_N$ of~\eqref{eq:Dyson hamil 2}. Why is that useful ? We shall make two observations:
\begin{itemize}
 \item Such a ground state $\Psit_N$ satisfies a \index{Moments estimates}{second moment estimate} akin to Lemma~\ref{lem:mom rep}. 
 \item The three-body term is controlled by the second moment of the kinetic energy.
\end{itemize}
Here it is important to (a) use the Dyson lemma to reduce the singularity of the interaction and (b) consider a ground state of~\eqref{eq:Dyson hamil 2} and not~\eqref{eq:Dyson hamil 1} directly. Indeed, we \emph{do not expect} that a second moment estimate can hold for a true ground state of the original Hamiltonian. We do expect that it holds for a ground state of~\eqref{eq:Dyson hamil 1}, but that would be harder to prove, because~\eqref{eq:Dyson hamil 1} contains attractive terms, unlike~\eqref{eq:Dyson hamil 2} (cf Section~\ref{sec:mom att}).  

We refer to~\cite{NamRouSei-15} for more details and state these two observations as lemmas, whose proofs are (lengthy) variations on the considerations of Sections~\ref{sec:mom}-\ref{sec:mom rep}. Essentially you should think that~\eqref{eq:Dyson hamil 2} is not much different from a bona fine Hamiltonian with pair interactions, in a \index{Dilute limit}{dilute scaling}. Computing with it is of course harder because of the cut-off killing the pair interaction whenever three particles sit at the same place.

Denote 
$$ \htilde := h - (1-\eps) \chi_K (\bp) (-\Delta) \chi_K (\bp)$$
the one-body Hamiltonian appearing in~\eqref{eq:Dyson hamil 2}. 

\begin{lemma}[\textbf{Second moment estimate for Dyson's Hamiltonian}]\label{lem:mom Dys}\mbox{}\\
For any $\eps,K>0$ in Corollary~\ref{cor:Dyson 2} and $R$ satisfying $R\gg N^{-2/3}$ when $N\to \infty$ we have 
$$ \Htilde_N ^2 \geq \frac{1}{3} \left(\sum_{j=1} ^N \htilde_{\bx_j} \right) ^2$$
as operators, where $\Htilde_N$ is as in~\eqref{eq:Dyson hamil 2}.
\end{lemma}

Clearly this implies that, for a ground state $\Psit_N,$ 
$$ \left\langle \Psit_N | \htilde_{\bx_1} \htilde_{\bx_2} | \Psit_N \right\rangle \leq C$$
independently of $N$ when $N\to \infty$, a convenient variant of~\eqref{eq:2mom rep}. 

The above is a crucial ingredient in the proof of 

\begin{lemma}[\textbf{Three particles expectations in ground states, again}]\label{lem:NamRouSei}\mbox{}\\
With the same assumptions and notation as above 
\begin{equation}\label{eq:Dyson hamil 3}
 \Htilde_N \geq \sum_{j= 1} ^N \htilde_{\bx_j} + (1-\eps)^2 \sum_{i\neq j } U_R (\bx_i - \bx_j) - C_{\eps,K} \frac{R^2}{N^2} \Htilde_N ^4. 
\end{equation}
\end{lemma}

This is stated as a bound on the full Hamiltonian~\eqref{eq:Dyson hamil 2} for conciseness, but this is really obtained by first reducing to~\eqref{eq:Dyson hamil 1} using~\eqref{eq:Dyson 2 app 2}, and then estimating the three-body term. 

With this at our disposal, we are left with studying the Hamiltonian on the right-hand side of~\eqref{eq:Dyson hamil 3}, which follows the lines of Chapter~\ref{cha:dilute} (remember that we take $R\gg N^{-2/3}$, which means $\beta < 2/3$). The extra error term in~\eqref{eq:Dyson hamil 3} is easily discarded: evaluated in a ground state of $\Htilde_N$
$$ \Htilde_N \Psit_N = E^{\rm Dys}_2 (N) \Psit_N$$
it is bounded by a constant times $N^2 R^2 $ (because we easily have $E^{\rm Dys}_2 (N) \leq CN$). This is $o(N)$ as needed provided $R \ll N^{-1/2}$, which is compatible with our other desiderata. 

\medskip

We are now done with our sketch of the proof of the energy lower bound. Let us say a few words of the convergence of reduced density matrices. The \index{Quantum de Finetti theorem}{quantum de Finetti theorem} is very handy to avoid a non-trivial bit of convex analysis (or, better said, the convex analysis part is included in the quantum de Finetti theorem).  

Note that the energy convergence directly gives information on a ground state $\Psit_N$ of~\eqref{eq:Dyson hamil 2}, which is not we are after. The way out is to apply all the above to a modified Hamiltonian (still acting on bosonic functions)
$$ H_{N,v,\ell} := H_N - \frac{\ell!}{N^{\ell - 1}} \sum_{1 \leq i_1 < \ldots < i_\ell \leq N}|v^{\otimes \ell} \rangle \langle v ^{\otimes \ell}|_{i_1,\ldots,i_\ell}$$
for $v \in L^2$ and $\ell \in \N$. The perturbation is a nice, bounded operator, so its inclusion destroys none of the methods we used, and we get the lower bound 
$$ \frac{E (N,v,\ell)}{N} \geq \inf_{\norm{u}_{L^2 = 1}} \left( \cEGP [u] -| \langle v | u \rangle | ^{2\ell} \right)   - o (1)$$
for the lowest eigenvalue of $H_{N,v,\ell}$.

Now, for a ground state $\Psi_N$ of the original Hamiltonian $H_N$ and its $\ell$-th \index{Reduced density matrix}{reduced density matrix} $\Gamma_N ^{(\ell)}$ 
\begin{align*}
{N \choose \ell} ^{-1}  \left\langle v^{\otimes \ell} | \Gamma_N ^{(\ell)} |  v^{\otimes \ell} \right\rangle &= N^{-1} \left(\left\langle \Psi_N | H_N | \Psi_N \right\rangle - \left\langle \Psi_N | H_{N,v,\ell} | \Psi_N \right\rangle \right) \\
&\leq N^{-1} \left(E (N) - E (N,v,\ell)\right) 
\end{align*}
Passing to the limit $N\to \infty$ we deduce from energy estimates that 
$$ \int | \langle u | v \rangle | ^{2\ell} \mathrm{d}\mu (u) \leq \EGP - \inf_{\norm{u}_{L^2 = 1}} \left( \cEGP [u] -| \langle v | u \rangle | ^{2\ell} | \right)$$
where $\mu$ is the de Finetti measure of the sequence $(\Psi_N)_N$, cf Theorem~\ref{thm:DeFinetti fort}. Fix $v$ with unit $L^2$ norm and apply this to $\lambda v$ for small $\lambda$. Simple perturbative arguments yield that a minimizer of   
$$ \cEGP [u] - \lambda ^{2\ell}| \langle v | u \rangle | ^{2\ell} $$
must, modulo subsequence, converge to a GP minimizer when $\lambda \to 0$. Taking this limit along all subsequences we find 
$$ \lambda ^{2\ell} \int | \langle u | v \rangle | ^{2\ell} \mathrm{d}\mu (u) \leq \lambda ^{2\ell}\sup_{u\in \MGP} | \langle v | u \rangle | ^{2\ell} + o(\lambda^{2\ell})$$
with $\MGP$ the set of all GP minimizers. Hence  
$$ \int | \langle u | v \rangle | ^{2\ell} \mathrm{d}\mu (u) \leq \left( \sup_{u\in \MGP}\left| \langle v | u \rangle \right| \right) ^{2\ell}$$
for all normalized $v\in L^2 (\R^d)$. This implies that the de Finetti measure $\mu$ is concentrated on $\MGP$. To get a feel as to why, consider the possibility that $\mu$ has a wrongly placed atom, i.e. assigns non-zero mass $m$ to a $v$ at finite $L^2$ distance from $\MGP$. Then, for this choice of $v$, the sup on the right-hand side is $<1$, whereas the left-hand side is bounded below by $m$. Taking $\ell \to \infty$ leads to a contradiction. Details for the general case are in~\cite[Section~4.3]{NamRouSei-15}.

\section{\index{Bogoliubov theory}{Bogoliubov methods} for GP ground states}\label{sec:Bogoliubov low}

We finally present two alternative methods~\cite{BocBreCenSch-18b,BocBreCenSch-18c,Fournais-20,Hainzl-20,NamNapRicTri-20} to deal with the GP limit. We shall not do them justice because we present them as ways of obtaining the

\begin{theorem}[\textbf{Gross-Pitaevskii limit, partial statement}]\label{thm:GP small}\mbox{}\\
Let $H_N$ be the many-body Hamiltonian~\eqref{eq:intro Schro op bis} in GP scaling, 
$$ w_N (\bx) = \lambda N^2 w (N \bx)$$
with $w$ fixed, satisfying Assumption~\ref{asum:GP pot} and with \index{Scattering length}{scattering length} $a$. Assume~\eqref{eq:trapping s bis} and that there is no magnetic field, $\bA \equiv 0$. Let the \index{Gross-Pitaevskii energy functional}{Gross-Pitaevskii functional} be 
$$ \cEGP[u] = \int_{\R^3} \left( |\nabla u | ^2 + V |u|^2 \right) + 4 \pi a\int_{\R^{3}} |u(\bx)|^4 \mathrm{d}\bx $$
with $a$ the scattering length of $w$.

There exists a $\lambda_0$ such that,  if $0< \lambda \leq \lambda_0$, we have, in the limit $N \to + \infty$:

\smallskip 

\noindent\textbf{Convergence of the energy:} 
$$
\frac{E(N)}{N} \to \EGP. 
$$

\smallskip 

\noindent\textbf{Convergence of reduced density matrices:} let $\Gamma_N ^{(1)}$ be the reduced density matrices of a many-body ground state $\Psi_N$ and $u= \uGP$ the (unique) minimizer of $\cEGP$. Along a subsequence, 
$$
{N \choose k} ^{-1} \Gamma_N ^{(k)} \to |u ^{\otimes k} \rangle \langle u ^{\otimes k} |  
$$
strongly in trace-class norm.
\end{theorem}

Here we assume some smallness of the unscaled potential $w$, which is not needed with the previous methods. Let me comment on this choice:
\begin{itemize}
 \item One significant aspect of the methods of~\cite{BocBreCenSch-18b,BocBreCenSch-18c,Hainzl-20,NamNapRicTri-20} is that they provide an alternative method to the \index{Dyson lemma}{Dyson lemma} to extract pair correlations. This works only for small $\lambda$, but is conceptually very interesting.
 \item The actual goal of the papers we cited is to obtain an optimal rate for the convergence of the first density matrix. Even with small $\lambda$ this is not something we can obtain with the previous methods.
 \item Actually, by using the previous results as starting point, one could bootstrap the arguments and obtain BEC with an optimal rate without assuming small~$\lambda$. So far this has been worked out only for the homogeneous Bose gas~\cite{BocBreCenSch-18b}. And then one has to use the Dyson lemma at least once.
 \item More recently, a full alternative to the Dyson lemma that works without assuming small $\lambda$ was proposed in~\cite{AdhBreSch-20}.
 \item We shall encounter arguments that are crucial to the rigorous derivation of Bogoliubov's theory for the excitation spectrum~\cite{Seiringer-11,GreSei-13,LewNamSerSol-13,DerNap-13,BocBreCenSch-17,BocBreCenSch-18,BriFouSol-19,BriSol-19,FouSol-19}. I chose not to discuss this topic in full details, but some introduction to the topic is worthwhile.
\end{itemize}

We will now sketch in two subsections some of the methods of~\cite{BocBreCenSch-18b,BocBreCenSch-18c} and~\cite{NamNapRicTri-20} for obtaining energy lower bounds in the GP regime. This treatment parallels the two subsections of Section~\ref{sec:Bogoliubov}. 

\subsection{Conjugating with the \index{Correlation map}{correlation map}}

In the first part of Section~\ref{sec:Bogoliubov}, appropriate correlations have been added to a condensed state by acting with the unitary operator from Definition~\ref{def:Bog trans}. It is very tempting then to un-act (act with the adjoint) on the true many-body ground state, in the hope that this will extract the correlations, and allow to treat the rest of the state in a mean-field like fashion. This is one of the main ideas of~\cite{BocBreCenSch-18b,BocBreCenSch-18c}, that we briefly sketch now.

What we discuss here probably can be adapted to inhomogeneous systems, but to follow~\cite{BocBreCenSch-18b,BocBreCenSch-18c,Hainzl-20} we restrict to the \index{Homogeneous Bose gas}{homogeneous Bose gas} in a fixed periodic box. As in Section~\ref{sec:thermo} we thus replace $\R^3$ by the unit torus $\mathbb{T}^3$ and set the external potential $V\equiv 0$. The $N$-particles Hamiltonian is then 
\begin{equation}\label{eq:hamil hom}
 H_N := \sum_{j= 1} ^N -\Delta_{\bx_j} + \lambda N^2 \sum_{1\leq i < j \leq N} w (N(\bx_i - \bx_j))
\end{equation}
with $w$ fixed as in Assumption~\ref{asum:GP pot}, with scattering length $a>0$. As in Theorem~\ref{thm:GP small} the coupling constant $\lambda>0$ will ultimately be small enough but fixed. By momentum conservation ($H_N$ commutes with translations) we can write this in \index{Second quantization}{second-quantized} form as 
\begin{equation}\label{eq:hamil hom 2}
H_N = \sum_{\bp \in (2 \pi \Z) ^3} |\bp| ^2 a^\dagger_ p a_p + \frac{\lambda}{2N} \sum_{\bk,\bp,\bq\in (2\pi \Z)^3} \widehat{w} \left(\frac{\bk}{N} \right) a ^{\dagger} _{\bp + \bk} a^\dagger _{\bq-\bk} a_\bp a_\bq 
\end{equation}
where $a^\dagger_\bp, a_\bp$ \index{Creation and annihilation operators}{create/annihilate} a particle in the plane-wave state $\bx \mapsto e^{\im \bp \cdot \bx}$.

What was done in Section~\ref{sec:Bogoliubov} was to conjugate the above by a \index{Weyl operator}{Weyl operator} removing the condensate part (here, particles in the zero momentum mode $\bp = 0$, aka the constant function), and with a \index{Bogoliubov transformation}{Bogoliubov transformation} to remove correlations. Applying the so-obtained operator to the vacuum led to a nice energy upper bound. To obtain a lower bound, it can be convenient to do the conjugations in such a way that the resulting operator only talks to the excited particles. For that purpose the Weyl operator from Definition~\ref{def:Weyl} is replaced by the \index{Excitation map}{excitation map} from Definition~\ref{def:excit map}, and associated calculations are not much more difficult, we merely replace~\eqref{eq:Weyl shift} by~\eqref{eq:excit act 1}-\eqref{eq:excit act 2}. 

More difficult is the replacement of the Bogoliubov transformation, because we want it to map the truncated \index{Fock space}{Fock space} $\gF ^{\leq N}(\gH^\perp)$ from Definition~\ref{def:excit map} to itself, where $\gH^\perp$ is the orthogonal of the constant function. The following originates from~\cite{BreSch-19}:

\begin{definition}[\textbf{Generalized Bogoliubov transformation}]\label{def:Bog gen}\mbox{}\\
Let the function $\eta(\bp)$ be defined by
$$ \eta (\bp) := \frac{1}{N^2} \widehat{f - 1} \left( \frac{\bp}{N}\right) $$
with $f$ the zero-energy \index{Scattering solution}{scattering solution} associated with $w$ (Theorem~\ref{thm:scat leng}). Let the modified creators/annihilators of excited particles be
\begin{equation}\label{eq:modif ops}
b ^\dagger _\bp := a^\dagger_\bp \sqrt{\frac{N-\cN^\perp}{N}}, \quad b_\bp := \sqrt{\frac{N-\cN^\perp}{N}}a_\bp
\end{equation}
with $\cN^\perp = \sum_{\bp \neq 0} a^{\dagger}_\bp a_\bp$ the \index{Number operator}{number} (operator) of excited particles. 

The generalized Bogoliubov transformation is the unitary map from $\gF^{\leq N}(\gH^\perp)$ to itself
\begin{equation}\label{eq:Bog tran gen}
\cT_b := \exp\left( \sum_{\bp \neq 0} \eta (\bp) \left(b^{\dagger} _\bp b^{\dagger} _{-\bp} - b^{\dagger} _\bp b^{\dagger} _{-\bp}\right)\right) 
\end{equation}
where $\gH^{\perp} = \mathrm{span} \left( e^{\im \bp \cdot \bx}\right)_{\bp \neq 0} $.
\end{definition}

\begin{proof}[Comments]
Compared to Definition~\ref{def:bog trial} we now work in Fourier variables and have replaced the original creators/annihilators $a^\dagger_\bp,a_{\bp}$ with~\eqref{eq:modif ops}. The latter have the virtue of preserving the particle number and the excited Fock space $\gF^{\leq N}(\gH^\perp)$. Heuristically you should think that on the full Fock space (cf the discussion around Definition~\ref{def:excit map})
$$
b ^\dagger _\bp :=  a^\dagger_\bp \frac{a_0}{\sqrt{N}} , \quad b_\bp := \frac{a^\dagger_0}{\sqrt{N}} a_\bp .
$$
We expect most particles to be condensed and hence $a^\dagger_0,a_0 \sim \sqrt{N}$, which means that $b^{\dagger}_\bp,b_\bp$ almost satisfy the \index{Canonical commutation relations (CCR)}{CCR}~\eqref{eq:CCR}. They do not satisfy it exactly however, and dealing with the remainders is sometimes tedious. The idea can be traced back at least to~\cite{LieSol-01,LieSol-04,Solovej-06} and has been used repeatedly~\cite{GiuSei-09,Seiringer-11,GreSei-13}. What we need here is a very strong version of this ``approximate CCR'', for we want to mimic~\eqref{eq:GP Bog rot} as closely as possible: 
\begin{equation}\label{eq:approx CCR}
 \cT_b ^* a_\bp \cT_b \simeq \cosh (\eta (\bp)) a_\bp + \sinh(\eta (\bp)) a^\dagger_{-\bp} 
\end{equation}
in order to conjugate the Hamiltonian with $\cT_b$. In~\cite[Lemma~3.4]{BocBreCenSch-18b}, an estimate on the difference between the two sides of~\eqref{eq:approx CCR} is derived, giving efficient bounds when $\cT_b$ is as in~\eqref{eq:Bog tran gen} with $\norm{\eta}_{L^2}$ small. This is the case when the map $\cT_b$ is meant to excite few particles out of the condensate.  
\end{proof}

The core of the proof of Theorem~\ref{thm:GP small} is given in the following~\cite{BocBreCenSch-18b,BocBreCenSch-18c}

\begin{theorem}[\textbf{Excitation Hamiltonian, correlations removed}]\label{thm:exc hamil}\mbox{}\\
With the previous notation, let $\mathcal{G}_N: \gF^{\leq N} (\gH^\perp) \mapsto \gF^{\leq N} (\gH^\perp)$ be defined as 
$$ \mathcal{G}_N := \cT_b^* \cU_N H_N \cU_N ^* \cT_b.$$
It is decomposed as 
\begin{equation}\label{eq:split excit}
 \cG_N = 4\lambda \pi a N + \mathcal{K} + \mathcal{W}  + \mathcal{E}  
\end{equation}
where $\mathcal{K}$ and $\mathcal{W}$ are respectively the kinetic and interaction energies of the excitations
$$
 \mathcal{K} = \sum_{0\neq \bp \in (2\pi \Z)^3} |p|^2 a^\dagger_\bp a_\bp,  \quad \mathcal{W} = \frac{\lambda}{2N} \sum_{\bp,\bq,\bp + \bk, \bq-\bk \neq 0} \widehat{w} \left(\frac{\bk}{N} \right) a ^{\dagger} _{\bp + \bk} a^\dagger _{\bq-\bk} a_\bp a_\bq 
$$
and, as an operator, 
\begin{equation}\label{eq:error BBCS}
 \pm \mathcal{E} \leq \lambda \left( \mathcal{K} + \mathcal{W} \right) + C_\lambda 
\end{equation}
with $C_\lambda$ a constant, only depending on $\lambda$.
\end{theorem}

\begin{proof}[Comments]
This is a simplified statement, as compared to the full result. In fact $\mathcal{E}$ contains Bogoliubov pairing terms that are physically relevant, so not all of it has to be discarded if one wants to continue in the direction of~\cite{BocBreCenSch-18}. It contains also a term due to interactions of triplets of excited particles, which is the main difficulty in using the above when $\lambda$ is not sufficiently small for~\eqref{eq:error BBCS} to be efficient. This term can ultimately be taken care of as in~\cite{BocBreCenSch-18} by using the Dyson lemma and another unitary transformation, this time of the form of the exponential of a cubic operator in creators/annihilators. In~\cite{AdhBreSch-20} a unitary of the form of the exponential of a quartic operator in creators/annihilators is used to deal with interactions of quadruplets of excited particles. All this goes beyond the scope of this review.
\end{proof}

Using~\eqref{eq:error BBCS} the proof of the energy lower bound is complete, for small $\lambda$, because $\cK + \cW$ is a positive operator and the constant $4\pi a N$ in~\eqref{eq:split excit} is none other than the GP energy of the homogeneous Bose gas in a fixed torus. By unitary equivalence with $H_N$ we even deduce, for $\lambda$ small enough  
$$ H_N \geq N \EGP + c \cN ^\perp - C$$
with $c,C>0$, and this can be used to control the condensation rate (difference between $N^{-1} \Gamma_N^{(1)}$ and $|\uGP \rangle \langle \uGP |$). 

\medskip

The philosophy of conjugating the \index{Many-body Schr\"odinger Hamiltonian}{many-body Schr\"odinger Hamiltonian} with well-chosen unitary maps removing correlations is quite systematic and can be carried much further than hinted at above~\cite{AdhBreSch-20,BocBreCenSch-17,BocBreCenSch-18,BocBreCenSch-18b,BocBreCenSch-18c}. A few basic intuitions can be briefly summarized. I will follow mostly~\cite{Hainzl-20}, where a slightly different model (\index{Grand-canonical}{grand-canonical} setting) is considered. The difference will not be very apparent in this summarized discussion, but for the fact that we do not conjugate with the excitation map from Definition~\ref{def:excit map} as above. Let thus ($\mathrm{h}.\mathrm{c}.$ stands for ``hermitian conjugate''.)
\begin{equation}\label{eq:Bog corr}
\cB := \exp\left( N ^{-1} \sum_{\bp \neq 0} \eta (\bp) \left(a^{\dagger} _\bp a^{\dagger} _{-\bp} a_0 a_0 - \mathrm{h.c.}\right)\right). 
\end{equation}
We consider the action of $e^{\cB}$ on $H_N$ as defined in~\eqref{eq:hamil hom 2} (and extended to the full Fock space $\gF (\gH)$). We will repeatedly use Duhamel-like formulas 
\begin{align}\label{eq:Duhamel}
e^{-\cB} \cA  e^{-\cB} &= \cA + \int_{0}^1 e^{-s\cB} [\cA,\cB] e^{s\cB} \mathrm{d} s \nonumber \\
&= \cA + [\cA,\cB] + \int_{0}^1 \int_{0} ^t e^{-t\cB} \left[[\cA,\cB],\cB\right] e^{t\cB} \mathrm{d} t \mathrm{d} s\nonumber \\
&= \cA + [\cA,\cB] + \frac{1}{2} \left[[\cA,\cB],\cB\right] + \ldots
\end{align}
We then write $H_N$ as 
\begin{equation}\label{eq:split hom hamil}
H_N = \bH_0 + \bH_1 + \bH_2 + Q_2 + Q_3 + Q_4 
\end{equation}
where (schematically...)
\begin{align}\label{eq:split hom hamil 2}
\bH_0 &= \mbox{ terms containing only } a_0^\dagger, a_0 \nonumber \\ 
\bH_1 &= \sum_{\bp \neq 0} |\bp|^2 a^\dagger_{\bp} a_\bp \\ 
\bH_2 &= \mbox{ interaction terms containing exactly one } a_\bp^\dagger \mbox{ and one } a_\bp \mbox{ for } \bp \neq 0 \nonumber \\ 
Q_2 &= \mbox{ terms containing exactly two } a_\bp ^\dagger \mbox{ or two } a_\bp \nonumber \mbox{ for } \bp \neq 0 \nonumber\\ 
Q_3 &= \mbox{ terms containing exactly three } a_\bp^\dagger, a_\bp \mbox{ for } \bp \neq 0 \nonumber \\ 
Q_4 &= \mbox{ terms containing exactly four } a_\bp^\dagger, a_\bp \mbox{ for } \bp \neq 0.  
\end{align}
Recall that we expect $a_0 ^\dagger, a_0 \sim \sqrt{N}$. Hence, essentially, the larger the number of $a_\bp^\dagger, a_\bp, \bp \neq 0$ a term contains, the smaller its contribution. In the Gross-Pitaevskii scaling however this is not so straightforward to see, nor easy to vindicate. Consider the quadratic (in $a^\dagger_\bp,a_\bp$)  
$$ Q_2 = \frac{\lambda}{2N} \sum_{\bp \neq 0} \widehat{w} \left( \frac{\bp}{N}\right) \left( a^\dagger _\bp a^{\dagger} _{-\bp} a_0 a_0 + \mathrm{h}.\mathrm{c}.\right) $$
This term is our main ennemy because it mixes the momenta $\bp$ and $-\bp$ and is thus responsible for the occurence of \index{Bogoliubov theory}{Bogoliubov pair-excitations}, which play a crucial role in the low-energy physics of the Bose gas. Another quadratic term, $\bH_2$ is easier to control. Other manifestly coercive terms, namely $\bH_1$ and $Q_4$, will ultimately control it and other remainders in the following argument. They will also give the final desired bound on the number of excited particles. The cubic term $Q_3$ turns out to be negligible at the ordrer of precision we aim at (although it does play a role at the next order). 

The role of the conjugation with the \index{Bogoliubov transformation}{Bogoliubov transformation}~\eqref{eq:Bog corr} is thus to dispose of the $Q_2$ term. The choice of the coefficients $\eta (\bp)$ using the \index{Scattering solution}{scattering solution} implies that 
$$
|\bp| ^2 \eta (\bp) + \frac{\lambda}{2N} \sum_{\bq\neq 0} \widehat{w} \left(\frac{\bp-\bq}{N}\right) \eta (\bq) = -\frac{\lambda}{2} \widehat{w} \left(\frac{\bp}{N}\right).
$$
This is essentially the scattering equation~\eqref{eq:scat eq} in Fourier variables. It yields some cancellations so that 
\begin{equation}\label{eq:cancel Hainzl}
 \left[ \bH_1 + Q_4,\cB \right] = - Q_2 + \mbox{ remainder } 
\end{equation}
where, as in the rest of the discussion, we stay unprecise as to what ``remainder'' actually means. A basic idea is that, since $\cB$ is quadratic in excited creators/annihilators, so is $\left[ \bH_1,\cB \right]$ (after using the CCR). The other term $\left[ Q_4,\cB \right]$ is quartic, but contains non \index{Normal order, anti-normal order}{normal-ordered}\footnote{Normal-ordered = all creators on the left of annihilators} terms whose normal-ordering produces another quadratic term. The sum of these two quadratic terms kills $Q_2$. 

The next important point is as follows: first~\eqref{eq:cancel Hainzl} leads to 
\begin{multline*}
 \int_{0}^1 \int_{0} ^t e^{-t\cB} \left[[\bH_1 + Q_4 ,\cB],\cB \right] e^{t\cB} \mathrm{d} t \mathrm{d} s + \int_{0}^1 e^{-t\cB} \left[Q_2,\cB \right] e^{t\cB} \mathrm{d} t  \\ =  \int_{0}^1 \int_{s} ^1 e^{-t\cB} \left[Q_2,\cB \right] e^{t\cB} \mathrm{d} t + \mbox{ remainder } 
\end{multline*}
and then 
\begin{align*}
\int_{0}^1 \int_{s} ^1 e^{-t\cB} \left[Q_2,\cB \right] e^{t\cB} \mathrm{d} t \,
&\simeq  \frac{1}{2}[Q_2,\cB] + \mbox{ remainder } \nonumber\\
&\simeq \frac{\lambda}{2} \sum_{\bp \neq 0} \widehat{w} \left(\frac{\bp}{N}\right) \eta (\bp)+ \mbox{ remainder } \nonumber\\
&\simeq \frac{\lambda}{2} N^2 \int w_N (f_N - 1) + \mbox{ remainder }
\end{align*}
with $f_N$ the solution to the scattering equation associated with $w_N$. 

Thus, if we use the first formula in~\eqref{eq:Duhamel} to conjugate $Q_2$ and the second to conjugate $\bH_1 + Q_4$, we find 
$$ e^{-\cB} H_N e^{\cB} = \bH_0 + \frac{\lambda}{2} N^2 \int w_N (f_N -1) + \bH_1 + Q_4  + e^{-\cB} \left( \bH_2 + Q_3\right) e^{\cB} + \mbox{ remainder }.$$
The contribution from $\bH_2$ and $Q_3$ can be neglected, and $\bH_0$ is basically a mean-field energy (of the constant function). It yields a term $\frac{\lambda}{2} N^2 \int w_N $. All in all
$$ e^{-\cB} H_N e^{\cB} \geq \frac{\lambda}{2} N^2 \int w_N f_N  + \bH_1 + Q_4 + \mbox{ remainder }$$
which gives the desired bound, for 
$$ \int w_N f_N = 8 \pi a_{w_N} = 8 \pi N^{-1} a_w$$
as discussed around~\eqref{eq:GP def scat bis}.

\subsection{Completing the square}

The approach of~\cite{NamNapRicTri-20} is inspired by the works~\cite{BriFouSol-19,BriSol-19,FouSol-19} on the \index{Thermodynamic limit}{thermodynamic limit} of the Bose gas. Similar tools, combined with a systematic use of \index{Localization}{localization} techniques~\cite{Fournais-20}, allow to prove that \index{Bose-Einstein condensate (BEC)}{Bose-Einstein condensation} holds not only on the macroscopic length scale (as discussed in this review), but also a certain range of short length scales.  

We complete a square by a variant of the simple observation~\eqref{eq:complete square}:
$$ \left( \1_{\gH_2} - P\otimes P f_N (\bx-\by) \right) w_N (\bx-\by) \left( \1_{\gH_2} - f_N (\bx-\by) P\otimes P  \right) \geq 0$$
with $P = |\uGP \rangle \langle \uGP|$ the projector on the GP ground state and $f_N$ the zero-energy scattering solution. This gives (all multiplication operators are understood in the variable $\bx-\by$)
\begin{multline}\label{eq:GP square complete}
 w_N  \geq P\otimes P \left( 2 f_N -f_N^2 \right) w_N P \otimes P + \left( P \otimes P f_N w_N Q\otimes Q + h.c.\right) \\
 + \left( P\otimes P f_N w_N P\otimes Q + P\otimes P f_N w_N Q\otimes + \mathrm{h.c.} \right)
\end{multline}
where 
$$P = |\uGP \rangle \langle \uGP|, \quad Q = \1 - P$$
and $\mathrm{h.c.}$ stands for ``hermitian conjugate''.

The rationale for throwing away some terms from the lower bound is that 
\begin{itemize}
 \item particles in the condensate (in the range of $P$)  interact via the modified potential $f_N w_N$.
 \item interactions of triples and quadruples of excited particles (in the range of $Q$) do not contribute to the GP energy. Hence all terms with more than two $Q$'s can be neglected.
 \item interactions between pairs of excited particles and pairs of condensed particles are crucial to reconstruct the scattering length.
\end{itemize}
In short, we are trying to vindicate the guesses backing the construction of the trial state from Definition~\ref{def:bog trial 2}. The first line of~\eqref{eq:GP square complete} will combine with the kinetic energy of excited particles to reproduce the scattering process, while the second line will mostly cancel when combined with terms $$ P h Q + Q h P$$ from the kinetic energy, because of the variational equation satisfied by $\uGP$. The need to assume a small coupling constant (in this simplified presentation) arises when controlling these ``mostly cancellations''. 

The method comes in two steps. We first define an operator that will be used to take into account the contribution of excited particles: 
\begin{equation}\label{eq:NNRT Bog}
\HB := \sum_{n,m > 0} \left\langle u_n | H | u_m \right\rangle a^\dagger_n a_m + \frac{1}{2} \sum_{n,m > 0} \left\langle u_n | K | u_m \right\rangle \left( a^\dagger_n a^\dagger_m  + a_m a_n  \right).
\end{equation}
This is a Bogoliubov-type\footnote{This is NOT \emph{the} Bogoliubov Hamiltonian which gives the next-to leading order of the energy as proven in~\cite{BocBreCenSch-18}.} Hamiltonian, meaning it is quadratic in \index{Creation and annihilation operators}{annihilators/creators} and has a non-particle number-conserving $a^\dagger a^\dagger + a a$  part. We have denoted $(u_n)_n\geq 0$ an orthonormal basis of $L^2 (\R^d)$, with $u_0 = \uGP$ and associated creators/annihilators $a_n, a^\dagger_n$. Thus $\HB$ acts on $\gF (Q \gH)$, the Fock space of excited particles (cf Definition~\ref{def:excit map}). The one-body operators it is made of are
\begin{align}\label{eq:Bog data}
 H &= Q \left(-\Delta + V - \mu \right)Q \nonumber\\
 K &= Q\otimes Q \, \widetilde{K} \nonumber\\
 \widetilde{K} (\bx,\by)&= N \uGP (\bx) \uGP (\by) f_N (\bx-\by) w_N (\bx-\by)  
\end{align}
where in the last equation we identify integral kernel and operator, whereas in the second we mean operator composition. The quantity $\mu >0$ in the first line is a suitable chemical potential. For techincal reasons it departs from the true GP chemical potential/Lagrange multiplier $\EGP + 4\pi a \int_{\R^3} |\uGP| ^4 $ from~\eqref{eq:GP eq}.

\begin{lemma}[\textbf{Lower bound with a \index{Bogoliubov Hamiltonian}{Bogoliubov Hamiltonian}}]\label{lem:low Bog}\mbox{}\\
Assumptions as in Theorem~\ref{thm:GP small} (in particular there is a unique GP minimizer $\uGP$). There exists a $\lambda_0 >0$ and a $c >0$ such that, for all $0< \lambda \leq \lambda_0$ 
\begin{multline}\label{eq:low Bog}
H_N \geq N \int_{\R^3} \left( |\nabla \uGP| ^2  + V |\uGP| ^2 \right) \\
+ \frac{N^2}{2} \iint_{\R^3 \times \R ^3}  |\uGP (\bx)| ^2 w_N (\bx - \by) \left( 2 f_N (\bx-\by)  - f_N ^2 (\bx-\by) \right) |\uGP (\by)| ^2 \mathrm{d}\bx \mathrm{d}\by \\
+ \eB + c \cN^{\perp} - C
\end{multline}
as an operator on $\gH_N$, with $\cN^{\perp} = N - a^{\dagger} (\uGP) a (\uGP)$ the \index{Number operator}{number} (operator) of particles outside the condensate and $\eB$ the lowest eigenvalue of the Bogoliubov-type Hamiltonian~\eqref{eq:NNRT Bog} acting on $\gF (Q \gH)$.
\end{lemma}

\begin{proof}[Comments]
The interaction term on the first line comes from the first term on the right of~\eqref{eq:GP square complete}. The Bogoliubov-like energy $\eB$ includes contributions from the second term on the right of~\eqref{eq:GP square complete}, plus the kinetic energy of the excited particles. The rest has been disposed of, interestingly by making only a $O(1)$ error, and keeping handy the term $c \cN^{\perp}$ that permits to finely control the condensation rate.  
\end{proof}

There is still a piece of GP interaction energy to be extracted from the Bogoliubov-like Hamiltonian:

\begin{lemma}[\textbf{Ground state energy of a Bogoliubov Hamiltonian}]\label{lem:Bog GSE}\mbox{}\\
With the previous notation, and if $0< \lambda$ is small enough we have that $\HB$ is bounded below on $\gF (Q \gH)$. The infimum of its spectrum is equal to
\begin{equation}\label{eq:Bog GSE ex}
\eB = \frac{1}{2} \tr_{Q\gH}\left( E-H\right) 
\end{equation}
where $E= \left( (H-K)^{1/2}  \left(H+K\right) (H-K)^{1/2}\right)^{1/2}$. Moreover, 
\begin{equation}\label{eq:Bog GSE}
\eB \geq -\frac{1}{4} \tr_{Q\gH} \left( H^{-1} K^2 \right) - C \norm{K} \tr_{Q\gH} \left( H^{-2} K^2 \right).
\end{equation}
\end{lemma}

\begin{proof}[Comments]
The assumption that $\lambda$ be small enough ensures that $H \geq K$, so that the definition of $E$ makes sense. The whole point of Bogoliubov Hamiltonians is that they can be solved explicitly, see~\cite{GreSei-13} where in particular~\eqref{eq:Bog GSE ex} is derived (beware of the different notation). It is part of the statement that the trace actually makes sense. Further discussions of Bogoliubov Hamiltonians are in~\cite{BacBru-16,BruDer-07,DerGer-13,Derezinski-17,NamNapSol-16}. 

It is not always easy to extract information from the explicit solution of a Bogoliubov Hamiltonian. In particular, it is here important to get the first term on the right-hand side of~\eqref{eq:Bog GSE} exactly right (a weaker bound with $-1/2$ instead of $-1/4$ in front of the first term is easier to derive but insufficient here). 

To see how~\eqref{eq:Bog GSE} might follow from~\eqref{eq:Bog GSE ex}, let us pretend that $H$ and $K$ commute (the actual proof in~\cite{NamNapRicTri-20} is of course more involved than that). Then 
$$E = \sqrt{H^2 - K^2} = H \left(1 - \frac{K^2}{H^2}\right)^{1/2}.$$
Furthermore, we are dealing with short-range correlations, which means large kinetic energy. One should then think that $H\gg K$, and then 
$$ E \simeq H - \frac{1}{2} \frac{K^2}{H}$$
which, inserted in~\eqref{eq:Bog GSE ex}, gives a rationale for~\eqref{eq:Bog GSE}. Note also that the $H\gg K$ heuristics hints at the reason why the second term in~\eqref{eq:Bog GSE} ultimately turns out to be a remainder.   
\end{proof}

One can further estimate the main term to find 
\begin{equation}\label{eq:Nam comp}
\tr_{Q\gH} \left( H^{-1} K^2 \right) \simeq 2N^2 \iint_{\R^3 \times \R^3} |\uGP (\bx)| ^2  \left( w_N f_N (1-f_N) \right) (\bx-\by) |\uGP(\by)|^2. 
\end{equation}
Indeed, the main contribution is kinetic, so let us replace $H\rightsquigarrow (-\Delta)^{-1}$. Let us also ignore the $Q$ operators for simplicity. We then find 
$$  \tr_{Q\gH} \left( H^{-1} K^2 \right) \simeq N^2 \tr\left( \uGP (\bx) \widehat{f_N w_N} (\bp)\uGP (\bx) |\bp| ^{-2} \uGP (\bx) \widehat{f_N w_N} (\bp) \uGP (\bx)\right)$$
where functions of $\bx,\bp$ are understood as multiplication operators in position/momentum space. One can then figure out that commuting $\uGP(\bx)$ and $|\bp|^{-2}$ yields a negligible remainder. Using the scattering equation~\eqref{eq:scat eq} to express $|\bp| ^{-2}\widehat{f_N w_N} (\bp)$ (and using the cyclicity of the trace) then leads to
$$  \tr_{Q\gH} \left( H^{-1} K^2 \right) \simeq N^2 \tr\left( \uGP (\bx) ^2 \widehat{f_N w_N} (\bp)\uGP  (\bx) ^2 \widehat{1-f_N } (\bp)\right).$$
The operator in the trace has an explicit integral kernel $O(\bx;\by)$, the trace~\cite[Section~VI.6]{ReeSim1} is the integral of $O(\bx;\bx)$, which leads to~\eqref{eq:Nam comp}. 

Then, combining~\eqref{eq:low Bog} and~\eqref{eq:Bog GSE} leads to 
\begin{multline*}
 H_N \geq  N \int_{\R^3} \left( |\nabla \uGP| ^2  + V |\uGP| ^2 \right) \\ + \frac{N^2}{2} \iint_{\R^3 \times \R ^3}  |\uGP (\bx)| ^2 w_N (\bx - \by) f_N (\bx-\by)  |\uGP (\by)| ^2 \mathrm{d}\bx \mathrm{d}\by 
 - C 
\end{multline*}
and there only remains to recall that, as per the discussion after Theorem~\ref{thm:scat leng} 
$$ \int w_N f_N = 8\pi a N^{-1} $$
and thus 
$$ N w_N f_N \wto 8\pi a \delta_0$$
as measures. This completes the proof of the GP energy lower bound, under a smallness assumption on $\lambda$. As in the previous subsection one can finely estimate the rate of convergence of the one-body density matrix to the (projector on the) GP minimizer, because~\eqref{eq:low Bog} has a term controlling the \index{Number operator}{number of excited particles} on the right-hand side. 

\begin{theindex}

  \item Berezin-Lieb inequalities, \hyperpage{76}, \hyperpage{84}
  \item Bogoliubov Hamiltonian, \hyperpage{34}, \hyperpage{89}, 
		\hyperpage{109}, \hyperpage{145}
  \item Bogoliubov theory, \hyperpage{32, 33}, \hyperpage{105}, 
		\hyperpage{109}, \hyperpage{111}, \hyperpage{138}, 
		\hyperpage{142}
  \item Bogoliubov transformation, \hyperpage{111, 112}, 
		\hyperpage{139}, \hyperpage{142}
  \item Boltzmann-Gibbs ensemble, \hyperpage{51, 52}, \hyperpage{54}, 
		\hyperpage{112}, \hyperpage{133}
  \item Born approximation, \hyperpage{101}, \hyperpage{105}
  \item Bose-Einstein condensate (BEC), \hyperpage{9}, \hyperpage{20}, 
		\hyperpage{29, 30}, \hyperpage{32, 33}, 
		\hyperpage{109, 110}, \hyperpage{144}

  \indexspace

  \item Canonical commutation relations (CCR), \hyperpage{18}, 
		\hyperpage{65}, \hyperpage{111}, \hyperpage{113}, 
		\hyperpage{140}
  \item Cell method, \hyperpage{127}, \hyperpage{130}
  \item Classical de Finetti theorem, \hyperpage{50}, \hyperpage{52}, 
		\hyperpage{57}
  \item Classical number substitution, \hyperpage{88}
  \item Classical statistical mechanics, \hyperpage{35}, \hyperpage{51}, 
		\hyperpage{54}, \hyperpage{57}
  \item Coherent states, \hyperpage{33}, \hyperpage{65}, 
		\hyperpage{73, 74}, \hyperpage{78}, \hyperpage{81}, 
		\hyperpage{85}, \hyperpage{91}, \hyperpage{110}, 
		\hyperpage{131}
  \item Cold atoms, \hyperpage{21}, \hyperpage{85}, \hyperpage{101}
  \item Concentration-compactness, \hyperpage{24}
  \item Correlation map, \hyperpage{139}
  \item Creation and annihilation operators, \hyperpage{16}, 
		\hyperpage{18}, \hyperpage{48}, \hyperpage{65}, 
		\hyperpage{82}, \hyperpage{90}, \hyperpage{111}, 
		\hyperpage{113}, \hyperpage{115}, \hyperpage{139}, 
		\hyperpage{145}
  \item Cwikel-Lieb-Rosenblum bound, \hyperpage{69}

  \indexspace

  \item Dilute limit, \hyperpage{22}, \hyperpage{85}, \hyperpage{94}, 
		\hyperpage{97}, \hyperpage{119}, \hyperpage{124}, 
		\hyperpage{136}
  \item Dyson Hamiltonian, \hyperpage{132}, \hyperpage{135}
  \item Dyson lemma, \hyperpage{101}, \hyperpage{119, 120}, 
		\hyperpage{131, 132}, \hyperpage{138}

  \indexspace

  \item Empirical measure, \hyperpage{53}
  \item Entropy, \hyperpage{50, 51}, \hyperpage{54}
  \item Excitation map, \hyperpage{115}, \hyperpage{139}

  \indexspace

  \item Feynman-Hellmann principle, \hyperpage{42}, \hyperpage{47}, 
		\hyperpage{81}, \hyperpage{91}, \hyperpage{135}
  \item Feynman-Kac formula, \hyperpage{36, 37}, \hyperpage{133}
  \item Fisher information, \hyperpage{50, 51}, \hyperpage{53}
  \item Fock space, \hyperpage{17}, \hyperpage{60}, \hyperpage{66}, 
		\hyperpage{70}, \hyperpage{78}, \hyperpage{86}, 
		\hyperpage{109}, \hyperpage{115}, \hyperpage{139}

  \indexspace

  \item Grand-canonical, \hyperpage{16, 17}, \hyperpage{110}, 
		\hyperpage{141}
  \item Gross-Pitaevskii energy functional, \hyperpage{25}, 
		\hyperpage{131}, \hyperpage{138}
  \item Gross-Pitaevskii limit, \hyperpage{22, 23}, \hyperpage{101}, 
		\hyperpage{104}, \hyperpage{119}, \hyperpage{126}, 
		\hyperpage{129}, \hyperpage{132}
  \item Ground state, \hyperpage{20}

  \indexspace

  \item Hartree energy functional, \hyperpage{23, 24}, \hyperpage{36}, 
		\hyperpage{43}, \hyperpage{50}, \hyperpage{57}, 
		\hyperpage{62}
  \item Heat kernel, \hyperpage{36}, \hyperpage{133}
  \item Heisenberg uncertainty principle, \hyperpage{11}
  \item Hewitt-Savage theorem, \hyperpage{50}, \hyperpage{52}, 
		\hyperpage{57}
  \item Hoffmann-Ostenhof$\,^2$ inequality, \hyperpage{38}, 
		\hyperpage{44}, \hyperpage{46}, \hyperpage{48}
  \item Homogeneous Bose gas, \hyperpage{102}, \hyperpage{124}, 
		\hyperpage{128}, \hyperpage{139}

  \indexspace

  \item Indistinguishability, \hyperpage{13--15}

  \indexspace

  \item L\'evy-Leblond's trick, \hyperpage{40}, \hyperpage{43}, 
		\hyperpage{46}, \hyperpage{48}
  \item Lieb-Oxford inequality, \hyperpage{43}
  \item Lieb-Thirring inequality, \hyperpage{69}, \hyperpage{80, 81}
  \item Local density approximation (LDA), \hyperpage{102}, 
		\hyperpage{125}, \hyperpage{128}, \hyperpage{131}
  \item Localization, \hyperpage{60}, \hyperpage{63}, \hyperpage{68}, 
		\hyperpage{72}, \hyperpage{85}, \hyperpage{88}, 
		\hyperpage{95}, \hyperpage{144}

  \indexspace

  \item Many-body Schr\"odinger equation, \hyperpage{14}, 
		\hyperpage{32}, \hyperpage{85}, \hyperpage{91}, 
		\hyperpage{93}
  \item Many-body Schr\"odinger Hamiltonian, \hyperpage{101}, 
		\hyperpage{141}
  \item Many-body Schr\"odinger operator, \hyperpage{7}, \hyperpage{19}
  \item Marginal density, \hyperpage{15}, \hyperpage{50}, 
		\hyperpage{52}, \hyperpage{130}
  \item Mean-field approximation, \hyperpage{19}
  \item Mean-field energy functional, \hyperpage{28}
  \item Mean-field limit, \hyperpage{21}, \hyperpage{77}, 
		\hyperpage{119}
  \item Moments estimates, \hyperpage{85}, \hyperpage{91}, 
		\hyperpage{93, 94}, \hyperpage{97}, \hyperpage{135}

  \indexspace

  \item Non-linear Schr\"odinger energy functional, \hyperpage{24}, 
		\hyperpage{36}, \hyperpage{63}, \hyperpage{86}
  \item Normal order, anti-normal order, \hyperpage{83}, \hyperpage{88}, 
		\hyperpage{113}, \hyperpage{143}
  \item Number operator, \hyperpage{17}, \hyperpage{78}, \hyperpage{87}, 
		\hyperpage{115}, \hyperpage{140}, \hyperpage{145}, 
		\hyperpage{147}

  \indexspace

  \item Onsager's inequality, \hyperpage{40}, \hyperpage{44}, 
		\hyperpage{47, 48}

  \indexspace

  \item Pairing matrix, \hyperpage{116, 117}
  \item Positivity improving, \hyperpage{35}, \hyperpage{37}, 
		\hyperpage{47}
  \item Pure state, mixed state, \hyperpage{12, 13}, \hyperpage{45}, 
		\hyperpage{58}, \hyperpage{115, 116}

  \indexspace

  \item Quantum de Finetti theorem, \hyperpage{48}, \hyperpage{55}, 
		\hyperpage{58}, \hyperpage{63, 64}, \hyperpage{66}, 
		\hyperpage{71}, \hyperpage{81}, \hyperpage{91}, 
		\hyperpage{94}, \hyperpage{131}, \hyperpage{136}
  \item Quantum statistics, bosons and fermions, \hyperpage{8}, 
		\hyperpage{15}
  \item Quasi-free state, \hyperpage{111}, \hyperpage{116}, 
		\hyperpage{118}

  \indexspace

  \item Reduced density matrix, \hyperpage{15}, \hyperpage{17}, 
		\hyperpage{29}, \hyperpage{48}, \hyperpage{57}, 
		\hyperpage{59}, \hyperpage{69}, \hyperpage{81}, 
		\hyperpage{110}, \hyperpage{129}, \hyperpage{137}

  \indexspace

  \item Scaling limits, \hyperpage{19}
  \item Scattering length, \hyperpage{23}, \hyperpage{101}, 
		\hyperpage{103}, \hyperpage{119}, \hyperpage{124}, 
		\hyperpage{128}, \hyperpage{132}, \hyperpage{138}
  \item Scattering solution, \hyperpage{23}, \hyperpage{103}, 
		\hyperpage{114}, \hyperpage{140}, \hyperpage{142}
  \item Schur's lemma, \hyperpage{65}, \hyperpage{75}, \hyperpage{87}
  \item Second quantization, \hyperpage{16}, \hyperpage{89}, 
		\hyperpage{109}, \hyperpage{115}, \hyperpage{139}
  \item Semiclassical analysis, \hyperpage{64}, \hyperpage{69}
  \item States, observables, \hyperpage{10}, \hyperpage{12}
  \item Symbol, upper and lower, \hyperpage{65}, \hyperpage{75, 76}, 
		\hyperpage{79}, \hyperpage{87}

  \indexspace

  \item Thermodynamic limit, \hyperpage{20}, \hyperpage{102}, 
		\hyperpage{124}, \hyperpage{126}, \hyperpage{144}
  \item Trapping potential, \hyperpage{24}, \hyperpage{63}, 
		\hyperpage{71}
  \item Trial state, Bogoliubov, \hyperpage{109}, \hyperpage{112}, 
		\hyperpage{117}
  \item Trial state, Dyson, \hyperpage{106}, \hyperpage{125}
  \item Trial state, Hartree, \hyperpage{20}, \hyperpage{28}, 
		\hyperpage{43}, \hyperpage{71}, \hyperpage{97}
  \item Trial state, Jastrow, \hyperpage{106, 107}, \hyperpage{118}
  \item Trotter product formula, \hyperpage{36, 37}, \hyperpage{133}

  \indexspace

  \item Vacuum vector, \hyperpage{17}, \hyperpage{110}, \hyperpage{113}

  \indexspace

  \item Weyl operator, \hyperpage{110}, \hyperpage{115}, 
		\hyperpage{139}
  \item Wick and anti-Wick quantization, \hyperpage{76}, \hyperpage{83}
  \item Wick's theorem, \hyperpage{112}, \hyperpage{116, 117}

\end{theindex}

%
%
%

\end{document}